\documentclass[10pt,reqno]{amsart}
\usepackage[utf8]{inputenc}
\usepackage[margin=1in]{geometry}
\usepackage{amsfonts,amssymb,amsmath,amsthm,tikz,comment,mathtools,setspace,stmaryrd,enumitem,bbm,mathrsfs}
\usepackage[export]{adjustbox}
\usepackage[hidelinks]{hyperref}
\usepackage{scalerel,stmaryrd} 
\numberwithin{equation}{section}
\usepackage[normalem]{ulem}

\allowdisplaybreaks

\newcommand{\N}{\mathbb{N}}
\newcommand{\R}{\mathbb{R}}
\newcommand{\Q}{\mathbb{Q}}
\newcommand{\Z}{\mathbb{Z}}

\newcommand{\EE}{\mathbf E}

\newcommand{\F}{\mathcal{F}}

\newcommand{\B}{\mathcal B}

\newcommand{\D}{\mathcal{D}}

\newcommand{\Cpin}{C_{\mathrm{pin}}}
\newcommand{\Ll}{\mathcal L}

\newcommand{\wt}{\widetilde}

\newcommand{\T}{\mathbb{T}}
\newcommand{\deq}{\overset{d}{=}}

\newcommand{\ve}{\varepsilon}

\newcommand{\f}{\frac}
\newcommand{\df}{\dfrac}
\newcommand{\mbf}{\mathbf}

\newcommand{\Pp}{\mathbb P}
\newcommand{\Ee}{\mathbb E}

\newcommand{\Var}{\mathrm{Var}}

\newcommand{\Ff}{\mathcal F}

\newcommand{\bfx}{\mbf{x}}

\newcommand{\ls}{[}
\newcommand{\rs}{]}

\newcommand{\lzb}{\llbracket}   
\newcommand{\rzb}{\rrbracket}   

\newcommand{\Pm}{\mathcal P}
\newcommand{\ICH}{\mathfrak g}

\newcommand{\la}{\langle}
\newcommand{\ra}{\rangle}
\newcommand{\cL}{\mathcal{L}}
\newcommand{\diff}{\nabla}
\newcommand{\Lapl}{\Delta}

\newcommand{\OCY}{Z^{N}}
\newcommand{\OCYp}{Z^{N, \text{per}}}
\newcommand{\scOCYp}{\wt{Z}^{N,\text{per}}_\beta}

\newcommand{\ind}{\mathbf 1}

\newcommand{\vecsum}{\mathfrak s}

\newcommand{\bT}{\mathbb{T}}

\newcommand{\cDm}{\Phi}
\newcommand{\FcDm}{\Psi}

\newcommand{\be}{\begin{equation}}
\newcommand{\ee}{\end{equation}}

\newtheorem{theorem}{Theorem}[section]
\newtheorem{proposition}[theorem]{Proposition}
\newtheorem{corollary}[theorem]{Corollary}
\newtheorem{lemma}[theorem]{Lemma}

\theoremstyle{definition}
\newtheorem{definition}[theorem]{Definition}

\newtheorem{conj}[theorem]{Conjecture}

\theoremstyle{remark}
\newtheorem{remark}[theorem]{Remark}

\RequirePackage{color}
\definecolor{darkblue}{rgb}{0.0,0.0,0.7}
\definecolor{darkred}{rgb}{0.5,0.0,0.0}
\definecolor{darkgreen}{rgb}{0.0,0.5,0.0}
\definecolor{indigo}{rgb}{0.3,0,0.5}

\newcommand{\pathsp}{\mathbb{X}}

\def\tspb{\hspace{0.9pt}}

\def\viiva{\hspace{.7pt}\vert\hspace{.7pt}}

\title[Periodic Pitman transforms and jointly invariant measures]{Periodic Pitman transforms and jointly invariant measures}
\author{Ivan Corwin}
\address{Ivan Corwin,  Columbia University, Mathematics Department,
New York, NY 10027, USA.}
\email{ivan.corwin@gmail.com}

\author{Yu Gu}
\address{Yu Gu, University of Maryland, Mathematics Department,  College Park, MD 20742,
USA.}
\email{yugull05@gmail.com}

\author{Evan Sorensen}
\address{Evan Sorensen,  Columbia University, Mathematics Department,
New York, NY 10027, USA.}
\email{evan.sorensen@columbia.edu}

\begin{document}

\begin{abstract}
We construct explicit jointly invariant measures for the periodic KPZ equation (and therefore also the stochastic Burgers' and stochastic heat equations) for general slope parameters and prove their uniqueness via a one force--one solution principle. The measures are given by polymer-like transforms of independent Brownian bridges. We describe several properties and limits of these measures, including an extension to a continuous process in the slope parameter that we term the periodic KPZ horizon. As an application of our construction, we prove a Gaussian process limit theorem with an explicit covariance function for the long-time height function fluctuations of the periodic KPZ equation when started from varying slopes. In connection with this, we conjecture a formula for the fluctuations of cumulants of the endpoint distribution for the periodic continuum directed random polymer.

To prove joint invariance, we address the analogous problem for a semi-discrete system of SDEs related to the periodic O'Connell-Yor polymer model and then perform a scaling limit of the model and jointly invariant measures. For the semi-discrete system, we demonstrate a bijection that maps our systems of SDEs to another system with product invariant measure. Inverting the map on this product measure yields our invariant measures. This map relates to a periodic version of the discrete geometric Pitman transform that we introduce and probe. As a by-product of this, we show that the jointly invariant measures for a periodic version of the inverse-gamma polymer are the same as those for the O'Connell-Yor polymer.
\end{abstract}

\maketitle

\setcounter{tocdepth}{1}
\tableofcontents

\section{Introduction and main result}
\subsection{Periodic equations}
The Kardar-Parisi-Zhang (KPZ) equation \cite{KPZ} was introduced in 1986 as a model for stochastic interface growth. Its spatial derivative, the stochastic Burgers' equation, was already studied earlier in 1977 by Forster, Nelson and Stephen \cite{FNS77} as a model for randomly stirred fluid (and since then, it is understood as a continuum model and limit of interacting particle systems, see, e.g., \cite{BertiniGiacomin, Goncalves2014,Gubinelli2017}). The exponential (i.e., Cole-Hopf transform) of the KPZ equation solves the multiplicative noise stochastic heat equation which, via the Feynman-Kac formula (see, e.g. \cite{Bertini1995}) is the partition function for a continuum directed random polymer \cite{Alberts-Khanin-Quastel-2014a}.

We consider these equations in the presence of spatial periodic noise. The spatial integral of the stochastic Burgers' equation (i.e., continuum analogue of the total number of particles in an interacting particle system) is conserved over time and for each value it is known that there exists a unique invariant measure. The stochastic Burgers' equation naturally couples together the evolution of all initial data and hence yields a coupling of the invariant measures with different values of the conserved spatial integral which, itself, is invariant under the evolution of the stochastic Burgers' equation.

In this paper we explicitly describe this jointly invariant measure in terms of elementary probabilistic objects such as Brownian bridges and their exponential functionals (periodic versions of the geometric Pitman transform  we introduce herein). The stochastic Burgers' equation is the most natural context in which to consider such periodic jointly invariant measures. However, they can also be defined in the KPZ and stochastic heat equation context. The advantage of those equations is that the solutions are function-valued, as opposed to the distribution-valued solutions to the stochastic Burgers' equation. For this reason, we will state our results in terms of the KPZ equation.

The {\em periodic stochastic Burgers' equation} is formally written as
\be \label{eq:SBE}\tag{SBE$_\beta$}
\partial_t u_\beta(t,x) = \f{1}{2} \partial_{xx} u_\beta(t,x) + \f{1}{2} \partial_x (u_\beta^2)(t,x) + \beta \partial_x \xi_\T(t,x),\qquad t \ge 0, x \in \T,
\ee
where $\beta>0$ is an inverse-temperature parameter, $\T:=\R/\Z$ is the one-dimensional unit-length   torus, and $\xi_\T(t,x)$ is a space-time (really, in our notation, time-space) white noise on $\R_{>0} \times \T$. The proper definition of solution for this equation is via the Cole-Hopf transform, as follows. 

Let $Z_\beta$ solve the well-posed {\em periodic stochastic heat equation} (see Section \ref{sec.convergeinitial})
\be\label{eq:SHE}\tag{SHE$_\beta$}
\partial_t Z_\beta(t,x) = \f{1}{2} \partial_{xx}Z_\beta(t,x) + \beta Z_\beta(t,x)\xi(t,x),\qquad t\ge 0, x\in \R
\ee
where here and below $\xi$ denotes the periodic extension of the white noise $\xi_\T$ to all of $\R$, i.e., so that (in a weak sense) $\xi(t,x+1)=\xi(t,x)$ for all $x\in \R$. See Section \ref{sec:white_noise} for a precise definition. The initial data is determined by a random function $F\in C[0,1]$ as follows. First we define the multiplicative periodic continuous extension of $F$ to $F\in C(\R)$ by the relation
$\frac{F(x+1)}{F(x)}=\frac{F(1)}{F(0)}=:e^\theta$ for all $x\in \R,
$
where we call $\theta\in \R$ defined above the {\em slope parameter} for the initial data. For all $x\in \R$ we set $Z(0,x)=F(x)$, thus defining the initial data at time zero.
When started from such initial data $F$, with probability one the function $(t,x) \mapsto Z_\beta(t,x)$ is continuous and strictly positive for all $t > 0$ and $x\in \R$ (Proposition \ref{prop:solve_SHE}\ref{itm:SHE_soln}) and satisfies $\frac{Z_\beta(t,x+1)}{Z_\beta(t,x)}=e^\theta$ for all $t\ge 0$ and all $x\in \R$ (Proposition \ref{prop:solve_SHE}\ref{itm:per_pres}).

\begin{figure}
    \centering
    \includegraphics[height =3in]{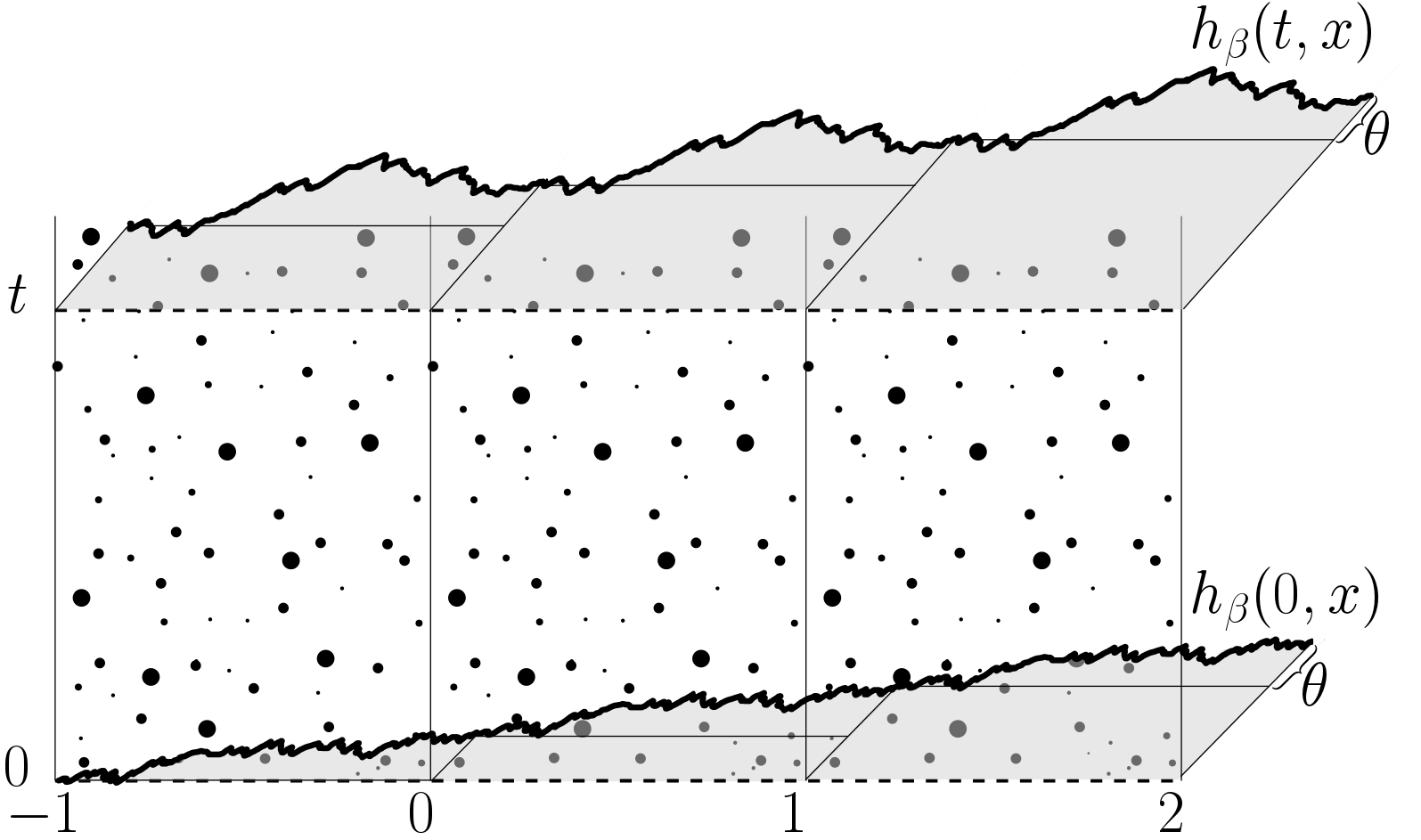}
    \caption{The evolution of the periodic KPZ equation (time and space in the vertical and horizontal directions, respectively). The periodic space-time white noise is depicted through random disks of varying radii and repeat in each horizontal width one strip. The initial data $h_{\beta}(0,x)$ is drawn out of the paper (sheared to illustrate this dimension) and repeats with the same periodicity up to a vertical shift by the slope $\theta$ (i.e., $h_{\beta}(0,x+1)-h_{\beta}(0,x)=\theta$ for all $x\in \R$). This evolves using the noise up to time $t$ to yield $h_{\beta}(t,x)$ likewise drawn. The slope $\theta$ is preserved by the evolution, as well as the periodicity.}
    \label{fig:inc_per_fun}
\end{figure}
From the stochastic heat equation solution $Z_\beta$, we define
$$
h_\beta(t,x):= \log Z_\beta(t,x)\qquad\textrm{and}\qquad u_\beta(t,x):=\partial_x h_\beta(t,x)
$$
(where taking the logarithm is justified by the almost-sure strict positivity, and the derivative in $x$ is distribution-valued) and call $h$ the Cole-Hopf solution to the {\em periodic KPZ equation}
\be \label{eq:KPZ}\tag{KPZ$_\beta$}
\partial_t h_\beta(t,x) = \f{1}{2} \partial_{xx} h_\beta(t,x) + \f{1}{2}\bigl(\partial_x h_\beta(t,x)\bigr)^2 +  \beta \xi(t,x),\qquad t\ge 0, x\in \R
\ee
with initial data $h(0,x)=\log Z(0,x)$
and $u_\beta$ the Cole-Hopf solution to the periodic stochastic Burgers' equation \eqref{eq:SBE} with initial data $u(0,x)=\partial_x h(0,x)$. As for the stochastic heat equation, it suffices to define the initial data for the KPZ and stochastic Burgers' equations when restricted to the domain $[0,1]$ and then to extend to $\R$ (in the KPZ case, by an additive periodic continuous extension, and in the stochastic Burgers' case, just by periodic extension).
Recalling the  slope parameter $\theta$ from above, it follows that
$$
h_\beta(t,x+1)-h_\beta(t,x)=\int_x^{x+1} u_\beta(t,x')dx'=\theta\quad \textrm{for all }t\ge 0, x\in \R.
$$ 
When $\theta=0$, the KPZ and stochastic heat equations can be defined restricted to $\T$ without the use of periodized space-time white noise. The reason for the periodization when $\theta\neq 0$ is the non-trivial slope around the torus; a similar periodization is necessary in the study of TASEP on a ring (see, e.g., \cite{BaikLiu}). 
See Figure \ref{fig:inc_per_fun} for an illustration of the periodic KPZ equation evolution described above.

The Cole-Hopf solution is the physically relevant notion of solution as it arises from various discrete or mollified equations (see, e.g., \cite{Bertini1995,HairerKPZ}), interacting particle systems and growth models  (see, e.g., \cite{BertiniGiacomin, Goncalves2014,Gubinelli2017}), and directed polymer models (see, e.g., \cite{Alberts-Khanin-Quastel-2014b}). We will focus below on the Cole-Hopf solution to the periodic KPZ equation and not the stochastic Burgers' equation, since the former is function valued while the later involves generalized functions and hence becomes more involved.

\subsection{Invariant measures}
We will be concerned with invariant measures for the Markov processes defined by these closely related equations. We say that a random function $f\in C[0,1]$ with $f(0)=0$ is an invariant initial condition for the KPZ equation (and that its law is an invariant probability measure) if for all $t\ge 0$
\be \label{eq:invar_def}
\bigl(h_\beta(t,x) - h_\beta(t,0): x \in [0,1] \bigr) \deq \bigl(f(x):x \in [0,1]\bigr),
\ee
where $h_\beta$ is the solution to the periodic KPZ equation started with initial data $f$ at time $0$. In terms of the stochastic Burgers' and stochastic heat equations, this translates into the equality in law for all $t>0$ of
$$
\bigl(u_\beta(t,x):x \in [0,1]\bigr) \deq \bigl(\partial_x f(x):x \in [0,1]\bigr),\quad\textrm{and} \quad \biggl(\frac{Z_\beta(t,x)}{Z_\beta(t,0)}:x \in [0,1]\biggr) \deq \biggl(e^{f(x)}:x \in [0,1]\biggr).
$$
It is important to note that for the KPZ and stochastic heat equations, it is necessary to center by $h_\beta(t,0)$ in order to have an invariant probability measure; for the stochastic Burgers' equation no centering is needed.

For each $\theta\in \R$ there is a unique invariant initial condition/measure for \eqref{eq:KPZ} with slope parameter $\theta$ (see, e.g., \cite{Funaki-Quastel-2015,Hairer-Mattingly-2018}). It is the additive periodic continuous extension of a slope $\theta$, variance $\beta^2$ Brownian bridge $\B_{\beta,\theta}$. Precisely, for $\beta > 0$ and $\theta \in \R$, let $x \mapsto \B_{\beta,\theta}(x)$ denote Gaussian process on $[0,1]$ with 
\begin{equation}\label{eq.bb}
\Ee[\B_{\beta,\theta}(x)]=\theta x \quad\textrm{for } x\in [0,1],\quad \textrm{and} \quad \mathrm{Cov}\bigl[\B_{\beta,\theta}(x),\B_{\beta,\theta}(y)\bigr] = \beta^2\big(\min(x,y) - xy\big)\quad \textrm{for }x,y\in [0,1].
\end{equation}
In other words, $\B_{\beta,\theta}(x) \deq \beta \B(x) + \theta x$, where $\B$ is a standard Brownian bridge with $\B(0) = \B(1) = 0$. For any initial data for the periodic KPZ equation with slope parameter $\theta$, in the long-time limit, $h_\beta(t,x)-h_\beta(t,0)$ converges as a process in $x$ to $\B_{\beta,\theta}(x)$ (see, e.g., Theorem \ref{thm:1f1s} with $k=1$). 

\subsection{Jointly invariant measures}
This paper answers the following question: What happens in the long-time limit to the periodic KPZ equation started from multiple initial conditions with different slope parameters, yet coupled to the same driving noise? Marginally, each solution will converge to a Brownian bridge with the corresponding slope. Yet, these solutions are coupled by a common noise, and it follows readily from the fact that the KPZ equation preserves ordered height functions that the joint distribution of these long-time limits should be suitably ordered and hence not independent.

The main result of this paper, Theorem \ref{thm:KPZ_invar_main}, is the explicit construction of these jointly invariant initial conditions / measures for the periodic KPZ equation. To state this result as well as a result that proves that these jointly invariant measures arise in the long-time limit as just outlined (see Theorem \ref{thm:1f1s}), we introduce the notation $h_\beta(t,y\viiva s,f)$ for the solution at time $t$ and location $y\in \R$ to the periodic KPZ equation \eqref{eq:KPZ} when started with the additive  periodic continuous extension of the initial condition $f\in C[0,1]$ at time $s<t$. For different choices of $f$, the solutions $h_\beta(t,y\viiva s,f)$ are coupled on the same probability space on which the common noise $\xi$ that drives them is defined (and on which the initial conditions are defined if they are random, in which case they are assumed to be independent of the noise after time $s$). Equivalently, $h_\beta(t,y\viiva s,f)$ can be defined from the four-parameter (random) Green's function for the stochastic heat equation (see Section \ref{sec.l2converge} and Proposition \ref{prop:solve_SHE})
$
\{Z_\beta(t,y \viiva s,x): t >s, x,y \in \R,\beta > 0\}
$
so that, for $f\in C[0,1]$,
\[
h_\beta(t,y \viiva s,f) := \log \int_\R Z_\beta(t,y \viiva s,x) e^{f(x)}\,dx,
\]
where on the right-hand side we extend $f$ to lie in $C(\R)$ by setting $f(x+1) - f(x) = \theta:=f(1)-f(0)$ for all $x \in\R$.
It is shown in Proposition \ref{prop:solve_SHE}\ref{itm:SHE_soln} that $h_\beta(t,y \viiva s,f)$ for $t >s$ and $y \in \R$ solves \eqref{eq:KPZ} with initial condition $f$ at time $s$. Notationally, if $s$ is not specified, we take $s = 0$.

\begin{definition}\label{def.jtinv}
Random $(f_1,\ldots,f_k) \in (C[0,1])^k$ satisfying $f_m(0) = 0$ for $1 \le m \le k$ are called \textit{jointly invariant initial conditions} for the periodic KPZ equation \eqref{eq:KPZ} if, for all $t > 0$, as processes on $(C[0,1])^k$,
$$
\Bigl(h_\beta(t,y \viiva f_m) - h_\beta(t,0 \viiva f_m):m\in \{1,\ldots, k\}, y \in [0,1] \Bigr)
\deq \Bigl(f_m(y):m\in \{1,\ldots, k\}, y \in [0,1]\Bigr).
$$
In that case, we say that the law of $\bigl(f_m(y):m\in \{1,\ldots k\}, y \in [0,1]\bigr)$ is a \textit{jointly invariant measure}. 
\end{definition}
To provide our  explicit description of the jointly invariant initial conditions, we need some further notation.
\[
\textrm{For a function }f:[0,1]\to \R \text{ and } x,y \in [0,1],  \textrm{ define } f(x,y) := f(y) - f(x) + \ind\{x > y\}
\bigl(f(1)-f(0)\bigr),
\]
where $\ind\{E\}$ is the indicator function for $E$, i.e., $1$ when the event $E$ occurs and 0 otherwise.
One can think of $f(x,y)$ as the increment of the extended function $f:\R \to \R$, from $x$ to $y$ (modulo $1$), when one is only allowed to travel from $x$ to $y$ in the positive direction. Let $\Cpin[0,1]$ be the space of functions $f \in C[0,1]$ satisfying $f(0) = 0$, endowed with the subspace topology inherited from the sup norm on $C[0,1]$ ($\Cpin[0,1]$ is a closed subset of $C[0,1]$ with the sup norm and is therefore a Polish space under the subspace topology).
Define $\cDm^2:\Cpin[0,1] \times \Cpin[0,1] \to \Cpin[0,1]$ by its action
\be \label{eq:Phi_def}
\begin{aligned}
\cDm^2(f_1,f_2)(x) &:= f_2(x) + Q^2(f_1,f_2)(x) - Q^2(f_1,f_2)(0),
\quad \textrm{with}\quad\\
Q^2(f_1,f_2)(x) &:= \log \int_0^1 e^{f_2(x,y) - f_1(x,y)}\,dy.
\end{aligned}
\ee
We iterate the map $\cDm^2$ to obtain a family of maps $\cDm^m: (\Cpin[0,1])^m \to \Cpin[0,1]$ by
\be \label{Phi_iter}
\cDm^1(f_1) := f_1,\quad \text{and for }m  > 1,\quad
\cDm^m(f_1,\ldots,f_k) := \cDm^2\bigl(f_1,\cDm^{m-1}(f_2,\ldots,f_k)\bigr).
\ee
It follows (see Lemma \ref{lem:Phi_pres_slope}) that $\cDm^m(f_1,\ldots,f_m)(0) = 0$ and $\cDm^m(f_1,\ldots,f_m)(1) = f_m(1)$.
By iterating \\(see, e.g., the proof of Lemma \ref{lem:Dnm_alt}), we obtain the  following polymer type description of the output of $\cDm^m$:
\begin{equation}
\begin{aligned}\label{eq:alternate_rep}
\cDm^m(f_1,\ldots,f_m)(x) &= f_m(x) + Q^m(f_1,\ldots,f_m)(x) - Q^m(f_1,\ldots,f_m)(0),\qquad \text{where}\\
Q^m(f_1,\ldots,f_m)(x) &= \log \int_{[0,1]^{m-1}, x_0 = x} \prod_{r = 1}^{m-1} \,dx_r\, e^{f_m(x_{r-1},x_r ) - f_r( x_{r-1}, x_r)}.
\end{aligned}
\end{equation}
Lastly, define the map $\FcDm^k: (\Cpin[0,1])^k \to (\Cpin[0,1])^k$ by its action
\be \label{Psi_map}
\FcDm^k(f_1,\ldots,f_k) := \bigl(\cDm^1(f_1),\cDm^2(f_1,f_2),\ldots, \cDm^k(f_1,\ldots,f_k)\bigr).
\ee

\begin{definition}
For any $\beta > 0$, $k \in \N$, and $(\theta_1,\ldots,\theta_k) \in \R^k$, let $f_m \deq \B_{\beta,\theta_m}$, $1\leq m\leq k$, be independent slope $\theta_m$, variance $\beta^2$ Brownian bridges, see \eqref{eq.bb}. Let $\Pm_\beta^{(\theta_1,\ldots,\theta_k)}$ denote the law on  $C([0,1])^k$ of 
\[
(g_1,\ldots,g_k) := \FcDm^k(f_1,\ldots,f_k).
\]
\end{definition}

We now state our main theorem, which is proved in Section \ref{sec:proof234}.
\begin{theorem}[Jointly invariant measure for periodic KPZ] \label{thm:KPZ_invar_main}
 For all $\beta > 0$, $k\in \N$ and $(\theta_1,\ldots,\theta_k) \in\R^k$, $\Pm_\beta^{(\theta_1,\ldots,\theta_k)}$ is a jointly invariant measure for \eqref{eq:KPZ}.
\end{theorem}

Here and below we adapt the notation $X\sim \mu$ to denote that a random variable (or set of random functions) $X$ is distributed according to a distribution $\mu$. For example, in Theorem \ref{thm:KPZ_invar_main} $(g_1,\ldots,g_k) \sim \Pm_\beta^{(\theta_1,\ldots,\theta_k)}$.

Figure \ref{fig:sim1} shows a simulation using R \cite{R} of independent sloped Brownian bridges, and the output of the map $\Psi^k$, thus producing a simulation of the measure $\Pm_\beta^{(\theta_1,\ldots,\theta_k)}$. Each function in the picture on the right is marginally a sloped Brownian bridge, and the joint functions satisfy a monotonicity of increments. This manifests itself in Proposition \ref{prop:g_prop_intro}\ref{itm:g_consis}-\ref{itm:g_mont} below. 

\begin{figure}
    \centering
    \includegraphics[width = 0.4\linewidth,valign  = c]{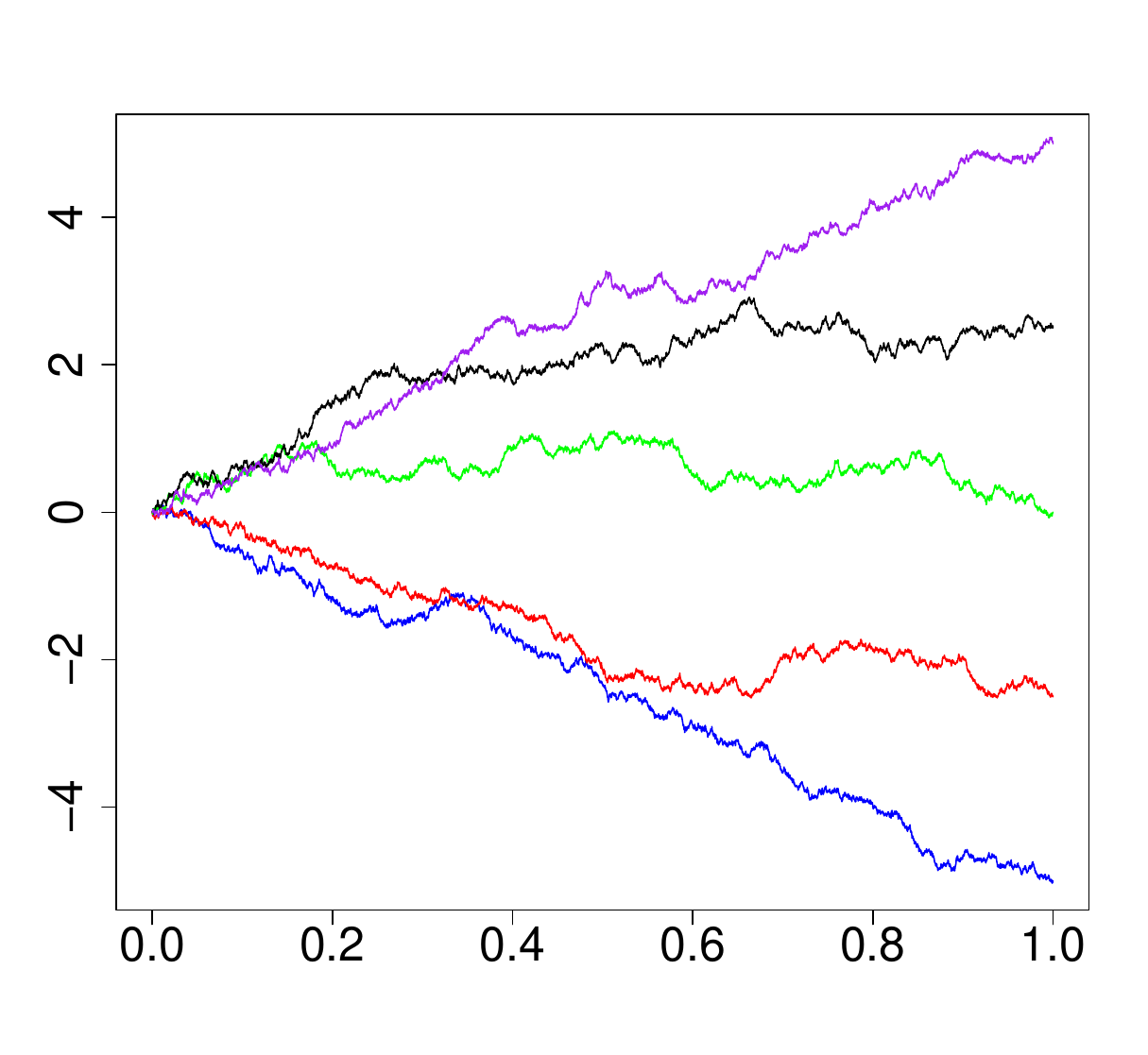}\;
    \scalebox{2}{$\overset{\Psi^5}{\relbar\joinrel\relbar\joinrel\longrightarrow}$} \;
    \includegraphics[width = 0.4\linewidth,valign = c]{
    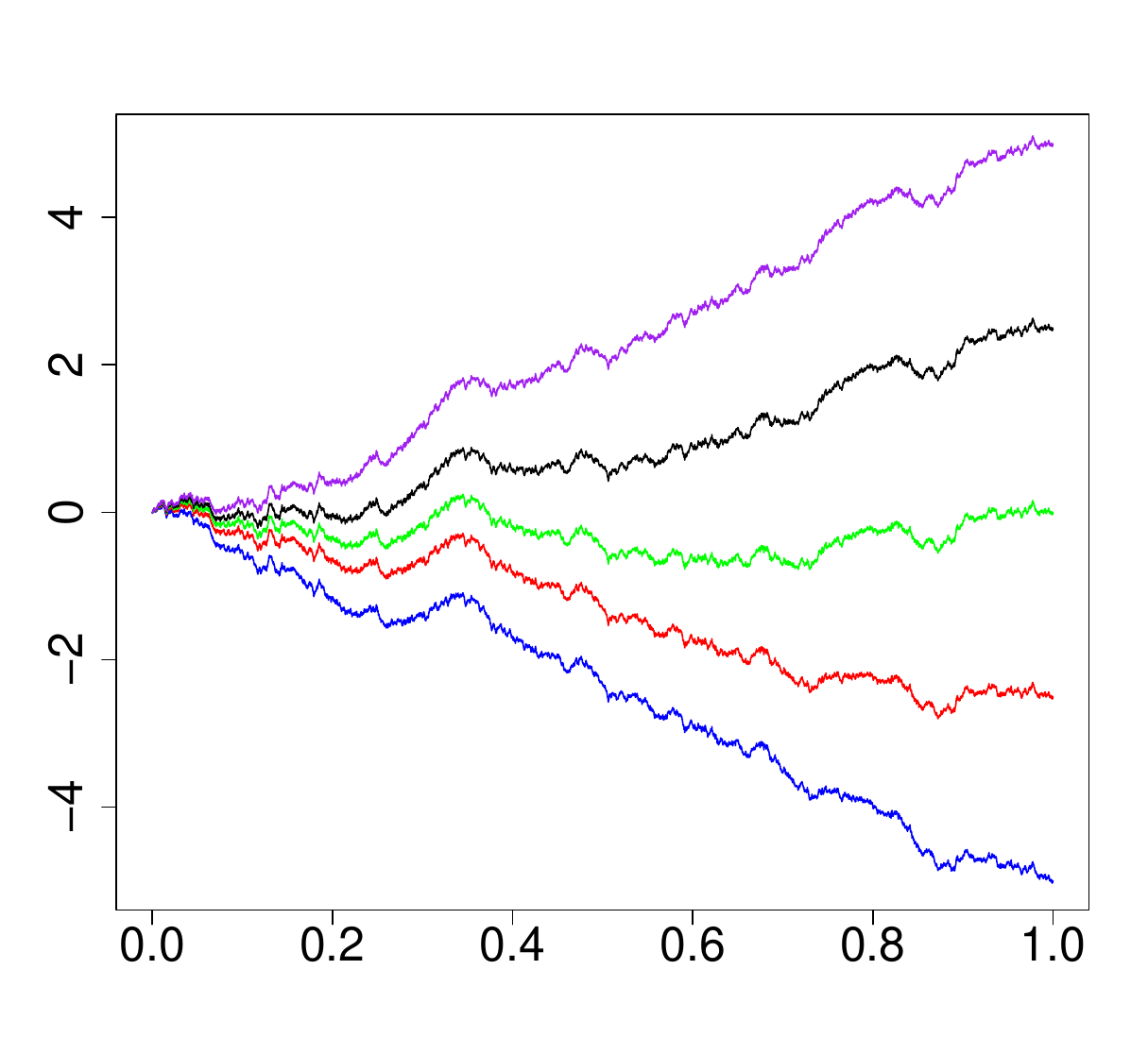}
    \caption{On the left are $k=5$ independent Brownian bridges with slopes $(\theta_1,\ldots,\theta_5) = (-5,-2.5,0,2.5,5)$. These are mapped via $\Psi^5$ to the output on the right, giving a sample of the invariant measure $\Pm_\beta^{(\theta_1,\ldots,\theta_5)}$ for $\beta = 1$. Each trajectory/color represents a different value of $\theta$.}
    \label{fig:sim1}
\end{figure}

The following result (proved in Section \ref{sec:proof234} as a corollary of results in \cite{GK21}) implies uniqueness of the jointly invariant measures and shows a stronger {\em one force--one solution (1F1S) principle}. In the context of the viscous or inviscid Burgers' equation with stochastic forcing, such a principle goes back to \cite{Sinai1991,Sinai1996,Iturriaga-Khanin-2003} and \cite{EKMS-1997,EKMS-2000,Iturriaga-Khanin-2003,GIKP-2005,TR22}, respectively (see also works involving non-compact domains, e.g., \cite{Bakhtin-Cator-Konstantin-2014,bakhtin2019thermodynamic,dunlap2021stationary,janjigian2022ergodicitysynchronizationkardarparisizhangequation}, and open boundary conditions \cite{parekh2023ergodicityresultsopenkpz}). For the stochastic Burgers' equation considered here, \cite{TR22} demonstrates such a principle in the special case of a single slope. 

\begin{theorem}[1F1S principle] \label{thm:1f1s}
For any $\theta_1,\ldots,\theta_m \in \R$ and random $(f_1,\ldots,f_k) \in (C[0,1])^k$ independent of  $\xi$  and satisfy $f_m(1) - f_m(0) = \theta_m$ for $1 \le m \le k$, the $(\Cpin[0,1])^k$ valued random variable 
\[
\Bigl(h_\beta(0,y \viiva -t, f_1) - h_\beta(0,0 \viiva -t,  f_1),\ldots,h_\beta(0,y \viiva  -t, f_k) - h_\beta(0,0 \viiva -t,f_k) \Bigr)
\]
converges in probability as $t \to \infty$, and the limit has the law given by $\mathcal{P}_\beta^{(\theta_1,\ldots,\theta_k)}$. As a result, $(g_1,\ldots,g_k) \sim \Pm_\beta^{(\theta_1,\ldots,\theta_k)}$, is the unique (in law) jointly invariant initial condition for \eqref{eq:KPZ} on $(\Cpin[0,1])^k$ satisfying  $g_m(0) = 0$ and $g_m(1) = \theta_m$ for $1\leq m\leq k$. 
\end{theorem}

The description of the jointly invariant measures is used in the following section, specifically Theorem \ref{t.conG} to give an explicit description of the covariance function of a limiting Gaussian process related to the associated polymer model.

\subsection{Overview of our approach and its relation to previous work}
In proving Theorem \ref{thm:KPZ_invar_main} there are two related challenges to surmount: (1) How to produce these measures; (2) How to check that they are jointly invariant. In principle once there is a proposed measure, one could hope to check the invariance directly. That would involving constructing the generator for the coupled KPZ evolution Markov process and then checking the associated Fokker-Planck equation. The work \cite{GubPerk20} provides a route to construct such generators but requires a priori knowledge of the invariant measure. So, while our results can be taken as inputs into \cite{GubPerk20}, it would be circular to try to check it using that method. Rather than trying to develop a (non-circular) route to verify our jointly invariant measures at the KPZ equation level, we follow Mark Kac's dictum to ``be wise---discretize'' (or semi-discretize). But how? In the following sections, we summarize some possible approaches before discussing the approach we took..

\subsubsection{Colored stochastic vertex models}
There are various finite-dimensional regularizations or discretizations of the periodic KPZ equation that (should) converge to it, and hence whose jointly invariant measures should converge to those we seek. Working with such a model makes the formulation of the Fokker-Planck equation clear. There are also several models, namely {\em colored stochastic vertex models} \cite{Corwin2016,BorodinWheeler} and various limits, for which there is significant progress towards understanding jointly invariant measures, both in the periodic setting and the full line. We  briefly describe the nature of those results and why it is challenging to take the KPZ limit from them, before going into the approach we use instead.

While uncolored stochastic vertex models are analogous to (and limit to in special cases, see, e.g., \cite{BertiniGiacomin,10.1214/16-AOP1101, Corwin2020,Lin2019}) the KPZ equation, their colored counterparts are analogous to (and limit to, see, e.g., \cite{PAREKH2023351}) the coupling of KPZ equations with different initial data yet the same noise considered herein.

The invariant measures for the colored stochastic vertex models are explicitly known. In the uncolored case they are certain product measures \cite{10.1215/00127094-2017-0029,10.1214/23-EJP1022}. In the colored case, they can be described explicitly in terms of partition functions for certain ``queue vertex models'' as shown directly from solutions to the colored Yang-Baxter equation in the recent result of \cite{Aggarwal-Nicoletti-Petrov-2023} (see also, e.g., \cite{Martin-2020,Ayyer-Mandelshtam-Martin-2023a,Ayyer-Mandelshtam-Martin-2023b,Bukh-Cox-2022} for examples of earlier work yielding similar descriptions of these measures in specific models, generally relying on the ``matrix product ansatz'' \cite{Prolhac-Evans-Mallick-2009}. See also the recent work \cite{Kuniba-Okado-Scrimshaw-24} for another description). This work applies to the periodic and full line setting. In systems with ``open'' boundaries (not considered here) the structure of invariant measures, even for just one color, becomes considerably more involved and is an active subject of research that we will not touch upon here. 

It is currently unclear how to extract the KPZ equation relevant limit of the queuing vertex model formulas for the colored (i.e., jointly) invariant measures of stochastic vertex models. To illustrate, for ASEP (a continuous time limit of the stochastic six vertex model which depends on a single {\em asymmetry} parameter $q\in [0,1]$)  \cite{Martin-2020} provides a description of the these invariant measures  (called multi-type therein) in terms of a randomized map of independent Bernoulli random walk bridges (one for each color) which can be seen as equivalent to the queue vertex model formulas of \cite{Aggarwal-Nicoletti-Petrov-2023}. In the special case of TASEP when $q=0$, the map becomes deterministic and goes back to \cite{Ferrari-Martin-2005}, see Section \ref{sec:TASEP} (although for $2$ colors, an alternate description of the invariant measures had previously been obtained \cite{Derrida-Janowsky-Lebowitz-1993,ANGEL-2006} using the matrix product ansatz and combinatorial methods). Taking the limit of that deterministic map and the underlying bridges, one readily sees convergence to a scaling limit (denoted here as the periodic stationary horizon) when the mesh-size of TASEP is scaled to zero. Convergence to a scaling limit in the $q = 0$ setting on the full line was previously shown in \cite{Busa-Sepp-Sore-22b}, noting that there are more difficulties in the full line case due to the non-compactness. The role of $q$ in the map has to do with the probability of accepting certain types of operations and for $q\in (0,1)$ fixed (as the mesh-size goes to zero), it is plausible that the effect of this randomness washes out and returns the same deterministic limiting map as for TASEP. This has not been proved directly, though follows in the full-line case from \cite[Corollary 2.14]{aggarwal2024scalinglimitcoloredasep}. The KPZ equation limit of ASEP \cite{BertiniGiacomin,PAREKH2023351} involves scaling $q\to 1$ as the mesh-size goes to zero, and it is unclear how the randomized ASEP maps produce the deterministic maps used to define the periodic KPZ jointly invariant measures. The same challenge seems to be present in the other stochastic vertex models that converge to the KPZ equation and makes for a compelling problem for future consideration.

\subsubsection{Polymer models}
Rather than directly taking the KPZ equation limit, one could first attempt a simpler $q\to 1$ limit of the colored stochastic vertex models and their associated invariant measures. In particular, the natural limits (see, e.g., \cite[Section 5.2]{Borodin-Corwin-2014} or \cite[Proposition 2.1]{Barraquand2017}) to consider are those that lead to the solvable directed polymer models. Indeed, for the discrete inverse-gamma polymer \cite{Seppalainen-2012} and semi-discrete O'Connell-Yor polymer \cite{brownian_queues,O'Connell-2012}  on the full line, the jointly invariant measures are known in \cite{Bates-Fan-Seppalainen} and \cite{GRASS-23}, respectively, and take the form of a deterministic map of random walks. The technique for finding these invariant measures goes back to the work of Ferrari and Martin (see also related results of \cite{Ferrari-Martin-2005,Ferrari-Martin-2007,Ferrari-Martin-2009,Fan-Seppalainen-20,Seppalainen-Sorensen-21b} in the particle system and last passage percolation settings, as well as \cite{Busani-2021,Busa-Sepp-Sore-22a,Busa-Sepp-Sore-22b,Busa-Sepp-Sore-23} for their scaling limits). This suggests that the queue vertex model formulas should converge to deterministic maps in those limits. Seeing that directly in either the full-line or periodic settings is currently an outstanding challenge.

\subsubsection{The periodic Pitman transform and our approach} \label{sec:approach}
While the above sketched approaches may eventually be realized, the approach that we develop here offers a new integrable structure---the periodic Pitman transform---that may be of independent interest. We  work with a semi-discretization of the KPZ equation that is related to the O'Connell-Yor polymer (In fact, we mainly work with a closely related system of SDEs that can be seen as a semi-discrete version of the stochastic Burgers' equation, see Equation \eqref{eq:OCY1_SDE} below).  
One immediate benefit of working with SDEs is that the Fokker-Planck equation is easily formulated, and we have tools of stochastic calculus available.

In the full-line (non-periodic) setting, the Pitman transform has arisen in describing jointly invariant measures, see, e.g., \cite[Section 6B]{Fan-Seppalainen-20}, \cite[Section 2.4]{Busani-2021}, \cite[Section 6]{Bates-Fan-Seppalainen}, \cite[Lemma 2.3.8]{Sorensen-thesis}, \cite[Theorem 1.2]{Dauvergne-Virag-2024}. That work relies on three key properties. The Pitman transform
\begin{enumerate}
    \item defines a {\em bijection} on suitable spaces,
    \item enjoys a {\em Burke property}, preserving certain product measures,
    \item and satisfies non-trivial {\em intertwining} relations.
\end{enumerate}
This structure in the full-line case served as a starting guide for our periodic investigation. However, as is often the case, it was unclear how to lift the structure so instead we began our investigation by considering the first non-trivial case of our problem: two joint invariant measures for the two-layer O'Connell-Yor polymer (i.e., $k=2$ and $N=2$ in the notation used later in Equation \eqref{eq:joint_OCY} below, where $k$ denotes the number of initial conditions, and $N$ denotes the spatial periodicity).

In this $k=2$, $N=2$ case, we were able to explicitly construct the jointly invariant measures by analysis of the system of SDEs in conjunction with an ansatz (informed by the full-line case) on the triangular structure of the deterministic map that we conjectured to exist. The main idea was to identify a transformation that maps the SDEs into another dual system of SDEs which has product form invariant measure. Certain natural structural properties of the transform and the dual  system of SDEs nailed down their description (see Section \ref{sec.simplecase} along with a discussion of the guiding structural ansatz).

Informed now by the result of the  $k=2$, $N=2$ calculation we were then able to formulate a periodic Pitman transform for general $N$ and verify the key properties need in order to use it to construct the jointly invariant measures (see the discussion below). Whereas this leap from $N=2$ to general $N$ took considerable effort, the leap from $k=2$ to general $k$ is straightforward and takes an iterative form as in earlier work (and also evident from the queuing vertex model structure in \cite{Aggarwal-Nicoletti-Petrov-2023}).
Theorem \ref{thm:OCY_joint} provides the output of this effort---the explicit formula for the jointly invariant measures for the periodic O'Connell-Yor polymer model.

Proving this formula proceeds in three steps: (1) We construct a bijection, see Proposition \ref{prop:transform}, $\D^{N,k}$ (a discrete analogue of $\cDm^k$) between suitable sets whose inverse map $\mathcal J^{N,k}$, see \eqref{Jimap}, (2) maps (see Proposition \ref{prop:NewSDE}) the system \eqref{eq:joint_OCY} of SDEs we care about into a new system \eqref{DxR_gen} of SDEs (3) which we show, see Proposition \ref{prop:prod_invar}, has a product form invariant measure dependent on the ordered set of slope parameters $(\theta_1,\ldots, \theta_k)$. This implies that the push-forward of the product measure under $\D^{N,k}$ yields the jointly invariant measures desired. Remarkably, if $\D^{N,k}$ is applied to the product measure with permuted slope parameters, the output is invariant in law, see Proposition \ref{prop:disc_consis}, or its KPZ analog, Proposition \ref{prop:g_prop_intro}.

To understand the origin of these results, we describe here how $\D^{N,k}$ is defined and in so doing introduce the periodic Pitman transform. The map $\D^{N,k}$ is built by a triangular iteration, see \eqref{eq:D_iter_intro} or Figure \ref{fig:DJ}, from a building block  map $D^{N,2}:\R^{\Z_N} \times \R^{\Z_N} \to \R^{\Z_N}$ (here $\R^{\Z_N}$ is the set of vectors $\mbf X = (X_0,\ldots,X_{N-1}) \in \R^N$, with $N$-periodic indices, e.g. $X_{i}=X_{i+N}$ for all $i$). 

We also make use of a related map $T^{N,2}: \R^{\Z_N} \times \R^{\Z_N} \to \R^{\Z_N}$. The maps $D^{N,2},T^{N,2}$ are defined by their coordinate outputs $D_i^{N,2},T_i^{N,2}$, for $i\in \Z_N$ as 
$$
D_i^{N,2}(\mbf X_1,\mbf X_2) = X_{2,i} + \log\Biggl(\f{\sum_{j \in \Z_N} e^{Y_{(i,j]}}}{\sum_{j \in \Z_N} e^{Y_{(i-1,j]}}}\Biggr),\qquad
T_i^{N,2}(\mbf X_1,\mbf X_2) = X_{1,i} + \log\Biggl(\f{\sum_{j \in \Z_N} e^{Y_{[i,j]}}  }{\sum_{j \in \Z_N} e^{Y_{[i+1,j]}}}\Biggr),
$$
where $\mbf Y = \mbf X_2 - \mbf X_1$ (as a vector) and the notation $Y_{(i,j]}$ means $\sum_{\ell = i+1}^j Y_\ell$, where the sum is taken in cyclic order, and $Y_{(i,i]} = 0$. The notation $Y_{[i,j]}$ is defined similarly, except that the sum starts at $\ell = i$. 

We define the {\em periodic discrete geometric Pitman transform} $W:\R^{\Z_N}\times \R^{\Z_N}\to \R^{\Z_N}\times \R^{\Z_N}$ by its action 
$$W(\mbf X_1,\mbf X_2) := \bigl(T^{N,2}(\mbf X_1,\mbf X_2),D^{N,2}(\mbf X_1,\mbf X_2)\bigr).$$
If we replaced $\Z_N$ by $\Z$, we would have the equality of intervals $(i-1,j]=[i,j]$ and $(i,j]=[i+1,j]$ (this fails on $\Z_N$ because $(i,i]=\varnothing$ while $[i+1,i]=\Z_N$) and thus $T^{N,2}$ and $D^{N,2}$ (and hence the transform) would exactly match the discrete geometric Pitman transform (see, e.g., \cite[Definition 2.3]{CorwinInvariance}), up to taking logarithms and translating between cumulative sums and increments. The full-line Pitman transform has proved to be a rich object, starting from its introduction in 1975 by Pitman  \cite{Pitman1975} in his celebrated $2M-X$ theorem, and developed considerably over the past three decades, e.g. in work such as \cite{brownian_queues,O'Connell-2003,rep_non_colliding,Biane-Bougerol-OConnell-2005,Biane09,O'Connell-2012,10.1215/00127094-2410289,O’Connell2014,CorwinInvariance,CroydonSasada21,Bougerol22, Dauvegne-Nica-Virag-2021}. The full-line Pitman transform is connected through some of these works to the RSK correspondence (especially its geometric lift from \cite{Noumi-Yamada-2004}). It would be interesting to see if the periodic Pitman transform introduced here enjoys all of the same structure as in the full-line case, and also whether it can be seen to arise from representation-theoretic origins such as crystal bases and Littelman paths (perhaps associated to other Lie algebras).

Here we just begin the investigation of the periodic Pitman transform by showing that it satisfies the same type of three key properties (needed to construct the jointly invariant measures) as in the full-line case. 

\begin{enumerate}
\item {\em Bijection}: $W$ is a bijection between $\R^{\Z_N} \times\R^{\Z_N}$ and itself, and further restrictions to a bijection between $\R^{\Z_N}_{\theta_1} \times\R^{\Z_N}_{\theta_2}$ and itself (where $\mbf X\in \R^{\Z_N}_{\theta}$ means that $\sum_{i\in \Z_N}X_i = \theta$). The inverse of map $M$ is given by $M = sWs$, where $s(\mbf X_1,\mbf X_2) = (\mbf X_2,\mbf X_1)$ permutes coordinates. These results are shown as Corollary \ref{cor:DR_bijection}. Observe the similarity with the form the inverse for the full-line Pitman transform in \cite[Proposition 2.2(iv)]{Biane-Bougerol-OConnell-2005} (see also \cite[Theorem 1.2]{Dauvegne-Nica-Virag-2021}, \cite[Lemma 2.3.8]{Sorensen-thesis}, and \cite[Theorem 1.2]{Dauvergne-Virag-2024}). We note that in those settings, there is also a spatial reflection involved. 
That is, the full analogue of these previous results is $M = sr W r s$, where $r$ is the spatial reflection map $r(\mbf X_1,\mbf X_2) = (\mbf X_1',\mbf X_2')$, where  $X_{m,i}' = -X_{m,N-i}$ (this has the effect of reversing time in the associated random walk). To reconcile with our result, it would be enough so show that $rWr = W$. This is actually not the case, but 
it is shown in Lemma \ref{lem:Pitman_reflect} that $rWr(\mbf X_1,\mbf X_2) = (\wt{\mbf X}_1,\wt{\mbf X}_2)$, where, for $i \in \Z_N$,
$
\wt X_{1,i} =  X_{2,i} - X_{2,i-1} + T_{i-1}^{N,2}(\mbf X_1,\mbf X_2)$ and $\wt X_{2,i} = X_{1,i} - X_{1,i+1} + D_{i+1}^{N,2}(\mbf X_1,\mbf X_2)$.
Hence, while we do not have $rWr = W$, the two sides differ by shifts and the addition of  telescoping terms. In a suitable scaling limit to a continuum periodic Pitman transform, the shifts become infinitesimal and the additional terms are negligible because of the telescoping.  

\item {\em Burke property}: If $(\mbf X_1,\mbf X_2)$ are independent, and each $\mbf X_m$ is a vector of $N$ i.i.d. log-inverse-gamma random variables, then $W(\mbf X_1,\mbf X_2) \deq (\mbf X_1,\mbf X_2)$ (see Proposition \ref{prop:Burke}). The preservation of this product measure is the direct analogue of the classical Burke theorem in \cite{Burke1956} (see also \cite[page 11]{Kelly-2011}) as well as several analogous Burke properties in various other settings, for example \cite{Hsu-Burke-1976} (discrete-time queue with Bernoulli service and arrivals), \cite[Lemma 8]{Aldous-Diaconis-1995} (Hammersley process), \cite[Theorem 3.1]{harrison1990} (Brownian queue), \cite[Theorem 5]{brownian_queues} (positive temperature Brownian queue), \cite[Theorem 3.3]{Seppalainen-2012} (inverse-gamma polymer).
\item {\em Intertwining}: We prove that the map $\mathcal D^{N,k}$ (built from $D^{N,2}$ iteratively) intertwines the generators of the systems \eqref{eq:joint_OCY} and \eqref{DxR_gen} of SDEs (see Proposition \ref{prop:NewSDE}). (Continuous-time intertwining of generators for particle systems was used to prove jointly invariant measures for multi-species particles systems, e.g. TASEP and the Hammersley process \cite{Ferrari-Martin-2007,Ferrari-Martin-2005,Ferrari-Martin-2009}, but, to our knowledge, the present paper is the first to deal with intertwining of diffusion processes in this setting.)

Besides the continuous-time intertwining, the map $\D^{N,k}$ also intertwines two discrete-time Markov chains (Proposition \ref{prop:DNk_disc_intertwine}), which allows us to show that the jointly invariant measures for \eqref{eq:joint_OCY} are the same as the jointly invariant measures for a discrete-time Markov chain that turns out to be a periodic version of the inverse-gamma polymer (see Section \ref{sec:disc_MC}). The key to this intertwining is a nontrivial identity involving the maps $D^{N,2}$ and $T^{N,2}$ given in Proposition \ref{prop:full_intertwine}. This identity is the periodic analog of a key identity found in several works on the full line (see \cite[Lemma 4.4]{Fan-Seppalainen-20}, \cite[Lemma A.5]{Busani-Seppalainen-2020}, \cite[Proposition 5.8]{Bates-Fan-Seppalainen}, \cite[Lemma 7.6]{Seppalainen-Sorensen-21b}, \cite[Lemma A.5]{Groathouse-Janjigian-Rassoul-21}, and \cite[Lemma 2.7]{GRASS-23}).
Our original motivation for Proposition \ref{prop:full_intertwine} was not to demonstrate jointly invariant measures for the periodic inverse-gamma polymer; instead, it was to prove the remarkable invariance noted above of the jointly invariant measures with respect to permuting the slope parameters $\theta_1,\ldots, \theta_k$. The periodic inverse-gamma polymer jointly invariant measure result simply came as a by-product.

Presumably, it should be possible to use the periodic inverse-gamma polymer jointly invariant measure result to recover those of the periodic O'Connell-Yor and KPZ equations through scaling limits, and the discrete intertwining should converge to the continuum intertwining. We do not attempt any of these limits here.

\end{enumerate}
We hope this work will prompt a fuller investigation of the structure of the periodic Pitman transform.

\subsubsection{Taking the KPZ equation limit}
This is done by showing that the polymer partition function converges to the solution to the periodic stochastic heat equation under {\em intermediate disorder} scaling. While such convergence results are known in the full-space \cite{Alberts-Khanin-Quastel-2014b,CSZ,Corwin2017, 10.1214/17-EJP32, Nica-2021,GRASS-23} and half-space \cite{Wu2020,10.1214/22-EJP775,10.1214/23-AOP1634} settings, this seems to be the first result for periodic polymers. As with these previous works, we rely upon convergence of the semi-discrete chaos series expansion for the O'Connell-Yor polymer partition function to the continuum chaos series expansion for the stochastic heat equation. We show $L^2$ convergence of the chaos series by embedding the semi-discrete noise in the continuum space-time white noise. On the semi-discrete side, the chaos series involves periodized Poissonian heat kernels (i.e., infinite sums of Poisson heat kernels), while on the continuum side, it involves periodized Gaussian heat kernels (i.e., infinite sums of Gaussian heat kernels). The convergence relies upon fine scale estimates on the convergence of these kernels and is complicated by the need to use different types of estimates depending on the time and length scales involved. The infinite sums of the Gaussian kernels and its semi-discrete analogue requires finer estimates than in the full-line case. As a specific example of this non-triviality, see the several cases handled after \eqref{eq:SumKj} in the proof Lemma \ref{lem:JkN_unif_bd}.  Unlike many previous works, we also address convergence of the polymer model with general boundary data (i.e. initial conditions).

Once the KPZ equation limit is established (Theorems \ref{L2_conv_main_theorem} and \ref{thm:SHE_convergence}), it remains to show that the limit of the semi-discrete jointly invariant measures are jointly invariant measures for the limiting KPZ equation. This requires verifying a strong form of convergence of the underlying log-inverse-gamma random walk bridges (that are present in the semi-discrete jointly invariant measures) to the Brownian bridges (in the KPZ limit). In particular, the convergence requires scaling the semi-discrete slope parameters and relies upon the KMT coupling for random walk bridges to Brownian bridges from \cite{Dmitrov-Wu-2021} (see Appendix \ref{sec:appBB}).

\subsubsection*{Outline}
In Section \ref{sec:proplim} we record several key properties of the jointly invariant measures (Proposition \ref{prop:g_prop_intro}) and use them to construct the periodic KPZ horizon (Corollary \ref{cor:ICH_process}). We also demonstrate scaling limits including to the periodic stationary horizon (Propositions \ref{prop:betalim}) and the full-line KPZ horizon (Proposition \ref{prop:width_lim}). Besides its inherent interest, we demonstrate (Theorem \ref{t.conG}) a Gaussian process limit  with an explicit covariance function for the long-time height function fluctuations of the periodic KPZ equation when started from varying slopes. This recovers and generalizes results of \cite{GK21,Dunlap-Gu-23}. This result provides evidence for new explicit formulas (Conjecture \ref{conj1}) for the fluctuations of cumulants of the endpoint distribution for the periodic continuum directed random polymer. Section \ref{sec.jt.intro} introduces our results for the periodic semi-discrete stochastic Burgers' equation, related to the O'Connell-Yor polymer as discussed above. The jointly invariant measures for that system are given (Theorem \ref{thm:OCY_joint}) along with a proof of that result in a special and informative case, and a discussion about the complexity and ideas in going to the general case. Section \ref{sec:disc_MC} discusses a discrete-time Markov chain that we relate to the periodic inverse-gamma polymer and gives its jointly invariant measures in Theorem \ref{thm:disc_MC}. 

The main technical body of the paper begins with Section \ref{sec:alg} which develops most of the key algebraic properties of the periodic  Pitman transform map through which our jointly invariant measures are defined and studied. Section \ref{sec:Burke_intertwine} contains the proof of an analogue of Proposition \ref{prop:g_prop_intro} (Proposition \ref{prop:disc_consis}). As a by-product of the investigation of the algebraic structure, we obtain the proof of Theorem \ref{thm:disc_MC} in Section \ref{sec:proof_disc_MC}.  Section \ref{sec:OCY} provides the full proof of Theorem \ref{thm:OCY_joint} by showing that the semi-discrete stochastic Burgers' equation transforms under our map into a system with a simple product form invariant measure (i.e., the continuous time intertwining). Section \ref{sec:SHE_conv_section} addresses convergence of the semi-discrete O'Connell-Yor polymer to the stochastic heat equation, which implies convergence of the semi-discrete stochastic Burgers' equation to its continuum analogue. This convergence is applied in Section \ref{sec:proofmaintthm} to prove our main KPZ result, Theorem \ref{thm:KPZ_invar_main}, the 1F1S principle (Theorem \ref{thm:1f1s}) and the other properties and limits discussed in Section \ref{sec:proplim}. The Gaussian limit process result, Theorem \ref{t.conG}, is proved in Section \ref{s.applicationSection}. Several appendices contain detailed technical computations and some tools, e.g. the KMT bridge coupling, used in the main body.

\subsection{Notation and conventions} \label{sec:not} We record here several pieces of notation and conventions we use. 
\begin{enumerate}
\item $\N:= \{1,2,\ldots\}$ is the natural numbers.
\item $\psi$ is the digamma function $\psi(x) := \f{d}{dx} \log (\Gamma(x))$, and $\psi_1$ is the trigamma function $\psi_1(x) := \psi'(x)$. 
\item $\rho$ is the full-space Gaussian heat kernel $\rho(t,x) := \f{1}{\sqrt{2\pi t}}e^{-\f{x^2}{2t}} \ind\{t > 0\}$.
\item For a real-valued random variable $W$ on a probability space $(\Omega,\Ff,\Pp)$, let $\|W\|_{L^2(\Pp)} = \sqrt{\Ee[W^2]}$
\item $\overset{p}{\longrightarrow}$ and $\Longrightarrow$ mean convergence in probability and distribution, resp., and $\deq$ equality in distribution.
\item $\ind\{E\}$ is the indicator function of the event $E$, taking value $1$ when the $E$ holds and $0$ otherwise.
\item $M_{n,m}(\R)$ denotes the set of real-valued $n \times m$ matrices and $A^T$ denotes the transpose of a matrix $A$.
\item $C[0,1]$ is the space of continuous functions $f:[0,1]\to \R$ and $\Cpin[0,1]$ is the subspace of $f\in C[0,1]$ satisfying $f(0) = 0$. More generally, $C(X,Y)$ denotes the space of continuous functions from $X\to Y$ endowed with the topology of uniform convergence on compact subsets.
\item For a function $f:[0,1]\to \R$ and $x,y \in [0,1]$,  define  $f(x,y) := f(y) - f(x) + \ind\{x > y\}
\bigl(f(1)-f(0)\bigr)$.
\item For $k \in \N$, $\mathcal S(k)$ denotes the set of permutations $\sigma$ on $\{1,\ldots,k\}$.  
\item $\Z_N$ is the cyclic group $\Z/N\Z$ of order $N$ with the operation of addition modulo $N$.
\item When integrating, $d \mbf s_{m:n}$ means $\prod_{i = m}^n ds_i$.
\item For $(s,m),(t,n) \in \R \times \Z$,  $(s,m) \le (t,n)$ if $s \le t$ and $m \le n$, and $(s,m) < (t,n)$ if $s < t$ and $m < n$.
\item For a vector $\mbf X = (X_i)_{i \in \Z_N} \in \R^{\Z_N}$, we define
$
\vecsum(\mbf X) := \sum_{i \in \Z_N} X_i
$.
\item
For two vectors $\mbf X,\mbf X' \in \R^{\Z_N}$, we say $\mbf X< \mbf X'$ if $X_i < X_i'$ for all $i \in \Z_N$.
\item For integers $m \le n$, $\lzb m,n \rzb = \{m,m+1,\ldots,n\}$.
\item For $i,j \in \Z_N$, $[i,j]=\{i,\ldots,j\}$ (in cyclic order), $[i,i]=\{i\}$, $(i,j]=\{i+1,\ldots,j\}$ (in cyclic order), and $(i,i]=\varnothing$, 
$[i,j)=\{i,\ldots,j-1\}$ (in cyclic order), and $[i,i)=\varnothing$.
\item For a sequence $(X_i)_{i \in \Z_N}$, define
\be \label{sum_notat}
X_{[i,j]} = \sum_{\ell \in [i,j]} X_\ell, \qquad X_{(i,j]} = \sum_{\ell \in (i,j]} X_\ell,\qquad\text{and}\qquad X_{[i,j)} = \sum_{\ell \in [i,j)} X_\ell.
\ee
Sometimes, we use this notation with doubly indexed sequences $(X_{m,i})_{1 \le m \le k,i \in \Z_N}$. There, 
\[
X_{m,(i,j]} = \sum_{\ell \in (i,j]} X_{m,\ell}.
\]
and analogously for $X_{m,[i,j]}$ and $X_{m,[i,j)}$. 
\item For a vector $\mbf X \in \R^{\Z_N}$, we define discrete derivative and Laplacian operators:
\[
\diff_i \mbf X = X_i - X_{i-1},\qquad \Lapl_i \mbf X = \diff_{i+1} \mbf X - \diff_{i}\mbf X =X_{i+1}- 2X_i + X_{i-1}.
\]
We will apply these to the inverse exponential $e^{-\mbf X}$. That is,  operators
\[
\diff_i e^{-\mbf X} = e^{-X_i}-  e^{-X_{i-1}},\quad\text{and}\quad \Lapl_i e^{-\mbf X} = e^{-X_{i+1}}  - 2e^{-X_i} + e^{-X_{i-1}}.
\]
For a vector of i.i.d. Brownian motions $\mbf B = (B_i)_{i \in \Z_N}$, we also apply $\diff_i$ to $d\mbf B(t)$, meaning
\[
\diff_i d\mbf B(t) = dB_i(t) - dB_{i-1}(t).
\]
\item For a set $\Lambda \subseteq \R$ and $x,y \in \R$, we write
\[
\Lambda_{x,y}^{k} := \{ 
y_i \in \Lambda, 0 \le i \le k+1,
\text{ with set values } y_0 = x,\; \text{and}\; y_{k+1} = y \} \quad\text{and} \quad \Lambda_y^k = \Lambda_{0,y}^k.
\]
\item For $t > s$, define
\[
\Delta^k(t \viiva s) := \{s= s_0 < s_1 < \cdots < s_k < s_{k + 1} = t: s_i \in \R\}.
\]
Set $\Delta^k := \Delta^k(\infty \viiva -\infty)$ and for $t > 0$, set
$\Delta^k(t) := \Delta^k(t \viiva 0)$.

\end{enumerate}

\subsection{Funding} I.C. was partially supported by the NSF through DMS:1811143, DMS:1937254, and \\ DMS:2246576, the Simons Foundation through a Simons Investigator Grant (Award ID 929852) and through the W.M. Keck Foundation through a Science and Engineering grant
on ``Extreme Diffusion''. Y. G. was partially supported by the NSF through DMS:2203014.  E.S. was partially supported by the Fernholz foundation. 

\subsection{Conflict of interest statement}
The authors have no conflicts of interest to declare.

\subsection{Data availability statement} There is no data associated to this manuscript. 

\subsection{Acknowledgments}
Y.G. acknowledges the hospitality of the Columbia University Math Department during the Fall 2023 semester when this work was initiated. E.S. wishes to thank Alex Dunlap for many insightful discussions, Yuri Bakhtin for pointers to the literature, and Eva Engel for pointing out several typos in an earlier version of the paper. We thank the two anonymous referees for their thoughtful reading and comments, which has greatly helped to improve the exposition of this paper.

\section{Properties, limits and applications of periodic joint invariant measures}\label{sec:proplim}

\subsection{Periodic KPZ horizon}
We record several key properties of $\Pm_\beta^{(\theta_1,\ldots,\theta_m)}$ from Theorem \ref{thm:KPZ_invar_main}.
\begin{proposition}
\label{prop:g_prop_intro}
Consider any $k\in \N$, $(\theta_1,\ldots,\theta_k) \in \R^k$, and $(g_1,\ldots,g_k) \sim \Pm_\beta^{(\theta_1,\ldots,\theta_k)}$. Then:
\begin{enumerate} [label=\textup{(\roman*)}]
\item \label{itm:g_perm_invar} \textup{(Consistency and symmetry)} For any permutation $\sigma \in \mathcal S(k)$ and any index $1 \le m \le k$, 
\be \label{12}
(g_{\sigma(1)},\ldots,g_{\sigma(m)}) \sim \Pm_\beta^{(\theta_{\sigma(1)},\ldots,\theta_{\sigma(m)})}.
\ee
In particular, the measures $\Pm_\beta^{(\theta_1,\ldots,\theta_k)}$ are a consistent family as $k\in \N$ and  $(\theta_1,\ldots,\theta_k) \in \R^k$ varies.
\item\label{itm:g_consis}\textup{(Distribution of a component function)} For $1 \le m \le k$, marginally $g_{m} \deq \B_{\beta,\theta_m}$, see \eqref{eq.bb}.
 \item \label{itm:g_mont}\textup{(Monotonicity)} If $\theta_r = \theta_m$ for some $1 \le r \le m \le k$, then $g_r = g_m$ almost surely. If $\theta_r > \theta_m$ for some  $r \neq m$, then  $g_r(x,y) > g_m(x,y)$ for all $x,y \in [0,1]$. In particular, 
 $x \mapsto g_r(x) - g_m(x)$ is strictly increasing. 
 \item \label{itm:Pm_cont}\textup{(Continuity)} If $(\theta_1^{(n)},\ldots,\theta_k^{(n)})_{n \ge 0}$ is a sequence converging to $(\theta_1,\ldots,\theta_k)$, and  $(g_1^{(n)},\ldots,g_k^{(n)}) \sim \Pm_\beta^{(\theta_1^{(n)},\ldots,\theta_k^{(n)})}$, then $(g_1^{(n)},\ldots,g_k^{(n)})$ converges in distribution to $(g_1,\ldots,g_k)$.
\end{enumerate}
\end{proposition} 
Proposition \ref{prop:g_prop_intro} is proved in Section \ref{sec:proof567}.

Items \ref{itm:g_perm_invar} and \ref{itm:g_consis} are direct consequences of the uniqueness in Theorem \ref{thm:1f1s}. However, we provide an alternative and direct proof using the structure of the map $\FcDm^k$. In fact, we prove an analogous result at the level of our semi-discretized system of SDEs and then pass them to the KPZ equation limit. The proof of these results are involved, and rely upon a Burke-type property (Proposition \ref{prop:Burke}) as well as an interlacing relation (Proposition \ref{prop:full_intertwine}) for the model. 

The symmetry of the measures $\Pm_\beta^{(\theta_1,\ldots,\theta_k)}$ under permutation of the slope parameters for the input Brownian bridges in item \ref{itm:g_perm_invar} is new compared to previous results on the full line \cite{Ferrari-Martin-2007,Ferrari-Martin-2005,Ferrari-Martin-2009,Martin-2020,Fan-Seppalainen-20,Seppalainen-Sorensen-21b,GRASS-23}. In those cases the analogous maps require the slope parameters for inputs to be ordered in order to ensure the output of the map is finite (i.e., the full-line maps generally involve infinite integrals that may not converge unless parameters are correctly ordered---see also the discussion after Proposition \ref{prop:width_lim}). Item \ref{itm:g_mont} is a strong form of monotonicity that is better understood at the level of the Burgers equation. Indeed, formally, we have $\partial_x g_r(x) > \partial_x g_m(x)$ if $\theta_r>\theta_m$. While these derivatives are distributions rather than bona-fide functions, the function $x \mapsto g_r(x) - g_m(x)$ is in fact, almost surely differentiable. This can be seen as follows: By item \ref{itm:g_perm_invar}, $(g_r,g_m)$ marginally has the law of $\bigl(g_r,\cDm^2(g_r,f_m)\bigr)$, where $f_m \deq \B_{\beta,\theta_m}$ is independent of $g_r$. It follows from definitions (see Lemma \ref{lem:Phiout}) that
\begin{align*}
g_m(x) =\cDm^2(g_r,f_m)(x) = g_r(x) + \log \f{\displaystyle{e^{\theta_m - \theta_r}\int_0^x e^{f_m(y) - g_r(y)}\,dy + \int_x^1 e^{f_m(y) - g_r(y)}\,dy}  }{\displaystyle{\int_0^1 e^{f_m(y) - g_r(y)}\,dy }},
\end{align*}
from which the differentiability of $g_r - g_m$ follows. This differentiability should hold more generally for the difference of any two solutions to the KPZ equation driven by the same noise, see e.g., \cite[Corollary 1.13]{HairerKPZ} for periodic KPZ when the slope parameter is $0$, or (57) in \cite{O’Connell2016} or \cite{GRASS-23} in the full-line KPZ case. This differentiability of the difference has some useful applications; it was recently used by Dunlap and the third author \cite{Dunlap-Sorensen-2024} in the setting of the KPZ equation on the full line (using the description in \cite{GRASS-23}) to show the existence of an invariant measure from the perspective of a shock.

A consequence of the consistency and monotonicity in Proposition \ref{prop:g_prop_intro} is the following.
\begin{corollary}[Periodic KPZ horizon] \label{cor:ICH_process}
For every $\beta > 0$, there exists a stochastic process $\ICH_\beta = (g_{\beta,\theta})_{\theta \in \R} \in C\bigl(\R,\Cpin[0,1]\bigr)$ (viewed as the function $\theta \mapsto g_{\beta,\theta})$ so that, for all $k\in \N$ and $(\theta_1,\ldots,\theta_k) \in \R^k$, marginally
\be \label{eq:marg}
(g_{\beta,\theta_1},\ldots,g_{\beta,\theta_k}) \sim \Pm_\beta^{(\theta_1,\ldots,\theta_k)}.
\ee
We call $\ICH_\beta$ the {\em periodic KPZ horizon} and denote its law by $\Pm_\beta$. 
\end{corollary}
\begin{proof}
By the consistency in Proposition \ref{prop:g_prop_intro}\ref{itm:g_perm_invar} and the Kolmogorov consistency theorem, there exists a process $(g_{\beta,\theta})_{\theta \in \Q} \in (\Cpin[0,1])^\Q$ satisfying \eqref{eq:marg} for all $k\in \N$ and $(\theta_1,\ldots,\theta_k) \in \Q^k$. By monotonicity (Proposition \ref{prop:g_prop_intro}\ref{itm:g_mont}), almost surely, for every rational pair $\theta_1 < \theta_2$, and $x \in [0,1]$,
\begin{equation}\label{eq.ctnty}
0 = g_{\beta,\theta_2}(0) - g_{\beta,\theta_1}(0) \le g_{\beta,\theta_2}(x) - g_{\beta, \theta_1}(x) \le g_{\beta,\theta_2}(1) - g_{\beta,\theta_1}(1) = \theta_2 - \theta_1.
\end{equation}
Hence, $\sup_{x \in [0,1]} |g_{\beta,\theta_2}(x) - g_{\beta,\theta_1}(x)| \le |\theta_2 - \theta_1|$ for all $(\theta_1,\theta_2) \in \Q^2$. This means that $(g_{\beta,\theta})_{\theta \in \Q}$ is a uniformly continuous function $\Q \to \Cpin[0,1]$, and hence has a unique extension to a process $(g_{\beta,\theta})_{\theta \in \R} \in C\bigl(\R,\Cpin[0,1]\bigr)$. The finite-dimensional projections satisfy \eqref{eq:marg} by continuity (Proposition \ref{prop:g_prop_intro}\ref{itm:Pm_cont}). 
\end{proof}

\begin{corollary} \label{cor:derivative}
As $\theta \to 0$, we have the following convergence in distribution on $\Cpin[0,1]$:
\[
\Biggl(\f{g_{\beta,\theta}(x) - g_{\beta,0}(x)}{\theta}:x \in [0,1] \Biggr) \Longrightarrow \Biggl(\f{\int_0^x e^{\beta (\B^{(1)}(y) + \B^{(2)}(y))}\,dy }{\int_0^1 e^{\beta (\B^{(1)}(y) + \B^{(2)}(y))}\,dy }: x \in [0,1]\Biggr),
\]
where $\B^{(1)},\B^{(2)}$ are independent standard Brownian bridges.
\end{corollary}
\begin{proof}
By definition, for $\theta \neq 0$, $(g_{\beta,0},g_{\beta,\theta})$ has the law of $\bigl(f_{\beta,0},\cDm^2(f_{\beta,0},f_{\beta,\theta})\bigr)$, where $f_{\beta,0},f_{\beta,\theta}$ are independent and $f_{\beta,\theta} \deq \B_{\beta,\theta}$. Rearranging the map $\cDm^2$ (see Lemma \ref{lem:Phiout}), we have 
\begin{align}
\Biggl(\f{g_{\beta,\theta}(x) - g_{\beta,0}(x)}{\theta}: x \in [0,1]\Biggr) &\deq \Biggl(\f{1}{\theta}\log \biggl( \f{e^{\theta}\int_0^x e^{f_{\beta,\theta}(y) - f_{\beta,0}(y)}\,dy + \int_x^1e^{f_{\beta,\theta}(y) - f_{\beta,0}(y)}\,dy }{\int_0^1 e^{f_{\beta,\theta}(y) - f_{\beta,0}(y)}\,dy }\biggr): x \in [0,1]\Biggr) \nonumber  \\
&=\Biggl(\f{1}{\theta}\log \biggl(1 + \f{(e^{\theta} - 1)\int_0^x e^{f_{\beta,\theta}(y) - f_{\beta,0}(y)}\,dy}{\int_0^1 e^{f_{\beta,\theta}(y) - f_{\beta,0}(y)}\,dy }\biggr):x \in [0,1]\Biggr). \label{eq:pre_deriv}
\end{align}
As $\theta \searrow 0$, $f_{\beta,\theta} - f_{\beta,0}$ converges in law to $\beta(\B^{(1)} + \B^{(2)})$, with respect to the topology of uniform convergence. The result now follows from a Taylor expansion in \eqref{eq:pre_deriv}. 
\end{proof}

The periodic KPZ horizon is the periodic counterpart to the (full-line) {\em KPZ horizon} constructed in \cite{GRASS-23}. The existence of a $C\bigl(\R,\Cpin[0,1]\bigr)$ process whose finite-dimensional projections are jointly invariant for the periodic KPZ equation was shown previously in \cite{Dunlap-Gu-23}. Several properties of that process 
were established therein. In particular, Corollary \ref{cor:derivative} is an alternate derivation of the law of the derivative in \cite[Theorem 1.3]{Dunlap-Gu-23}, but our joint distribution of $(g_{\beta,0},g_{\beta,\theta})$ contains more information.  The contribution of Corollary \ref{cor:ICH_process} is both the explicit description of its finite dimensional distributions in terms of $\Pm_\beta^{(\theta_1,\ldots,\theta_k)}$ as well as a direct proof of the existence of the process in $C\bigl(\R,\Cpin[0,1]\bigr)$ based from that description (i.e., using the consistency that follows from our formulas for the jointly invariant measures).

The continuity of $\ICH_\beta$ should be contrasted with the results for models on the full line, namely the KPZ horizon \cite{GRASS-23} (as well as its long-time limit, the stationary horizon \cite{Busani-2021,Seppalainen-Sorensen-21b}, and some discrete variants such as for  exponential last-passage percolation and the inverse-gamma polymer \cite{Fan-Seppalainen-20,Bates-Fan-Seppalainen}). In the full-line setting, it has been shown that the analogous processes to $\ICH_\beta$ actually have a dense set of discontinuities in $\theta$. As a side note, continuity in the periodic setting is more general, as seen from the argument in \eqref{eq.ctnty}. For instance, recently \cite{Dunlap-2024} showed that for the inviscid Burgers equation with periodic Poisson forcing, while discontinuities exist in the jointly invariant measure process on the level of the Burgers' equation, they disappear when integrated in space, i.e., at the level of the the Hamilton-Jacobi equation that is analogous to the KPZ equation in that setting. It is easy to see how the continuity argument used in \eqref{eq.ctnty} breaks down in the limit where the periodic KPZ horizon converges to the full-line KPZ horizon. In particular, in Proposition \ref{prop:width_lim} below, we show that, as $L \to \infty$, $g_{\sqrt L \beta,\theta L}(\f{x}{L})$ (as a process in $\theta$ and $x$) converges in the sense of finite-dimensional projections to the jointly invariant measures for the KPZ equation on the full line, i.e., the KPZ horizon \cite{GRASS-23}. Under that scaling \eqref{eq.ctnty} becomes
\[
\sup_{x \in [0,L]} \Bigl|g_{\sqrt L \beta,\theta_2 L}\bigl(\f{x}{L}\bigl) - g_{\sqrt L \beta,\theta_1 L}\bigl(\f{x}{L}\bigr)\Bigr| \le L|\theta_2 - \theta_1|,
\]
and the right-hand side diverges as $L \to \infty$; hence the continuity argument cannot be applied.

\subsection{Periodic stationary horizon limit} 
We record two scaling limits of $\ICH_\beta$ as $\beta \searrow 0$ and as $\beta \to \infty$. The first, $\wt g_{0,\theta}$ below, is the relatively trivial jointly invariant measures of the periodic Edwards-Wilkinson equation (i.e., additive stochastic heat equation) while the second, $\wt g_{\infty,\theta}$ below, should be the jointly invariant initial conditions for the (conjectural) periodic KPZ fixed point. In the Edwards-Wilkinson case this is straightforward to check. On the other hand, the periodic KPZ fixed point has not yet been constructed as a Markov process and taking any sort of asymptotic (e.g. at the level of various marginal distribution as done in the TASEP case in \cite{BaikLiu}) of the periodic KPZ equation is an outstanding open problem.  In analogy with the full-line setting \cite{Busani-2021,Seppalainen-Sorensen-21b,Busa-Sepp-Sore-22a,Sorensen-thesis} we call $\wt g_{\infty,\theta}$ the {\em periodic stationary horizon}.

\begin{proposition} \label{prop:betalim}
For $\beta > 0$, let $\ICH_\beta \sim \Pm_\beta$, and define the rescaled process $\wt \ICH_\beta = (\wt g_{\beta,\theta})_{\theta \in \R}$ by 
\[
\wt g_{\beta,\theta}(x) := \f{1}{\beta} g_{\beta,\beta \theta}(x), \qquad \theta \in \R,x \in [0,1].
\]
Then, there exist scaling limits of $\wt \ICH_\beta$, with respect to the uniform topology on $C\bigl(\R,\Cpin[0,1]\bigr)$, as $\beta \searrow 0$ and $\beta \to \infty$. We call these $\wt \ICH_0 = (\wt g_{0,\theta})_{\theta \in \R}$ and $\wt \ICH_\infty = (\wt g_{\infty,\theta})_{\theta \in \R}$ respectively. Each marginal $\wt g_{0,\theta}$ and $\wt g_{\infty,\theta}$ has the law of $B_{1,\theta}$, and the joint laws are characterized as follows: 
\begin{enumerate} [label=\textup{(\roman*)}]
\item \label{itm:beta0lim} For $x \in [0,1]$ and $\theta \in \R$, $\wt g_{0,\theta}(x) = \wt g_{0,0}(x) + \theta x$, almost surely. 
\item \label{itm:betainflim}\textup{(Periodic stationary horizon)} For any $k\in \N$ and $(\theta_1,\ldots,\theta_k) \in \R^k$, 
let $f_m=\B_{1,\theta_m}$, for $1\leq m\leq k$, be independent slope $\theta_m$, variance $1$ Brownian bridges \eqref{eq.bb} and set, for $x\in [0,1]$,
\be \label{eq:grtilde}
\begin{aligned}
\wt g_m(x):= f_m(x) &+ \sup_{x_1,\ldots,x_{m-1} \in [0,1],x_0 = x} \Bigl\{\sum_{r = 1}^{m-1}\bigl(f_m(x_{r-1},x_r) - f_r(x_{r-1},x_r)\bigr)  \Bigr\} \\
&- \sup_{x_1,\ldots,x_{m-1} \in [0,1],x_0 = 0} \Bigl\{\sum_{r = 1}^{m-1}\bigl(f_m(x_{r-1},x_r) - f_r(x_{r-1},x_r)\bigr)  \Bigr\}.
\end{aligned}
\ee
Then, $(\wt g_{\infty,\theta_1},\ldots,\wt g_{\infty,\theta_k}) \deq (\wt g_1,\ldots,\wt g_k)$.
\end{enumerate}
\end{proposition}
Proposition \ref{prop:betalim} is proved in Section \ref{sec:proof8910}. Note, \eqref{eq:grtilde} can also be written in a similar iterated way as \eqref{Phi_iter}. 
Figure \ref{fig:KPZH_lim} shows a simulation in R of $\wt g_{\beta,\theta}(x)$ for small and large values of $\beta$ and $\theta  \in \{-5,-2.5,0,2.5,5\}$. For small $\beta$, the functions becomes effectively linear shift of each other, while for large $\beta$, the functions tend to stick very close to each other near the origin. In the limit, one can see from \eqref{eq:grtilde} that the functions do in fact stick together in a random interval near $x = 0$.
\begin{figure}
    \centering
    \includegraphics[width=0.4\linewidth]{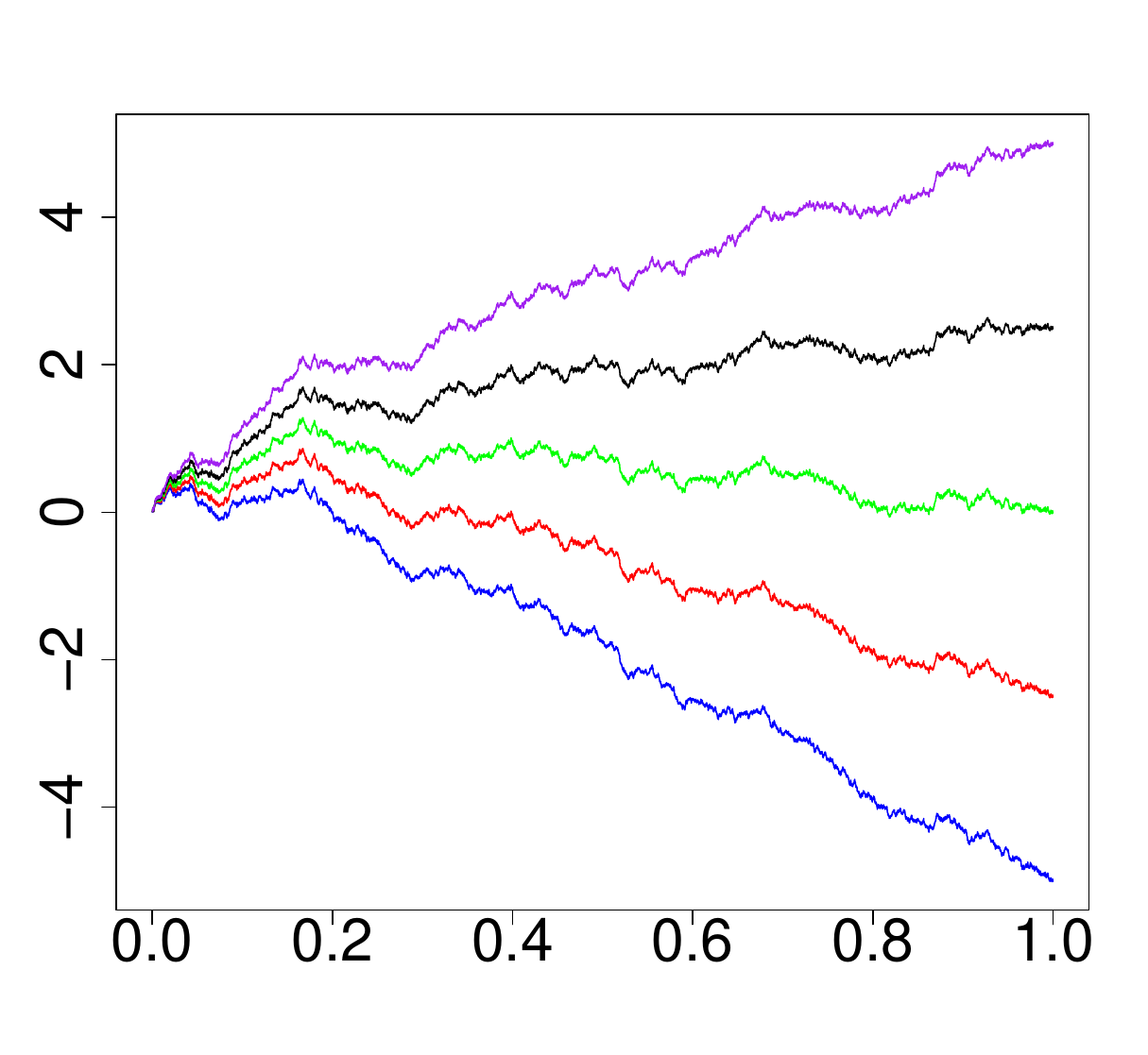}
     \includegraphics[width=0.4\linewidth]{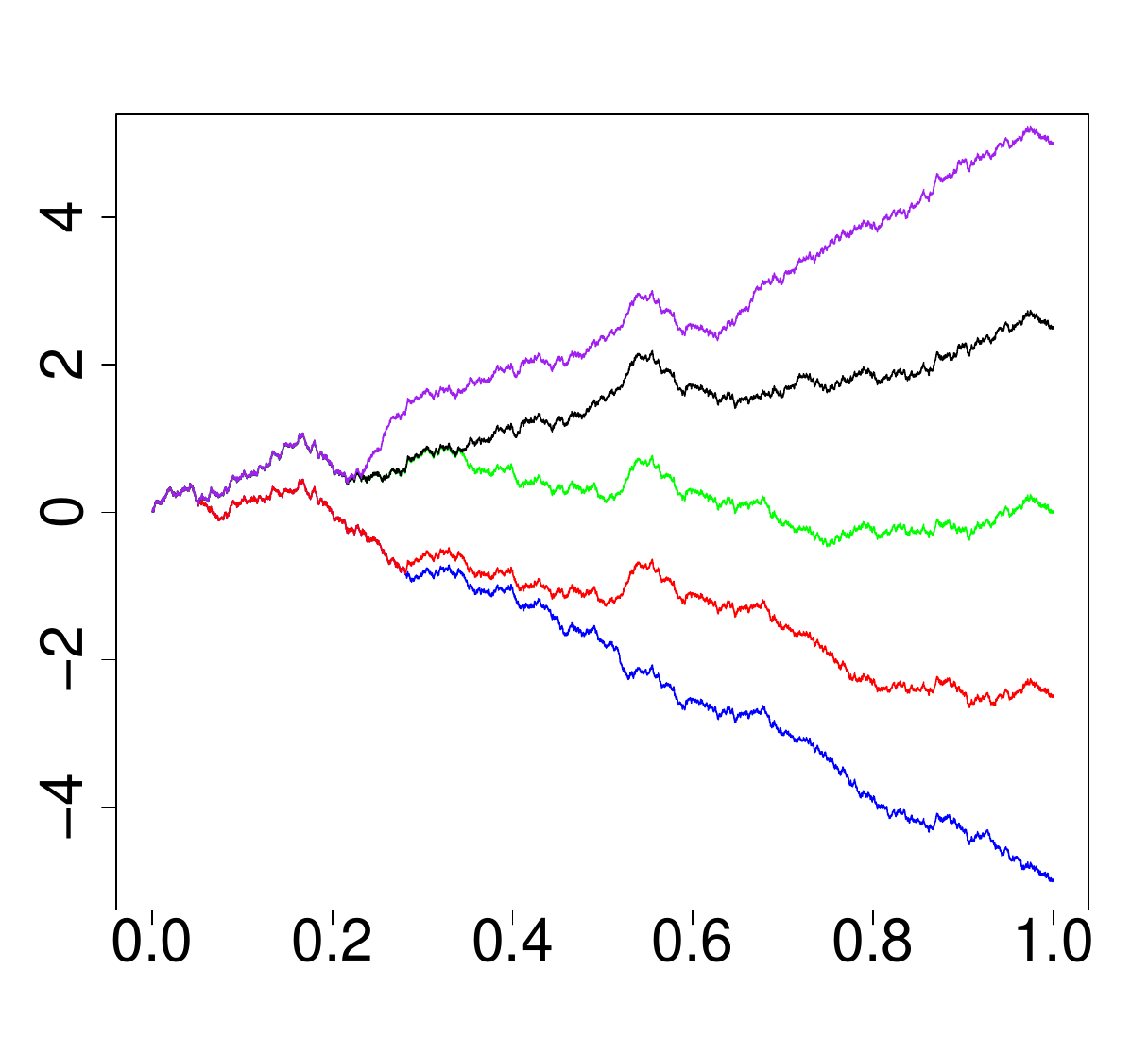}
    \caption{Simulation of $\wt g_{\beta,\theta}$ from Proposition \ref{prop:betalim} for $\beta  = 0.01$ (left) and $\beta = 30$ (right), and $\theta \in \{-5,-2.5,0,2.5,5\}$. The pictures are produced by applying $\Psi^5$ to the same set of $5$ independent Brownian bridges (as in Theorem \ref{thm:KPZ_invar_main}), with the scaling factor $\beta$ changed. Consequently the bottom (blue) curve is the same in each picture.}
    \label{fig:KPZH_lim}
\end{figure}

\subsection{Full-line KPZ horizon limit}

For $\beta,L>0$ define the scaled solution to \eqref{eq:KPZ}
\[
h_{L,\beta}(t,x) := h_{\sqrt L \beta}\Bigl(\f{t}{L^2},\f{x}{L}\Bigr), \qquad \textrm{where }x\in [0,L].
\]
Noting the distributional equality 
$\xi\bigl(\f{t}{L^2},\f{x}{L}\bigr) \deq L^{3/2} \xi_L(t,x)$, 
where we let $\wt\xi(t,x)$ denote the periodic extension of the space-time white noise on $\R\times L\T$ (for the size $L$ torus $L\T=\R/L\Z$), it follows that $h_{L,\beta}$ solves \eqref{eq:KPZ} with the noise $\xi$ (which has spatial period $1$) replaced by $\xi_L$ (which has spatial period $L$). As $L\to \infty$ solutions to the $L$-periodic KPZ equation should converge to those of the full-line KPZ equation (where the noise has no periodicity in space). Likewise, the jointly invariant measures for the $L$-periodic equation should converge to those of the full-line equation, i.e., the periodic KPZ horizon should converge to a (full-line) KPZ horizon defined recently in \cite{GRASS-23}. We prove the latter claim here.

To describe these limits, we introduce some additional notations. Consider the space 
 \[
 \mathcal Y = \Bigl\{f \in C(\R): f(0) = 0, \lim_{|x| \to \infty} \f{f(x)}{x} \text{ exists in }\R \Bigr\}.
 \]
 For each $f \in \mathcal Y$, define $\mathfrak c(f) = \lim_{|x| \to \infty} \f{f(x)}{x}$, and for $\theta > 0$, let 
 $
 \mathcal Y_\theta = \{f \in \mathcal Y: \mathfrak c(f) = \theta\}.
 $ 
 Next, consider the map
$\cDm^2_\infty:\mathcal Y \times \mathcal Y \to \mathcal Y$ defined as follows (compare to \eqref{eq:Phi_def}):
\be \label{eq:Phi_inf_def}
\begin{aligned}
\cDm^2_\infty(f_1,f_2)(x) &:=  \begin{cases}
f_1(x), & \mathfrak c(f_1) = \mathfrak c(f_2), \\
f_1(x) + \log \f{\int_{-\infty}^x e^{f_2(y) - f_1(y)}\,dy}{\int_{-\infty}^0 e^{f_2(y) - f_1(y)}\,dy} , &\mathfrak c(f_1) < \mathfrak c(f_2) ,\\
 f_1(x) + \log \f{\int_x^\infty e^{f_2(y) - f_1(y)}\,dy}{\int_0^\infty e^{f_2(y) - f_1(y)}\,dy} , &\mathfrak c(f_1) > \mathfrak c(f_2).
\end{cases}
\end{aligned}
\ee
Note the conditions on $f_1$ and $f_2$ ensure that $\cDm_\infty^2$ is well-defined. That $\cDm^2_\infty$ is indeed a map $\mathcal Y \times \mathcal Y \to \mathcal Y$ follows from \cite[Lemma 2.2]{GRASS-23}. In fact, $\cDm^2_\infty$ is a map from $\mathcal Y_{\theta_1} \times \mathcal Y_{\theta_2}$ to $ \mathcal Y_{\theta_2}$ for any $\theta_1,\theta_2 \in \R$. We iterate $\cDm^2_\infty$ to create maps $\cDm_\infty^m: \mathcal Y^m \to \mathcal Y$ by 
\begin{equation}\label{Phi_infty_iter}
\cDm_\infty^1(f_1) := f_1,\quad \textrm{and for }m  > 1, \quad 
\cDm_\infty^m(f_1,\ldots,f_k) := \cDm_\infty^2(f_1,\cDm_\infty^{m-1}(f_2,\ldots,f_k)),
\end{equation}
and define the map $\FcDm_\infty^k:\mathcal Y^k \to \mathcal Y$ by 
\[
\FcDm_\infty^k(f_1,\ldots,f_k) = \bigl(\cDm_\infty^1(f_1),\cDm_\infty^2(f_1,f_2),\ldots, \cDm_\infty^k(f_1,\ldots,f_k)\bigr).
\]

Analogous to our notation for $\B_{\beta,\theta}$, let $B_{\beta,\theta}$ denote the function $x \mapsto \beta B(x) + \theta x$ when $B$ has the law of a standard two-sided Brownian motion. Note that $B_{\beta,\theta}\in \mathcal{Y}_{\theta}$.
If $(f_1,\ldots,f_k) \in C(\R)^k$ has independent components with $f_m \deq B_{\beta,\theta_m}$ for $1 \le m \le k$, let $\Pm_{\infty,\beta}^{(\theta_1,\ldots,\theta_k)}$ denote the law of 
\[
(\wt g_1,\ldots,\wt g_k) := \FcDm_\infty^k(f_1,\ldots,f_k).
\]

\begin{proposition} \label{prop:width_lim}
For $\beta > 0$, let $\ICH_\beta = (g_{\beta,\theta})_{\theta \in \R} \sim \Pm_\beta$, and extend each $g_{\beta,\theta}$ to functions on $\R$ by letting $g_{\beta,\theta}(x+1)-g_{\beta,\theta}(x)=\theta$ for all $x\in \R$. For $x,\theta \in \R$ and $\beta > 0$, define  
\[
\wt g_{L,\beta,\theta}(x) := g_{\sqrt L \beta,L \theta}(x/L).
\]
Then, for any $(\theta_1,\ldots,\theta_k) \in \R^k$, as $L \to \infty$,
\[
\bigl(\wt g_{L,\beta,\theta_1},\ldots,\wt g_{L,\beta,\theta_k}\bigr) \Longrightarrow (\wt g_1,\ldots,\wt g_k)
\]
weakly, with respect to the uniform-on-compact topology on $C(\R)^k$, where $(\wt g_1,\ldots,\wt g_k) \sim \Pm_{\infty,\beta}^{(\theta_1,\ldots,\theta_k)}$.
\end{proposition}
Proposition \ref{prop:width_lim} is proved in Section \ref{sec:proof111213}. It is natural to expect that the limit of the jointly invariant measures will be jointly invariant for the limiting equation, i.e., the full-line KPZ equation  \eqref{eq:KPZ} with $\xi$ space-time white noise on $\R_{>0}\times\R$. Indeed, it is already known from \cite[Theorem 1.1]{GRASS-23} that when $\theta_1 < \cdots < \theta_k$, $\Pm_{\infty,\beta}^{(\theta_1,\ldots,\theta_k)}$ is a jointly invariant measure for the full-line KPZ equation \eqref{eq:KPZ} (note, earlier \cite{Janjigian-Rassoul-2020b} showed uniqueness of jointly invariant measures given the set of drifts $\{\theta_1,\ldots, \theta_k\}$, without giving a description for them). \cite{GRASS-23} also extended these measures to a process in $\theta$, named the {\em KPZ horizon}. 

Our measures $\Pm_{\infty,\beta}^{(\theta_1,\ldots,\theta_k)}$ are well-defined for any order of $\theta$ variables and it follows from taking the limit of the symmetry in Proposition \ref{prop:g_prop_intro}\ref{itm:g_perm_invar} that $\Pm_{\infty,\beta}^{(\theta_1,\ldots,\theta_k)}$ is invariant under permuting the order of the $\theta$ parameters. Thus,  Proposition \ref{prop:width_lim} provides $k!$ equivalent (not straightforwardly) descriptions for the jointly invariant measures for the  full-line KPZ equation; matching that of \cite{GRASS-23} when $\theta_1 < \cdots < \theta_k$.

\subsection{Gaussian process from height fluctuations} \label{sec:appplication_intro}

We now present an application of Theorem~\ref{thm:KPZ_invar_main} to the study of height fluctuations/directed polymers in a periodic environment. Define

\[
Z_\beta^{(\theta)}(t):=\int_{\R} e^{\theta y} Z_\beta(t,y \viiva 0,0)dy
\]  
where $Z_\beta(t,y\viiva s,x)$ is the Green's function of stochastic heat equation with noise $\xi$ defined in Section \ref{sec.l2converge}. 
Formally, by the Feynman-Kac formula we can write
\[
Z_\beta^{(\theta)}(t)=\Ee_B\left[e^{\theta B(t)} :\!\exp\!:\Bigl(\beta \int_0^t \xi\bigl(s,B(s)\bigr)ds\Bigr)\right],
\]
where $:\!\exp\!:$ is the Wick exponential, $B$ is a standard Brownian motion starting at the origin, see \cite[Section 2.2]{Bertini1995}, and the expectation $\Ee_B$ is with respect to $B$ only. This shows that
\[
\frac{
Z_\beta^{(\theta)}(t)}{Z_\beta^{(0)}(t)}=\frac{\int_{\R} e^{\theta x} Z_\beta(t,x \viiva 0,0)dx}{\int_{\R}   Z_\beta(t,x \viiva 0,0)dx}
\]
should be interpreted as the quenched moment generating function of the point-to-line polymer endpoint for a periodic environment continuum directed random polymer (in the spirit of \cite{Alberts-Khanin-Quastel-2014a}). In other words, if we define the {\em point-to-line polymer endpoint density} as  
\begin{equation}\label{eq.endptdist}
\rho_\beta(t,x\viiva 0,0)= \frac{Z_\beta(t,x\viiva 0,0)}{\int_{\R} Z_\beta(t,x'\viiva 0,0)dx'}\qquad \textrm{then}\qquad
\frac{Z_\beta^{(\theta)}(t)}{Z_\beta^{(0)}(t)}=\int_{\R} e^{\theta x} \rho_\beta(t,x\viiva 0,0)dx.
\end{equation}
In terms of surface growth, one can interpret $\log Z_\beta^{(\theta)}(t)$ as the height at $(t,0)$ when started with slope $\theta$. Up to a time reversal, $\log Z_\beta^{(\theta)}(t)$ is also the solution to the KPZ equation at $(t,0)$, with initial data $\theta x$.

By \cite[Theorem 1.1]{GK21}, we know that, with $\theta=0$, $\log Z_\beta^{(0)}(t)$ satisfies a central limit theorem, and by the shear invariance of the noise $\xi$, that result can be generalized to any fixed $\theta\in\R$, yielding, as $t\to \infty$,
\begin{equation}\label{e.defGamma}
\mathcal{X}_\beta^{(\theta)}(t):=\frac{\log Z_\beta^{(\theta)}(t)-\gamma_{\beta}^{(\theta)}t}{\sqrt{t}}\Rightarrow N(0,\sigma^2_\beta),    
\qquad\textrm{where the drift is } \quad
\gamma_{\beta}^{(\theta)}=\frac12\theta^2-\frac{1}{2}\beta^2-\frac{1}{24}\beta^4.
\end{equation}
The calculation of the drift is from \cite[Proposition 4.1]{GK211}, and the variance from  \cite[Corollary 2.3]{ADYGTK22}:
\begin{equation}\label{e.defsigma}
\sigma_\beta^2=   \beta^2 \Ee\bigg[  \frac{ \int_0^1 e^{\beta \mathcal{B}^{(1)}(y)  + \beta  \mathcal{B}^{(2)}(y) +2\beta\mathcal{B}^{(3)}}(y)\,dy}{\int_0^1 e^{\beta \mathcal{B}^{(1)}(y) + \beta\mathcal{B}^{(3)}(y)}dy \int_0^1 e^{\beta \mathcal{B}^{(2)}(y) +\beta\mathcal{B}^{(3)}(y)}dy}\bigg],
\end{equation}
where $\mathcal{B}^{(i)},i=1,2,3$ are independent copies of standard Brownian bridges.

We prove convergence in law of $\{\mathcal{X}_\beta^{(\theta)}(t)\}_{\theta\in\R}$   to a Gaussian process whose covariance function can, remarkably, be expressed in terms of the jointly invariant measure obtained in Theorem~\ref{thm:KPZ_invar_main}. 
\begin{theorem}\label{t.conG}
As $t\to\infty$, we have convergence in distribution in $C(\R)$ of
\[
\left\{\mathcal{X}_\beta^{(\theta)}(t)\right\}_{\theta\in\R}\Longrightarrow\{\mathcal{A}_\beta^{(\theta)}\}_{\theta\in\R},
\]
where $\mathcal{A}_\beta^{(\theta)}$ is a centered stationary Gaussian process with the covariance function given by
\begin{equation}\label{e.defA}
R_\beta^{(\theta)}:=\Ee[\mathcal{A}_\beta^{(0)}\mathcal{A}_\beta^{(\theta)}]=\beta^2 \Ee\left[  \frac{\displaystyle{\int_0^1 e^{\beta \mathcal{B}^{(1)}(y) + g_{\beta,0}(y) + \beta \mathcal{B}^{(2)}(y) + g_{\beta,-\theta}(y) + \theta y}\,dy}}{\displaystyle{\int_0^1 e^{\beta \mathcal{B}^{(1)}(y) + g_{\beta,0}(y)}dy \int_0^1 e^{\beta \mathcal{B}^{(2)}(y) + g_{\beta,-\theta}(y)+\theta y}dy}}\right]
\end{equation} 
Here $\mathcal{B}^{(i)},i=1,2,$ are independent standard Brownian bridges, and $\{g_{\beta,\theta}(\cdot)\}_{\theta\in\R}$ is the periodic KPZ horizon obtained in Corollary~\ref{cor:ICH_process}, independent of $\mathcal{B}^{(i)},i=1,2$.
The expectation $\Ee$ is on all source of randomnesses.
\end{theorem}
The proof of Theorem~\ref{t.conG} is in Section~\ref{s.applicationSection}, and can be read independently from the other sections of the paper. The formula for the covariance $R_\beta^{(\theta)}$ actually only depends on the joint law $\Pm_\beta^{(0,-\theta)}$ of $g_{\beta,0}$ and $g_{\beta,-\theta}$. If we take the special case of $\theta=0$ in \eqref{e.defA}, since $g_{\beta,0} \deq \beta\mathcal{B}^{(3)}$, we recover the expression for $\sigma_\beta^2$ in \eqref{e.defsigma} and hence the result of \cite[Theorem 1.1]{GK21} and \cite[Corollary 2.3]{ADYGTK22}.

We end this section with a conjecture and a formal calculation based on Theorem \ref{t.conG} that points to its validity. Recall \eqref{eq.endptdist} whose right-hand side involves  $\rho(t,x;0,0)$, the quenched endpoint density of the point-to-line directed polymer. Using $\kappa_n(t)$ to denote the quenched $n$-th cumulant of this endpoint, we have 
\[
\kappa_n(t)=\partial_\theta^n \log Z_\beta^{(\theta)}(t)\Big|_{\theta=0}, \quad\quad n\geq0.
\]
For the polymer on a cylinder, it is well-accepted by physicists 
(see, e.g., \cite{Brunet03})
that the free energy and all its cumulants 
are extensive functions, meaning they scale with the order of $t$. By shear invariance of $\xi$ (see Lemma \ref{l.reductiontheta0}), $\log Z_\beta^{(\theta)}(t)\deq\log Z_\beta^{(0)}(t)+\frac12\theta^2t$, 
from which it follows that $\Ee[\kappa_n(t)]=0$ for all $n\in \N\setminus\{2\}$ and $\Ee[\kappa_2(t)]=t$.

 It is natural to study the next order fluctuations, namely, whether the central limit theorem (CLT) holds for $\frac{\kappa_n(t)-\Ee[\kappa_n(t)]}{\sqrt{t}}$. We expect the CLT holds, and even more interesting that the asymptotic variance should be expressed in terms of the covariance obtained in Theorem~\ref{t.conG}.
\begin{conj}\label{conj1}
 Recalling $R_\beta^{(\theta)}$ from \eqref{e.defA}, $\lim_{t\to\infty}
 \frac{1}{t}\Var[\kappa_n(t)]=(-1)^n \partial_\theta^{2n}R_\beta^{(\theta)}\Big|_{\theta=0}$.
\end{conj}
 
Here is a heuristic argument for this conjecture. For each $t>0$, the process $\mathcal{X}_\beta^{(\theta)}(t)=\log Z_\beta^{(\theta)}(t)-\gamma_{\beta}^{(\theta)} t$ is stationary in the $\theta$ variable, and $\gamma_{\beta}^{(\theta)}=\frac12\theta^2-\frac{\beta^2}{2}-\frac{\beta^4}{24}$. Note that for a smooth stationary process $V(x)$ with sufficiently high moments for $V(0),V'(0)$, we have the integration by parts formula $\Ee[V(0)V'(x)] = - \Ee[V'(0)V(x)]$, which follows by taking limits from  
\[
\Ee\Bigl[V(0) \f{V(x+\delta) - V(x)}{\delta}\Bigr] = \Ee\Bigl[V(x) \f{V(-\delta) - V(0)}{\delta}\Bigr],
\]
the equality following from $\Ee[V(0)V(x+\delta)] = \Ee[V(-\delta)V(x)]$.
Then, for $n\neq 2$,
 \[
 \begin{aligned}
 \Ee\big[\kappa_n(t)^2\big]&=\Ee\big[|\partial_\theta^n \log Z_\beta^{(\theta)}(t)|^2\big] \Big|_{\theta=0}=\Ee\big[ |\partial_\theta^n \mathcal{X}_\beta^{(\theta)}(t)+\partial_\theta^n (\gamma_{\beta}^{(\theta)} t)|^2\big]\Big|_{\theta=0}=\Ee\big[ |\partial_\theta^n \mathcal{X}_\beta^{(\theta)}(t)|^2 \big]\Big|_{\theta=0}\\
 &=(-1)^n\Ee \big[\mathcal{X}_\beta^{(0)}(t)\partial_\theta^{2n}\mathcal{X}_\beta^{(\theta)}(t)\big]\Big|_{\theta=0}=(-1)^n \partial_\theta^{2n} \Ee \big[\mathcal{X}_\beta^{(0)}(t)\mathcal{X}_\beta^{(\theta)}(t)\big]\Big|_{\theta=0}.
  \end{aligned}
 \]
A similar argument treats the $n=2$ case, and in the end, we obtain that  for any $n\geq1$, 
\begin{equation}\label{e.varkappa}
\frac{1}{t}\Var[\kappa_n(t)]=(-1)^n \partial_\theta^{2n} \frac{ \Ee \big[\mathcal{X}_\beta^{(0)}(t)\mathcal{X}_\beta^{(\theta)}(t)\big]}{t}\bigg|_{\theta=0}.
\end{equation}
Our proof of Theorem~\ref{t.conG} also shows that
\[
\frac{ \Ee \big[\mathcal{X}_\beta^{(0)}(t)\mathcal{X}_\beta^{(\theta)}(t)\big]}{t}\to R_\beta^{(\theta)},
\]
so, {\em supposing that we can pass the limit under the derivative in \eqref{e.varkappa}},  this would imply Conjecture \ref{conj1}.

For $n=1$, the above heuristics have been justified rigorously in \cite{GK23}, using a different approach. The main result obtained in this paper could greatly simplify the argument in \cite{GK23}; in particular, Section 4 of \cite{GK23} could be removed. For arbitrary $n\geq2$, we expect a similar approach as in \cite{GK23}, combined with Theorem~\ref{t.conG}, should lead to Conjecture \ref{conj1}, see more discussion in Remark~\ref{r.proofconjecture} below. 

\section{Jointly invariant measures for a periodic semi-discrete stochastic Burgers' equation}\label{sec.jt.intro} 

\subsection{Periodic semi-discrete stochastic Burgers' equation}
The proof of Theorem \ref{thm:KPZ_invar_main} is achieved by first finding the jointly invariant measures for an integrable semi-discrete version of the KPZ equation (or rather, the stochastic Burgers' equation), then passing to a limit. We consider the system of  SDEs
\be \label{eq:OCY1_SDE}\tag{SBE$^N_{\beta}$}
dU_{i}(t) = \diff_i e^{-\mbf U(t)} \,dt + \beta \diff_i d\mbf B(t), \qquad i \in \Z_N.
\ee 
Here $\beta>0$, $N \in \N$ denotes the dimension of the system, $\Z_N$ denotes the cyclic group $\Z/N\Z$ of order $N$ under addition modulo $N$, $\mbf U(t) := (U_i(t))_{i\in \Z_N}$, and $\mbf{B}:=(B_i)_{i \in \Z_N}$ are i.i.d. standard Brownian motions. For $\mbf X \in \R^{\Z_N}$ we define the discrete derivative and (for later use) Laplacian  
\begin{equation}\label{eq.discretederiv}
\diff_i\mbf X := X_i - X_{i-1},\qquad \Lapl_i \mbf X := \diff_{i+1} \mbf X - \diff_i \mbf X = X_{i+1} - 2 X_{i} + X_{i-1}
\end{equation}
keeping in mind the cyclic nature of $\Z_N$. Thus, with some abuse of notations, $\diff_i e^{-\mbf U(t)} = e^{-U_i(t)}-e^{-U_{i-1}(t)}$.

As explained in Section \ref{sec:OCY_def}, \eqref{eq:OCY1_SDE} is satisfied by the logarithm of ratios of the periodic O'Connell-Yor polymer partition functions (i.e., free energy increments) and is therefore a semi-discrete version of the stochastic Burgers' equation. We show in Lemma \ref{lem:sd_SDE}\ref{itm:SDE} that, for each initial condition $\mbf U(0) \in (\R^{\Z_N})$, there is a pathwise-unique strong solution to this system. 

For $\mbf X = (X_i)_{i \in \Z_N} \in \R^{\Z_N}$, let $\vecsum(\mbf X) := \sum_{i \in \Z_N} X_i$ as defined in Section \ref{sec:not}. Since \eqref{eq:OCY1_SDE} is conservative and periodic, $d\vecsum(\mbf U(t)) =  0$ for all $t$, which implies the preservation of the slope parameter $\theta:=\vecsum(\mbf U(0))$.

\subsection{Invariant measures} The invariant measures for the continuum stochastic Burgers' equation are the derivatives of sloped Brownian bridges. In the semi-discrete case, this is replaced by independent log-inverse-gamma random variables conditioned on their total sum.
\begin{definition} \label{def:p_meas}
A random variable $X$ is said to be {\em log-inverse-gamma distributed} with shape $\gamma > 0$ and scale $\beta^{-2} > 0$ if $X$ has a probability density on $\R$  equal to
\[
\f{\beta^{-2\gamma}}{\Gamma(\gamma)}e^{-\gamma x} e^{-\beta^{-2} e^{-x}}.
\]
For $\theta \in \R$ and $\beta > 0$, let $\nu_\beta^{N,(\theta)}$ be the probability measure on $\R^{\Z_N}$ under which $\mbf X = (X_0,\ldots,X_{N-1})$ are i.i.d. log-inverse-gamma random variables with shape $\gamma$ and scale $\beta^{-2}$, conditioned on $\vecsum(\mbf X) = \theta$. This means that $(X_0,\ldots,X_{N-1})$ has density, on the hyperplane in $\R^{\Z_N}$ whose coordinates sum to $\theta$, proportional to
\be \label{p_dens1}
\prod_{i \in \Z_N} e^{-\beta^{-2}e^{-x_{i}}}.
\ee
We use here and elsewhere the convention that if $X_i$ denotes a random variable, then $x_i$ denotes the variable for its density. Note that this conditioned law does not depend on $\gamma$. Let $p_\beta^{N,(\theta)}(x_1,\ldots,x_{N-1})$ be the probability density of $(X_1,\ldots,X_{N-1})$ on $\R^{N-1}$, when $\mbf X = (X_0,\ldots, X_{N-1}) \sim \nu_\beta^{N,(\theta)}$. Note that $p_\beta^{N,(\theta)}(x_1,\ldots,x_{N-1})$ is a density on $N-1$ dimensional space, as we set $X_{r,0} = X_{r,N} = \theta_r - \sum_{i = 1}^{N-1} X_{r,i}$.

For  $(\theta_1,\ldots,\theta_k) \in \R^k$, let $\nu_\beta^{N,(\theta_1,\ldots,\theta_k)}$ be the product measure $\nu_\beta^{N,(\theta_1)} \otimes \cdots \otimes \nu_\beta^{N,(\theta_k)}$ on  $\R^{\Z_N} \times \cdots \times \R^{\Z_N}$. Let $p_\beta^{N,(\theta_1,\ldots,\theta_k)}$ be the probability density of the vector $\bigl((X_{1,i})_{1 \le i \le N-1},\ldots,(X_{k,i}\bigr)_{1 \le i \le N-1})$ when $(\mbf X_1,\ldots,\mbf X_k) \sim \nu_\beta^{N,(\theta_1,\ldots,\theta_k)}$ and $\mbf X_r = (X_{r,i}\bigr)_{1 \le i \le N-1}$.
\end{definition}

It is straightforward to check (and essentially known since \cite{brownian_queues}) from the Fokker-Planck equation for  \eqref{eq:OCY1_SDE} that $\nu_\beta^{N,(\theta)}$ is an invariant measure for each $\theta$, i.e., if $\mbf U(0) \sim \nu_\beta^{N,(\theta)}$, $\mbf U(t) \deq \mbf U(0)$ for all $t \ge 0$. 

\subsection{Jointly invariant measures}
We seek to identify the jointly invariant measures for \eqref{eq:OCY1_SDE}, i.e., the joint law of $k$ initial conditions which, when evolved via \eqref{eq:OCY1_SDE} subject to the same set of driving Brownian motions, remain distributionally unchanged in time. Specifically, we consider the process $\bigl(\mbf U_1(t),\ldots, \mbf U_k(t) \bigr)_{t \ge 0} \in C((0,\infty),\R^{\Z_N})$, where, for $1 \le r \le k$,  $\mbf U_r(t) = \bigl(U_{r,0}(t),\ldots,U_{r,N-1}(t)\bigr)_{t \ge 0} \in \R^{\Z_N}$ satisfies the coupled (i.e., the same Brownian motions $\mbf B$ drive the evolution for each $r \in \{1,\ldots,k\}$) SDEs
\be \label{eq:joint_OCY}\tag{SBE$^N_{\beta,k}$}
dU_{r,i}(t) = \diff_ie^{-\mbf U_r(t)}\,dt + \beta \diff_i d\mbf{B}(t),\quad  i \in \Z_N,\, 1 \le r \le k.
\ee
\begin{definition}
A random vector $\bigl(\mbf U_1(0),\ldots, \mbf U_k(0) \bigr)$ and its law is a {\em jointly invariant initial condition and measure} for \eqref{eq:OCY1_SDE} if it is an invariant initial condition and measure for \eqref{eq:joint_OCY}, i.e., if for all $t>0$
\[
\bigl(\mbf U_1(t),\ldots,\mbf U_k(t)\bigr) \deq \bigl(\mbf U_1(0),\ldots,\mbf U_k(0)\bigr).
\]
\end{definition}

The jointly invariant measures cannot be the product measure $\nu_\beta^{N,(\theta_1,\ldots, \theta_k)}$. Indeed, if we define a partial order on $\R^{\Z_N}$ by stating that $\mbf X < \mbf X'$ if and only if $X_i < X_i'$ for $i \in \Z_N$, then \eqref{eq:joint_OCY} is strongly attractive in the sense that, if $\mbf U_m(0) < \mbf U_r(0)$ for some $m,r \in \{1,\ldots,k\}$, then $\mbf U_m(t) < \mbf U_r(t)$  for all subsequent $t > 0$ (Lemma \ref{lem:SDE_strong_comp}). This suggests that the extremal invariant measures for \eqref{eq:joint_OCY} should be supported on vectors $(\mbf U_1,\ldots,\mbf U_k)$ such that, for each choice of $r,m \in \{1,\ldots,k\}$, either $\mbf U_r < \mbf U_m$ or $\mbf U_r > \mbf U_m$, and $\vecsum(\mbf U_m) = \theta_m$ for some distinct values $\theta_1,\ldots,\theta_k$.

We proceed to define a map (a discrete version of the construction that goes into Theorem \ref{thm:KPZ_invar_main}) that pushes-forward the product measure to the jointly invariant measure. We start by defining $D^{N,2}:\R^{\Z_N} \times \R^{\Z_N} \to \R^{\Z_N}$ as follows. For $(\mbf X_1,\mbf X_2) \in \R^{\Z_N} \times \R^{\Z_N}$, let $D^{N,2}(\mbf X_1,\mbf X_2) = \bigl(D_i^{N,2}(\mbf X_1,\mbf X_2)\bigr)_{i \in \Z_N}$ with 
\be \label{eq:D_intro}
\begin{aligned}
D_i^{N,2}(\mbf X_1,\mbf X_2) &= X_{2,i} + Q_i^{N,2}(\mbf X_1,\mbf X_2) - Q_{i-1}^{N,2}(\mbf X_1,\mbf X_2),  \quad \text{where}\\
Q_i^{N,2}(\mbf X_1,\mbf X_2) &= \log\sum_{j \in \Z_N} e^{X_{2,(i,j]}-X_{1,(i,j]}},\quad \textrm{and}\quad X_{m,(i,j]}=\sum_{\ell = i+1}^j X_{m,\ell},
\end{aligned}
\ee
with the sum taken in cyclic order over $\Z_N$ and the convention $X_{m,(i,i]} = 0$. In the limit to the KPZ equation, the output of the map $D^{N,2}$ can be thought of as producing the infinitesimal increments of the function-valued output of the map $\cDm^2$ from \eqref{eq:Phi_def}, while $Q^{N,2}$ is the semi-discrete analogue of $Q^2$ in \eqref{eq:Phi_def}. 
We iterate this map to define $D^{N,k}: (\R^{\Z_N})^k \to \R^{\Z_N}$ by setting
\be \label{eq:D_iter_intro}
D^{N,1}(\mbf X_1) = \mbf X_1, \textrm{ and for }m > 1,\quad 
D^{N,m}(\mbf X_1,\ldots,\mbf X_m) = D^{N,2}\bigl(\mbf X_1, D^{N,m-1}(\mbf X_2,\ldots,\mbf X_m)\bigr).
\ee
Equivalently, as shown in Lemma \ref{lem:Dnm_alt},
\begin{align*}
D^{N,m}_i(\mbf X_1,\ldots,\mbf X_m) &= X_{m,i} + Q_i^{N,m}(\mbf X_1,\ldots,\mbf X_m) - Q_{i-1}^{N,m}(\mbf X_1,\ldots,\mbf X_m) \quad \text{where}\\
Q_i^{N,m}(\mbf X_1,\ldots,\mbf X_m) &= \log \sum_{\substack{j_1,\ldots,j_{m-1} \in \Z_N\\ j_0 = i}} \prod_{r = 1}^{m-1} e^{X_{m,(j_{r - 1},j_r]}-X_{r,(j_{r - 1},j_r]}}. 
\end{align*}
Finally, define the map $\D^{N,k}:(\R^{\Z_N})^k \to (\R^{\Z_N})^k$ by
\be \label{DNk_intro}
\D^{N,k}(\mbf X_1,\ldots,\mbf X_k) = \bigl(\mbf X_1,D^{N,2}(\mbf X_1,\mbf X_2),\ldots, D^{N,k}(\mbf X_1,\ldots,\mbf X_k)\bigr).
\ee

\begin{definition} \label{def:mu_meas}
    For $k\in \N$ and $\theta_1,\ldots,\theta_k \in \R$, let $\mu_\beta^{N,(\theta_1,\ldots,\theta_k)}$ be the probability measure of the vector $(\mbf U_1,\ldots,\mbf U_k) = \D^{N,k}(\mbf X_1,\ldots,\mbf X_k)$ when $(\mbf X_1,\ldots,\mbf X_k) \sim \nu_\beta^{N,(\theta_1,\ldots,\theta_k)}$.
\end{definition}

The following is the semi-discrete analogue of
Proposition \ref{prop:g_prop_intro} (and implies that proposition after a suitable limit). We prove these properties  directly from the structure of the maps $\D^{N,k}$.
\begin{proposition} \label{prop:disc_consis}
Let $N,k \ge 1$, and let $(\theta_1,\ldots,\theta_k) \in \R^k$.  Let $(\mbf U_1,\ldots,\mbf U_k) \sim \mu_\beta^{N,(\theta_1,\ldots,\theta_k)}$. Then, 
\begin{enumerate} [label=\textup{(\roman*)}]
\item \label{itm:mu_perm} \textup{(Consistency and symmetry)} For $\sigma \in \mathcal S(k)$ and $1 \le m \le k$,
$
(\mbf U_{\sigma(1)},\ldots,\mbf U_{\sigma(m)}) \sim \mu_\beta^{N,(\theta_{\sigma(1)},\ldots,\theta_{\sigma(m)})}.
$
\item \label{itm:mu_consis} \textup{(Distribution of a component vector)} For $1 \le m \le k$, $\mbf U_m \sim \nu_\beta^{N,(\theta_m)}$.
\item \label{itm:mu_order} \textup{(Monotonicity)} If $\theta_r = \theta_m$ for some $r \neq m$, then $\mbf U_r = \mbf U_m$ almost surely. Furthermore, if $\theta_r < \theta_m$, then $\mbf U_r < \mbf U_m$ almost surely.
\item \label{itm:mu_cont} \textup{(Continuity)} If $(\theta_1^{(n)},\ldots,\theta_k^{(n)})_{n \ge 0}$ is a sequence converging to $(\theta_1,\ldots,\theta_k)$, and  $(\mbf U_1^{(n)},\ldots,\mbf U_k^{(n)}) \sim \mu_\beta^{N,(\theta_1^{(n)},\ldots,\theta_k^{(n)})}$, then  $(\mbf U_1^{(n)},\ldots,\mbf U_k^{(n)})$ converges in distribution to $(\mbf U_1,\ldots,\mbf U_k)$. 
\end{enumerate}
\end{proposition}

Proposition \ref{prop:disc_consis} is proved at the very end of Section \ref{sec:Burke_intertwine}. We now state the theorem for the jointly invariant measures in this semi-discrete model. We do not include a proof of the, undoubtedly true, 1F1S principle in this semi-discrete setting.
\begin{theorem}[Jointly invariant measure for \eqref{eq:joint_OCY}] \label{thm:OCY_joint}
For all $\beta > 0$, $k\in \N$ and $(\theta_1,\ldots,\theta_k) \in \R^k$, $\mu_\beta^{N,(\theta_1,\ldots,\theta_k)}$ is a jointly invariant measure for \eqref{eq:joint_OCY}.
\end{theorem}

This result is the analogue of Theorem \ref{thm:KPZ_invar_main} and used to prove it in Section \ref{sec:proofmaintthm}. Theorem \ref{thm:OCY_joint} is proved in Section \ref{sec:OCY_full_proof}. However, the main idea behind its proof and the origin of the structure of the jointly invariant measures can be already seen by considering the simple case of $N=k=2$, as we do now.

\subsection{An informative special case of the proof of Theorem \ref{thm:OCY_joint}}\label{sec.simplecase}
We give here a proof of Theorem \ref{thm:OCY_joint} in the case $N = k = 2$ and for $\theta_1 \neq \theta_2$. We also set $\beta = 1$ for ease of notation. In this simple case, we can derive our description of the jointly invariant measures from first principles, after making an ansatz based on prior work on joint stationary measures. This case also informs our understanding of how to formulate and check the jointly invariant measures for general $k$ and $N$.

Consider $\bigl(\mbf U_1(t),\mbf U_2(t)\bigr)_{t\ge 0} = \Bigl(\bigl(U_{1,0}(t),U_{1,1}(t)\bigr),\bigl(U_{2,0}(t),U_{2,1}(t)\bigr)\Bigr)_{t \ge 0}$ evolving via the equation \eqref{eq:joint_OCY} with $k=N=2$. Because $\vecsum\bigl(\mbf U_1(t)\bigr)$ and $\vecsum\bigl(\mbf U_2(t)\bigr)$ are conserved quantities, knowing the evolution of $U_1(t):=U_{1,0}(t)$ and $U_2(t):=U_{2,0}(t)$, together with the knowledge of $\theta_1:=\vecsum\bigl(\mbf U_1(0)\bigr)$ and $\theta_2:=\vecsum\bigl(\mbf U_2(0)\bigr)$ (which we assume are deterministic), is sufficient to determine the evolution of the full system. In particular,we find that by letting  $B(t) := \f{1}{\sqrt 2}\bigl(B_0(t) - B_1(t)\bigr)$ we can rewrite \eqref{eq:joint_OCY} as (with $\beta=1$)
\be \label{eq:U1U2eq}
\begin{aligned}
dU_1(t) &= \bigl(e^{-U_1(t)} - e^{-(\theta_1 - U_1(t))}\bigr)\,dt + \sqrt 2 dB(t) \\
dU_2(t) &= \bigl(e^{-U_2(t)} - e^{-(\theta_2 - U_2(t))}\bigr)\,dt + \sqrt 2 dB(t).
\end{aligned}
\ee
Let us first consider the marginally invariant measures. When restricted to $U_1$, the generator of \eqref{eq:U1U2eq} is
\[
\mathcal Lf(x) = \bigl(e^{-x} - e^{-(\theta_1 -x)}\bigr) f'(x) + f''(x),
\]
so to check that $p$ is invariant, we verify that 
$p$ satisfies the stationary Fokker-Planck equation
\be \label{eq:simple_FPE}
\mathcal L^\star p := -\df{d}{dx}\Bigl[(e^{-x} - e^{-(\theta_1 - x)}) p(x)\Bigr] + p''(x) = 0,
\ee
where $\mathcal L^\star$ is the formal adjoint of $\Ll$. To prove that proving \eqref{eq:simple_FPE} is sufficient to show invariance requires a nontrivial amount of explanation, due to the fact that the coefficients of the equation are unbounded. This is handled for the general $N,k$ case in Appendix \ref{appx:invariance}. 
The unique probability density satisfying \eqref{eq:simple_FPE} is $p_{\theta_1}(x)$ where
$$p_{\theta}(x) := C\exp\Bigl(-e^{-x} - e^{-(\theta - x)}\Bigr)$$ 
and $C>0$ is for normalization. This is the density at $U_{1,0} = x$ when $U_{1,0}\sim \nu^{N,(\theta_1)}_1$ (Definition \ref{def:p_meas}). 

Inspired by previous results on jointly invariant measures (see, e.g., \cite{Ferrari-Martin-2005,Ferrari-Martin-2007,Ferrari-Martin-2009,Martin-2020,Fan-Seppalainen-20,Seppalainen-Sorensen-21b,Bates-Fan-Seppalainen,GRASS-23} and Appendix \ref{sec:comparison}), we make the ansatz that the invariant initial conditions for \eqref{eq:U1U2eq} have the form 
\begin{align*}
\bigl(U_1(0),&U_2(0)\bigr) = \bigl(X_1(0),F(X_1(0),X_2(0))\bigr) 
\quad \textrm{for some } F:\R^2 \to \R, 
\\&\textrm{where }X_1(0)\sim p_{\theta_1}\textrm{ and }X_2(0)\sim p_{\theta_2}\textrm{ are independent}.
\end{align*}
To show this is the case, it suffices to:
\begin{enumerate} [label=\textup{(\roman*)}]
\item \label{itm:X1proc} Show that there exists a diffusion process $\bigl(X_1(t),X_2(t)\bigr)_{t \ge 0}$ with invariant  density $p_{\theta_1}(x_1)p_{\theta_2}(x_2)$.  
\item \label{itm:F_map_inv} Using It\^o's lemma, find a function $F:\R^2 \to \R$ so that the image of the diffusion from \ref{itm:X1proc} under the map $(X_1,X_2) \mapsto \bigl(X_1,F(X_1,X_2)\bigr)$ satisfies \eqref{eq:U1U2eq}. 
\end{enumerate} 
By triangularity of the map in \ref{itm:F_map_inv}, we see that the SDE for $X_1(t)$ must be the same as that for $U_1(t)$. Simplifying things further, we postulate that the diffusion equation for $\bigl(X_1(t),X_2(t)\bigr)$ is of the form
\be \label{eq:X_sys}
\begin{aligned}
    dX_1(t) &= \bigl(e^{-X_1(t)} - e^{-(\theta_1 - X_1(t))}\bigr)\,dt + \sqrt 2 dB(t) \\
    dX_2(t) &=  A\bigl(X_1(t),X_2(t)\bigr)\,dt + \sqrt 2 dB(t)
\end{aligned}
\ee
for some function $A:\R^2 \to \R$ (i.e., both components have the diffusion coefficient $\sqrt 2$ as in the original system of equations). This system of SDEs' generator is
\[
(e^{-x_1} - e^{-(\theta_1 - x_1)})\partial_{x_1} + A(x_1,x_2)\partial_{x_2} + \sum_{m,r = 1}^2 \partial_{x_m}\partial_{x_r}
\]
and thus for \ref{itm:X1proc} to hold, $A$ must be such that the following 2-dimensional Fokker-Planck equation is satisfied:
\[
\df{\partial}{\partial x_1}\Bigl[(e^{-x_1} - e^{-(\theta_1 - x_1)})p_{\theta_1}(x_1)\Bigr]p_{\theta_2}(x_2) + \df{\partial}{\partial x_2}\Bigl[A(x_1,x_2)p_{\theta_2}(x_2)\Bigr]p_{\theta_1}(x_1) = \sum_{r,m = 1}^2\f{\partial^2}{\partial_{x_m}\partial_{x_r}} p_{\theta_1}(x_1) p_{\theta_2}(x_2). 
\]
Observe that, $\df{\partial}{\partial x} p_{\theta}(x) = \bigl(e^{-x} - e^{-(\theta - x)}\bigr) p_{\theta}(x)$. Hence, the right-hand side above is
\[
\df{\partial}{\partial x_1}\Bigl[(e^{-x_1} - e^{-(\theta_1 - x_1)})p_{\theta_1}(x_1) \Bigr]p_{\theta_2}(x_2) + \df{\partial}{\partial x_2} \Bigl[ \Bigl(e^{-x_2} - e^{-(\theta_2 - x_2)} + 2e^{-x_1} - 2e^{-(\theta_1 - x_1)}\Bigr)p_{\theta_2}(x_2)  \Bigr] p_{\theta_1}(x_1). 
\]
Comparing this to the left-hand side, we see that $A$ is unique (up to  the possible addition of $\f{f(x_1)}{p_{\theta_2}(x_2)}$ for some function $f:\R \to \R$) and given by 
\be \label{eq:Adef}
A(x_1,x_2) = e^{-x_2} - e^{-(\theta_2 - x_2)} + 2(e^{-x_1} - e^{-(\theta_1 - x_1)}).
\ee
What remains is to find a function $F$ (as described in \ref{itm:F_map_inv}) so that  
$(U_1(t),U_2(t)) := \Bigl(X_1(t),F\bigl(X_1(t),X_2(t)\bigr)\Bigr)$ solves \eqref{eq:U1U2eq}. The nontrivial part is to check the equation for $U_2(t)$ since $U_1(t)$ trivially satisfies the desired SDE. By It\^o's lemma (see, e.g., \cite[Theorem 5.10]{LeGall-book}),
\be \label{U2_Ito}
\begin{aligned}
dU_2(t) &=\Biggl( \Bigl(e^{-X_1(t)} - e^{-(\theta_1 - X_1(t))}  \Bigr) \df{\partial F}{\partial x_1}\bigl(X_1(t),X_2(t)\bigr) + A\bigl(X_1(t),X_2(t)\bigr) \df{\partial F}{\partial x_2}\bigl(X_1(t),X_2(t)\bigr)  \\
&+ \f{1}{2}\text{tr}\Bigl(\sigma^T H\bigl(X_1(t),X_2(t)\bigr) \sigma\Bigr) \Biggr)dt+\biggl(\df{\partial F}{\partial x_1}\bigl(X_1(t),X_2(t)\bigr)   + \df{\partial F}{\partial x_2}\bigl(X_1(t),X_2(t)\bigr) \biggr)\sqrt 2 dB(t),
\end{aligned}
\ee
where $H(x_1,x_2)$ is the Hessian matrix of the function $F(x_1,x_2)$, and 
$\sigma = \begin{pmatrix} \sqrt 2 &0 \\
\sqrt 2 &0
\end{pmatrix}$ 
is the diffusion matrix of the system (as there is only one Brownian motion, $\sqrt 2 dB$ present). To match the diffusion term with that in \eqref{eq:U1U2eq}, $F$ must satisfy the transport equation
\be \label{transport}
\df{\partial F}{\partial x_1}(x_1,x_2) + \df{\partial F}{\partial x_2}(x_1,x_2) = 1,
\ee
whose general solution is 
$
F(x_1,x_2) = x_1 + \phi(x_1 - x_2) 
$
for any function $\phi: \R\to \R$. From $\eqref{transport}$, we see that
\[
\df{\partial^2 F}{\partial x_1^2} = \df{\partial^2 F}{\partial x_2^2} = -\df{\partial^2 F}{\partial x_1 \partial x_2},
\]
and thus $\sigma^T H(x_1,x_2) \sigma$ is the zero matrix for all $(x_1,x_2) \in \R^2$. Hence, the drift term in \eqref{U2_Ito} simplifies to
 \begin{align*} 
&\quad \, \Bigl(e^{-X_1(t)} - e^{-(\theta_1 - X_1(t))}\Bigr)\Bigl(1 + \phi'\bigl(X_1(t) - X_2(t)\bigr) \Bigr) - \phi'\bigl(X_1(t) - X_2(t)\bigr) A\bigl(X_1(t),X_2(t)\bigr)
 \\
 &= \Bigl(e^{-X_1(t)} - e^{-(\theta_1 - X_1(t))}\Bigr) - \phi'\bigl(X_1(t) - X_2(t)\bigr)\Bigl(e^{-X_1(t)} - e^{-(\theta_1 - X_1(t))} + e^{-X_2(t)} - e^{-(\theta_2 - X_2(t))}\Bigr).
 \end{align*}
 To recover \eqref{eq:U1U2eq}, $\phi$ must be chosen so the right-hand side above equals
\[
e^{-X_1(t) - \phi(X_1(t)-X_2(t))} - e^{-\bigl(\theta_2 - X_1(t) - \phi(X_1(t) - X_2(t))\bigr)}.
\]
For this equality to hold, it suffices to break this up into two equations and prove that, for $x_1,x_2 \in \R$,
\be \label{eq:splitansatz}
\begin{aligned}
e^{-x_1 - \phi(x_1-x_2)} &= e^{-x_1} - \phi'(x_1 - x_2)(e^{-x_1} + e^{-x_2}),\qquad \text{and} \\
e^{-(\theta_2 - x_1 - \phi(x_1 - x_2))} &= e^{-(\theta_1 - x_1)} - \phi'(x_1 - x_2)(e^{-(\theta_1 - x_1)} + e^{-(\theta_2 - x_2)} ).
\end{aligned}
\ee
Observe that if we set $\chi = x_1 - x_2$ then these equations are equivalent to two ODEs in the single variable $\chi$:
$$
e^{-\phi(\chi)} = 1 - \phi'(\chi)(1 + e^{\chi}), \qquad\textrm{and}\qquad
    e^{\phi(\chi)} = e^{\theta_2 - \theta_1} - \phi'(\chi)(e^{\theta_2 - \theta_1} + e^{-\chi}).
$$
The unique solution that satisfies both equations is
\[
\phi(\chi) = \log\Biggl(\f{1 + e^{\theta_2 - \theta_1 + \chi}}{1 + e^\chi}\Biggr).
\]
Thus we have found a solution to our ansatz. Let us compare the outcome to the definition of the map  $D^{2,2}$ \eqref{eq:D_intro}. For $\mbf X_1 = (X_{1,0},X_{1,1}) := (x_1,\theta_1 - x_1)$ and $\mbf X_2 := (x_2,\theta_2 - x_2)$, we have
\begin{align*}
 D_0^{2,2}(\mbf X_1,\mbf X_2) &= x_2 + \log\Biggl(\f{1 + e^{\theta_2 - x_2 - (\theta_1 - x_1)}}{1 + e^{x_2 - x_1}}\Biggr) = x_1 + (x_2 - x_1) + \log\Biggl(\f{1 + e^{\theta_2 - x_2 - (\theta_1 - x_1)}}{1 + e^{x_2 - x_1}}\Biggr) \\
&= x_1 + \log\Biggl(\f{1 + e^{\theta_2 - \theta_1 + x_1 - x_2 }}{1 + e^{x_1 - x_2}}\Biggr)  = x_1 + \phi(x_1 - x_2) = F(x_1,x_2),
\end{align*}
and so, we recover the $N = k = 2$ case of Theorem \ref{thm:OCY_joint}. 

\subsection{Challenges in moving beyond the special case}\label{sec:challenges}
Working out this special case was our starting point in the course of writing this paper.  The major difficulty in finding the jointly invariant measures in the case of general $N,k \ge 2$ lies in proving the theorem for $k = 2$ and general $N \ge 2$. Once that case is covered, the previous works \cite{Ferrari-Martin-2005,Ferrari-Martin-2007,Ferrari-Martin-2009,Martin-2020,Fan-Seppalainen-20,Seppalainen-Sorensen-21b,Bates-Fan-Seppalainen,GRASS-23} suggest how one should iterate the transformations as in \eqref{eq:D_iter_intro}.  For $k=2$ and general $N \ge 2$, we follow the same two steps outlined in Items \ref{itm:X1proc}-\ref{itm:F_map_inv}, which is the general same technique introduced originally in \cite{Ferrari-Martin-2007}. 

Generating a natural candidate for the dual system of SDEs needed in Item \ref{itm:X1proc} was not terribly difficult. By making a similar generator calculation as in $N=2$, one sees that
  \be \label{eq:XR_intro}
  d X_{r,i}(t) = \Bigl(\diff_i e^{-\mbf X_r(t)} - \sum_{m = 1}^{r-1}\Lapl_i e^{-\mbf X_m(t)} \Bigr)\,dt 
 + \beta\diff_i d\mbf B(t),\qquad 1 \le r \le 2, i \in \Z_N,
  \ee
has $\nu_\beta^{N,(\theta_1,\theta_2)}$ as its invariant measure (recall $\diff_i$ and $\Lapl_i$ notation from \eqref{eq.discretederiv}).
It is worth noting a difference here to previous works. In each of the previous works mentioned above, there is a natural Burke property embedded into the model which implies that the evolution of the Markov process produces ``dual" driving weights; and the analogue of the equation for $\mbf X_2$ is the same as for the original process, but driven by those dual weights. We do not see such an interpretation in this setting. We also note that because the drift coefficients in \eqref{eq:XR_intro} are unbounded and the diffusion matrix is degenerate (we only use $N$ Brownian motions for a $2N$-dimensional system), there are nontrivial technical details to rigorously prove that the product measure is invariant for \eqref{eq:XR_intro}. These are worked out in Appendix \ref{appx:invariance}.  
 
Finding the general $N$ mapping needed for Item \ref{itm:F_map_inv} was more difficult and is achieved in Proposition \ref{prop:NewSDE}. Previous works suggested triangularity of the transformation, but did not pin it down beyond that. An initial guess was to take the same transformation as used for the inverse-gamma polymer in the full-line setting of \cite{Bates-Fan-Seppalainen} but with periodic inputs. This failed, so we sought a different means to produce the transform.

In the $N=2$ case of Item \ref{itm:F_map_inv} worked out in Section \ref{sec.simplecase}, there were two key simplifications that allowed the proof to work. The first was to postulate that the drift term in the dual system \eqref{eq:X_sys} is the same as the drift term in the original equation \eqref{eq:U1U2eq}. This led to the transport equation in \eqref{eq:Adef}, which implied that the second-order terms in It\^o's lemma vanish. An analogous simplification holds for general $N$. See in particular, \eqref{FGder} and \eqref{eq:H_transport} in the proof of Proposition \ref{prop:NewSDE}, where it is shown that this form of the transform produces the same diffusion terms and a zero second-order term in the multi-dimensional It\^o's lemma.  The second simplification came by postulating that $\phi$ solved the two equations in \eqref{eq:splitansatz}. This simplified the problem as follows: Before, we needed to find a function $\phi$ that satisfied an equation involving two variables. Equation \eqref{eq:splitansatz} broke this up into two equations, each of which could be interpreted as an ODE in a single variable. We do not know of an analogue of this simplification of \eqref{eq:splitansatz} for general $N$. However, because a multivariate transport equation does arise in the general case, we knew that the output of the map $D_i^{N,2}(\mbf X_1,\mbf X_2)$ should be equal to $X_{2,i}$ plus some function of the sequence of differences $(X_{2,\ell} - X_{1,\ell})_{\ell \in \Z_N}$. Motivated by the known structure for the discrete geometric Pitman transform on the real line (see \cite[Definition 2.3]{CorwinInvariance}), this led to a natural conjecture for extending the $N = 2$ case, which we could then verify directly. 

After determining $D_i^{N,2}$ we were able to recognized this as a suitable shifted version of the periodized transformation from \cite{Bates-Fan-Seppalainen}; see Appendix \ref{sec:comparison} for details. We do not have an explanation for why this must be the form of the transformation, though it makes for a nice periodic generalization of the full line picture.

\section{Jointly invariant measures for a periodic discrete-time Markov chain} \label{sec:disc_MC} The algebraic structure we uncover in Proposition \ref{prop:disc_consis} results in the following by-product: the measures $\mu_\beta^{N,(\theta_1,\ldots,\theta_k)}$ are also jointly invariant for a discrete-time Markov chain described now. The remainder of the paper may be read independently of this section. 

Let $(\mbf U_1^{(m)},\ldots,\mbf U_k^{(m)}) \in (\R^{\Z_N})^k$ be the state of the chain at time $m$, and let $\mbf W\in \R^{\Z_N}$ be independent of the chain up to time $m$. Then, the time $m + 1$ state is defined as 
\be \label{eq:coupled_disc_MC}
(\mbf U_1^{(m+1)},\ldots,\mbf U_k^{(m+1)}) = \bigl(D^{N,2}(\mbf W,\mbf U_1^{(m)}),\ldots,D^{N,2}(\mbf W,\mbf U_k^{(m)}) \bigr).
\ee 
The same driving sequence $\mbf W$ is applied to each component hence this Markov chain is the coupled evolution of the $k = 1$ case of the chain, using the same driving weights. We will consider two cases of the distribution of the driving weights $\mbf W \in \R^{\Z_N}$, both leading to the same jointly invariant measures.

\begin{theorem} \label{thm:disc_MC}
Let $\alpha \in \R,\gamma,\beta > 0$. Consider the coupled Markov chain \eqref{eq:coupled_disc_MC} when the driving weights $\mbf W \in \R^{\Z_N}$ are either 
(i) a sequence of $N$ i.i.d. log-inverse-gamma random variables with a fixed shape $\gamma > 0$ and scale $\beta^{-2}>0$ and (ii) that sequence conditioned on their sum to equal $\alpha$ for some $\alpha\in \R$, i.e. having law $\nu_\beta^{N,(\alpha)}$. Then, for any $(\theta_1,\ldots,\theta_k) \in \R^k$, $\mu_\beta^{N,(\theta_1,\ldots,\theta_k)}$ is invariant for this Markov chain. 
\end{theorem}
The proof of Theorem \ref{thm:disc_MC} is in Section \ref{sec:proof_disc_MC}. The key to it is the relation in Proposition \ref{prop:full_intertwine}, which we use to show that this Markov chain intertwines with a dual Markov chain (see Proposition \ref{prop:DNk_disc_intertwine}) whose invariant measure is $\nu_\beta^{N,(\theta_1,\ldots,\theta_k)}$. Once Proposition \ref{prop:full_intertwine} is shown, the proof follows the same procedure as in \cite{Fan-Seppalainen-20,Seppalainen-Sorensen-21b,Bates-Fan-Seppalainen,GRASS-23}.

\subsection{Polymer interpretation}\label{sec:polymerinterpretation}
The discrete-time Markov chain in \eqref{eq:coupled_disc_MC} can be related to a directed polymer in a periodic environment as follows. Let $\bigl(W_{r,i}\bigr)_{r \in \Z, i \in \Z_N}$ be real-valued and extend these to a periodic environment $\bigl(W_{r,i}\bigr)_{r \in \Z,i \in \Z}$ by the condition that $W_{r,i} = W_{r,j}$ whenever $i \equiv j \mod N$. For $r \in \Z$, let $\mbf W_r:=(W_{r,i})_{i \in \Z_N}$. For integers $m \le r$ and $i + (r-m) \le j$, define the allowable paths from $(m,i)$ to $(r,j)$
\[
\Pi_{(m,i),(r,j)} = \Bigl\{(\mbf x_\ell)_{0 \le \ell \le n}: \mbf x_0 = (m,i),\mbf x_n = (r,j), \mbf x_\ell - \mbf x_{\ell - 1} \in \{(0,1),(1,1)\} \Bigr\},
\]
noting that they take steps up or up/right (see the left side of Figure \ref{fig:disc_per_environment}). Define
\be \label{eq:Zdis}
Z^N(r,j \viiva m,i) := \sum_{\pi \in \Pi_{(m,i),(r,j)}} \prod_{(n,\ell) \in \pi}  {e^{W_{n,\ell}}}
\ee
to be the associated point to point polymer partition function.

\begin{figure}
    \centering
    \includegraphics[height = 8cm]{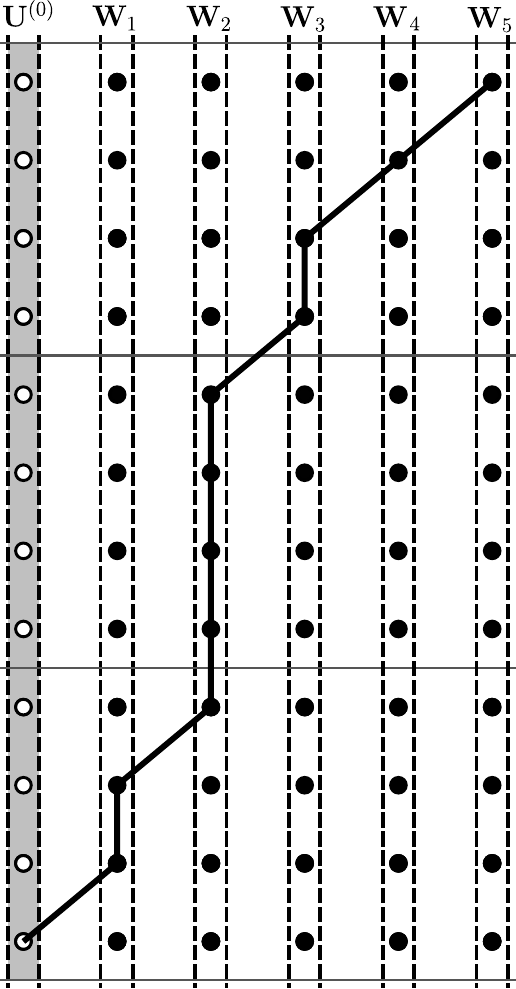}\hspace{2cm}
    \includegraphics[height = 8cm]{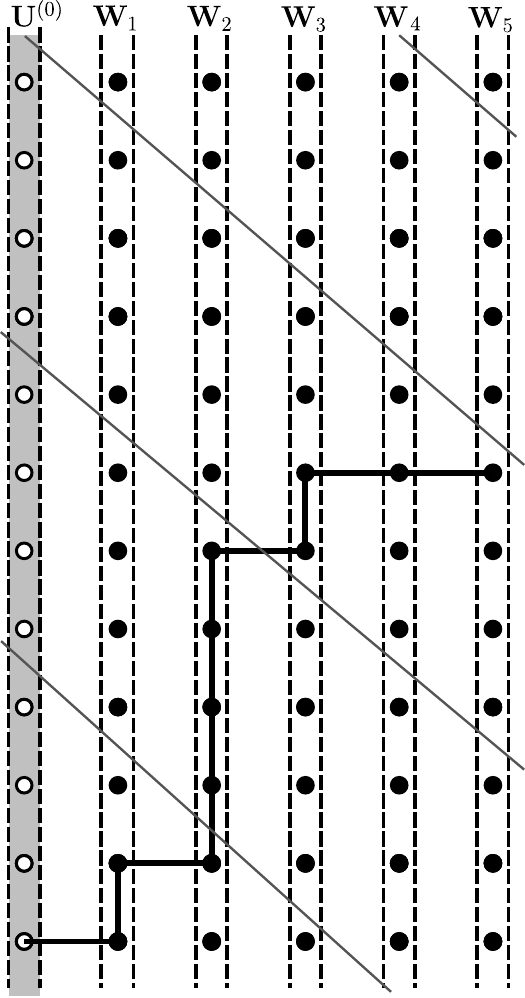}
    \caption{On the left is an illustration of the polymer defined in \eqref{eq:Zdis}. The bulk weights are given by the vectors $\mbf W_1,\ldots, \mbf W_5\in \R^{\Z_N}$ with $N=4$. The horizontal gray lines signify where this environment periodically repeats. The first column (shaded in gray) contains the initial condition $\mbf U^{(0)}$, also periodically repeated. The allowable paths (thick black) considered in defining the partition function for this model involve up or up/right steps and the weight assigned to a path is the exponential of the sum of all bulk weights along the path and the cumulative sum (measured from height zero) of the initial condition. The figure on the right is the result of shearing this model so that the allowable paths are those that go up or right. This is the periodic inverse-gamma polymer.}
    \label{fig:disc_per_environment}
\end{figure}

We also define the model with an initial condition $\mbf U^{(0)}= (U_0,\ldots, U_{N-1}) \in \R^{\Z_N}$ as follows (see also Figure \ref{fig:disc_per_environment}). Extend $\mbf U^{(0)}$ to a sequence in $\R^\Z$ by the condition $U^{(0)}_i = U^{(0)}_j$ whenever $i \equiv j \mod N$ and let $F:\Z \to \R_{>0}$ be the function satisfying $F(0) = 1$ and $\f{F(j)}{F(j-1)} = e^{U^{(0)}_j}$ for $j \in \Z$. For $j \in \Z$ and $r > 0$, define 
\be \label{eq:disc_poly_init}
Z^N(r,j \viiva F) := \sum_{i \le j-r} F(i) Z^N(r,j \viiva 1,i+1).
\ee
This sum is convergent if and only if $\vecsum(\mbf U^{(0)}) > \max_{1 \le r \le m} \vecsum(\mbf W_r)$ (see Lemma \ref{lem:inv_gamma}\ref{itm:finsumcond}). If this holds, then for $m \ge 1$ and $i \in \Z$, set
\be \label{eq:poly_part}
U^{(m)}_i := \log \f{Z^N(m,i \viiva F)}{Z^N(m,i-1 \viiva F)},
\ee

It is shown in Lemma \ref{lem:inv_gamma}\ref{itm:Xper} that, if $U^{(m)}_i$ is well-defined, then $U^{(m)}_i = U^{(m)}_j$ whenever $i \equiv j \mod N$. Thus, under this condition we may consider the vector $\mbf U^{(m)}$ as an element of $\R^{\Z_N}$.  Furthermore, Lemma \ref{lem:inv_gamma}\ref{itm:DN2map} shows that
\[
\mbf U^{(m)} = D^{N,2}\bigl(\mbf W_m,\mbf U^{(m-1)}\bigr).
\]
Thus, the $k=1$ Markov chain \eqref{eq:coupled_disc_MC} describes the evolution of the ratios of the discrete polymer with initial condition $\mbf U^{(0)}$, and weights $\mbf W_1,\ldots,\mbf W_m$ used at each time step of the chain. The general $k$ Markov chain comes from considering this evolution for a common environment but with respect to different initial conditions $\mbf U^{(0)}_1,\ldots, \mbf U^{(0)}_k$. 
It is important to note that the Markov chain \eqref{eq:coupled_disc_MC} defined from the map $D^{N,2}$ does not involve an infinite sum and is therefore well-defined for any choice of driving weights $\mbf W_1,\ldots,\mbf W_m$ and any initial condition $\mbf U^{(0)}$. On the other hand, if we let $(\mbf W_r)_{r \ge 1}$ be independent i.i.d. sequences of log-inverse-gamma random variables, then with probability one, there exists some $m \ge 1$ so that $\vecsum(\mbf U^{(0)}) < \vecsum(\mbf W_m)$. In particular, $Z(r,j \viiva F)$ blows up in finite time, and the \eqref{eq:poly_part} is only well-defined on an event of probability strictly less than one. Perhaps one can make sense of the log ratio of these partition functions even when their values of infinite. In essence, that is achieved by \eqref{eq:coupled_disc_MC}.

If, instead, we let the driving weights $\mbf W_r$ have distribution $\nu_\beta^{N,(\alpha)}$, then \eqref{eq:poly_part} is well-defined for all $m$ as long as $\vecsum(\mbf U^{(0)}) > \alpha$, and from Theorem \ref{thm:disc_MC}, $\mu_\beta^{N,(\theta_1,\ldots,\theta_k)}$ is a jointly invariant measure for the coupled ratios of the polymer partition function as long as $\theta_\ell  > \alpha$ for $1 \le \ell \le k$. Of course, using the definition of the Markov chain in \eqref{eq:coupled_disc_MC}, there is no condition on the $\theta_\ell$.

In light of the above polymer interpretation for the Markov chain \eqref{eq:coupled_disc_MC}, it is quite natural to conjecture that under a suitable scaling limit this chain convergences to \eqref{eq:joint_OCY}, which itself is related to log ratios of partition functions for the semi-discrete O'Connell-Yor polymer, see \eqref{Uudef} in Section \ref{sec:OCY_def}. In seeing this, it helps to perform a shear transformation as illustrated between the left and right of Figure \ref{fig:disc_per_environment}.
For $m\in \Z$, define the operator $\tau_m: \R^{\Z_N} \to \R^{\Z_N}$ by 
$(\tau_m \mbf U)_m = X_{i+m}$, and then define a new Markov chain $\tilde{\mbf U}^{(m)} := \tau_{-m} \mbf U^{(m)}$. It follows from Corollary \ref{cor:shift} that the jointly invariant measures for this chain are the same as those given in Theorem \ref{thm:disc_MC} for $\mbf U^{(m)}$. The shearing involved in defining $\tilde{\mbf U}^{(m)}$ leads to a polymer description in which the admissible paths are up or right (i.e., increments $(0,1)$ or $(1,0)$); otherwise the environment and initial condition are defined exactly the same (this uses the fact that the environment is invariant in law under the $\tau$ shearing). This just defined polymer model is a periodic version of the inverse-gamma polymer. See Section \ref{sec:comparison} for further discussion on the inverse-gamma polymer.

In light of the fact that the full-line inverse-gamma polymer has a scaling limit to the corresponding full-line O'Connell-Yor polymer (see Section \ref{sec:OCY}), it is natural to conjecture the same is true of the periodic models. Proving this is complicated by the fact (as discussed above and shown in Lemma \ref{lem:inv_gamma}\ref{itm:finsumcond}) that the quantity \eqref{eq:disc_poly_init} will explode almost surely in finite time. Presumably under the appropriate scaling, the time scale in which this occurs outpaces the time scaling in which the O'Connell-Yor polymer arises. 

\begin{conj} \label{conj:disctoSDE}
For fixed $N$, there should be a suitable scaling of time and tuning of the parameters $\gamma$ under which the coupled Markov chain \eqref{eq:coupled_disc_MC} converges to the system of SDEs \eqref{eq:joint_OCY} and under this scaling  the dual Markov chain \eqref{eq:multline} converges to the dual system of SDEs \eqref{DxR_gen}. Similarly, there should be a scaling of time and $N$ to infinity along with a tuning of $\gamma$ under which the  coupled Markov chain  \eqref{eq:coupled_disc_MC} converges directly to coupled solutions to the KPZ equation \eqref{eq:KPZ} with the same white noise $\xi$.
\end{conj}
Note that while we have a clear sense of what the fixed $N$ scaling limit should be of the dual Markov chain, it is a compelling question to determine what becomes of the dual Markov chain in the KPZ equation limit. Likewise, this question can be posed in terms of understanding the scaling limit of the dual system of SDEs \eqref{DxR_gen} under the KPZ equation scaling. It should be noted, though, that knowledge of this dual scaling limit is unnecessary from the perspective of proving the jointly invariant measures.

Resolution of the first statement of Conjecture \ref{conj:disctoSDE} would yield an alternative method of proving Theorem \ref{thm:OCY_joint}. Resolution of the second statement combined with the type of arguments in Section \ref{sec:proofmaintthm} (e.g. Lemma \ref{lem:initial_data_conv}) would provide an alternative proof to Theorem \ref{thm:KPZ_invar_main}.

\section{Algebraic structure of maps, and proof of Proposition \ref{prop:disc_consis} and Theorem \ref{thm:disc_MC}}\label{sec:alg}

As stated in Theorem \ref{thm:OCY_joint}, the semi-discrete jointly invariant initial conditions (i.e., invariant measures for \eqref{eq:joint_OCY}) are given by the push-forward under a map called  $\mathcal D^{N,k}$ of independent conditioned sequences of i.i.d. log-inverse-gamma random variables. In this section, we investigate the algebraic structure of this and related transformations that drive the proof of Theorem \ref{thm:OCY_joint} (given in Section \ref{sec:OCY}) and the key consistency and symmetry property of Proposition \ref{prop:disc_consis} (given at the end of this section).

\subsection{The $\mathcal D^{N,k}$ bijection}\label{sec.sst}
We will consider vectors of the form $\mbf U = (U_i)_{i \in \Z_N} \in \R^{\Z_N}$, meaning that $U_i = U_j$ whenever $i \equiv j \mod N$.
For vectors $\mbf U = (U_i)_{i \in \Z_N},\mbf U' = (U_i')_{i \in \Z_N} \in \R^{\Z_N}$, we say that $\mbf U < \mbf U'$ if $U_i < U_i'$ for all $i \in \Z_N$.
The following sets will be important in describing the domains and ranges of our transforms: 
\be \label{RA2}
\begin{aligned}
\mathcal R^{N,1}&:=  \big\{\mbf U\in \R^{\Z_N}\big\},\\
\mathcal R^{N,2} &:= \big\{(\mbf U_1,\mbf U_2)\in (\R^{\Z_N})^2: \mbf U_1 < \mbf U_2 \text{ or }\mbf U_1 > \mbf U_2\big\},\\
\mathcal R^{N,2}_{\theta_1,\theta_2} &:= \big\{(\mbf U_1,\mbf U_2) \in \mathcal R^{N,2}: \vecsum(\mbf U_m) = \theta_m,\;\; m \in \{1,2\}\big\}\quad \text{for $\theta_1 \neq \theta_2$}\\
\R^{\Z_N}_\theta &:= \big\{\mbf X  \in \R^{\Z_N}: \vecsum(\mbf X) = \theta \big\}\\
\R^{N,k}_{\neq} &:= \big\{(\mbf X_1,\ldots,\mbf X_k) \in (\R^{\Z_N})^k: \vecsum(\mbf X_m) \neq \vecsum(\mbf X_r) \text{ for } r \neq m \big\}.
\end{aligned}
\ee
All of these sets are $G_\delta$ subsets of their respective product spaces, and hence they are Polish spaces under the subspace topology by Alexandrov's Theorem (see, e.g., \cite[Theorem 2.2.1]{Srivastava-1998}).

Recall the definitions of the maps $D^{N,m}: (\R^{\Z_N})^m \to \R^{\Z_N}$, $1\leq m\leq k$ and  $\D^{N,k}:(\R^{\Z_N})^k  \to (\R^{\Z_N})^k$ from \eqref{eq:D_intro}, \eqref{eq:D_iter_intro}, and \eqref{DNk_intro}.
We will show in Proposition \ref{prop:transform} that $\mathcal D^{N,k}$ is injective if we restrict to the domain $\R^{N,k}_{\neq}$ and we will also describe its inverse map $\mathcal J^{N,k}$. Recall from \eqref{eq:D_iter_intro}, and \eqref{DNk_intro} that $\mathcal D^{N,k}$ is built from iterative application of the basic building block $D^{N,2}$. This can be represented graphically, as in the left half of Figure \ref{fig:DJ} (see also the caption). Similarly, the inverse map $\mathcal J^{N,k}$ is built from iterative application of a basic building block $J^{N,2}$ (see the right half of the Figure). The proof that these maps are inverses  can also be represented graphically as in Figure \ref{fig:bijection}, and relies on showing that $D^{N,2}$ and $J^{N,2}$ are mutual inverses, when restricted to the right spaces.

\begin{figure}
    \centering
    \includegraphics[width=.9\linewidth]{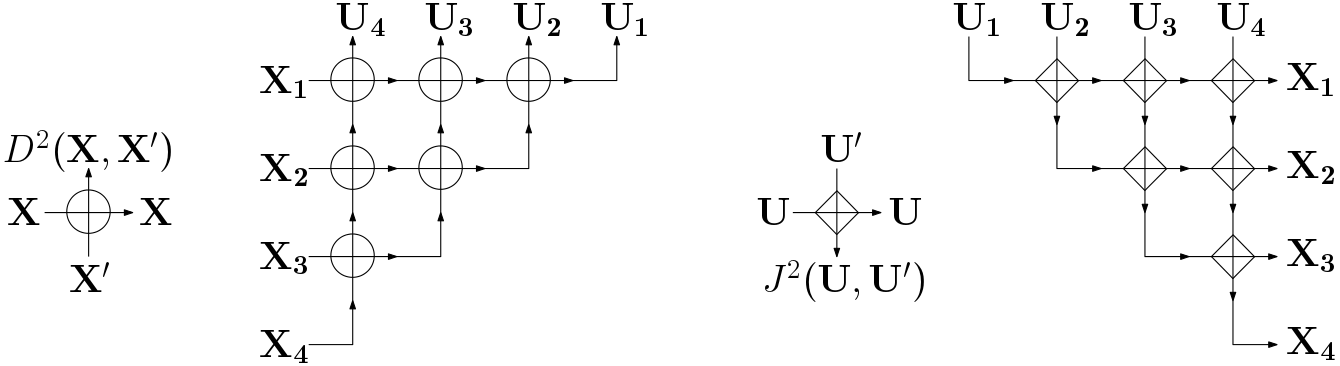}
    \caption{The first figure on the left graphically defines via a circle the transformation from the pair $(\mbf X,\mbf X')$ to $(D^{N,2}(\mbf X,\mbf X'),\mbf X)$ (we drop the $N$ superscript here an in all of the graphical map figures). The direction of the arrows go from input to output, and at times below we will use rotated or reflected versions of this and the $J$ building blocks to mean the same map. We will, however, always maintain the convention that the blocks are oriented so that the horizontal input passes through unchanged to become the horizontal output. The map  $\mathcal D^{N,k}$ with $k=4$ is depicted graphically in the second figure on the left as a composition of these building block maps. In particular $(\mbf U_1,\mbf U_2,\mbf U_3,\mbf U_4):=\mathcal D^{N,k}(\mbf X_1,\mbf X_2,\mbf X_3,\mbf X_4)$ is the output, given input $(\mbf X_1,\mbf X_2,\mbf X_3,\mbf X_4)$. The output of each circle maps into the inputs of the next map. The first figure on the right depicts via a diamond the transformation of a pair $(\mbf U,\mbf U')$ to $J^{N,2}(\mbf U,\mbf U')$ with the input down on the left and top and outputs on the right and bottom. The map $\mathcal J^{N,k}$ for $k=4$  is depicted in the rightmost figure  as a composition of these building block maps. In particular $(\mbf X_1,\mbf X_2,\mbf X_3,\mbf X_4):=\mathcal J^{N,k}(\mbf U_1,\mbf U_2,\mbf U_3,\mbf U_4)$ is the output, given input $(\mbf U_1,\mbf U_2,\mbf U_3,\mbf U_4)$.}
    \label{fig:DJ}
\end{figure}

Define $J^{N,2}:\mathcal R^{N,2} \to  \R^{\Z_N}$ so that for $(\mbf U_1,\mbf U_2) \in \mathcal R^{N,2}$, $J^{N,2}(\mbf U_1,\mbf U_2) = (J^{N,2}_i(\mbf U_1,\mbf U_2))_{i \in \Z_N}$, where
\be \label{Jimap}
J_i^{N,2}(\mbf U_1, \mbf U_2) := U_{2,i}+ \log\Bigl(\f{e^{ U_{2,i+1} - U_{1,i+1}} - 1}{e^{U_{2,i} - U_{1,i}} - 1}\Bigr).
\ee
It is essential that the domain of $J^{N,2}$ is $\mathcal R^{N,2}$ since this implies that we will only consider pairs $(\mbf U_1,\mbf U_2)$ satisfying $\mbf U_1 < \mbf U_2$ or $\mbf U_1 > \mbf U_2$, which guarantees that the ratio inside the logarithm is strictly positive.

The restriction on the domain of the building block $J^{N,2}$ to $\mathcal R^{N,2}$ results in a restriction on the domain of $\mathcal J^{N,k}$. We proceed inductively with the base cases $k=1$ and $2$. 
Recalling $\mathcal R^{N,1}$ and $\mathcal R^{N,2}$ from \eqref{RA2}, define 
$$
\mathcal J^{N,1}:\mathcal R^{N,1} \to \R^{\Z_N} \textrm{ by }\mathcal J^{N,1}(\mbf U_1) := \mbf U_1,\quad \mathcal J^{N,2}:\mathcal R^{N,2} \to (\R^{\Z_N})^2 \textrm{ by }
\mathcal J^{N,2}(\mbf U_1,\mbf U_2) = (\mbf U_1,J^{N,2}(\mbf U_1,\mbf U_2)).
$$
Inductively define the map $\mathcal J^{N,k}:\mathcal R^{N,k} \to (\R^{\Z_N})^k$ by setting (we define its domain $\mathcal R^{N,k}$ below)
$$\mathcal J^{N,k+1}(\mbf U_1,\ldots,\mbf U_{k+1}):= (\mbf X_1,\ldots,\mbf X_{k+1}),
$$
where $(\mbf X_1,\ldots,\mbf X_{k})$ is given by setting
\begin{equation}\label{eq.XJ1}
(\mbf X_1,\ldots,\mbf X_{k-1},\mbf X_{k}) = \mathcal J^{N,k}(\mbf U_1,\ldots,\mbf U_{k-1},\mbf U_{k})
\end{equation}
and then specifying $\mbf X_{k+1}$ as
\begin{equation}\label{eq.XJ2}
\mbf X_{k+1} = J^{N,2}(\mbf X_{k},\wt{\mbf X}_{k+1}),\quad\text{where} \quad
(\mbf X_1,\ldots,\mbf X_{k-1},\wt{\mbf X}_{k+1}) = \mathcal J^{N,k}(\mbf U_1,\ldots,\mbf U_{k-1},\mbf U_{k+1}).
\end{equation}
The above definition relies on the consistency relation
$$
(\mbf X_1,\ldots,\mbf X_{k-1}, \mbf X_{k}) = \mathcal J^{N,k}(\mbf U_1,\ldots,\mbf U_{k}) \Longrightarrow (\mbf X_1,\ldots,\mbf X_{k-1}) = \mathcal J^{N,k-1}(\mbf U_1,\ldots,\mbf U_{k-1}),
$$
which clearly holds true for $k=2$. By the inductive definition, this relation continues to hold true for all $k\geq 3$. This inductive definition  matches the graphical construction in Figure \ref{fig:DJ}. The domain $\mathcal R^{N,k+1}$ is also inductively defined to be the set of all $(\mbf U_1,\ldots,\mbf U_{k+1}) \in (\R^{\Z_N})^{k+1}$ such that the arguments in the maps in \eqref{eq.XJ1} and \eqref{eq.XJ2} are in the inductively defined domains. In other words, $(\mbf U_1,\ldots,\mbf U_{k-1},\mbf U_{k}) \in \mathcal R^{N,k}$, 
$(\mbf U_1,\ldots,\mbf U_{k-1},\mbf U_{k+1}) \in \mathcal R^{N,k}$ and 
$(\mbf X_{k},\wt{\mbf X}_{k+1}) \in \mathcal R^{N,2}$
where $\mbf X_{k}$ and $\wt{\mbf X}_{k+1}$ are defined in \eqref{eq.XJ1} and \eqref{eq.XJ2}.

For real numbers $\theta_1,\ldots,\theta_k$, we can further refine $\mathcal R^{N,k}$ by specifying the slope of each $\mbf U_m$:
\[
\mathcal R^{N,k}_{\theta_1,\ldots,\theta_k} = \{(\mbf U_1,\ldots,\mbf U_k) \in \mathcal R^{N,k}: \vecsum(\mbf U_m) = \theta_m, 1 \le m \le k\}.
\]
It follows from Proposition \ref{prop:transform} and Corollary \ref{cor:RNKcontained} that $\mathcal R^{N,k}_{\theta_1,\ldots,\theta_k}\neq \varnothing $ if and only if the $\theta_1,\ldots,\theta_k$ are distinct.

The main result of this section is the following. We delay its proof to state a few useful lemmas.
\begin{proposition} \label{prop:transform}
For $k,N \in \N$, $\D^{N,k}:\R^{N,k}_{\neq} \to \mathcal R^{N,k}$ is a bijection with inverse $\mathcal J^{N,k}$.  Additionally, for distinct  $\theta_1,\ldots,\theta_k\in \R$, $\mathcal D^{N,k}$ restricts to a bijection $\R^{\Z_N}_{\theta_1} \times \cdots \times \R^{\Z_N}_{\theta_k} \to \mathcal R^{N,k}_{\theta_1,\ldots,\theta_k}$.  
\end{proposition}

We note that the function $\mathcal D^{N,k}$ is not injective on the larger domain $(\R^{\Z_N})^k$. Indeed, if $\vecsum(\mbf X_1) = \vecsum(\mbf X_2) = \vecsum(\mbf X_2')$ for $\mbf X_1,\mbf X_2,\mbf X_2' \in \R^{\Z_N}$, then Lemma \ref{lem:sum_v_order} below implies that $\D^{N,2}(\mbf X_1,\mbf X_2) = \D^{N,2}(\mbf X_1,\mbf X_2') = (\mbf X_1,\mbf X_1)$.

The proof of Proposition \ref{prop:transform} is illustrated graphically in Figure \ref{fig:bijection}. Because the domains of the mappings $\D^{N,k}$ and $\mathcal J^{N,k}$ are complicated, we need to prove some intermediate lemmas.  Before getting to these lemmas, we give a formal algebraic sketch of the proof. Starting from $k = 2$, we recall that $\mathcal D^{N,2}(\mbf X_1,\mbf X_2) = (\mbf X_1,D^{N,2}(\mbf X_1,\mbf X_2))$ and $\mathcal J^{N,2} (\mbf U_1,\mbf U_2) = (\mbf U_1,J^{N,2}(\mbf U_1,\mbf U_2))$. Hence, to show these maps are inverses of each other, it suffices to show that for each fixed vector $\mbf X_1 \in \R^{\Z_N}$, we have $\mbf U_2 = D^{N,2}(\mbf X_1,\mbf X_2)$ for some vector $\mbf X_2 \in \R^{\Z_N}$ satisfying $\vecsum(\mbf X_1) \neq \vecsum(\mbf X_2)$, if and only if $\mbf X_2 = J^{N,2}(\mbf X_1,\mbf U_2)$. This is fairly straightforward to verify from the definitions; it is done below Equations \eqref{ZtoY}-\eqref{YtoZ}.

In general, we wish to show that
\be \label{eq:DJ_casek}
(\mbf U_1,\ldots,\mbf U_k) = \D^{N,k}(\mbf X_1,\ldots,\mbf X_k) \iff (\mbf X_1,\ldots,\mbf X_k) = \mathcal J^{N,k}(\mbf U_1,\ldots,\mbf U_k).
\ee
We assume by induction that this equivalence holds for some $k \ge 2$. By definition of the map $\D^{N,k}$ \eqref{DNk_intro}, we have $(\mbf U_1,\ldots,\mbf U_{k+1}) = \D^{N,k+1}(\mbf X_1,\ldots,\mbf X_{k+1})$ if and only if $(\mbf U_1,\ldots,\mbf U_k) = \D^{N,k}(\mbf X_1,\ldots,\mbf X_k)$ and $\mbf U_{k+1} = D^{N,k+1}(\mbf X_1,\ldots,\mbf X_k,\mbf X_{k+1})$. It is straightforward to see from the definition of the map $D^{N,k+1}$ (see Lemma \ref{lem:Diter} below) that $D^{N,k+1}(\mbf X_1,\ldots,\mbf X_k,\mbf X_{k+1}) = D^{N,k}(\mbf X_1,\ldots,\mbf X_{k-1}, D^{N,2}(\mbf X_{k},\mbf X_{k+1}))$. Putting this all together, we have $(\mbf U_1,\ldots,\mbf U_{k+1}) = \D^{N,k+1}(\mbf X_1,\ldots,\mbf X_{k+1})$ if and only if 
\begin{align*}
 &(\mbf U_1,\ldots,\mbf U_{k-1},\mbf U_k) = \D^{N,k}(\mbf X_1,\ldots,\mbf X_{k-1},\mbf X_k),\quad\text{and} \\
 &(\mbf U_1,\ldots,\mbf U_{k-1}, \mbf U_{k+1}) = \D^{N,k}(\mbf X_1,\ldots, \mbf X_{k-1},\wt{\mbf X}_{k+1}),\quad\text{where}\quad \wt{\mbf X}_{k+1} = D^{N,2}(\mbf X_k,\mbf X_{k+1}).
\end{align*}
By applying \eqref{eq:DJ_casek} twice along with an application of the $k = 2$ case,  this is equivalent to
\begin{align*}
&(\mbf X_1,\ldots,\mbf X_{k-1},\mbf X_k) = \mathcal J^{N,k}(\mbf U_1,\ldots,\mbf U_{k-1},\mbf U_k), \quad \text{and}\quad \mbf X_{k+1} = J^{N,2}(\mbf X_k,\wt{\mbf X}_{k+1}),\quad \text{where} \\
&(\mbf X_1,\ldots,\mbf X_{k-1},\wt{\mbf X}_{k+1}) = \mathcal J^{N,k}(\mbf U_1,\ldots,\mbf U_{k-1},\mbf U_{k+1}).
\end{align*}
By inspection of the definition in \eqref{eq.XJ1}-\eqref{eq.XJ2}, this is exactly the definition of $\mathcal J^{N,k+1}(\mbf U_1,\ldots,\mbf U_{k+1}) = (\mbf X_1,\ldots,\mbf X_{k+1})$, completing the sketch of the proof. We turn to making this argument precise by taking care of the appropriate domains. We need the following lemmas.
\begin{lemma} \label{lem:Diter}
For $N,m,r\in \N$ with $m > r$, and $(\mbf X_1,\ldots,\mbf X_r) \in (\R^{\Z_N})^r$, 
\be \label{eq:Diter}
D^{N,m}(\mbf X_1,\ldots,\mbf X_m) = D^{N,r+1}\bigl(\mbf X_1,\ldots,\mbf X_r,D^{N,m-r}(\mbf X_{r+1},\ldots,\mbf X_m)\bigr).
\ee
\end{lemma}
\begin{proof}
The proof is easy to see from Figure \ref{fig:DJ}. The left-hand side of \eqref{eq:Diter} is the output of the first column of the $D$ transform (shown on the left side of the figure for $m=4$). Observe, if one cuts the first column between $\mbf X_{r}$ and $\mbf X_{r+1}$, then the output on the top of the column (i.e. $D^{N,m}(\mbf X_1,\ldots,\mbf X_m)$) decomposes as the output under the map $= D^{N,r+1}$ of the the top $r$ inputs $\mbf X_1,\ldots,\mbf X_r$  combined with the output $D^{N,m-r}(\mbf X_{r+1},\ldots,\mbf X_m)$ of the bottom $m-r$ part of the column, precisely as claimed in the lemma.  
\end{proof}
To establish that $D^{N,m}$ maps between the desired spaces, we utilize the below monotonicity result.

\begin{lemma} \label{lem:Hmont}
For $m \ge 2$, if $\mbf X_1,\ldots,\mbf X_{m-1},\mbf X_m,\mbf X_m' \in \R^{\Z_N}$ and $\mbf X_m < \mbf X_m'$, then 
\[
D^{N,m}(\mbf X_1,\ldots,\mbf X_{m-1},\mbf X_m) < D^{N,m}(\mbf X_1,\ldots,\mbf X_{m-1},\mbf X_m').
\]
\end{lemma}
\begin{proof}
We prove this by induction, starting with the $m = 2$ case.
Set $Y_\ell = X_{2,\ell} - X_{1,\ell}$, and $Y_\ell' = X_{2,\ell}' - X_{1,\ell}$, noting that $X_{2,\ell} < X_{2,\ell}'$ iff $Y_\ell < Y_\ell'$. Then, by the definition \eqref{eq:D_intro}, (recall the notation of $Y_{(i,j]}$ from \eqref{sum_notat})
\begin{align*}
D_i^{N,2}(\mbf X_1,\mbf X_2) &= X_{1,i} + \log\Biggl(\f{e^{Y_i}\sum_{j \in \Z_N} e^{Y_{(i,j]}}     }{\sum_{j \in \Z_N} e^{Y_{(i-1,j]}  }}\Biggr) \\
&= X_{1,i} + \log\Biggl(\f{e^{Y_i} + e^{Y_i + Y_{i+1}} + \cdots + e^{Y_i + Y_{i+1} + \cdots + Y_{i-1}}}{1 + e^{Y_i} + e^{Y_i + Y_{i+1}} + \cdots + e^{Y_i + Y_{i+1} \cdots + Y_{i-2} }}\Biggr) \\
&= X_{1,i} - \log\Biggl(1 + \f{1 - e^{Y_i + Y_{i+1}\cdots + Y_{i-1}}}{e^{Y_i} + e^{Y_i + Y_{i+1}} + \cdots + e^{Y_i + Y_{i+1} \cdots + Y_{i-1}}}   \Biggr).
\end{align*}
It suffices to show, if we replace $Y_\ell$ with $Y_\ell'$, one-by-one, for each $\ell \in \Z_N$, then the term
\be \label{eq:Yimont}
\f{1 - e^{Y_i + Y_{i+1}\cdots + Y_{i-1}}}{e^{Y_i} + e^{Y_i + Y_{i+1}} + \cdots + e^{Y_i + Y_{i+1} \cdots + Y_{i-1}}}
\ee
strictly decreases. For each choice of $\ell$, if we fix $Y_j$ for $j \neq \ell$, then \eqref{eq:Yimont}, as a function of $Y_\ell$, takes the form
\[
\f{1 - A e^{Y_\ell}}{B + Ce^{Y_\ell}},
\]
where $A,C>0$ and $B \ge 0$. The derivative of $x \mapsto \f{1 - A e^{x}}{B + Ce^{x}}$ is $-\f{(AB + C)e^x}{(B + Ce^x)^2}<0$  for all $x \in \R$. 

Next, assume the statement of the lemma holds for some $m \ge 2$. Then, if $\mbf X_{m+1} < \mbf X_{m+1}'$, we have 
\[
D^{N,m}(\mbf X_2,\ldots,\mbf X_m,\mbf X_{m+1}) < D^{N,m}(\mbf X_2,\ldots,\mbf X_m,\mbf X_{m+1}'),
\]
and then, by the $m = 2$ case,
\begin{align*}
D^{N,m+1}(\mbf X_1,\ldots,\mbf X_m,\mbf X_{m+1}) &= D^{N,2}\bigl(\mbf X_1,D^{N,m}(\mbf X_2,\ldots,\mbf X_m,\mbf X_{m+1})\bigr) \\
&< D^{N,2}\bigl(\mbf X_1,D^{N,m}(\mbf X_2,\ldots,\mbf X_m,\mbf X_{m+1}')\bigr) = D^{N,m+1}(\mbf X_1,\ldots,\mbf X_m,\mbf X_{m+1}').\quad \qedhere
\end{align*}
\end{proof}

The next lemma states that the maps $\D^{N,k}$ and $\mathcal J^{N,k}$ preserve the sums of each coordinate.
\begin{lemma} \label{lem:Dsum_pres}
For $N,k\in \N$, if $(\mbf X_1,\ldots,\mbf X_k) \in (\R^{\Z_N})^k$, and $(\mbf U_1,\ldots,\mbf U_k)  = \D^{N,k}(\mbf X_1,\ldots,\mbf X_k)$, then $\vecsum(\mbf U_m) = \vecsum(\mbf X_m)$ for $1 \le m \le k$; and if $(\mbf U_1,\ldots,\mbf U_k) \in \mathcal R^{N,k}$ and $(\mbf X_1,\ldots,\mbf X_k) = \mathcal J^{N,k}(\mbf U_1,\ldots,\mbf U_k)$, then $\vecsum(\mbf X_m) = \vecsum(\mbf U_m)$ for $1\le m \le k$.
\end{lemma}
\begin{proof}
We start by proving the statement for $\D^{N,k}$. By definition of $\D^{N,k}$ \eqref{DNk_intro}, it suffices to prove that
\be \label{dNM_as}
\text{for all } (\mbf X_1,\ldots,\mbf X_m) \in (\R^{\Z_N})^m,\quad \vecsum(D^{N,m}(\mbf X_1,\ldots,\mbf X_m)) = \vecsum(\mbf X_m).
\ee
We prove this by induction on $m$, starting from $m = 2$ (the $m = 1$ case is trivial since $D^{N,1}(\mbf X_1) = \mbf X_1$). For $\mbf X_1,\mbf X_2 \in \R^{\Z_N}$, if we define $Y_\ell = X_{2,\ell} - X_{1,\ell}$, then
\[
\sum_{i \in \Z_N} D_i^{N,2}(\mbf X_1,\mbf X_2) = \sum_{i \in \Z_N}\Biggl(X_{2,i}  + \log \Bigl(\f{\sum_{j \in \Z_N} e^{Y_{(i,j]}}  }{\sum_{j \in \Z_N} e^{Y_{(i-1,j]}}}\Bigr)\Biggr) = \sum_{i \in \Z_N} X_{2,i}.
\] 
Next, assume that \eqref{dNM_as} holds for some $m \ge 2$. Then, for $\mbf X_1,\ldots,\mbf X_{m+1} \in \R^{\Z_N}$, $\vecsum\bigl(D^{N,m}(\mbf X_2,\ldots,\mbf X_{m+1})\bigr) = \vecsum(\mbf X_{m+1})$. Then, using the $m = 2$ case,
\[
\vecsum\bigl(D^{N,m+1}(\mbf X_1,\ldots,\mbf X_{m+1})\bigr) = \vecsum\bigl(D^{N,2}(\mbf X_1,D^{N,m}(\mbf X_2,\ldots,\mbf X_{m+1}))\bigr) = \vecsum\bigl(D^{N,m}(\mbf X_2,\ldots,\mbf X_{m+1})\bigr) = \vecsum(\mbf X_{m+1}).
\]

Turning to the statement for $\mathcal J^{N,k}$, the $k=1$ case is trivial since  $\mathcal J^{N,k}(\mbf U_1) = \mbf U_1$. In the $k = 2$ case, since 
$\mathcal J^{N,2}(\mbf U_1,\mbf U_2) = \bigl(\mbf U_1,J^{N,2}(\mbf U_1,\mbf U_2)\bigr)$ the statement follows from the fact that 
\[
\sum_{i \in \Z_N}J_i^{N,2}(\mbf U_1,\mbf U_2) = \sum_{i \in \Z_N} \Biggl( U_{2,i} +\log\Bigl(\f{e^{ U_{2,i+1} - U_{1,i+1}} - 1}{e^{U_{2,i} - U_{1,i}} - 1}\Bigr)  \Biggr) = \sum_{i \in \Z_N} U_{2,i}.
\]
Now, assume the statement in the lemma holds for some $k \ge 2$, and for $(\mbf U_1,\ldots,\mbf U_{k+1}) \in \mathcal R^{N,k}$, let $(\mbf X_1,\ldots,\mbf X_{k+1}) = \mathcal J^{N,k+1}(\mbf U_1,\ldots,\mbf U_{k+1})$. By the definition of $\mathcal J^{N,k+1}$, namely \eqref{eq.XJ1} and \eqref{eq.XJ2}, 
\[
\begin{aligned}
(\mbf X_1,\ldots,\mbf X_{k}) &:= \mathcal J^{N,k}(\mbf U_1,\ldots,\mbf U_{k}), \quad\text{and}\quad \mbf X_{k+1} := J^{N,2}(\mbf X_{k},\wt{\mbf X}_{k+1}),\quad\text{where} \\
(\mbf X_1,\ldots,\mbf X_{k-1},\wt{\mbf X}_{k+1}) &:= \mathcal J^{N,k}(\mbf U_1,\ldots,\mbf U_{k-1},\mbf U_{k+1}).
\end{aligned}
\]
By induction $\vecsum(\mbf X_m) = \vecsum(\mbf U_m)$ for $1 \le m \le k$, and $\vecsum(\wt{\mbf X}_{k+1}) = \vecsum(\mbf U_{k+1})$. Finally, by the $k = 2$ case, 
\[
\vecsum(\mbf X_{k+1}) = \vecsum\bigl(J^{N,2}(\mbf X_k,\wt{\mbf X}_{k+1})\bigr) = \vecsum(\wt{\mbf X}_{k+1}) = \vecsum(\mbf U_{k+1}). \qedhere
\]
\end{proof}

The next lemma shows that the output of $\D^{N,k}$ is sorted pointwise depending on the input $\vecsum$-ordering.
\begin{lemma} \label{lem:sum_v_order}
For $N,k \in \N$ and $(\mbf X_1,\ldots,\mbf X_{k}) \in (\R^{\Z_N})^{k}$, let
$(\mbf U_1,\ldots,\mbf U_k) = \D^{N,k}(\mbf X_1,\ldots,\mbf X_k)$. Then, 
\begin{enumerate} [label=\textup{(\roman*)}]
\item If $\vecsum(\mbf X_{m}) = \vecsum(\mbf X_r)$, then $\mbf U_r = \mbf U_m$.
\item If $\vecsum(\mbf X_{m}) < \vecsum(\mbf X_r)$, then $\mbf U_m < \mbf U_r$.
\item If $\vecsum(\mbf X_{m}) > \vecsum(\mbf X_r)$, then $\mbf U_m > \mbf U_r$.
\end{enumerate}
\end{lemma}
\begin{proof}
We first prove the statement for $r = m-1$, starting from the $m = 2$ case. Let $(\mbf X_1,\mbf X_2) \in (\R^{\Z_N})^2$, and define $Y_\ell = X_{2,\ell} - X_{1,\ell}$ for $\ell \in \Z_N$. Then,
\begin{align} \label{eq:D-X1}
U_{2,i} - U_{1,i} = D_i^{N,2}(X_1,X_2) - X_{1,i} = \log\Biggl(\f{e^{Y_i}\sum_{j \in \Z_N} e^{Y_{(i,j]}}  }{\sum_{j \in \Z_N} e^{Y_{(i-1,j]}} }  \Biggr) = \log\Biggl(\f{e^{Y_i} + e^{Y_{i} + Y_{i+1}} + \cdots + e^{Y_i + \cdots + Y_{i-1}}}{1 + e^{Y_i} + \cdots + e^{Y_i + \cdots + Y_{i-2}}   }\Biggr).
\end{align}
Notice the numerator and denominator are both positive, and the difference between numerator and denominator is $e^{\sum_{\ell \in \Z_N} Y_\ell} - 1 = e^{\vecsum(X_2) - \vecsum(X_1)} - 1$. Hence, $U_{2,i} - U_{1,i}$ has the same sign as $\vecsum(X_2) - \vecsum(X_1)$.

Now, let $m > 2$. By definition of $\D^{N,k}$ and Lemma \ref{lem:Diter}, 
\[
\mbf U_m = D^{N,m}(\mbf X_1,\ldots,\mbf X_m) = D^{N,m-1}\bigl(\mbf X_1,\ldots, \mbf X_{m-2},D^{N,2}(\mbf X_{m-1},\mbf X_m)\bigr),
\]
and by the $m = 2$ case, $D^{N,2}(\mbf X_{m-1},\mbf X_m) <  \mbf X_{m-1}$ if $\vecsum(\mbf X_m) > \vecsum(\mbf X_{m-1})$, $D^{N,2}(\mbf X_{m-1},\mbf X_m) > \mbf X_{m-1}$ if $\vecsum(\mbf X_m) < \vecsum(\mbf X_{m-1})$, and $D^{N,2}(\mbf X_{m-1},\mbf X_m) = \mbf X_{m-1}$ if $\vecsum(\mbf X_m) = \vecsum(\mbf X_{m-1})$. Combined with the monotonicity of Lemma \ref{lem:Hmont}, this proves the statement of the Lemma when $r = m-1$. For general $r < m$, Lemma \ref{lem:Diter}  shows that
$\mbf U_m =D^{N,r+1}\bigl(\mbf X_1,\ldots,\mbf X_r,D^{N,m-r}(\mbf X_{r+1},\ldots,\mbf X_m)\bigr)$
and by Lemma \ref{lem:Dsum_pres}, $\vecsum\bigl(D^{N,m-r}(\mbf X_{r+1},\ldots,\mbf X_m)\bigr) = \vecsum(\mbf X_m)$. Hence, the conclusion of the Lemma follows from the $r = m-1$ case. 
\end{proof}

We are now ready to prove Proposition \ref{prop:transform}.
\begin{proof}[Proof of Proposition \ref{prop:transform}]
Recall that, by definition, for $\theta_1,\ldots,\theta_k$ all distinct, $\R^{\Z_N}_{\theta_1} \times \cdots \times \R^{\Z_N}_{\theta_k}$ and $\mathcal R_{\theta_1,\ldots\theta_k}^{N,k}$ are the subsets of  $\R^{N,k}_{\neq}$ and $\mathcal R^{N,k}$, respectively, such that the $m$th coordinate has sum $\theta_m$. Hence, the second bijection stated in the proposition follows from the first and Lemma \ref{lem:Dsum_pres}. It now suffices to show the following.
\begin{enumerate} [label=\textup{(\roman*)}]
\item \label{itm:Dim} If $(\mbf X_1,\ldots,\mbf X_k) \in \R_{\neq}^{N,k}$, then $\D^{N,k}(\mbf X_1,\ldots,\mbf X_k) \in \mathcal R^{N,k}$, and 
\[
\mathcal J^{N,k}\bigl(\D^{N,k}(\mbf X_1,\ldots,\mbf X_k)\bigr) = (\mbf X_1,\ldots,\mbf X_k).
\]
\item \label{itm:Jim} If $(\mbf U_1,\ldots,\mbf U_k) \in \mathcal R^{N,k}$, then $\mathcal J^{N,k}(\mbf U_1,\ldots,\mbf U_k) \in \R_{\neq}^{N,k}$, and
\[
\D^{N,k}\bigl(\mathcal J^{N,k}(\mbf U_1,\ldots,\mbf U_k)\bigr) = (\mbf U_1,\ldots,\mbf U_k).
\]
\end{enumerate}

\begin{figure}
    \centering
    \includegraphics[width=.8\linewidth]{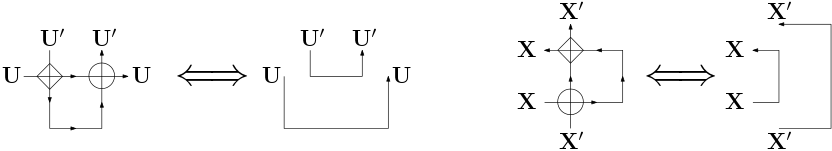}
    \caption{The $k=2$ case of Proposition \ref{prop:transform} depicted here serves as the basic relation to prove the general $k$ case. The left figure records the fact that that applying $\mathcal J^{N,2}$ and then $\mathcal D^{N,2}$ is equivalent to applying the identity map on $\mathcal R^{N,2}$. The right figure records the fact that that applying $\mathcal D^{N,2}$ and then $\mathcal J^{N,2}$ is equivalent to applying the identity map on $\R^{N,2}_{\neq}$. Here and elsewhere arrows (without decorated crosses) depict identity maps.}
    \label{fig:k2case}
\end{figure} 

\begin{figure}
    \centering
    \includegraphics[width=1\linewidth]{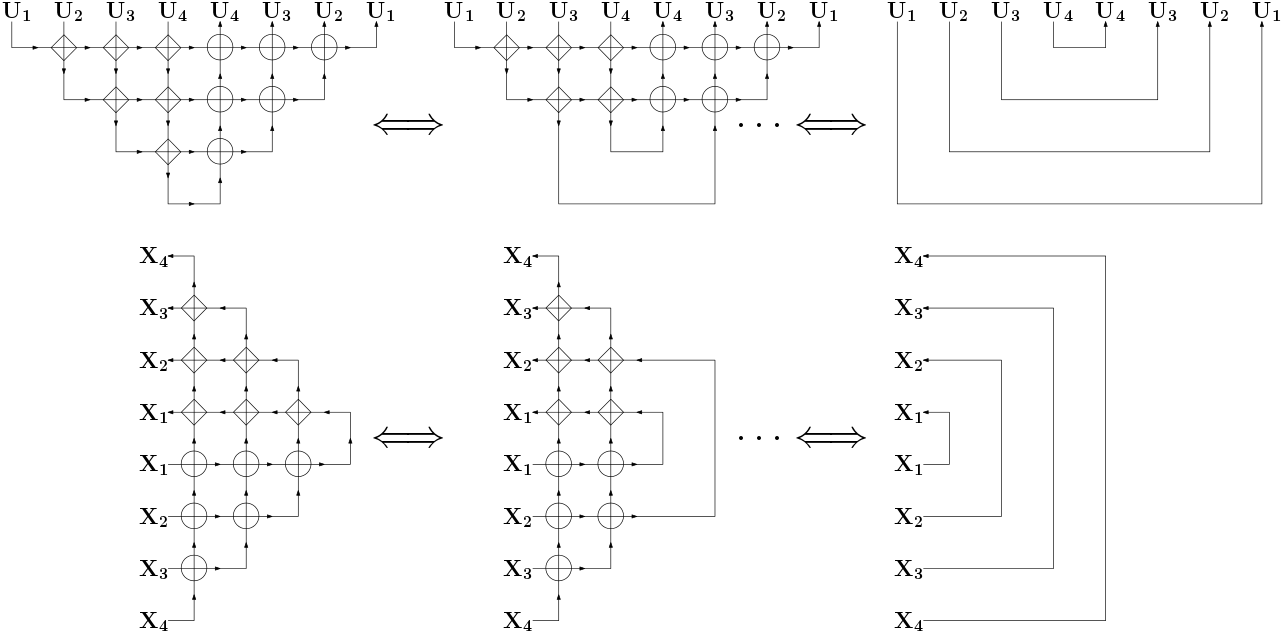}
    \caption{The proof of Proposition \ref{prop:transform} is shown graphically for $k=4$. Repeated application of the unwinding identities shown in Figure \ref{fig:k2case} imply (on the top of the figure) that applying $\mathcal J^{N,k}$ (recall Figure \ref{fig:DJ}) and then $\mathcal D^{N,k}$ is equivalent to applying the identity map on $\mathcal R^{N,k}$ and (on the bottom of the figure) that applying $\mathcal D^{N,k}$ and then $\mathcal J^{N,k}$ is equivalent to applying the identity map on $\R^{N,k}_{\neq}$.}
    \label{fig:bijection}
\end{figure} 

 We prove Items \ref{itm:Dim} and \ref{itm:Jim} together using induction, starting from the $k = 2$ case, see Figure \ref{fig:k2case} (the $k = 1$ case is trivial). If $(\mbf X_1,\mbf X_2) \in \R_{\neq}^{N,2}$ (meaning $\vecsum(\mbf X_1) \neq \vecsum(\mbf X_2)$), then Lemma \ref{lem:sum_v_order} implies that 
    \[
    \D^{N,2}(\mbf X_1,\mbf X_2) = (\mbf X_1,D^{N,2}(\mbf X_1,\mbf X_2)) \in \mathcal R^{N,2}.
    \]
    On the other hand, if $(\mbf U_1,\mbf U_2) \in \mathcal R^{N,2}$, then it necessarily follows that $\vecsum(\mbf U_1) \neq \vecsum(\mbf U_2)$. By the preservation of sums from Lemma \ref{lem:Dsum_pres}, we have $\mathcal J^{N,2}(\mbf U_1,\mbf U_2) \in \R_{\neq}^{N,2}$.

    Next, since $J^{N,1}(\mbf U_1) = \mbf U_1$ and $D^{N,1}(\mbf X_1) = \mbf X_1$, we must show that, for any given $\mbf X_1 \in \R^{\Z_N}$,  if $(\mbf X_1,\mbf U_2) \in \mathcal R^{N,2}$, then   $J^{N,2}(\mbf X_1,\mbf U_2) = \mbf X_2$ if and only if $D^{N,2}(\mbf X_1,\mbf X_2) = \mbf U_2$. For this, it is more convenient to use the coordinates $Y_i = X_{2,i} - X_{1,i}$ and $W_i = U_{2,i} - X_{1,i}$. The assumption $(\mbf X_1,\mbf U_2) \in \mathcal R^{N,2}$ means that $W_i > 0$ for all $i \in \Z_N$ or $W_i < 0$ for all $i \in \Z_N$.  By \eqref{Jimap}, $J^{N,2}(\mbf X_1,\mbf U_2) = \mbf X_2$ is equivalent to 
\be \label{ZtoY}
Y_i = W_i + \log\Bigl(\f{e^{W_{i+1}} - 1}{e^{W_i} - 1}\Bigr),\quad i \in \Z_N,
\ee
and by \eqref{eq:D_intro}, $D^{N,2}(\mbf X_1,\mbf X_2) = \mbf U_2$ is equivalent to 
\be \label{YtoZ}
W_i = Y_i + \log\Biggl(\f{\sum_{j \in \Z_N} e^{Y_{(i,j]}}     }{\sum_{j \in \Z_N} e^{Y_{(i-1,j]}  }}\Biggr),\quad i \in \Z_N.
\ee
We must show that \eqref{ZtoY} and \eqref{YtoZ} are equivalent. Assuming \eqref{ZtoY},  we make the observation that 
\[
Y_{(i,j]} = W_{(i,j]} + \log \Bigl(\f{e^{W_{j+1}} - 1}{e^{W_{i+1}} - 1}\Bigr),
\]
from which it follows that
\begin{align*}
    Y_i + \log\Biggl(\f{\sum_{j \in \Z_N} e^{Y_{(i,j]}}     }{\sum_{j \in \Z_N} e^{Y_{(i-1,j]}  }}\Biggr) &= W_i + \log\Bigl(\f{e^{W_{i+1}} - 1}{e^{W_i} - 1}\Bigr) + \log \Biggl(\f{\sum_{j \in \Z_N} e^{W_{(i,j]}} \Bigl(\f{e^{W_{j+1}} - 1}{e^{W_{i+1}} - 1}\Bigr)   }{\sum_{j \in \Z_N} e^{W_{(i-1,j]}} \Bigl(\f{e^{W_{j+1}} - 1}{e^{W_{i}} - 1}\Bigr)   }   \Biggr) \nonumber \\
    &= W_i + \log\Biggl(\f{\sum_{j \in \Z_N}e^{W_{(i,j]}}(e^{W_{j+1}} - 1)   }{\sum_{j \in \Z_N}e^{W_{(i-1,j]}}(e^{W_{j+1}} - 1)  }\Biggr) \\
    &= W_i + \log\Biggl(\f{\prod_{j \in \Z_N}e^{W_j} - 1 }{\prod_{j \in \Z_N}e^{W_j} - 1 }   \Biggr) = W_i.
\end{align*}
In the penultimate step, we note there is a telescoping of terms. Next, assuming \eqref{YtoZ} we have
\begin{align*}
W_i + \log\Bigl(\f{e^{W_{i+1}} - 1}{e^{W_i} - 1}\Bigr) &= Y_i + \log\Bigl(\f{\sum_{j \in \Z_N} e^{Y_{(i,j]}}}{\sum_{j \in \Z_N} e^{Y_{(i-1,j]}}}\Bigr) + \log\Biggl(\f{e^{Y_{i+1}} \Bigl(\f{\sum_{j \in \Z_N} e^{Y_{(i+1,j]}}}{\sum_{j \in \Z_N} e^{Y_{(i,j]}}}\Bigr) - 1}{e^{Y_{i}} \Bigl(\f{\sum_{j \in \Z_N} e^{Y_{(i,j]}}}{\sum_{j \in \Z_N} e^{Y_{(i-1,j]}}}\Bigr) - 1}   \Biggr) \\
&= Y_i + \log\Biggl(\f{\sum_{j \in \Z_N} [e^{Y_{[i+1,j]}}- e^{Y_{(i,j]}}]   }{\sum_{j \in \Z_N} [e^{Y_{[i,j]}}- e^{Y_{(i-1,j]}}] }\Biggr) = Y_i + \log\Biggl(\f{\prod_{j \in \Z_N}e^{Y_j} - 1   }{\prod_{j \in \Z_N}e^{Y_j} - 1}\Biggr) = Y_i.
\end{align*}
In summary, we have shown \ref{itm:Dim} and \ref{itm:Jim} for $k = 2$.

The induction in $k$ now proceeds as depicted in Figure \ref{fig:bijection}. Assume that \ref{itm:Dim} and \ref{itm:Jim} hold for some $k \ge 2$. We first prove \ref{itm:Dim} for $k + 1$. Let $(\mbf X_1,\ldots,\mbf X_{k+1}) \in \R_{\neq}^{N,k+1}$, and set $(\mbf U_1,\ldots,\mbf U_{k+1}) := \D^{N,k+1}(\mbf X_1,\ldots,\mbf X_{k+1})$. By definition of $\D^{N,k}$, we see that 
\be \label{101}
(\mbf U_1,\ldots,\mbf U_k)  = \D^{N,k}(\mbf X_1,\ldots,\mbf X_k).
\ee
The assumption $(\mbf X_1,\ldots,\mbf X_{k+1}) \in \R_{\neq}^{N,k+1}$ immediately implies that $(\mbf X_1,\ldots,\mbf X_{k}) \in \R_{\neq}^{N,k}$. Hence, by \eqref{101} and \ref{itm:Dim} of the induction assumption, $(\mbf U_1,\ldots,\mbf U_k) \in \mathcal R^{N,k}$. Furthermore, by  definition of the map $\D^{N,k+1}$ and Lemma \ref{lem:Diter}, we observe that 
\[
\mbf U_{k+1} = D^{N,k+1}(\mbf X_1,\ldots,\mbf X_{k+1}) = D^{N,k}(\mbf X_1,\ldots,\mbf X_{k-1},D^{N,2}(\mbf X_k,\mbf X_{k+1})).
\]
Hence, if we define $\wt{\mbf X}_{k+1} := D^{N,2}(\mbf X_k,\mbf X_{k+1})$, then 
\be \label{102}
(\mbf U_1,\ldots,\mbf U_{k-1},\mbf U_{k+1}) = \D^{N,k}(\mbf X_1,\ldots,\mbf X_{k-1},\wt{\mbf X}_{k+1}).
\ee
By Lemma \ref{lem:Dsum_pres}, $\vecsum(\wt{\mbf X}_{k+1}) = \vecsum(\mbf X_{k+1})$, so $(\mbf X_1,\ldots,\mbf X_{k-1},\wt{\mbf X}_{k+1}) \in \R^{N,k}_{\neq}$. Then \eqref{102} and \ref{itm:Dim} of the induction assumption implies that $(\mbf U_1,\ldots,\mbf U_{k-1},\mbf U_{k+1}) \in \mathcal R^{N,k}$. By \eqref{101}-\eqref{102} and \ref{itm:Dim} of the induction assumption, 
\be \label{103}
\begin{aligned}
(\mbf X_1,\ldots,\mbf X_{k-1},\mbf X_k) &= \mathcal J^{N,k}(\mbf U_1,\ldots,\mbf U_{k-1},\mbf U_{k}),\quad\text{and}  \\
(\mbf X_1,\ldots,\mbf X_{k-1},\wt{\mbf X}_{k+1}) &= \mathcal J^{N,k}(\mbf U_1,\ldots,\mbf U_{k-1},\mbf U_{k+1}).
\end{aligned}
\ee
Since $\vecsum(\wt{\mbf X}_{k+1}) = \vecsum(\mbf X_{k+1}) \neq \vecsum(\mbf X_k)$, the $m = 2, r =1$ case of Lemma \ref{lem:Hmont} implies that $(\mbf X_{k},\wt{\mbf X}_{k+1}) \in \mathcal R^{N,2}$. In summary, we have shown that 
\[
\begin{aligned}
    (\mbf U_1,\ldots,\mbf U_{k-1},\mbf U_k) \in \mathcal R^{N,k},\quad (\mbf U_1,\ldots,\mbf U_{k-1},\mbf U_{k+1}) \in \mathcal R^{N,k}, \quad\text{and}\\
    (\mbf X_k,\wt{\mbf X}_{k+1}) \in \mathcal R^{N,2}, \quad\text{where}\quad \mbf X_k, \wt{\mbf X}_{k+1} \text{ are determined
 by }\eqref{103}
 \end{aligned}
\]
These are precisely the conditions for $(\mbf U_1,\ldots,\mbf U_{k+1}) \in \mathcal R^{N,k+1}$. Furthermore, since we defined $\wt{\mbf X}_{k+1} = D^{N,2}(\mbf X_k,\mbf X_{k+1})$, the $k = 2$ case implies that $\mbf X_{k+1} = J^{N,2}(\mbf X_k,\wt{\mbf X}_{k+1})$. Combined with \eqref{103} and recalling the definition of $\mathcal J^{N,k + 1}$, we see that 
\[
\mathcal J^{N,k+1}\bigl(\D^{N,k+1}(\mbf X_1,\ldots,\mbf X_{k+1})\bigr) = \mathcal J^{N,k+1}\bigl(\mbf U_1,\ldots,\mbf U_{k+1}\bigr) = (\mbf X_1,\ldots,\mbf X_{k+1}),
\]
thus proving \ref{itm:Dim} for $k+1$.

We finish by proving \ref{itm:Jim} for $k+1$. Assume $(\mbf U_1,\ldots,\mbf U_{k+1}) \in \mathcal R^{N,k+1}$, and define 
\[
(\mbf X_1,\ldots,\mbf X_{k+1}) := \mathcal J^{N,k+1}(\mbf U_1,\ldots,\mbf U_{k+1}).
\]
We seek to show that $\vecsum(\mbf X_m)$, $1 \le m \le k+1$, are all distinct, and $\D^{N,k+1}(\mbf X_1,\ldots,\mbf X_{k+1}) =(\mbf U_1,\ldots,\mbf U_{k+1}) $.

Since  $(\mbf U_1,\ldots,\mbf U_{k+1}) \in \mathcal R^{N,k+1}$, it follows that $(\mbf U_1,\ldots,\mbf U_{k-1},\mbf U_k),(\mbf U_1,\ldots,\mbf U_{k-1},\mbf U_{k+1})  \in \mathcal R^{N,k}$; and, recalling \eqref{eq.XJ1} and \eqref{eq.XJ2} (which define $\mbf X_1,\ldots ,\mbf X_{k+1}$ and $\wt{\mbf X}_{k+1}$),
it also follows that $(\mbf X_k,\wt{\mbf X}_{k+1}) \in \mathcal R^{N,2}$.
By \eqref{eq.XJ1}, \eqref{eq.XJ2}, and \ref{itm:Jim} of the induction assumption, $(\mbf X_1,\ldots,\mbf X_{k-1},\mbf X_k), (\mbf X_1,\ldots,\mbf X_{k-1},\wt{\mbf X}_{k+1}) \in \R_{\neq}^{N,k}$. In order to show $(\mbf X_1,\ldots,\mbf X_{k+1}) \in \R_{\neq}^{N,k}$, it remains to show that $\vecsum(\mbf X_m) \neq \vecsum(\mbf X_{k+1})$ for $1 \le m \le k$.  Note that $\vecsum(\mbf X_k) \neq \vecsum(\wt{\mbf X}_{k+1})$ because $(\mbf X_k,\wt{\mbf X}_{k+1}) \in \mathcal R^{N,2}$. Also, $\vecsum(\wt{\mbf X}_{k+1}) \neq \vecsum(\mbf X_m)$ for $1 \le m \le k$ since $(\mbf X_1,\ldots,\mbf X_{k-1},\wt{\mbf X}_{k+1}) \in \R_{\neq}^{N,k}$. By definition of the map $\mathcal J^{N,k+1}$, we have
$
\mbf X_{k+1} = J^{N,2}(\mbf X_k,\wt{\mbf X}_{k+1}),
$
and by Lemma \ref{lem:Dsum_pres}, $\vecsum(\mbf X_{k+1}) = \vecsum(\wt{\mbf X}_{k+1})$. Hence, $(\mbf X_1,\ldots,\mbf X_{k+1}) \in  \R_{\neq}^{N,k+1}$. 

Finally we show  $\D^{N,k+1}(\mbf X_1,\ldots,\mbf X_{k+1}) = (\mbf U_1,\ldots,\mbf U_{k+1})$. By \eqref{eq.XJ1} and \ref{itm:Jim} of the inductive assumption, 
\[
\D^{N,k}(\mbf X_1,\ldots,\mbf X_k) = (\mbf U_1,\ldots,\mbf U_k),
\]
so by definition of $\D^{N,k+1}$ we just need to show that 
$\mbf U_{k+1} = D^{N,k+1}(\mbf X_1,\ldots,\mbf X_{k+1})$.
By \eqref{eq.XJ2} and \ref{itm:Jim} of the inductive assumption, $(\mbf U_1,\ldots,\mbf U_{k-1},\mbf U_{k+1}) = \D^{N,k}(\mbf X_1,\ldots,\mbf X_{k-1},\wt{\mbf X}_{k+1})$, which by definition, implies
\be \label{106}
 \mbf U_{k+1} = D^{N,k}(\mbf X_1,\ldots,\mbf X_{k-1},\wt{\mbf X}_{k+1}).
\ee
We have shown that $\mbf X_{k+1} = J^{N,2}(\mbf X_k,\wt{\mbf X}_{k+1})$, so the $k = 2$ case of the bijection implies that $\wt{\mbf X}_{k+1} = D^{N,2}(\mbf X_k,\mbf X_{k+1})$. Combined with \eqref{106} and Lemma \ref{lem:Diter}, we have
\[
\mbf U_{k+1} =  D^{N,k}\bigl(\mbf X_1,\ldots,\mbf X_{k-1},D^{N,2}(\mbf X_k,\mbf X_{k+1})\bigr) = D^{N,k}(\mbf X_1,\ldots,\mbf X_{k+1}), 
\]
as desired.
\end{proof}

We close by noting two facts (neither is used below) about the structure of the domain  $\mathcal R^{N,k}$. The first is recorded as Corollary \ref{cor:RNKcontained} and shows the strict inclusion of $\mathcal R^{N,k}$ in the set of all strictly ordered (up to permutation) $\mbf U_m$. The second, proved later as Corollary \ref{cor:RNK_permute}, is that $\mathcal R^{N,k}$ is invariant under permutations; i.e.,  whenever $(\mbf U_1,\ldots,\mbf U_k) \in \mathcal R^{N,k}$ and $\sigma \in \mathcal S(k)$, then we have $(\mbf U_{\sigma(1)},\ldots,\mbf U_{\sigma(k)}) \in \mathcal R^{N,k}$ too.

\begin{corollary} \label{cor:RNKcontained}
    For $N,k \ge 1$, we have the inclusion 
    \[
    \mathcal R^{N,k} \subseteq \bigcup_{\sigma \in \mathcal S(k)}\Bigl\{(\mbf U_1,\ldots,\mbf U_k) \in (\R^{\Z_N})^k: \mbf U_{\sigma(1)} < \mbf U_{\sigma(2)} < \cdots < \mbf U_{\sigma(k)}\Bigr\}.
    \]
    When $k \ge 3$, the inclusion is strict.
\end{corollary}
\begin{proof}
Let $(\mbf U_1,\ldots,\mbf U_k) \in \mathcal R^{N,k}$. By Proposition \ref{prop:transform}, $(\mbf U_1,\ldots,\mbf U_k) = \D^{N,k}(\mbf X_1,\ldots,\mbf X_k)$ for some \\
$(\mbf X_1,\ldots,\mbf X_k) \in \R_{\neq}^{N,k}$.
By definition of $\R_{\neq}^{N,k}$ \eqref{RA2}, for each pair $r \neq m$, $\vecsum(\mbf X_m) \neq \vecsum(\mbf X_r)$. Then, by Lemma \ref{lem:sum_v_order}, either $\mbf U_m < \mbf U_r$ or $\mbf U_r < \mbf U_m$. This proves the inclusion.  
We will show the strict inclusion by direct inspection when $k=3$; for $k>3$ it follows from the inductive definition of $ \mathcal R^{N,k}$. 
By definition, 
\[
\mathcal J^{N,3}(\mbf U_1,\mbf U_2,\mbf U_3) = \Bigl(\mbf U_1,J^{N,2}(\mbf U_1,\mbf U_2),J^{N,2}\bigl(J^{N,2}(\mbf U_1,\mbf U_2),J^{N,2}(\mbf U_1,\mbf U_3)\bigr)\Bigr),
\]
which is defined on the domain
\[
\mathcal R^{N,3} = \Bigl\{(\mbf U_1,\mbf U_2,\mbf U_3): (\mbf U_1,\mbf U_2) \in \mathcal R^{N,2},(\mbf U_1,\mbf U_3) \in \mathcal R^{N,2}, \text{ and }\bigl(J^{N,2}(\mbf U_1,\mbf U_2),J^{N,2}(\mbf U_1,\mbf U_3)\bigr) \in \mathcal R^{N,2}   \Bigr \}.
\]
We note that the last condition implies, but is not equivalent to $(\mbf U_2,\mbf U_3) \in \mathcal R^{N,2}$. If, without loss of generality, $J^{N,2}(\mbf U_1,\mbf U_2) < J^{N,2}(\mbf U_1,\mbf U_3)$, then the $k = 2$ case of the bijection in Proposition \ref{prop:transform} below, along with the monotonicity of Lemma \ref{lem:Hmont} implies that 
\[
\mbf U_2 = D^{N,2}\bigl(\mbf U_1,J^{N,2}(\mbf U_1,\mbf U_2)\bigr) < D^{N,2}\bigl(\mbf U_1,J^{N,2}(\mbf U_1,\mbf U_3)\bigr) = \mbf U_3.
\]
However, one can see directly from the definition of the map $J^{N,2}$ that $\mbf U_2 < \mbf U_3$ does not imply that $J^{N,2}(\mbf U_1,\mbf U_2) <  J^{N,2}(\mbf U_1,\mbf U_3)$. This implies that the inclusion for $k=3$ is strict.
\end{proof}

\subsection{Properties of the measures and the proof of Proposition \ref{prop:disc_consis}} \label{sec:Burke_intertwine}
The culmination of this subsection is the proof of  Proposition \ref{prop:disc_consis} (the main challenge is to prove Proposition \ref{prop:disc_consis}\ref{itm:g_perm_invar}, the consistency and symmetry). The key tools needed for the proof are the bijection in Corollary \ref{cor:DR_bijection}, the Burke property in Proposition \ref{prop:Burke}, and the intertwining in Proposition \ref{prop:full_intertwine}.

We introduce two additional maps.
Define $T^{N,2}:\R^{\Z_N} \times \R^{\Z_N} \to \R^{\Z_N}$ by 
\be \label{Rdef}
T_i^{N,2}(\mbf X_1,\mbf X_2) = X_{1,i} + \log\Biggl(\f{\sum_{j \in \Z_N} e^{Y_{[i,j]}}  }{\sum_{j \in \Z_N} e^{Y_{[i+1,j]}}}\Biggr),\quad \mbox{ with } \mbf Y=\mbf X_2-\mbf X_1,
\ee
and $L^{N,2}:\mathcal R^{N,2} \to \R^{\Z_N}$ by
\be \label{Ldef}
L_i^{N,2}(\mbf U_1,\mbf U_2) = U_{1,i} + \log\Bigl( \f{1 - e^{U_{1,i-1} - U_{2,i-1}}}{1-e^{U_{1,i} - U_{2,i}}}\Bigr).
\ee
Lemma \ref{lem:Rbij} below shows that $T^{N,2}$ and $L^{N,2}$ are inverse to each other in a certain sense, while Lemma \ref{lem:R-intertwine} shows that $T^{N,2}$ is actually the composition of the $J^{N,2}$ and $D^{N,2}$ maps. Specifically, $T^{N,2}(\mbf X_1,\mbf X_2) = J^{N,2}(D^{N,2}(\mbf X_1,\mbf X_2),\mbf X_1)$. Note that this is different than the composition $J^{N,2}(\mbf X_1,D^{N,2}(\mbf X_1,\mbf X_2))$, which is simply $\mbf X_2$ by Proposition \ref{prop:transform}.

We first note that $T^{N,2}$ and $L^{N,2}$ preserve the sum of their coordinates. Letting $Y_\ell = X_{2,\ell} - X_{1,\ell}$, we have 
\be \label{eq.TL}
\begin{aligned}
    \sum_{i \in \Z_N} T_i^{N,2}(\mbf X_1,\mbf X_2) &= \sum_{i \in \Z_N} \Biggl(X_{1,i} + 
  \log\Bigl(\f{\sum_{j \in \Z_N} e^{Y_{[i,j]}}  }{\sum_{j \in \Z_N} e^{Y_{[i+1,j]}}}\Bigr)\Biggr) = \sum_{i \in \Z_N} X_{1,i}, \\
   \sum_{i \in \Z_N} L_i^{N,2}(\mbf U_1,\mbf U_2) &= \sum_{i \in \Z_N}\Biggl( U_{1,i} + \log\Bigl(\f{1 - e^{U_{1,i-1} - U_{2,i-1}}}{1 - e^{U_{1,i} - U_{2,i}}}\Bigr)\Biggr) = \sum_{i \in \Z_N} U_{1,i}. 
\end{aligned}
\ee

\begin{lemma} \label{lem:Rbij}
For $N\in \N$ and $\theta_1 \neq \theta_2$, the map $(\mbf X_1,\mbf X_2) \mapsto (T^{N,2}(\mbf X_1,\mbf X_2),\mbf X_2)$ is a bijection $\R^{\Z_N}_{\theta_1} \times \R^{\Z_N}_{\theta_2} \to \mathcal R^{N,2}_{\theta_1,\theta_2}$ and also a bijection $\R^{N,2}_{\neq} \to \mathcal R^{N,2}$; in both cases the inverse is $(\mbf U_1,\mbf U_2) \mapsto (L^{N,2}(\mbf U_1,\mbf U_2),\mbf U_2)$.

\end{lemma}
\begin{proof}
   That the map and, respectively, its inverse take $\R^{\Z_N}_{\theta_1} \times \R^{\Z_N}_{\theta_2}$ to $\mathcal R^{N,2}_{\theta_1,\theta_2}$ and visa-versa follows immediately from the coordinate sum preservation in  \eqref{eq.TL}. The same reasoning shows that the inverse takes $\mathcal R^{N,2}$ to $\R^{N,2}_{\neq}$, again by coordinate sum preservation. The other direction requires more than coordinate sum preservation since $\mathcal R^{N,2}$  is defined by coordinate-wise ordering, not just of the sum. In particular we must show that if $(\mbf X_1,\mbf X_2) \in \R_{\neq}^{N,k}$, then $(T^{N,2}(\mbf X_1,\mbf X_2),\mbf X_2) \in \mathcal R^{N,2}$. By definition, 
    \be \label{Rmont}
    \begin{aligned}
    X_{2,i} - T_i^{N,2}(\mbf X_1,\mbf X_2) =  Y_i + \log\Biggl(\f{\sum_{j \in \Z_N} e^{Y_{[i+1,j]}}  }{\sum_{j \in \Z_N} e^{Y_{[i,j]}}}\Biggr) &= \log\Biggl(\f{e^{Y_{i+1}} + e^{Y_{i+1} + Y_{i+2}} + \cdots + e^{Y_{i+1} + \cdots + Y_i} }{1 + e^{Y_{i+1}} + \cdots + e^{Y_{i+1} + \cdots + Y_{i-1}}  }\Biggr)\\
    &=  \log\Biggl(\f{e^{\vecsum}+e^{Y_{i+1}} + e^{Y_{i+1} + Y_{i+2}} + \cdots + e^{Y_{i+1} + \cdots + Y_{i-1}} }{1 + e^{Y_{i+1}} + \cdots + e^{Y_{i+1} + \cdots + Y_{i-1}}  }\Biggr),
    \end{aligned}
    \ee
    where we let $\vecsum = \sum_{\ell \in \Z_N} Y_\ell = \vecsum(\mbf X_2) - \vecsum(\mbf X_1)$. The final equality simply comes from replacing the term $e^{Y_{i+1} + \cdots + Y_i}$ in the numerator by $e^\vecsum$.   The sign of $\vecsum$ and the right-hand side of \eqref{Rmont} are the same as desired to show that either  $T^{N,2}(\mbf X_1,\mbf X_2)>\mbf X_2$ or   $T^{N,2}(\mbf X_1,\mbf X_2)<\mbf X_2$. 

    Having shown above that the map and its inverse go between the claimed spaces, it remains to prove that it is a bijection. To do this it suffices to show that, if $(\mbf X_1,\mbf X_2) \in \R_{\neq }^{N,2}$ and $(\mbf U_1,\mbf X_2) \in \mathcal R^{N,2}$, then $T^{N,2}(\mbf X_1,\mbf X_2) = \mbf U_1$ if and only if $\mbf X_1 = L^{N,2}(\mbf U_1,\mbf X_2)$. Assume first that $\mbf U_1 = T^{N,2}(\mbf X_1,\mbf X_2)$. Then, for $i \in \Z_N$, with $Y_\ell = X_{2,\ell} - X_{1,\ell}$,
    \begin{align*}
    L_i^{N,2}(\mbf U_1,\mbf X_2) &= U_{1,i} + \log\Bigl(\f{1 - e^{U_{1,i-1} - X_{2,i-1}}}{1 - e^{U_{1,i} - X_{2,i}}}\Bigr)  \\
    &= X_{1,i} + \log\Biggl(\f{\sum_{j \in \Z_N} e^{Y_{[i,j]}}  }{\sum_{j \in \Z_N} e^{Y_{[i+1,j]}}}\Biggr) + \log\Biggl(1 - e^{-Y_{i-1}}\f{\sum_{j \in \Z_N} e^{Y_{[i-1,j]}}    }{\sum_{j \in \Z_N}e^{Y_{[i,j]}}}     \Biggr) \\
       &\qquad\qquad\qquad\qquad\qquad - \log\Biggl(1 - e^{-Y_{i}}\f{\sum_{j \in \Z_N} e^{Y_{[i,j]}}    }{\sum_{j \in \Z_N}e^{Y_{[i+1,j]}}}     \Biggr) \\
    &= X_{1,i} + \log\Biggl(\f{\sum_{j \in \Z_N} e^{Y_{[i,j]}} - \sum_{j \in \Z_N} e^{Y_{(i-1,j]}}   }{\sum_{j \in \Z_N} e^{Y_{[i+1,j]}} - \sum_{j \in \Z_N} e^{Y_{(i,j]}}  }\Biggr) = X_{1,i} + \log\Biggl(\f{\prod_{j \in \Z_N} e^{Y_j} - 1}{\prod_{j \in \Z_N} e^{Y_j} - 1}\Biggr) = X_{1,i}.
    \end{align*}
    Conversely, assume that $\mbf X_1 = L^{N,2}(\mbf U_1,\mbf X_2)$. Observe by definition that 
    \[
    X_{1,[i,j]} = U_{1,[i,j]} + \log\Bigl( \f{1 - e^{U_{1,i-1} - X_{2,i-1}}}{1 - e^{U_{1,j} - X_{2,j}}}\Bigr).
    \]
    Then, again with $Y_\ell = X_{2,\ell} - X_{1,\ell}$, we have, for $i \in \Z_N$, 
    \begin{align*}
    T_i^{N,2}(\mbf X_1,\mbf X_2) &= X_{1,i} + \log\Biggl(\f{\sum_{j \in \Z_N} e^{Y_{[i,j]}}  }{\sum_{j \in \Z_N} e^{Y_{[i+1,j]}}}\Biggr) \\
    &= U_{1,i} + \log\Biggl(\f{1 - e^{U_{1,i-1} - X_{2,i-1}}   }{1 - e^{U_{1,i} - X_{2,i}}}\Biggr) +\log\Biggl(\sum_{j \in \Z_N} e^{X_{2,[i,j]} - U_{1,[i,j]}} \Bigl( \f{1 - e^{U_{1,j} - X_{2,j}}}{1 - e^{U_{1,i-1} - X_{2,i-1}}}\Bigr) 
   \Biggr) \\
   &\qquad\qquad\qquad\qquad\qquad - \log\Biggl(\sum_{j \in \Z_N} e^{X_{2,[i+1,j]} - U_{1,[i+1,j]}} \Bigl( \f{1 - e^{U_{1,j} - X_{2,j}}}{1 - e^{U_{1,i} - X_{2,i}}}\Bigr) 
   \Biggr) \\
   &= U_{1,i} + \log\Biggl(\f{\sum_{j \in \Z_N}[e^{X_{2,[i,j]} - U_{1,[i,j]}} - e^{X_{2,[i,j)} - U_{1,[i,j)}}]  }{\sum_{j \in \Z_N}[e^{X_{2,[i+1,j]} - U_{1,[i+1,j]}} - e^{X_{2,[i+1,j)} - U_{1,[i+1,j)}}]  }\Biggr) \\
   &= U_{1,i} + \log\Biggl( \f{\prod_{j \in \Z_N} e^{X_{2,j} - U_{1,j}} - 1}{\prod_{j \in \Z_N} e^{X_{2,j} - U_{1,j}} - 1}\Biggr) = U_{1,i}. \qedhere
    \end{align*}
\end{proof}

\begin{lemma} \label{lem:R-intertwine}
    For $N\in \N$ and all $(\mbf X_1,\mbf X_2) \in \R^{N,2}_{\neq}$,
$T^{N,2}(\mbf X_1,\mbf X_2) = J^{N,2}(D^{N,2}(\mbf X_1,\mbf X_2),\mbf X_1)$.
\end{lemma}
\begin{proof}
 Lemma \ref{lem:sum_v_order} says that $\bigl(D^{N,2}(\mbf X_1,\mbf X_2),\mbf X_1\bigr) \in \mathcal R^{N,2}$ if and only if $(\mbf X_1,\mbf X_2) \in \R_{\neq}^{N,2}$. Since the domain of $J^{N,2}$ is $\mathcal R^{N,2}$, the condition $(\mbf X_1,\mbf X_2) \in \R^{N,2}_{\neq}$ ensures that the right-hand side is well-defined.

Set $Y_\ell = X_{2,\ell} - X_{1,\ell}$, and observe that
\begin{align*}
    X_{1,i} - D_i^{N,2}(\mbf X_1,\mbf X_2) = -Y_i + \log\Bigl(\f{\sum_{j \in \Z_N} e^{Y_{(i-1,j]}}   }{\sum_{j \in \Z_N} e^{Y_{(i,j]}} }    \Bigr) = \log\Bigl(\f{\sum_{j \in \Z_N} e^{Y_{(i-1,j]}}   }{\sum_{j \in \Z_N} e^{Y_{[i,j]}} }    \Bigr).
\end{align*}
Then, we have
\begin{align*}
&\quad \, J_i^{N,2}(D^{N,2}(\mbf X_1,\mbf X_2),\mbf X_1) = X_{1,i} + \log\Bigl(\f{e^{X_{1,i+1} - D_{i+1}^{N,2}(\mbf X_1,\mbf X_2)} - 1}{e^{X_{1,i} - D_i^{N,2}(\mbf X_1,\mbf X_2)} - 1}\Bigr) \\
&= X_{1,i} + \log\Bigl(\f{\sum_{j \in \Z_N} e^{Y_{[i,j]}} }{\sum_{j \in \Z_N} e^{Y_{[i+1,j]}} }\Bigr) + \log\Biggl(\f{\sum_{j \in \Z_N} e^{Y_{(i,j]}} -    \sum_{j \in \Z_N} e^{Y_{[i+1,j]}} }{\sum_{j \in \Z_N} e^{Y_{(i-1,j]}} -   \sum_{j \in \Z_N} e^{Y_{[i,j]}} }\Biggr) \\
&= T_i^{N,2}(\mbf X_1,\mbf X_2) + \log\Biggl(\f{1 - \prod_{j \in \Z_N} e^{Y_j}}{1 - \prod_{j \in \Z_N} e^{Y_j}}\Biggr) = T_i^{N,2}(\mbf X_1,\mbf X_2). \qedhere
\end{align*}
\end{proof}

The following bijection is at the heart of the consistency in Proposition \ref{prop:disc_consis}\ref{itm:mu_perm}. Its proof follows from twice applying the $k=2$ case of Proposition \ref{prop:transform}. A graphical depiction of the proof is given in Figure \ref{fig:Tidentity}.

\begin{figure}
    \centering
    \includegraphics[width=0.7\linewidth]{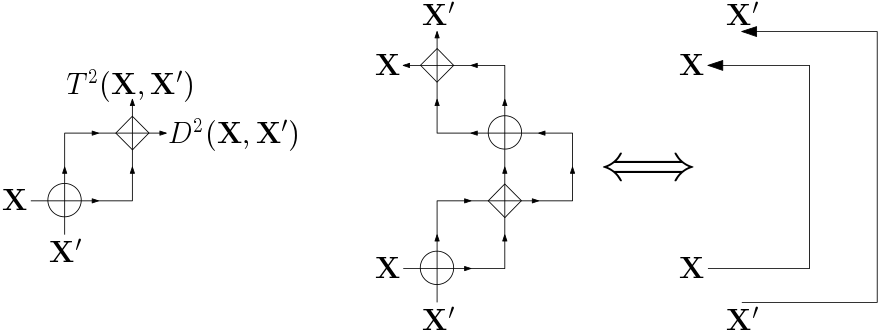}
    \caption{On the left, the map in Corollary \ref{cor:DR_bijection} is depicted as the composition of the circle and diamond maps (see Figure \ref{fig:DJ}). In the middle the map and its proposed inverse are depicted. Twice applying the $k=2$ case of  Proposition \ref{prop:transform}, as graphically depicted in Figure \ref{fig:k2case}, the map is reduced to the identity, proving the corollary.}
    \label{fig:Tidentity}
\end{figure}
\begin{corollary}[Periodic discrete geometric Pitman transform] \label{cor:DR_bijection}
For any $N\in\N$ and $\theta_1, \theta_2 \in \R$, the map $(\mbf X_1,\mbf X_2) \mapsto \bigl(T^{N,2}(\mbf X_1,\mbf X_2),D^{N,2}(\mbf X_1,\mbf X_2)\bigr)$ is a bijection from $\R_{\theta_1}^{\Z_N} \times \R_{\theta_2}^{\Z_N} \to \R_{\theta_1}^{\Z_N} \times \R_{\theta_2}^{\Z_N}$ with inverse $(\mbf V_1,\mbf V_2) \mapsto \bigl(D^{N,2}(\mbf V_2,\mbf V_1),T^{N,2}(\mbf V_2,\mbf V_1)\bigr)$. 
\end{corollary}

\begin{proof}
In the case $\theta_1 \neq \theta_2$, \eqref{eq.TL} implies  that the map and its inverse go between the claimed spaces. If $\theta_1 = \theta_2$, then Lemma \ref{lem:sum_v_order} shows that $D^{N,2}(\mbf X_1,\mbf X_2) = \mbf X_1$, and \eqref{Rmont} shows $T^{N,2}(\mbf X_1,\mbf X_2) = \mbf X_2$, so  the map $(\mbf X_1,\mbf X_2) \mapsto \bigl(T^{N,2}(\mbf X_1,\mbf X_2),D^{N,2}(\mbf X_1,\mbf X_2)\bigr)$ is simply the map $(\mbf X_1,\mbf X_2) \mapsto (\mbf X_2,\mbf X_1)$, which is clearly a bijection. For the rest of the proof, we let $\theta_1 \neq \theta_2$. Assume that
\[
(\mbf V_1,\mbf V_2) = \bigl(T^{N,2}(\mbf X_1,\mbf X_2),D^{N,2}(\mbf X_1,\mbf X_2)\bigr).
\]
By Lemma \ref{lem:R-intertwine}, we have $\mbf V_1 = J^{N,2}(D^{N,2}(\mbf X_1,\mbf X_2),\mbf X_1)$. Proposition \ref{prop:transform} for $k=2$ then implies that 
\[
\mbf X_1 = D^{N,2}(D^{N,2}(\mbf X_1,\mbf X_2),\mbf V_1) = D^{N,2}(\mbf V_2,\mbf V_1).
\]
Then, since $\mbf V_2= D^{N,2}(\mbf X_1,\mbf X_2)$, Proposition \ref{prop:transform}, followed by Lemma \ref{lem:R-intertwine}, implies that 
\[
\mbf X_2 = J^{N,2}(\mbf X_1,\mbf V_2) = J^{N,2}(D^{N,2}(\mbf V_2,\mbf V_1),\mbf V_2) = T^{N,2}(\mbf V_2,\mbf V_1).
\]
If, on the other hand, we assume $(\mbf X_1,\mbf X_2) = (D^{N,2}(\mbf V_2,\mbf V_1),T^{N,2}(\mbf V_2,\mbf V_1))$, then a symmetric argument shows that $\mbf V_1 = T^{N,2}(\mbf X_1,\mbf X_2)$ and $\mbf V_2 = D^{N,2}(\mbf X_1,\mbf X_2)$. 
\end{proof}

\begin{lemma} \label{lem:DT_Add}
For all $N \in \N$, $(\mbf X_1,\mbf X_2) \in \R^{\Z_N} \times \R^{\Z_N}$ and $i \in \Z_N$,
\[
D_i^{N,2}(\mbf X_1,\mbf X_2) + T_{i-1}^{N,2}(\mbf X_1,\mbf X_2) = X_{1,i} + X_{2,i-1}.
\]
\end{lemma}
\begin{proof}
This follows quickly from the definitions: setting $Y_\ell = X_{2,\ell} - X_{1,\ell}$,
\begin{align*}
T_{i-1}^{N,2}(\mbf X_1,\mbf X_2) &= X_{1,i-1} + \log\Biggl(\f{\sum_{j \in \Z_N} e^{Y_{[i-1,j]}}   }{\sum_{j \in \Z_N}e^{Y_{[i,j]}}  }\Biggr) \\
&=  X_{1,i-1}  + Y_{i-1} - Y_{i} - \log\Biggl(\f{\sum_{j \in \Z_N} e^{Y_{(i,j]}}   }{\sum_{j \in \Z_N}e^{Y_{(i-1,j]}}  }\Biggr) \\
&= X_{2,i-1} + X_{1,i} - D_{i}^{N,2}(\mbf X_1,\mbf X_2).\qedhere 
\end{align*}
\end{proof}

\begin{lemma} \label{lem:exp_add_pres}
For  $N\in \N$, $(\mbf X_1,\mbf X_2) \in \R^{\Z_N} \times \R^{\Z_N}$ and $i \in \Z_N$,
\[
e^{-X_{1,i}} + e^{-X_{2,i}} = e^{-T_i^{N,2}(\mbf X_1,\mbf X_2)} + e^{-D_i^{N,2}(\mbf X_1,\mbf X_2)}.
\]
\end{lemma}
\begin{proof}
Let $(\mbf V_1,\mbf V_2) := \bigl(T^{N,2}(\mbf X_1,\mbf X_2),D^{N,2}(\mbf X_1,\mbf X_2)\bigr)$. For $\ell \in \Z_N$, let $Y_\ell = X_{2,\ell} - X_{1,\ell}$ so that by definition
\[
V_{1,i} = X_{1,i} + \log\Biggl(\f{\sum_{j \in \Z_N} e^{Y_{[i,j]}}  }{\sum_{j \in \Z_N} e^{Y_{[i+1,j]}}}\Biggr),\quad \text{and}\quad V_{2,i} = X_{2,i} + \log\Biggl(\f{\sum_{j \in \Z_N} e^{Y_{(i,j]}}  }{\sum_{j \in \Z_N} e^{Y_{(i-1,j]}}}\Biggr).
\]
Then, 
\begin{align*}
 e^{-V_{1,i}} + e^{-V_{2,i}} &= e^{-X_{1,i}}\Biggl(\f{\sum_{j \in \Z_N} e^{Y_{[i+1,j]}}}{\sum_{j \in \Z_N} e^{Y_{[i,j]}}  }\Biggr) + e^{-X_{2,i}}\Biggl(\f{\sum_{j \in \Z_N} e^{Y_{(i-1,j]}}}{\sum_{j \in \Z_N} e^{Y_{(i,j]}}  }\Biggr) \\
&=\f{e^{-X_{2,i}}}{\sum_{j \in \Z_N} e^{Y_{(i,j]}}}\Biggl(\sum_{j \in \Z_N} e^{Y_{[i+1,j]}} + \sum_{j \in \Z_N} e^{Y_{(i-1,j]}} \Biggr). 
\end{align*}
Hence, the statement of the lemma is equivalent to 
\[
(1 + e^{Y_i})\sum_{j \in \Z_N} e^{Y_{(i,j]}} = \sum_{j \in \Z_N} e^{Y_{[i+1,j]}} + \sum_{j \in \Z_N} e^{Y_{(i-1,j]}},
\]
which may be rearranged to 
\[
\sum_{j \in \Z_N} e^{Y_{[i,j]}} - \sum_{j \in \Z_N} e^{Y_{(i - 1,j]}} = \sum_{j \in \Z_N} e^{Y_{[i+1,j]}} - \sum_{j \in \Z_N} e^{Y_{(i,j]}},
\]
and after a cancellation of terms, both sides are seen to be equal to $\prod_{j \in \Z_N} e^{Y_j} - 1$, completing the proof.
\end{proof}

\begin{proposition}[Burke property] \label{prop:Burke}
For $N\in \N$, $\theta_1,\theta_2 \in \R$ and $\beta > 0$, the bijection in Corollary \ref{cor:DR_bijection} preserves the measure $\nu_\beta^{N,(\theta_1,\theta_2)}$ from Definition \ref{def:p_meas}. More generally, for $\gamma_1,\gamma_2 \in \R$, if $\mbf X_1,\mbf X_2 \in \R^{\Z_N}$ are independent, and each is an (unconditioned) i.i.d. sequence of log-inverse-gamma random variables with shape $\gamma_1,\gamma_2$, respectively, and common scale $\beta^{-2}$, then 
this bijection preserves the distribution of $(\mbf X_1,\mbf X_2)$
\end{proposition}
\begin{proof}
We prove the second statement. Then, the first claim of the proposition follows by combining this claim with the fact that the bijection preserves the coordinate sums (Lemmas \ref{lem:Dsum_pres} and \ref{lem:Rbij}), and the fact that the conditional law of $(\mbf X_1,\mbf X_2)$ (well-defined since the log-inverse-gamma distribution has a density) given coordinate sums $\vecsum(\mbf X_1) = \theta_1$ and  $\vecsum(\mbf X_2) = \theta_2$ is precisely $\nu_\beta^{N,(\theta_1,\theta_2)}$.

The joint density of $(X_{1,i},X_{2,i})_{i \in \Z_N}$ on $\R^{\Z_N} \times \R^{\Z_N}$  is proportional to 
\be \label{eq:prod_lg_dens}
\exp\Bigl(- \sum_{i \in \Z_N} (\gamma_1 x_{1,i} + \gamma_2 x_{2,i}) - \beta^{-2} \sum_{i \in \Z_N} (e^{-x_{1,i}} +e^{- x_{2,i}})\Bigr).
\ee
We claim the same for the joint density of $(V_{1,i},V_{2,i})_{i \in \Z_N}$. By Proposition \ref{prop:transform} and Lemmas \ref{lem:Rbij} and \ref{lem:exp_add_pres}, we have 
$$
\sum_{i \in \Z_N} V_{r,i} = \sum_{i \in \Z_N} X_{r,i},\quad \textrm{for }r \in \{1,2\},\qquad \textrm{and}\qquad
e^{-X_{1,i}} + e^{-X_{2,i}} = e^{-V_{1,i}} + e^{-V_{2,i}}\quad \textrm{for }i \in \Z_N.
$$
Substituting these relations into \eqref{eq:prod_lg_dens} preserves the form of the density. All the remains is to show that the determinant of the Jacobian matrix of the bijection has absolute value $1$ (though not the same, this is reminiscent of a similar property of the geometric RSK correspondence, see \cite[Theorem 3.1]{O’Connell2014}). The Jacobian matrix is the $2N \times 2N$ block matrix
\[
\begin{pmatrix}
A_{1,1} & A_{1,2} \\
A_{2,1} & A_{2,2}
\end{pmatrix},
\]
where, for $r,m \in \{1,2\}$ $A_{r,m} = A_{r,m}(\mbf X_1,\mbf X_2)$ is the matrix of derivatives whose $(i,j)$ entry is
\[
\df{\partial V_{r,i}}{\partial X_{m,j}}.
\]
As a convention, $0$ is the starting index for these matrices. We make the observation that by the definition of $D^{N,2}$ and $T^{N,2}$, both $V_{2,i} - X_{2,i}$ and $V_{1,i} -  X_{1,i}$ are functions of $(X_{2,\ell} - X_{1,\ell})_{\ell \in \Z}$. Hence, for $r \in \{1,2\}$, we have the relations
\[
\df{\partial V_{r,i}}{\partial X_{1,j}} + \df{\partial V_{r,i}}{\partial X_{2,j}} = \ind\{i = j\},
\]
which implies that $A_{1,1} + A_{1,2} = A_{2,1} + A_{2,2} = I_N$, where $I_N$ is the $N \times N$ identity matrix. Furthermore, by Lemma \ref{lem:DT_Add}, $V_{2,i} + V_{1,i-1} = X_{1,i} + X_{2,i-1}$, which gives us the relation
\[
\df{\partial V_{2,i}}{\partial X_{1,j}} + \df{\partial V_{1,i-1}}{\partial X_{1,j}} = \ind\{i = j\}.
\]
In other words, the $i-1$st row (in cyclic order) of $A_{1,1}$ plus the $i$th row of $A_{2,1}$ has an entry of $1$ at $i$ and $0$ elsewhere. Below, we use these relations and row operations to get the following equality for the determinant. In the second equality below, we add the columns from the right half of the matrix to those on the left, and in the last equality, we add the $i-1$ row of the top of the matrix to the $i$th row of the bottom of the matrix:
\[
\det \begin{pmatrix}
A_{1,1} & A_{1,2} \\
A_{2,1} & A_{2,2}
\end{pmatrix} = \det \begin{pmatrix}
A_{1,1} & I_N - A_{1,1} \\
A_{2,1} & I_N - A_{2,1}
\end{pmatrix} = \det \begin{pmatrix}
A_{1,1} & I_N \\
A_{2,1} & I_N
\end{pmatrix} = \det \begin{pmatrix}
A_{1,1} & I_N \\
I_N & M
\end{pmatrix}.
\]
Here, we define the $N \times N$ matrix $M = (m_{i,j})_{i,j \in \Z_N}$ so that, for $i \in \Z_N$, $m_{i,i} = m_{i,i-1} = 1$ and $m_{i,j} = 0$ for $j \notin \{i-1,1\}$. 
To compute this determinant, we use the following standard identity for block matrices (with square matrix entries of the same size), under the assumption that $A$ is invertible:
\[
\det \begin{pmatrix}
A & B \\
C & D
\end{pmatrix} = \det(A)\det(D - CA^{-1} B).
\]
Applied in our setting,
\[
\det \begin{pmatrix}
A_{1,1} & I \\
I & M
\end{pmatrix} = \det(A_{1,1})\det(M - A_{1,1}^{-1}).
\]
To complete the proof, we show the right-hand side equals $(-1)^{N-1}$ by showing that $\det(M - A_{1,1}^{-1}) = (-1)^{N-1}\det(A_{1,1}^{-1})$. 
We first justify that $A_{1,1} = A_{1,1}(\mbf X_1,\mbf X_2)$ is invertible for Lebesgue-a.e. $\mbf X_1,\mbf X_2$ as follows: recall that $A_{1,1}$ is the matrix of partial derivatives $\df{\partial V_{1,i}}{\partial X_{1,j}}$. For a fixed choice of $\mbf X_2$, Lemma \ref{lem:Rbij} shows that $\mbf X_1 \mapsto T^{N,2}(\mbf X_1,\mbf X_2)$ is a bijection between the set 
\[
\{\mbf X_1 \in \R^{\Z_N}: (\mbf X_1,\mbf X_2) \in \R^{N,2}_{\neq} \}\quad\text{and the set}\quad \{\mbf U_1 \in \R^{\Z_N}: (\mbf U_1,\mbf X_2) \in \mathcal R^{N,2} \}.
\]
For each fixed choice of $\mbf X_2$, $A_{1,1}$ is the matrix of partial derivatives of this bijection and hence is invertible. Furthermore, $A_{1,1}^{-1}$ is the Jacobian matrix of the inverse transformation $\mbf V_1 \mapsto L^{N,2}(\mbf V_1,\mbf X_2)$. We recall the definition of the map $L^{N,2}$ \eqref{Ldef}:
\[
X_{1,i} = L_i^{N,2}(V_1,X_2) = V_{1,i} + \log\Biggl(\f{1 - e^{V_{1,i-1} - X_{2,i-1}}}{1 - e^{V_{1,i} - X_{1,i}}}\Biggr).
\]
We see that, for $i \in \Z_N$,
$
\df{\partial X_{1,i}}{\partial V_{1,i} } + \df{\partial X_{1,i+1}}{\partial V_{1,i} } = 1,
$
while $\df{\partial X_{1,i}}{ \partial V_{1,j}} = 0$ for $j \notin \{i,i-1\}$. 

This implies that $\det(M - A_{1,1}^{-1}) = (-1)^{N-1}\det(A_{1,1}^{-1})$, as desired. To see this last deduction, we observe more generally that for any choice $a_0,\ldots,a_{N-1}\in \R$, if we define  
\begin{align*}
  B:= \begin{pmatrix}
  a_0 &0 &0 &\cdots &0 &1-a_{N-1} \\
  1 - a_0 &a_1 &0 &\cdots &0 &0 \\
  0 &1 - a_1 &a_2 &\cdots &0 &0 \\
  \vdots &\ddots & & & \vdots &\vdots \\
  0 &0 &0 &\cdots &1 - a_{N-2} &a_{N-1}
  \end{pmatrix},\qquad \textrm{and} \qquad
  M:= \begin{pmatrix}
  1 &0 &0 &\cdots &0 &1 \\
  1  &1 &0 &\cdots &0 &0 \\
  0 &1 &1 &\cdots &0 &0 \\
  \vdots &\ddots & & & \vdots &\vdots \\
  0 &0 &0 &\cdots &1 &1
  \end{pmatrix},
\end{align*}
then we have $\det(M-B) = (-1)^{N-1} \det(B)$. This is proved by cofactor expansion which shows that 
\begin{align*}
\det B &= \prod_{i \in \Z_N} a_i + (-1)^{N-1} \prod_{i \in \Z_N}(1 - a_i),\qquad\text{and}\\
\det(M - B) &= \prod_{i \in \Z_N}(1-a_i) + (-1)^{N-1}\prod_{i \in \Z_N} a_i = (-1)^{N-1} \det(B).\qedhere
\end{align*}
\end{proof}

The final ingredient we will use to prove the consistency and symmetry of the measures $\mu_\beta^{N,(\theta_1,\ldots,\theta_k)}$ from \eqref{def:mu_meas} is the following intertwining result. This identity is depicted graphically in Figure \ref{fig:intertwining}.
\begin{figure}
    \centering
    \includegraphics[width=0.6\linewidth]{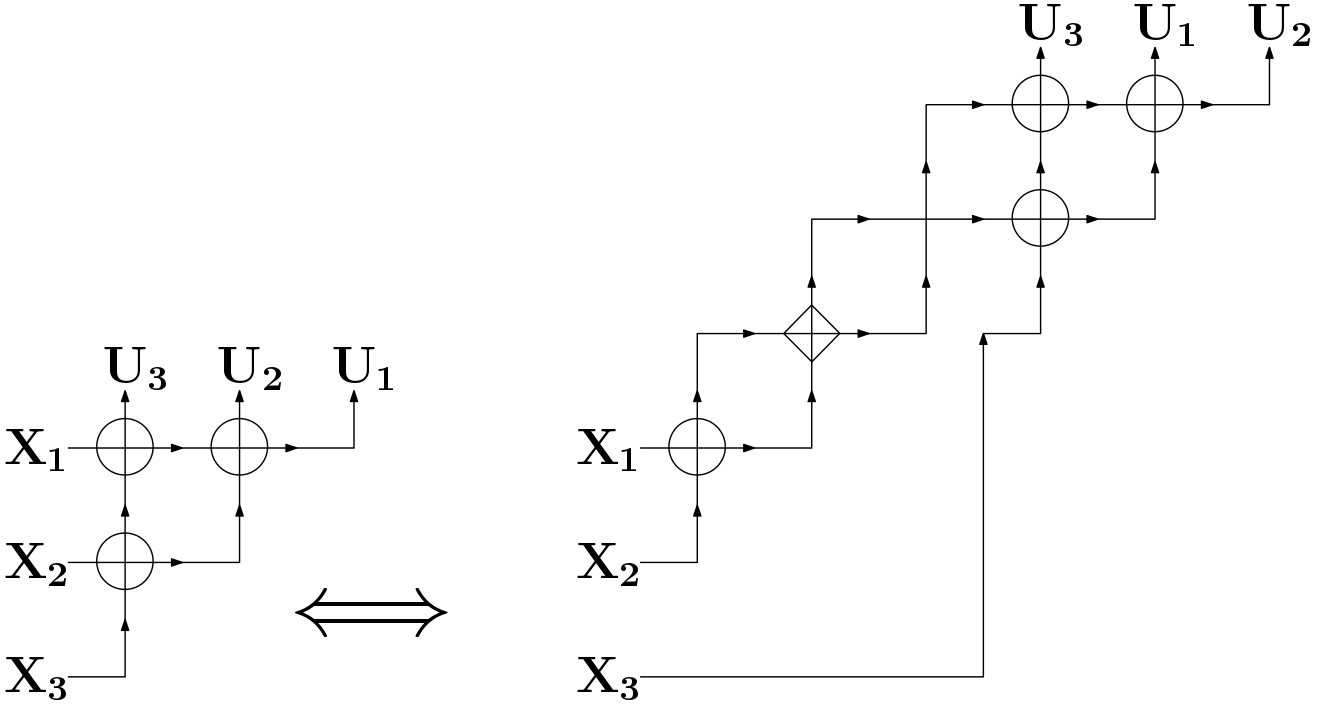}
    \caption{Graphical depiction of the identity in Proposition \ref{prop:full_intertwine}. The crossed paths (without a circle or diamond) act as the identity, mapping the horizontal/vertical inputs to their respective horizontal/vertical outputs. The identity only claims that the left-most vertical output matches between the two sides. However, it is easily deduced from Corollary \ref{cor:DR_bijection} that the other two vertical outputs are swapped between the two sides as depicted above.}
    \label{fig:intertwining}
\end{figure}
\begin{proposition} \label{prop:full_intertwine}
For $(\mbf X_1,\mbf X_2,\mbf X_3) \in (\R^{\Z_N})^3$, 
\[
D^{N,3}(\mbf X_1,\mbf X_2,\mbf X_3) = D^{N,3}\Bigl(D^{N,2}(\mbf X_1,\mbf X_2),T^{N,2}(\mbf X_1,\mbf X_2), \mbf X_3\Bigr).
\]
\end{proposition}
\begin{proof}
For ease of notation, let $\mbf X_1' = D^{N,2}(\mbf X_1,\mbf X_2)$ and $\mbf X_2' = T^{N,2}(\mbf X_1,\mbf X_2)$. By Lemma \ref{lem:Dnm_alt}, it suffices to show that, for $i \in \Z_N$, $Q_i^{N,3}(\mbf X_1,\mbf X_2,\mbf X_3) = Q_i^{N,3}(\mbf X_1',\mbf X_2',\mbf X_3)$, meaning 
\be \label{eq:123_eq}
\sum_{j_1,j_2 \in \Z_N} e^{X_{3,(i,j_1]} - X_{1,(i,j_1]} + X_{3,(j_1,j_2]} - X_{2,(j_1,j_2]}} = \sum_{j_1,j_2 \in \Z_N} e^{X_{3,(i,j_1]} - X_{1,(i,j_1]}' + X_{3,(j_1,j_2]} - X_{2,(j_1,j_2]}'}.
\ee
We adopt the following shorthand notation:
\be \label{400}
Y_{1,\ell} = X_{3,\ell}- X_{1,\ell}, \qquad Y_{1,\ell}' = X_{3,\ell} - X_{1,\ell}',\qquad Y_{2,\ell} = X_{3,\ell}- X_{2,\ell},\qquad\text{and}\qquad  Y_{2,\ell}' = X_{3,\ell}- X_{1,\ell}'.
\ee
After changing the indices of summation we may rewrite the sum on the left-hand side of \eqref{eq:123_eq} as
\be \label{401}
\sum_{\substack{j,\ell \in \Z \\ 0 \le j \le \min(\ell,N-1) \\ 0 \le \ell - j \le N-1}} e^{Y_{1,(i,i+j]} +  Y_{2,(i+j,i+\ell]}}.
\ee
We note the subtle choice here that we take $j,\ell \in \Z$ instead of $\Z_N$, although we identify the indices of the vectors with their respective elements on $\Z_N$. We do this for the following reason. Writing 
\[
Y_{1,(i,i+j]}+   Y_{2,(i+j,i+\ell]} = \sum_{m = 1}^{j} Y_{1,i+m} + \sum_{m = j+1}^\ell Y_{2,i+m},
\]
(where the first sum is $0$ if $j = 0$ and the second sum is $0$ if $\ell = j$), we have $j$ terms of the form $Y_{1,i+m}$ and $\ell - j$ terms of the form $Y_{2,i+m}$; a total of $\ell$ terms. The conditions on $\ell$ and $j$ in \eqref{401} imply that $0 \le \ell \le N-1 + j \le 2N - 2$. Then, the sum in \eqref{401} is equal to
\be \label{402}
\sum_{\ell = 1}^{2N - 2}\!\!\!\!\!\!\!\! \sum_{\substack{j \in \Z \\ 0 \le j \le \min(\ell,N-1) \\ 0 \le \ell - j \le N-1 }}e^{Y_{1,(i,i+j]} +  Y_{2,(i+j,i+\ell]} } 
=\sum_{\ell = 0}^{N-1} \sum_{j = 0}^\ell e^{Y_{1,(i,i+j]} + Y_{2,(i+j,i+\ell]} } + \sum_{\ell = N}^{2N - 2} \sum_{j = \ell - N+1}^{N-1} e^{Y_{1,(i,i+j]}  +  Y_{2,(i+j,i+\ell]}}.
\ee
We prove the proposition by showing that the sum in \eqref{402} is unchanged if we replace $Y_{r,j}$ with $Y_{r,j}'$ for $r \in \{1,2\}$ and $j \in \Z_N$.
For this, it suffices to show two things:
\begin{enumerate} [label=\textup{(\roman*)}]
\item \label{itm:sl} For $i \in \Z_N$ and $0 \le \ell \le N-1$,
\be \label{eq:sl}
\sum_{j = 0}^\ell e^{Y_{1,(i,i+j]} +  Y_{2,(i+j,i+\ell]}} = \sum_{j = 0}^\ell e^{Y_{1,(i,i+j]}' +   Y_{2,(i+j,i+\ell]}'}.
\ee
\item \label{itm:BL} For $i \in \Z_N$ and $N \le \ell \le 2N - 2$,
\be \label{eq:BL}
\sum_{j = \ell - N+1}^{N-1} e^{Y_{1,(i,i+j]} + Y_{2,(i+j,i+\ell]}} = \sum_{j = \ell - N+1}^{N-1} e^{Y_{1,(i,i+j]}' + Y_{2,(i+j,i+\ell]}'}.
\ee
\end{enumerate}

To prove these, we make frequent use of the two key identities from Lemmas \ref{lem:DT_Add} and \ref{lem:exp_add_pres}: for $i \in \Z_N$,
\[
e^{-X_{1,i}'} + e^{-X_{2,i}'} = e^{-X_{1,i}} + e^{-X_{2,i}}\qquad \text{and}\qquad X_{1,i}' + X_{2,i-1}' = X_{1,i} + X_{2,i-1},
\]
which, by the definitions in \eqref{400}, imply the following:
\begin{align}
\quad e^{Y_{1,i}'} + e^{Y_{2,i}'} &= e^{Y_{1,i}} + e^{Y_{2,i}},\qquad\text{and} \label{403}\\
 Y_{1,i}' + Y_{2,i-1}' &= Y_{1,i} + Y_{2,i-1}. \label{404}
\end{align}

We start by proving \ref{itm:sl} by strong induction on $\ell$. For $i \in \Z_N$, the $\ell = 0$ case of \eqref{eq:sl} is simply $1 = 1$. Replacing $i$ with $i+1$ in \eqref{403} gives the $\ell = 1$ case of \eqref{eq:sl}. Now, assume by induction, that, for some $1 \le q \le N-2$, \eqref{eq:sl} holds for $\ell = q,\ell = q-1$, and all $i \in \Z_N$. Then, for $i \in \Z_N$ and $\ell = q+1$, we write 
\begin{align*}
&\quad \, \sum_{j = 0}^{q + 1} e^{Y_{1,(i,i+j]} +  Y_{2,(i+j,i+q + 1]}} = e^{Y_{2,i+1} + Y_{2,(i+1,i+q+1]}} + e^{Y_{1,i+1}}\sum_{j = 1}^{q + 1} e^{Y_{1,(i+1 ,i+j]} +  Y_{2,(i+j,i+q+1]}} \\
&=e^{Y_{2,i+1}} \Biggl(\sum_{j = 1}^{q + 1} e^{Y_{1,(i+1 ,i+j]} +  Y_{2,(i+j,i+q+1]}} - \sum_{j = 2}^{q+1}e^{Y_{1,(i+1 ,i+j]} +  Y_{2,(i+j,i+q+1]}}  \Biggr) + e^{Y_{1,i+1}}\sum_{j = 1}^{q + 1} e^{Y_{1,(i+1 ,i+j]} +  Y_{2,(i+j,i+q+1]}}.
\end{align*}
Now, shift the indices of the sums to start at $j = 0$ and combine terms to obtain
\be \label{405}
\begin{aligned} 
&\quad \, \sum_{j = 0}^{q + 1} e^{Y_{1,(i,i+j]} +  Y_{2,(i+j,i+q + 1]}} \\
&= \bigl(e^{Y_{2,i+1}} + e^{Y_{1,i+1}}\bigr)\sum_{j = 0}^{q } e^{Y_{1,(i+1 ,i+1+j]} +  Y_{2,(i+1 +j,i+1+q]}} - e^{Y_{2,i+1} + Y_{1,i+2}}\sum_{j = 0}^{q-1} e^{Y_{1,(i+2 ,i+2+j]} +  Y_{2,(i+2+j,i+2 + q-1]}}.
\end{aligned}
\ee
Note that \eqref{405} simply is a rearranging of terms, and hence, it also holds if we replace $Y$ with $Y'$.
Next, by applying \eqref{403} and \eqref{404}, followed by the strong induction assumption to $\ell = q,i+1$ and $\ell = q-1,i+2$, we obtain
\begin{align*}
&\quad \, \bigl(e^{Y_{2,i+1}} + e^{Y_{1,i+1}}\bigr)\sum_{j = 0}^{q } e^{Y_{1,(i+1 ,i+1+j]} +  Y_{2,(i+1 +j,i+1+q]}} - e^{Y_{2,i+1} + Y_{1,i+2}}\sum_{j = 0}^{q-1} e^{Y_{1,(i+2 ,i+2+j]} +  Y_{2,(i+2+j,i+2 + q-1]}} \\
&=\bigl(e^{Y_{2,i+1}'} + e^{Y_{1,i+1}'}\bigr)\sum_{j = 0}^{q } e^{Y_{1,(i+1 ,i+1+j]} +  Y_{2,(i+1 +j,i+1+q]}} - e^{Y'_{2,i+1} + Y'_{1,i+2}}\sum_{j = 0}^{q-1} e^{Y_{1,(i+2 ,i+2+j]} +  Y_{2,(i+2+j,i+2 + q-1]}} \\
&= \bigl(e^{Y_{2,i+1}'} + e^{Y_{1,i+1}'}\bigr)\sum_{j = 0}^{q } e^{Y_{1,(i+1 ,i+1+j]}' +  Y_{2,(i+1 +j,i+1+q]}'} - e^{Y_{2,i+1}' + Y_{1,i+2}'}\sum_{j = 0}^{q-1} e^{Y_{1,(i+2 ,i+2+j]}' +  Y_{2,(i+2+j,i+2 + q-1]}'}. 
\end{align*}
Then, applying \eqref{405} to both $Y$ and $Y'$, we have
\[
\sum_{j = 0}^{q + 1} e^{Y_{1,(i,i+j]} +  Y_{2,(i+j,i+q + 1]}} = \sum_{j = 0}^{q + 1} e^{Y_{1,(i,i+j]}' +  Y_{2,(i+j,i+q + 1]}'},
\]
thus proving \eqref{eq:sl} for $\ell = q+1$ and all $i \in \Z_N$. 

We turn to proving \ref{itm:BL}. Observe that, for a vector $\mbf X \in \R^{\Z_N}$, whenever $0 \le j \le m \le N-1$, 
\be \label{406}
X_{(i,i+m]} = X_{(i,i+j]} + X_{(i+j,i+m]}.
\ee
The condition that $m \le N-1$ is important; otherwise the index $i+m$ starts to wrap around the circle. For example, if $m = N$ and $j =1$, then
\[
X_{(i,i+m]} = X_{(i,i]} = 0 \neq X_{(i,i+1]} + X_{(i+1,i]} = X_{(i,i+j]} + X_{(i+j,i+m]}.
\]
Let $N \le \ell \le 2N -2$. Using \eqref{406}, observe that 
\be \label{407}
\begin{aligned}
&\sum_{j = \ell - N+1}^{N-1} e^{Y_{1,(i,i+j]} + Y_{2,(i+j,i+\ell]}} \\
&= e^{Y_{1,(i,i+\ell -N+1 ]} + Y_{2,(i+N-1,i+\ell]}} \sum_{j = \ell - N + 1}^{N-1} e^{Y_{1,(i+\ell - N +1,i+j]} + Y_{2,(i+j,i +N-1]}} \\
&= e^{Y_{1,(i,i+\ell-N +1 ]} + Y_{2,(i+N-1,i+\ell]}} \sum_{j = 0}^{2N - 2 - \ell} e^{Y_{1,(i+\ell-N +1,i+\ell-N + 1+ j ]} + Y_{2,(i+\ell-N + 1+ j,i+ N- 1]}},
\end{aligned}
\ee
where in the last line, we shifted the index $j$. Observe \eqref{407} is a rearranging of terms; it also holds if we replace $Y$ with $Y'$. Since $Y_i = Y_{i+N}$ for $i \in \Z_N$,
\be \label{408}
\begin{aligned}
\, Y_{1,(i,i+\ell-N +1 ]} + Y_{2,(i+N-1,i+\ell]} & \,\,\,\,= \sum_{j = i+1}^{i+\ell-N +1} Y_{1,j} + \sum_{j = i+N}^{i+\ell} Y_{2,j}\\
 = \sum_{j = i+1}^{i+\ell - N+1} [Y_{1,j} + Y_{2,j-1}] & \overset{\eqref{404}}{=} \sum_{j = i+1}^{i+\ell - N+1} [Y_{1,j}' + Y_{2,j-1}'] =  Y_{1,(i,i+\ell-N +1 ]}' + Y_{2,(i+N-1,i+\ell]}',
\end{aligned}
\ee
where the last equality follows by the first two equalities with $Y'$ in place of $Y$. 
Furthermore, we observe that $2N - 2 - \ell \le N-1$, so by Item \ref{itm:sl}, we obtain
\be \label{409}
\begin{aligned}
&\quad \, \sum_{j = 0}^{2N - 2 - \ell} e^{Y_{1,(i+\ell-N +1,i+\ell-N + 1+ j ]} + Y_{2,(i+\ell-N + 1+ j,i+ N- 1]}} \\
&= \sum_{j = 0}^{2N - 2 - \ell} e^{Y_{1,(i+\ell-N +1,i+\ell-N + 1+ j ]}' + Y_{2,(i+\ell-N + 1+ j,i+ N- 1]}'}.
\end{aligned}
\ee
Then, by \eqref{408} and \eqref{409}, and applying \eqref{407} to both $Y$ and $Y'$, we obtain
\begin{align*}
&\, \sum_{j = \ell - N+1}^{N-1} e^{Y_{1,(i,i+j]} + Y_{2,(i+j,i+\ell]}} \\
&= e^{Y_{1,(i,i+\ell-N +1 ]} + Y_{2,(i+N-1,i+\ell]}} \sum_{j = 0}^{2N - 2 - \ell} e^{Y_{1,(i+\ell-N +1,i+\ell-N + 1+ j ]} + Y_{2,(i+\ell-N + 1+ j,i+ N- 1]}} \\
&= e^{Y_{1,(i,i+\ell-N +1 ]}' + Y_{2,(i+N-1,i+\ell]}'} \sum_{j = 0}^{2N - 2 - \ell} e^{Y_{1,(i+\ell-N +1,i+\ell-N + 1+ j ]}' + Y_{2,(i+\ell-N + 1+ j,i+ N- 1]}'} \\
&= \sum_{j = \ell - N+1}^{N-1} e^{Y_{1,(i,i+j]}' + Y_{2,(i+j,i+\ell]}'}. \qedhere 
\end{align*}
\end{proof}

\begin{figure}
    \centering
    \includegraphics[width=0.7\linewidth]{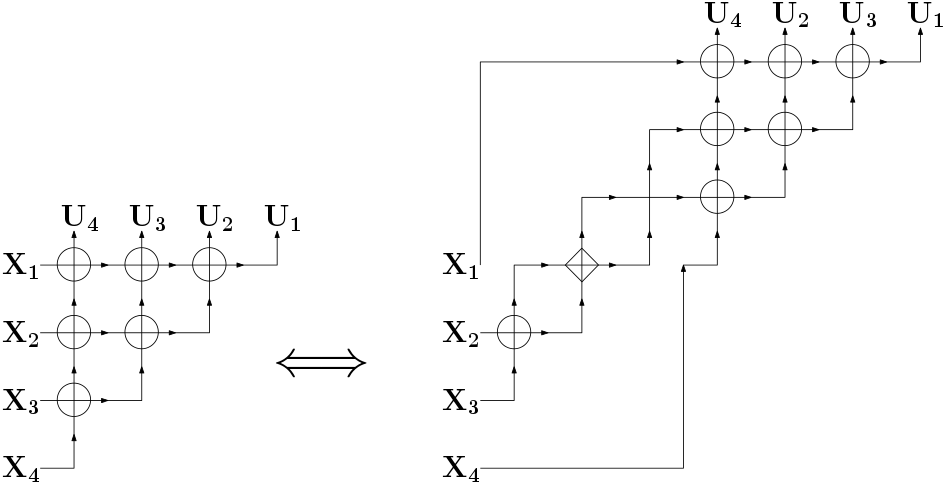}
    \caption{Graphical depiction of the result of Corollary \ref{cor:D_perm} for $k=4$. If the output of the first figure is the vector $(\mbf U_1,\mbf U_2,\mbf U_3,\mbf U_4)$ then the output of the second figure is  $(\mbf U_1,\mbf U_3,\mbf U_2,\mbf U_4)$, i.e., permuted at slots $2$ and $3$. The proof of the corollary relies on the identity depicted in Figure \ref{fig:intertwining}.}
    \label{fig:k4intertwining}
\end{figure}

We obtain the following as a corollary of Corollary \ref{cor:DR_bijection} and Proposition \ref{prop:full_intertwine}. It shows how we may write a permutation of the output of $\D^{N,k}$ as the output of $\D^{N,k}$ for a different choice of inputs. In particular, if we  apply the bijection from Corollary \ref{cor:DR_bijection} to the input in slots $m-1$ and $m$, then reverse the order of these inputs, this results in permuting the output in those slots. See Figure \ref{fig:k4intertwining} for a graphical depiction of this result. 
\begin{corollary} \label{cor:D_perm}
    Let $N \ge 1$, $k \ge 2$, and $(\mbf X_1,\ldots,\mbf X_k) \in (\R^{\Z_N})^k$, and define
    \[
    (\mbf U_1,\ldots,\mbf U_k) := \mathcal D^{N,k}(\mbf X_1,\ldots,\mbf X_k).
    \]
    Then, for any $1 \le m \le k$,
    \begin{multline*}
    (\mbf U_1,\ldots,\mbf U_{m-2},\mbf U_m,\mbf U_{m-1},\mbf U_{m+1},\ldots,\mbf U_{k}) \\
    = \mathcal D^{N,k}\bigl(\mbf X_1,\ldots,\mbf X_{m-2}, D^{N,2}(\mbf X_{m-1},\mbf X_m),T^{N,2}(\mbf X_{m-1},\mbf X_m),\mbf X_{m+1},\ldots,\mbf X_k\bigr).
    \end{multline*}
\end{corollary}
\begin{proof}
To ease the notation, define $\mbf X_{m-1}' = D^{N,2}(\mbf X_{m-1},\mbf X_m)$, and $\mbf X_m' = T^{N,2}(\mbf X_{m-1},\mbf X_m)$. For $r \notin \{m-1,m\}$, set $\mbf X_r' = \mbf X_r$ and define $(\mbf U_1',\ldots,\mbf U_k') = \D^{N,k}(\mbf X_1',\ldots,\mbf X_k')$. We wish to prove that $\mbf U_r = \mbf U_r'$ for $r \notin \{m-1,m\}$, $\mbf U_{m-1}' = \mbf U_m$, and $\mbf U_m' = \mbf U_{m-1}$. We break this up into three cases:

\medskip \noindent \textbf{Case 1:} ($r < m - 1$)
We have $\mbf X_\ell = \mbf X_\ell'$ for $\ell \le r$, so by definition of the map $\D^{N,k}$, 
\[
\mbf U_r = D^{N,r}(\mbf X_1,\ldots,\mbf X_r) = D^{N,r}(\mbf X_1',\ldots,\mbf X_r') = \mbf U_r'.
\]

\medskip \noindent \textbf{Case 2:} (Showing $\mbf U_{m-1}' = \mbf U_m$ and $\mbf U_m' = \mbf U_{m-1}$)  
Using the definition of $\mbf U_{m-1}'$ and the map $\mathcal D^{N,k}$,
\[
\mbf U_{m-1}' = D^{N,m-1}(\mbf X_1',\ldots,\mbf X_{m-1}') = D^{N,m-1}(\mbf X_1,\ldots,\mbf X_{m-2},D^{N,2}(\mbf X_{m-1},\mbf X_m)).
\]
 By Lemma \ref{lem:Diter}, this equals $D^{N,m}(\mbf X_1,\ldots,\mbf X_m) = \mbf U_m$. Using Lemma \ref{lem:Diter} again, 
\begin{align*}
\mbf U_m' &= D^{N,m}(\mbf X_1,\ldots,\mbf X_{m-2},\mbf X_{m-1}',\mbf X_m') = D^{N,m-1}\Bigl(\mbf X_1,\ldots,\mbf X_{m-2},D^{N,2}\bigl(\mbf X_{m-1}',\mbf X_{m}'\bigr)\Bigr).
\end{align*}
Since $\mbf X_{m-1}' = D^{N,2}(\mbf X_{m-1},\mbf X_m)$, and $\mbf X_m' = T^{N,2}(\mbf X_{m-1},\mbf X_m)$, Corollary \ref{cor:DR_bijection} implies that\\ $D^{N,2}(\mbf X_{m-1}',\mbf X_{m}') = \mbf X_{m-1}$. Hence, 
\[
\mbf U_m' = D^{N,m-1}(\mbf X_1,\ldots,\mbf X_{m-1}) = \mbf U_{m-1}.
\]

\medskip \noindent \textbf{Case 3:} ($r > m$) In this case, we seek to show that $\mbf U_r = \mbf U_r'$. By definition and Lemma \ref{lem:Diter},
\begin{multline*}
\mbf U_r' = D^{N,r}(\mbf X_1,\ldots,\mbf X_{m-2},\mbf X_{m-1}',\mbf X_{m}',\mbf X_{m+1},\ldots,\mbf X_r) \\
=D^{N,m-1}\bigl(\mbf X_1,\ldots,\mbf X_{m-2}, D^{N,r-m+2}(\mbf X_{m-1}',\mbf X_m',\mbf X_{m+1},\ldots,\mbf X_r)\bigr).
\end{multline*}
By another application of Lemma \ref{lem:Diter}, it suffices to show that 
\be \label{201}
D^{N,r-m+2}(\mbf X_{m-1}',\mbf X_m',\mbf X_{m+1},\ldots,\mbf X_r) = D^{N,r-m+2}(\mbf X_{m-1},\mbf X_m,\mbf X_{m+1},\ldots,\mbf X_r).
\ee
Using Lemma \ref{lem:Diter}, then Proposition \ref{prop:full_intertwine}, the left-hand side of \eqref{201} is
\[
D^{N,3}\bigl(\mbf X_{m-1}',\mbf X_m', D^{N,r-m}(\mbf X_{m+1},\ldots,\mbf X_r)\bigr) = D^{N,3}\bigl(\mbf X_{m-1},\mbf X_m,D^{N,r-m}(\mbf X_{m+1},\ldots,\mbf X_r) \bigr),
\]
and this is equal to the right-hand side of \eqref{201} by another application of Lemma \ref{lem:Diter}. 
\end{proof}

For the next corollary, recall the set $\mathcal R^{N,k}$ defined early in Section \ref{sec.sst}. This corollary is not used later, but is a nice symmetry of the sets $\mathcal R^{N,k}$.
\begin{corollary} \label{cor:RNK_permute}
For $k,N \ge 1$, if $(\mbf U_1,\ldots,\mbf U_k) \in \mathcal R^{N,k}$, then for any permutation $\sigma \in \mathcal S(k)$, 
\[
\bigl(\mbf U_{\sigma(1)},\ldots,\mbf U_{\sigma(k)}\bigr) \in \mathcal R^{N,k}.
\]
\end{corollary}
\begin{proof}
It suffices to prove this for transpositions $(m-1,m)$ for $m \in \{2,\ldots,k\}$.
By Proposition \ref{prop:transform}, and $(\mbf U_1,\ldots,\mbf U_k) \in \mathcal R^{N,k}$ is equal to $\D^{N,k}(\mbf X_1,\ldots,\mbf X_k)$ for some $(\mbf X_1,\ldots,\mbf X_k) \in \R_{\neq}^{N,k}$. Define $\mbf X_r' = \mbf X_r$ for $r \notin \{m-1,m\}$, $\mbf X_{m-1}' = D^{N,2}(\mbf X_{m-1},\mbf X_m)$, and $\mbf X_m' = T^{N,2}(\mbf X_{m-1},\mbf X_m)$.  By Lemmas \ref{lem:Dsum_pres} and \ref{lem:Rbij}, $\vecsum(\mbf X_{m-1}') = \vecsum(\mbf X_m)$ and $\vecsum(\mbf X_m') = \vecsum(\mbf X_{m-1})$, so $(\mbf X_1',\ldots,\mbf X_k') \in \R_{\neq}^{N,k}$ By Corollary \ref{cor:D_perm}, $(\mbf U_{\sigma(1)},\ldots,\mbf U_{\sigma(k)}) = \D^{N,k}(\mbf X_1',\ldots,\mbf X_k')$, so by Proposition \ref{prop:transform}, $(\mbf U_{\sigma(1)},\ldots,\mbf U_{\sigma(k)}) \in \mathcal R^{N,k}$. 
\end{proof}

We are now prepared to prove Proposition \ref{prop:disc_consis}.
\begin{proof}[Proof of Proposition \ref{prop:disc_consis}] We address each of the four items sequentially. 

\noindent \textbf{Item \ref{itm:mu_perm}:}
Let $(\mbf X_1,\ldots,\mbf X_k) \sim \nu_\beta^{N,(\theta_1,\ldots,\theta_k)}$, and let $(\mbf U_1,\ldots,\mbf U_k) = \D^{N,k}(\mbf X_1,\ldots,\mbf X_k)$ so that $(\mbf U_1,\ldots,\mbf U_k) \sim \mu_\beta^{N,(\theta_1,\ldots,\theta_k)}$. We wish to show that, for any $\sigma \in \mathcal S(k)$ and $1 \le m \le k$,
\[
(\mbf U_{\sigma(1)},\ldots,\mbf U_{\sigma(m)}) \sim \mu_\beta^{N,(\theta_{\sigma(1)},\ldots,\theta_{\sigma(m)})}.
\]
We first show that
\be \label{k_cutoff}
(\mbf U_{\sigma(1)},\ldots,\mbf U_{\sigma(k)}) \sim \mu_\beta^{N,(\theta_{\sigma(1)},\ldots,\theta_{\sigma(k)})}.
\ee
For this, it  suffices to consider transpositions $(m-1,m)$ for any $m \in \{2,\ldots,k\}$. 
For $r \notin \{m-1,m\}$, define $\mbf X_r' := \mbf X_r$, define $\mbf X_{m-1}' := D^{N,2}(\mbf X_{m-1},\mbf X_m)$, and define $\mbf X_m' := T^{N,2}(\mbf X_{m-1},\mbf X_m)$. By Proposition \ref{prop:Burke}, $(\mbf X_1',\ldots,\mbf X_k') \sim \nu_\beta^{N,(\theta_1,\ldots,\theta_{m-2},\theta_m,\theta_{m-1},\theta_{m+1},\ldots,\theta_k)}$, and by Corollary \ref{cor:D_perm},
\[
(\mbf U_1,\ldots,\mbf U_{m-2},\mbf U_m,\mbf U_{m-1},\mbf U_{m+1},\ldots,\mbf U_k) = \D^{N,k}(\mbf X_1',\ldots,\mbf X_k') \sim \mu _\beta^{N,(\theta_1,\ldots,\theta_{m-2},\theta_m,\theta_{m-1},\theta_{m+1},\ldots,\theta_k)}, 
\]
and thus we have proved \eqref{k_cutoff}. Hence, 
$
(\mbf U_{\sigma(1)},\ldots,\mbf U_{\sigma(k)}) \deq \D^{N,k}(\wt{\mbf X}_1,\ldots,\wt{\mbf X}_k)
$
where $(\wt{\mbf X}_1,\ldots,\wt{\mbf X}_k) \sim \nu_\beta^{N,(\theta_{\sigma(1)},\ldots,\theta_{\sigma(k)})}$.
Since this is a product measure, for $1 \le m \le k$, $(\wt{\mbf X}_1,\ldots,\wt{\mbf X}_m) \sim \nu_\beta^{N,(\theta_{\sigma(1)},\ldots,\theta_{\sigma(m)})}$.
Then, by the inductive definition of $\D^{N,k}$, \eqref{DNk_intro}, under restriction we have
\[
(\mbf U_{\sigma(1)},\ldots,\mbf U_{\sigma(m)}) \deq \D^{N,m}(\wt{\mbf X}_1,\ldots,\wt{\mbf X}_m) \sim \mu_\beta^{N,(\theta_{\sigma(1)},\ldots,\theta_{\sigma(m)})}. 
\]

\medskip \noindent \textbf{Item \ref{itm:mu_consis}:} If $(\mbf U_1,\ldots,\mbf U_k) = \D^{N,k}(\mbf X_1,\ldots,\mbf X_k)$, then $\mbf U_1 = \mbf X_1$ by definition. Hence, $\mbf U_1 \sim \mu_\beta^{N,(\theta_1)}$ whenever $(\mbf U_1,\ldots,\mbf U_k) \sim \mu_\beta^{N,(\theta_1,\ldots,\theta_k)}$. The case of general $\mbf U_m$ follows by Item \ref{itm:mu_perm}.

\medskip \noindent \textbf{Item \ref{itm:mu_order}:} This is a direct consequence of Lemma \ref{lem:sum_v_order}.

\medskip \noindent \textbf{Item \ref{itm:mu_cont}:} It is immediate that the function $D^{N,2}:\R^{\Z_N} \times \R^{\Z_N} \to \R^{\Z_N}$ is continuous, so the function $\D^{N,k}:(\R^{\Z_N})^k \to (\R^{\Z_N})^k$ is continuous, since it is built inductively from $D^{N,2}$ via composition of functions. Hence, it suffices to show that $\nu_\beta^{N,(\theta_1^{(n)},\ldots,\theta_k^{(n)})}$ converges weakly to $\nu_\beta^{N,(\theta_1,\ldots,\theta_k)}$ whenever $(\theta_1^{(n)},\ldots,\theta_k^{(n)}) \to (\theta_1,\ldots,\theta_k)$. Define
\[
\wt p_\beta^{N,(\theta_1,\ldots,\theta_k)}(\mbf x_1,\ldots,\mbf x_r) = \wt p_\beta^{N,(\theta_1,\ldots,\theta_k)}\bigl((x_{r,i})_{1\le r \le k, 1 \le i \le N-1}\bigr) := \prod_{r = 1}^k \prod_{i = 0}^{N-1} e^{-e^{-\beta^{-2} x_{r,i}}},
\]
where $x_{r,0} = \theta_r - \sum_{i = 1}^{N-1} x_i$. Then, recalling the definition of the probability density $p_\beta^{N,(\theta_1,\ldots,\theta_k)}$ from Definition \ref{def:p_meas},  $p_\beta^{N,(\theta_1,\ldots,\theta_k)}(\mbf x_1,\ldots,\mbf x_r) = \f{1}{C(\theta_1,\ldots,\theta_k)} \wt p_\beta^{N,(\theta_1,\ldots,\theta_k)}(\mbf x_1,\ldots,\mbf x_r)$ for a constant $C(\theta_1,\ldots,\theta_k)$. Immediately, we have that $\wt p_\beta^{N,(\theta_1^{(n)},\ldots,\theta_k^{(n)})} \to \wt p_\beta^{N,(\theta_1,\ldots,\theta_k)}$ pointwise. Furthermore, $\wt p_\beta^{N,(\theta_1,\ldots,\theta_k)}$ increases pointwise if any of the $\theta_1,\ldots,\theta_k$ increase. Hence, we have $\wt p_\beta^{N,(\theta_1^{(n)},\ldots,\theta_k^{(n)})} \le \wt p_\beta^{N,(\gamma_1,\ldots,\gamma_k)}$ pointwise for some choice of $\gamma_1,\ldots,\gamma_k$ and all $n \ge 1$. By the dominated convergence theorem, as $n \to \infty$, 
\[
C(\theta_1^{(n)},\ldots,\theta_k^{(n)}) = \int_{(\R^{N-1})^k} \wt p_\beta^{N,(\theta_1^{(n)},\ldots,\theta_k^{(n)})} \to \int_{(\R^{N-1})^k} \wt p_\beta^{N,(\theta_1,\ldots,\theta_k)}  = C(\theta_1,\ldots,\theta_k).
\]  
Hence, $p_\beta^{N,(\theta_1^{(n)},\ldots,\theta_k^{(n)})} \to p_\beta^{N,(\theta_1,\ldots,\theta_k)}$ pointwise, and so by Scheffe's Lemma, $\nu_\beta^{N,(\theta_1^{(n)},\ldots,\theta_k^{(n)})} \to \nu_\beta^{N,(\theta_1,\ldots,\theta_k)}$ in total variation distance as $n \to \infty$. 
\end{proof}

\subsection{Proof of Theorem \ref{thm:disc_MC}} \label{sec:proof_disc_MC}

To prove Theorem \ref{thm:disc_MC}, we define here a dual Markov chain with state space $(\R^{\Z_N})^k$. This Markov chain will be shown to satisfy an intertwining with the chain in \eqref{eq:coupled_disc_MC}. In the terminology introduced in \cite{Ferrari-Martin-2007} (and used also in \cite{Fan-Seppalainen-20,Seppalainen-Sorensen-21b,GRASS-23}) this could also be called a multiline process; while in \cite{Bates-Fan-Seppalainen}, the analogous process is called the parallel process. For a state $(\mbf X_1^{(m)},\ldots,\mbf X_k^{(m)}) \in (\R^{\Z_N})^k$ at time $m$ and an independent driving sequence $\mbf W \in \R^{\Z_N}$, let $\mbf W_1 = \mbf W$. Then, the state at time $m + 1$, $(\mbf X_r^{(m+1)})_{1 \le r \le k}$ is defined inductively as follows: For $r \ge 1$,
\be \label{eq:multline}
\begin{aligned}
\mbf X_r^{(m+1)} &= D^{N,2}(\mbf W_r,\mbf X_r^{(m)}),\qquad \text{and the input to the next level is}\qquad  \mbf W_{r+1} = T^{N,2}(\mbf W_r,\mbf X_r^{(m)}).
\end{aligned}
\ee
For $\mbf W \in \R^{\Z_N}$, let $\mathcal G_{\mbf W}^{N,k},\mathcal H_{\mbf W}^{N,k}: (\R^{\Z_N})^k \to (\R^{\Z_N})^k$ be defined so that 
\[
\mathcal G_{\mbf W}^{N,k}(\mbf U_1^{(m)},\ldots,\mbf U_k^{(m)}):=(\mbf U_1^{(m+1)},\ldots,\mbf U_k^{(m+1)}),\qquad 
\mathcal H_{\mbf W}^{N,k}(\mbf X_1^{(m)},\ldots,\mbf X_k^{(m)}):=(\mbf X_1^{(m+1)},\ldots,\mbf X_k^{(m+1)})
\]
where the first map is defined by \eqref{eq:coupled_disc_MC} and the second by  \eqref{eq:multline}.

The proof of Theorem \ref{thm:disc_MC} comes from the following intertwining relation, which is an inductive corollary of Proposition \ref{prop:full_intertwine}. Recall the map $\D^{N,k}:(\R^{\Z_N})^k \to (\R^{\Z_N})^k$ defined in \eqref{DNk_intro}.
\begin{proposition} \label{prop:DNk_disc_intertwine}
For each $\mbf W \in \R^{\Z_N}$, we have the intertwining relation on  $(\R^{\Z_N})^k$:
\[
\mathcal G_{\mbf W}^{N,k} \circ \D^{N,k} = \D^{N,k} \circ \mathcal H_{\mbf W}^{N,k}.
\]
\end{proposition}
\begin{proof}
Let $\mbf W_1 = \mbf W$, and let $\mbf W_r$ for $r > 1$ be defined inductively by \eqref{eq:multline}. As both sides of the equation are functions $(\R^{\Z_N})^k \to (\R^{\Z_N})^k$, it suffices to show equality of each component map $(\R^{\Z_N})^k \to \R^{\Z_N}$. Recalling the definition of $\D^{N,k}$ \eqref{DNk_intro}, we must show that, for $1 \le r \le k$, 
\be \label{eq:int_id}
D^{N,2}\bigl(\mbf W_1,D^{N,r}(\mbf X_1,\ldots,\mbf X_r)\bigr) = D^{N,r}\bigl(D^{N,2}(\mbf W_1,\mbf X_1),\ldots,D^{N,2}(\mbf W_r,\mbf X_r)\bigr). 
\ee
The statement for $r = 1$ is immediate. For $r \ge 2$, we prove following more general identity: for $1 \le \ell \le r-1$,
\be \label{eq:gen_id}
\begin{aligned}
&D^{N,2}\bigl(\mbf W_1,D^{N,r}(\mbf X_1,\ldots,\mbf X_r)\bigr)  \\
&\qquad\qquad= D^{N,\ell + 1}\bigl(D^{N,2}(\mbf W_1,\mbf X_1),\ldots,D^{N,2}(\mbf W_\ell,\mbf X_\ell), D^{N,r - \ell + 1}(\mbf W_{\ell + 1},\mbf X_{\ell + 1},\ldots,\mbf X_{r})  \bigr),
\end{aligned}
\ee
noting that \eqref{eq:int_id} is the $\ell = r-1$ case of \eqref{eq:gen_id}. When $r = 2$, the only option is $\ell = 1$, and  since $\mbf W_2 = T^{N,2}(\mbf W_1,\mbf X_1)$ by definition, the equality \eqref{eq:gen_id} is exactly Proposition \ref{prop:full_intertwine}. We prove the case of general $r \ge 2$ by induction. We note that, given Proposition \ref{prop:full_intertwine}, the proof follows that of \cite[Lemma 4.5]{Fan-Seppalainen-20}, repeated also in \cite[Theorem 7.7]{Seppalainen-Sorensen-21b}, \cite[Lemma 6.11]{Bates-Fan-Seppalainen}, and \cite[Lemma A.3]{GRASS-23}.

Assume, for some $r \ge 2$ that \eqref{eq:gen_id} holds for all $1 \le \ell \le r-1$. We prove it holds for $r + 1$, starting with $\ell = 1$. In this case, the right-hand side of \eqref{eq:gen_id} is
\begin{align*}
&\quad \, D^{N,2}\bigl(D^{N,2}(\mbf W_1,\mbf X_1), D^{N,r+1}(\mbf W_2,\mbf X_2,\ldots,\mbf X_{r+1})   \bigr) \\
&= D^{N,3}\Bigl(D^{N,2}(\mbf W_1,\mbf X_1),\mbf W_2, D^{N,r}(\mbf X_2,\ldots,\mbf X_{r+1})\Bigr).
\end{align*}
Since $\mbf W_2 = T^{N,2}(\mbf W_1,\mbf X_1)$, by Proposition \ref{prop:full_intertwine} followed by Lemma \ref{lem:Diter}, this is equal to 
\begin{align*}
D^{N,3}\bigl(\mbf W_1,\mbf X_1,D^{N,r}(\mbf X_2,\ldots,\mbf X_{r+1})\bigr) = D^{N,2}\bigl(\mbf W_1,D^{N,r+1}(\mbf X_2,\ldots,\mbf X_{r+1})\bigr),
\end{align*}
and this is the left-hand side of \eqref{eq:gen_id} with $r$ replaced by $r + 1$.

Next, let $2 \le \ell \le r-1$. Then, by definition of the map $D^{N,\ell + 1}$, 
\begin{align*}
&\quad \, D^{N,\ell + 1}\bigl(D^{N,2}(\mbf W_1,\mbf X_1),\ldots,D^{N,2}(\mbf W_\ell,\mbf X_\ell), D^{N,r + 1 - \ell + 1}(\mbf W_{\ell + 1},\mbf X_{\ell + 1},\ldots,\mbf X_{r+1})  \bigr) \\
&= D^{N,2}\Bigl(D^{N,2}(\mbf W_1,\mbf X_1),D^{N,\ell}\bigl(D^{N,2}(\mbf W_2,\mbf X_2),\ldots,D^{N,2}(\mbf W_\ell,\mbf X_\ell),D^{N,r + 1-\ell + 1}(\mbf W_{\ell + 1},\mbf X_{\ell + 1},\ldots,\mbf X_{r+1}) \bigr).
\end{align*}
Applying the induction assumption, this is equal to
\begin{align*}
D^{N,2}\Bigl(D^{N,2}(\mbf W_1,\mbf X_1), D^{N,r+1}(\mbf W_2,\mbf X_2,\ldots,\mbf X_{r+1})    \Bigr).
\end{align*}
We have now reduced this to the $\ell = 1$ case, completing the proof of \eqref{eq:gen_id} for $r + 1$.
\end{proof}
We now complete the section by proving Theorem \ref{thm:disc_MC}.
\begin{proof}[Proof of Theorem \ref{thm:disc_MC}]
Let $\alpha \in \R$, $(\theta_1,\ldots,\theta_k)\in \R^k$, and let $(\mbf X_1,\ldots,\mbf X_k) \sim \nu_\beta^{N,(\theta_1,\ldots,\theta_k)}$. Let $(\mbf U_1,\ldots,\mbf U_k) = \D^{N,k}(\mbf X_1,\ldots,\mbf X_k)$ so that $(\mbf U_1,\ldots,\mbf U_k) \sim \mu_\beta^{N,(\theta_1,\ldots,\theta_k)}$. By definition of the Markov chain in \eqref{eq:coupled_disc_MC}, the statement to prove is that, if (i) $\mbf W$ is a sequence of $N$ i.i.d. log-inverse gamma random variables with shape $\gamma > 0$ and scale $\beta^{-2}$ or (ii) $\mbf W \sim \nu_\beta^{N,(\alpha)}$, independent of $(\mbf X_1,\ldots,\mbf X_k)$, then $\mathcal G_{\mbf Z}^{N,k}(\mbf U_1,\ldots,\mbf U_k) \sim \mu_\beta^{N,(\theta_1,\ldots,\theta_k)}$. In both of these cases, by inductively applying Proposition \ref{prop:Burke} and the fact that $D^{N,2}$ preserves the sum of the second coordinate (Lemma \ref{lem:Dsum_pres}), we have 
\[
 \mathcal H_{\mbf Z}^{N,k}(\mbf X_1,\ldots,\mbf X_k) \deq (\mbf X_1,\ldots,\mbf X_k).
\]
In other words, $\nu_\beta^{N,(\theta_1,\ldots,\theta_k)}$ is invariant for the dual process in both of these cases for the driving weight $\mbf W$.
Then, by Proposition \ref{prop:DNk_disc_intertwine}, 
\[
\mathcal G_{\mbf Z}^{N,k}(\mbf U_1,\ldots,\mbf U_k) = \D^{N,k} \circ \mathcal H_{\mbf Z}^{N,k}(\mbf X_1,\ldots,\mbf X_k)  \deq \D^{N,k}(\mbf X_1,\ldots,\mbf X_k) = (\mbf U_1,\ldots,\mbf U_k),
\]
as desired. 
\end{proof}

\section{The O'Connell-Yor polymer in a periodic environment and proof of Theorem \ref{thm:OCY_joint}} \label{sec:OCY} 

\subsection{Definition of the model} \label{sec:OCY_def}
 On a probability space $(\Omega,\Ff,\Pp)$, let $(B_r)_{r \in \Z_N}$ be a collection of i.i.d. two-sided standard Brownian motions. We periodically extend it to $(B_r)_{r \in \Z}$ by setting $B_r = B_j$ whenever $r \equiv j \mod N$. For $(s,m) < (t,n) \in \R \times \Z$,  define the space
\[
\pathsp_{(s,m),(t,n)} := \{(s_{m - 1},s_m,\ldots,s_n) \in \R^{n - m + 2}: s = s_{m - 1} < s_m < \cdots < s_n = t   \}.
\] 
For $(s,m) < (t,n) \in \R \times \Z$ and a parameter $\beta > 0$, we define a point-to-point partition function   as
\begin{equation} \label{OCYpart}
\OCY_\beta(t,n\viiva s,m) := \int_{\pathsp_{(s,m),(t,n)} }\exp \bigg\{\beta \sum_{r = m}^n \big(B_r(s_r) - B_r(s_{r-1})\big)\bigg\} d  \mbf s_{m:n-1},
\end{equation}
where $d \mbf s_{m:n-1}:=\prod_{i = m}^{n-1} ds_i$. As shorthand drop the $s$ if it equals $0$.  Define boundary cases
\[
\OCY_\beta(t,m\viiva s,m )= e^{\beta (B_m(t) - B_m(s))}\quad \textrm{for }m=n,\quad\textrm{and} \quad 
\OCY_\beta(s,n \viiva s,m) = \ind\{m = n\} \quad \textrm{for }s=t.
\]
If $t  < s$ or $n<m$,  set $\OCY_\beta(t,n \viiva s,m) = 0$. We can think of $\OCY_\beta(t,n \viiva s,m)$ as the integral over all up-right paths between $(s,m)$ and $(t,n)$, in a semi-discrete lattice, of the weight collected by the path. The weight in this case is the sum of Brownian increments along each integer level. See Figure \ref{fig:OCY} for a visual representation.
\begin{figure}
    \centering
    \includegraphics[width=0.5\linewidth]{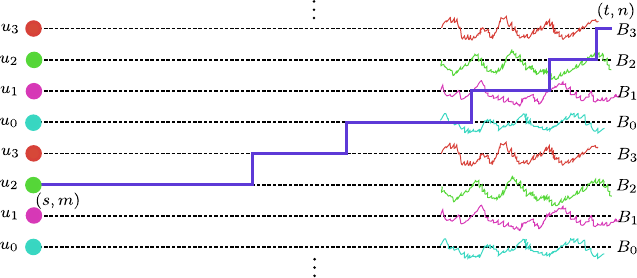}
    \caption{A depiction of the O'Connell-Yor polymer in a periodic environment. An up-right path is shown in blue from points $(s,m)$ to $(t,n)$ in the semi-discrete lattice. The environment is made up of a collection of two-sided Brownian motions $(B_r)_{r  \in \{1,2,3,4\}}$, one associated to each line and repeated periodically (these Brownian motions extend across the figure, but are truncated for the purposes of illustration). The periodic initial condition $\mbf u = (u_0,u_1,u_2,u_3)$ is represented by the disks starting on the left boundary at time $s$. The partition function is the integral over all up-right paths of the exponential of the sum of the Brownian increments along the path plus the cumulative sum of the initial condition to the path starting point.}
    \label{fig:OCY}
\end{figure}

The function $\OCY_\beta$ is the Green's function for a semi-discrete version of the SHE. We also consider the model with initial data. For a function $F:\{0,\ldots,N\} \to \R_{>0}$, extend $F$ to a function $\Z \to \R_{>0}$ by the condition that $\f{F(j+N)}{f(j)} = \f{F(N)}{F(0)}$ for all $j \in \Z$. Then, we define 
\be \label{Zfdef}
\OCY_\beta(t,i \viiva s,F) = \sum_{j \in \Z} F(j) \OCY_\beta(t,i \viiva s,j).
\ee
Next, for a sequence $\mbf u = (u_i)_{i \in \Z_N}$, extend to $(u_\ell)_{\ell \in \Z}$ by the condition $u_i = u_j$ whenever $i \equiv j \mod N$. Let $F:\Z \to \R_{>0}$ be the function satisfying $F(0) = 1$ and $\f{F(j)}{F(j-1)} = e^{u_j}$ for $j \in \Z$, and define
\begin{align}\label{Uudef}
U_\beta^N(t,i \viiva s,\mbf u) := \log \f{\OCY_\beta(t,i \viiva s,F)}{\OCY_\beta(t,i -1 \viiva s,F)}.
\end{align}
One should view the relation $u_j=\log F(j)-\log F(j-1)$ and the definition in \eqref{Uudef} as a semi-discrete Hopf-Cole transformation between the stochastic Burgers' and stochastic heat equations, and 
 \eqref{Uudef} and \eqref{Zfdef} are the corresponding solutions. Pictorially, one can think of $\mbf u$ as a discrete, periodic initial condition living on the left boundary in Figure \ref{fig:OCY}.

If $\vecsum(\mbf u) := \sum_{i\in \Z_N} u_i= \theta \in \R$, then for all $j \in \Z$, $\f{F(j+N)}{F(j)} = e^\theta$, and we may rewrite \eqref{Zfdef} as
\begin{align*}
\OCY_\beta(t,n \viiva s, F) = \sum_{m  = 0}^{N-1} \sum_{j \in \Z} F(m + jN) \OCY_\beta(t,n \viiva s,m+jN) = \sum_{m =0}^{N-1} F(m) \sum_{j \in \Z} e^{\theta j} \OCY_\beta(t,n \viiva s, m+ jN).
\end{align*}
This motivates the definition
\be \label{eq:reg_conv}
\OCYp_\beta(t,n \viiva s,m; \theta) := \sum_{j \in \Z} e^{\theta j} \OCY_\beta(t,n \viiva s,m+jN)
\ee
so that 
\be \label{per_conv}
\OCY_\beta(t,n \viiva s, F) = \sum_{m  = 0}^{N-1} F(m) \OCYp_\beta(t,n \viiva s,m;\theta).
\ee
 We note that the terms of the sum in \eqref{eq:reg_conv} are nonzero only for $j$ satisfying $m+ jN \le n$. Since the sum in \eqref{per_conv} is finite, Lemma \ref{lem:OCYp_finite} below ensures that $\OCY_\beta(t,n \viiva s,F)$ is finite as well. The following two results have standard proofs, which are deferred to Appendix \ref{appx:OCY_SDE_proofs}.

\begin{lemma} \label{lem:OCYp_finite}
With probability one, $\OCYp_\beta (t,n \viiva s,m; \theta)\in (0, \infty)$ for all $\theta \in \R$, $t > s$, and $n,m \in \Z$.
\end{lemma}

\begin{lemma} \label{lem:sd_SDE}
Fix any $s \in \R$, $k\in \N$, $\theta_1,\ldots,\theta_k\in \R$ and (random) $\mbf u_r = (u_{r,i})_{i \in \Z_N}$ satisfying $\vecsum(\mbf u_r) = \theta_r$ for $1\leq r\leq k$. With $U_{r,i}(t):=U_\beta^N(t,i \viiva s,\mbf u_r)$ from \eqref{Uudef}, define $\mbf U_r(t) := \bigl((U_{r,i}(t))_{i \in \Z_N}\bigr)_{t \ge s}$ for $1\leq r\leq k$.

\begin{enumerate} [label=\textup{(\roman*)}]
\item \label{itm:cont_lim} For $i \in \Z$ and $1 \le r \le k$, with probability one, $t \mapsto U_{r,i}(t)$ is continuous in $t \ge s$. Thus, $\lim_{t \searrow s} U_{r,i}(t) = U_{r,i}(s)=u_{r,i}$.
\item \label{itm:Zper} For $1 \le r \le k$ and $i,j \in \Z$, $U_{r,i}(t) = U_{r,j}(t)$ whenever $i \equiv j \mod N$.
\item \label{itm:prod_pres} For all $t \ge s$ and $1\leq r\leq k$,  $\vecsum\big(\mbf U_r(t)\big) = \theta_r$. 
\item \label{itm:SDE} $\big(\mbf U_1(t),\ldots,\mbf U_k(t)\big) $ is the unique (pathwise and in law) strong solution to the system of SDEs \eqref{eq:joint_OCY}
satisfying $\lim_{t \searrow s} \mbf U_r(t) = \mbf U_r(s)$ for $1 \le r \le k$,
where $(B_i)_{i \in \Z_N}$ are the same in \eqref{OCYpart} and   \eqref{eq:joint_OCY}.
\end{enumerate}
\end{lemma}
\begin{remark}
When we define $U_{r,i}(t)$, this is done for all $i \in \Z$. Lemma \ref{lem:sd_SDE}\ref{itm:Zper} then show that $U_{r,i}(t) = U_{r,i+N}(t)$ for all $i \in \Z$. Hence it suffices to consider $\mbf U_r(t)$ (defined in Lemma \ref{lem:sd_SDE}) which only tracks $i\in \Z_N$. 
\end{remark}

\subsection{Transforming the system of SDEs and the proof of Theorem \ref{thm:OCY_joint}} \label{sec:OCY_full_proof}
There are two main steps toward proving Theorem \ref{thm:OCY_joint}. The first is Proposition \ref{prop:NewSDE}, which demonstrates how the inverse map of $\D^{N,k}$, namely $\mathcal J^{N,k}$, transforms the system of SDEs \eqref{eq:joint_OCY}. Proposition \ref{prop:prod_invar} then shows that the product measure $\nu_\beta^{N,(\theta_1,\ldots,\theta_k)}$ is invariant for this transformed system of SDEs. 

We recall  the notation $\diff_i \mbf X = X_i - X_{i-1}$ and $\Lapl_i \mbf X = X_{i-1} + X_{i+1} - 2X_i$. When we write $\diff_i e^{-\mbf X}$ or $\Lapl_i e^{-\mbf X}$, we mean that we apply these operators to the sequence $e^{-\mbf X} := (e^{-X_i})_{i \in \Z_N}$.

\begin{proposition} \label{prop:NewSDE}
For $N,k \in \N$, let $\bigl(\mbf U_1(t),\ldots,\mbf U_k(t)\bigr)_{t \ge 0}$ be the unique strong solution to \eqref{eq:joint_OCY} from Lemma \ref{lem:sd_SDE} \ref{itm:SDE}, started from an initial condition $(\mbf U_1(0),\ldots,\mbf U_k(0)) \in \mathcal R^{N,k}$. With probability one, for all $t > 0$, $(\mbf U_1(t),\ldots,\mbf U_k(t)) \in \mathcal R^{N,k}$, and  $\bigl(\mbf X_1(t),\ldots,\mbf X_k(t) \bigr)_{t \ge 0} := \bigl(\mathcal J^{N,k}(\mbf U_1(t),\ldots,\mbf U_k(t)) \bigr)_{t \ge 0}$ is the unique (pathwise and in law) strong solution to the system of SDEs below with initial condition $\bigl(\mbf X_1(0),\ldots,\mbf X_k(0)\bigr)$:
\begin{equation}
\label{DxR_gen}
d X_{r,i}(t) = \Bigl(\diff_ie^{-\mbf X_r(t)} - \sum_{m = 1}^{r-1}\Lapl_ie^{-\mbf X_m(t)} \Bigr)\,dt 
 + \beta\diff_i d\mbf B(t),\qquad 1 \le r \le k, i \in \Z_N.
\end{equation}
\end{proposition}
\begin{proof}
 In showing that $\bigl(\mbf X_1(t),\ldots,\mbf X_k(t)\bigr)$ is the unique strong solution to \eqref{DxR_gen}, we only need to show that $\bigl(\mbf X_1(t),\ldots,\mbf X_k(t)\bigr)$ is a strong solution to  \eqref{DxR_gen}. Indeed, since the coefficients of the SDE are locally Lipschitz, \cite[Theorem 5.3.7, page 297]{ethi-kurt}, implies pathwise uniqueness of solutions to the system of SDEs. Then \cite[theorem 5.3.6]{ethi-kurt} implies uniqueness of solutions in law.

For convenience of the reader, we recall here \eqref{eq:joint_OCY}:
\[
dU_{r,i}(t) = \diff_i e^{-\mbf X_r(t)}\,dt + \beta \diff_i d\mbf B(t) ,\qquad 1 \le r \le k,\; i \in \Z_N.
\]
We proceed by induction. When $k = 1$, $\mathcal J^{N,1}$ is the identity, and $\mathcal R^{N,1} = \R^{\Z_N}$, so there is nothing to show.
Assume the proposition holds for some $k \ge 1$. There are two things to show:
\begin{enumerate}[label=\textup{(\roman*)}]
    \item \label{itm:domain} If $\bigl(\mbf U_1(t),\ldots,\mbf U_{k+1}(t)\bigr)_{t \ge 0}$ solves \eqref{eq:joint_OCY} with $\bigl(\mbf U_1(0),\ldots,\mbf U_{k+1}(0)\bigr) \in \mathcal R^{N,k+1}$ then for all $t> 0$, we also have $\bigl(\mbf U_1(t),\ldots,\mbf U_{k+1}(t)\bigr) \in \mathcal R^{N,k+1}$.
    \item \label{itm:Ito} $\bigl(\mathcal J^{N,k+1}(\mbf U_1(t),\ldots,\mbf U_{k+1}(t)) \bigr)_{t \ge 0}$ is a strong solution to the SDE \eqref{DxR_gen} satisfying 
    \[
    \lim_{t \searrow 0} \mathcal J^{N,k+1}\bigl(\mbf U_1(t),\ldots,\mbf U_{k+1}(t)\bigr) = \mathcal J^{N,k+1}\bigl(\mbf U_1(0),\ldots,\mbf U_{k+1}(0)\bigr).
    \] 
\end{enumerate}

\noindent {\bf Item \ref{itm:domain}:} 
Let $\bigl(\mbf U_1(t),\ldots,\mbf U_{k+1}(t)\bigr)_{t \ge 0}$ solve \eqref{eq:joint_OCY} (with $k\mapsto k+1$) with $\bigl(\mbf U_1(0),\ldots,\mbf U_{k+1}(0)\bigr) \in \mathcal R^{N,k+1}$. Recalling $\mathcal R^{N,k}$  and $\mathcal J^{N,k}$ from  Section \ref{sec.sst}, the assumption $(\mbf U_1(0),\ldots,\mbf U_{k+1}(0)) \in \mathcal R^{N,k+1}$ means:
\be \label{eq:Rnk+1}
\begin{aligned}
    &\bigl(\mbf U_1(0),\ldots,\mbf U_{k-1}(0),\mbf U_{k}(0)\bigr), \bigl(\mbf U_1(0),\ldots,\mbf U_{k-1}(0),\mbf U_{k+1}(0)\bigr) \in \mathcal R^{N,k}, \qquad  \bigl(\mbf X_k(0),\wt{\mbf X}_{k+1}(0)\bigr) \in \mathcal R^{N,2},\\
    &\text{where}\quad \bigl(\mbf X_1(0),\ldots,\mbf X_{k-1}(0),\mbf X_{k}(0)\bigr) := \mathcal J^{N,k}\bigl(\mbf U_1(0),\ldots,\mbf U_{k-1}(0),\mbf U_{k}(0)\bigr),\\
    &\text{and} \quad \bigl(\mbf X_1(0),\ldots,\mbf X_{k-1}(0),\wt{\mbf X}_{k+1}(0)\bigr) := \mathcal J^{N,k}\bigl(\mbf U_1(0),\ldots,\mbf U_{k-1}(0),\mbf U_{k+1}(0)\bigr),
\end{aligned}
\ee
and the output of $\mathcal J^{N,k+1}(\mbf U_1(0),\ldots,\mbf U_{k+1}(0))$, is $(\mbf X_1(0),\ldots,\mbf X_{k+1}(0))$, where $(\mbf X_1(0),\ldots,\mbf X_k(0))$ is defined above (consistency follows from the definition of $\mathcal J^{N,k}$), and $\mbf X_{k+1}(0) = J^{N,2}\bigl(\mbf X_k(0),\wt{\mbf X}_{k+1}(0)\bigr)$.

Since $\bigl(\mbf U_1(t),\ldots,\mbf U_{k+1}(t)\bigr)_{t \ge 0}$ solve \eqref{eq:joint_OCY} (with $k\mapsto k+1$), marginally $\bigl(\mbf U_1(t),\ldots,\mbf U_{k-1}(t),\mbf U_{k}(t)\bigr)$ and $\bigl(\mbf U_1(t),\ldots,\mbf U_{k-1}(t),\mbf U_{k+1}(t)\bigr)$ both solve \eqref{eq:joint_OCY} with their respective initial data. Thus, applying Item \ref{itm:domain} (with $k$) of the induction assumption twice, we have that, for all $t \ge 0$,
\be \label{eq:UinR}
\bigl(\mbf U_1(t),\ldots,\mbf U_{k-1}(t),\mbf U_{k}(t)\bigr), \bigl(\mbf U_1(t),\ldots,\mbf U_{k-1}(t),\mbf U_{k+1}(t)\bigr) \in \mathcal R^{N,k}.
\ee
Therefore, if we define, for $t \ge 0$,
\begin{align*}
\big(\mbf X_1(t),\ldots,\mbf X_{k-1}(t),\mbf X_k(t)\big) &:= \mathcal J^{N,k}\big(\mbf U_1(t),\ldots,\mbf U_{k-1}(t),\mbf U_k(t)\big),\\
\big(\mbf X_1(t),\ldots,\mbf X_{k-1}(t),\wt{\mbf X}_{k+1}(t)\big) &:= \mathcal J^{N,k}\bigl(\mbf U_1(t),\ldots,\mbf U_{k-1}(t),\mbf U_{k+1}(t)\bigr),
\end{align*}
it follows from two applications of Item \ref{itm:Ito} (with $k$) of the induction assumption that \\ $(\mbf X_1(t),\ldots,\mbf X_k(t),\wt{\mbf X}_{k+1}(t))$ is a strong solution to the system of SDEs
\be \label{eq:pre_Ito}
\begin{aligned}
d X_{r,i}(t) &= \Bigl(\diff_i e^{-\mbf X_r(t)} - \sum_{m = 1}^{r-1}\Lapl_ie^{-\mbf X_m(t)} \Bigr)\,dt 
 + \beta\diff_i d\mbf B(t), & 1 \le r \le k, i \in \Z_N,  \\
d \wt{X}_{k + 1,i}(t) &= \Bigl(\diff_i e^{-\wt{\mbf X}_{k+1}(t)} - \sum_{m = 1}^{k-1}\Lapl_ie^{-\mbf X_m(t)} \Bigr)\,dt 
 +\beta \diff_i d\mbf B(t), &i \in \Z_N.
\end{aligned}
\ee
If we look at the different $\wt X_{k+1,i}(t) - X_{k,i}(t)$ for $i\in \Z_N$ then it satisfies
\[
\df{d}{dt} (\wt X_{k+1,i}(t) - X_{k,i}(t)) = \diff_ie^{-\wt{\mbf X}_{k+1}(t)} - \diff_ie^{-\mbf X_{k}(t)} = e^{-\wt X_{k+1,i}(t)} - e^{-X_{k,i}(t)} + e^{-X_{k,i-1}(t)} - e^{-\wt X_{k+1,i-1}(t)}.
\]
By \eqref{eq:Rnk+1}, $\big(\mbf X_k(0),\wt{\mbf X}_{k+1}(0)\big) \in \mathcal R^{N,2}$, so Lemma \ref{lem:SDE_strong_comp} implies that $\big(\mbf X_k(t),\wt{\mbf X}_{k+1}(t)\big) \in \mathcal R^{N,2}$ for all $t \ge 0$.
Proposition \ref{prop:transform} shows that  $\big(\mbf X_1(t),\ldots,\mbf X_{k-1}(t),\mbf X_k(t)\big),\big(\mbf X_1(t),\ldots,\mbf X_{k-1}(t),\wt{\mbf X}_{k+1}(t)\big)\in \R^{N,k}_{\neq}$. Combined with  $\big(\mbf X_k(t),\wt{\mbf X}_{k+1}(t)\big) \in \mathcal R^{N,2}$ (shown above) this implies that 
 $\big(\mbf X_1(t),\ldots,\mbf X_{k-1}(t),\mbf X_k(t),\wt{\mbf X}_{k+1}(t)\big)\in \R^{N,k+1}_{\neq}$.  Defining $\mbf X_{k+1}(t) := J^{N,2}\bigl(\mbf X_k(t),\wt{\mbf X}_{k+1}(t)\bigr)$, 
 Proposition \ref{prop:transform} implies that $\vecsum\big(\wt{\mbf X}_{k+1}(t)\big) = \vecsum\big(\mbf X_{k+1}(t)\big)$ and hence $\big(\mbf X_1(t),\ldots,\mbf X_{k-1}(t),\mbf X_k(t),\mbf X_{k+1}(t)\big)\in \R^{N,k+1}_{\neq}$. All the remains is to again invoke Proposition \ref{prop:transform} (with $k+1$) which shows the desired $k+1$ case of the inductive assumption, that 
 $$
 \mathcal D^{N,k+1}\big(\mbf X_1(t),\ldots,\mbf X_{k-1}(t),\mbf X_k(t),\mbf X_{k+1}(t)\big)\in \R^{N,k+1}_{\neq}= \big(\mbf U_1(t),\ldots,\mbf U_{k+1}(t)\big)\in \mathcal R^{N,k+1}.$$
 
\noindent{\bf Item \ref{itm:Ito}:} By the continuity of the sample paths $(\mbf U_1(t),\ldots,\mbf U_k(t))$ from Lemma \ref{lem:sd_SDE}\ref{itm:cont_lim} and the continuity of the map $\mathcal J^{N,k}$, the sample paths of $\bigl(\mbf X_1(t),\ldots,\mbf X_k(t)\bigr)$ are continuous. In particular, $\lim_{t \searrow 0} \mbf X_m(t) = \mbf X_m(0)$ for $1 \le m \le k$.  By \eqref{eq:pre_Ito} and inductive definition of the map $\mathcal J^{N,k+1}$, it only remains to show that, if we set $\mbf X_{k+1}(t) := \mathcal J^{N,2}(\mbf X_k(t),\wt{\mbf X}_{k+1}(t))$ for all $t \ge 0$, then $\mbf X_{k+1}(t)$ satisfies
\be \label{eq:k+1_eq}
d  X_{k+1,i}(t) = \Bigl(\diff_ie^{-\mbf X_{k+1}(t)} - \sum_{m = 1}^{k}\Lapl_ie^{-\mbf X_m(t)} \Bigr)\,dt 
 +\beta \diff_i d\mbf B(t),\quad i \in \Z_N.
\ee

We recall here  It\^o's lemma for a multi-dimensional diffusion (see, e.g., \cite[Theorem 5.10]{LeGall-book}): Assume that $\mbf Y(t)$ is an $d$-dimensional diffusion process satisfying  
\[
d \mbf Y(t) = \mbf A(\mbf Y(t))\,dt + \Sigma\,d \mbf B(t),
\]
where $\mbf A:\R \to \R^d$,  $\Sigma \in M_{d\times d}(\R)$, and
where $\mbf B$ is an $d$-dimensional Brownian motion. Then, for a twice continuously-differentiable function $f:\R^d \to \R$, we have
\be
d f(\mbf Y(t)) = \Biggl( (\text{grad} f)^T \mbf A(\mbf Y(t)) + \f{1}{2} \text{Tr}\Bigl(\Sigma^T (H f)\Sigma\Bigr)\Biggr)\,dt + (\text{grad} f)^T \Sigma d\mbf B(t), \label{eq:Ito}
\ee
where $\text{grad} f$ denotes the (column) gradient vector of $f$, and $H_{\mbf Y}f$ denotes the Hessian matrix of $f$. We apply Ito's lemma to the system \eqref{eq:pre_Ito} and the coordinate functions $J_i^{N,2}$. The key is that the diffusion term will be the same, and the Hessian term in \eqref{eq:Ito} will be shown to be $0$.

To properly apply It\^o's lemma, one starts from the full system \eqref{eq:pre_Ito} and considers $\mbf X_{k+1}(t)$ as a function of the $(k+1)N$-dimensional system $\bigl(\mbf X_1(t),\ldots,\mbf X_k(t),\wt{\mbf X}_{k+1}(t)\bigr)_{t \ge 0}$. Then, the covariance matrix for this system is of dimension $(k+1)N \times (k+1)N$. However, $\mbf X_{k+1}(t)$ depends only on $\mbf X_k(t)$ and $\wt{\mbf X}_{k+1}(t)$, and so we may reduce to $2N$-dimensions by considering only the equations obeyed by $\mbf X_k(t)$ and $\wt{\mbf X}_{k+1}(t)$. That is,
\be \label{eq:tXR_sys}
\begin{aligned}
d X_{k,i}(t) &= \Bigl(\diff_i e^{-\mbf X_k(t)} - \sum_{m = 1}^{k-1}\Lapl_ie^{-\mbf X_m(t)} \Bigr)\,dt 
 + \beta\diff_i d\mbf B(t), i \in \Z_N,  \\
d \wt{X}_{k + 1,i}(t) &= \Bigl(\diff_ie^{-\wt{\mbf X}_{k+1}(t)} - \sum_{m = 1}^{k-1}\Lapl_ie^{-\mbf X_m(t)} \Bigr)\,dt 
 + \beta \diff_i d\mbf B(t), i \in \Z_N.
 \end{aligned}
\ee
Then,  we consider the diffusion matrix for the system as the $2N \times 2N$ block matrix
\[
\Sigma = \begin{pmatrix}
    \sigma &0 \\
    \sigma &0
\end{pmatrix}, 
\]
where the $0$s above represent  $N \times N$ blocks of $0$s, and $\sigma$ is the $N \times N$ matrix with $\sigma_{i,i} = -\beta$, $\sigma_{i,i+1} = \beta$ (with addition modulo $N$), and $\sigma_{i,j} = 0$ otherwise. The zero blocks are present because we only use one set of Brownian motions for the system.

For $j \in \Z_N$, and $t \ge 0$, define
$
W_{k,j}(t) = \wt X_{k+1,j}(t) - X_{k,j}(t).
$
Then, for $i \in \Z_N$, by definition,
\[
X_{k+1,i}(t) = J_i^{N,2}\bigl(\mbf X_k(t),\wt{\mbf X}_{k+1}(t)\bigr) = \wt X_{k+1,i}(t) + \log \Bigl(\f{e^{W_{k,i+1}(t)} - 1}{e^{W_{k,i}(t)}-1}\Bigr).
\]
Then, for $j \in \{i-1,1\}$,
\be \label{FGder}
\df{\partial J_i^{N,2}}{\partial \wt X_{k+1,j}} = \ind\{i = j\} + (-1)^{i+j-1}\f{e^{W_{k,j}(t)}}{e^{W_{k,j}(t)}-1},\qquad\text{and}\qquad\df{\partial J_i^{N,2}}{\partial X_{k,j}} + \df{\partial J_i^{N,2}}{\partial \wt X_{k+1,j}} = \ind\{i = j\},
\ee
By It\^o's lemma and the second equality of \eqref{FGder}, the diffusion term for $X_{r+1,i}(t) = J_i^{N,2}\bigl(\mbf X_r(t),\wt{\mbf X}_{r+1}(t)\bigr)$ is 
\begin{align*}
\sum_{j = i}^{i+1} \Bigl(\df{\partial J_i^{N,2}}{\partial X_{k,j}} + \df{\partial J_i^{N,2}}{\partial\wt X_{k+1,j}}\Bigr)\Bigl(\beta dB_{j}(t) - \beta dB_{j-1}(t)\Bigr) = \beta \diff_i d\mbf B(t),
\end{align*}
and this matches the diffusion term in \eqref{eq:k+1_eq}. We turn our attention to the drift term. Let $H = H(i,k)$ be $2N \times 2N$  block Hessian matrix 
\[
H = \begin{pmatrix}
H_{k,k} & H_{k,k+1} \\
H_{k+1,k} & H_{k+1,k+1}
\end{pmatrix},
\]
where $(H_{k,k})_{m,n}$ is the $N \times N$ matrix of  partial derivatives,
$
\df{\partial^2 J_i^{N,2}}{\partial X_{k,m} \partial X_{k,n}},
$
$(H_{k,k+1})_{m,n}$ is the matrix of partial derivatives
$
\df{\partial^2 J_i^{N,2}}{\partial \wt X_{k+1,m} \partial \wt X_{k+1,n}},
$
and analogously for $H_{k+1,k}$ and $H_{k+1,k+1}$. We note these matrices depend on $i$ also, but suppress the dependence in the notation.  Note that the second equality in \eqref{FGder} implies 
\[
\df{\partial^2 J_i^{N,2}}{\partial X_{k,m}\partial X_{k,n}} = \df{\partial}{\partial X_{k,m}}\Bigl(\ind\{i = n\} - \df{\partial J_i^{N,2}}{\partial \wt X_{k+1,n}}\Bigr) = - \df{\partial^2 J_i^{N,2}}{\partial X_{k,m} \partial \wt X_{k+1,n}},
\]
and a similar computation shows that the right-hand side is also equal to $\df{\partial^2 J_i^{N,2}}{\partial \wt X_{k+1,m}\partial \wt X_{k+1,m}}$ and \\ $- \df{\partial^2 J_i^{N,2}}{\partial \wt X_{k+1,m} \partial X_{k,n}}$. Thus, 
\begin{equation} \label{eq:H_transport}
H_{k,k} = H_{k+1,k+1} = - H_{k,k+1} = -H_{k+1,k}
\end{equation}
Hence, the contribution of the Hessian term in It\^o's lemma \eqref{eq:Ito} to the drift of $X_{k+1,i}$ is
\[
\f{1}{2} \text{Tr} \big(\Sigma^T H \Sigma\big) = \f{1}{2} \text{Tr} \begin{pmatrix}\sigma^T H_{k,k} \sigma + \sigma^T H_{k+1,k+1} \sigma + 2\sigma^T H_{k,k+1} \sigma &0 \\
0 &0
\end{pmatrix} = \f{1}{2}\text{Tr} \begin{pmatrix} 0 &0 \\
0 &0
\end{pmatrix} = 0.
\]
Note that thus far, we have not yet used the specific form of $J^{N,2}$; the matching of diffusion terms and the disappearance of the Hessian term followed from the second equality of \eqref{FGder}. To compute the drift term of $X_{k+1,i}$, we now use the precise formula for $J^{N,2}$. Referring back to \eqref{eq:tXR_sys} and using the first equality of \eqref{FGder}, the drift term for $X_{k+1,i}(t) = J_i^{N,2}\bigl(\mbf X_k(t),\wt{\mbf  X}_{k+1}(t)\bigr)$ is  
\begin{align*}
&\quad \,\df{\partial J_i^{N,2}}{\partial X_{k,i}} \times \Bigl(\diff_i e^{-\mbf X_k(t)}- \sum_{m = 1}^{k-1} \Lapl_ie^{-\mbf X_m(t)}\Bigr)  + \df{\partial J_i^{N,2}}{\partial X_{k,i+1}}\times \Bigl(\diff_{i+1}e^{-\mbf X_k(t)} - \sum_{m = 1}^{k-1} \Lapl_{i+1}e^{-\mbf X_m(t)}\Bigr)  \\
&+ \df{\partial J_i^{N,2}}{\partial \wt X_{k+1,i}} \times \Bigl(\diff_i e^{-\wt{\mbf X}_{k+1}(t)}  - \sum_{m = 1}^{k-1} \Lapl_ie^{-\mbf X_m(t)}\Bigr) + \df{\partial J_i^{N,2}}{\partial \wt X_{k+1,i+1}(t)}\times \Bigl(\diff_{i+1}e^{-\wt{\mbf X}_{k+1}(t)}- \sum_{m = 1}^{k-1} \Lapl_{i+1}e^{-\mbf X_m(t)}\Bigr)  \\
&= \f{e^{W_{k,i}}}{e^{W_{k,i}(t)} - 1} \diff_ie^{-\mbf X_k(t)} - \f{e^{W_{k,i+1}(t)}}{e^{W_{k,i+1}} - 1} \diff_{i+1}e^{-\mbf X_k(t)} \\
&+ \Bigl(1 - \f{e^{W_{k,i}(t)}}{e^{W_{k,i}(t)} - 1}\Bigr)\diff_ie^{-\wt{\mbf X}_{k+1}(t)} +  \f{e^{W_{k,i+1}(t)}}{e^{W_{k,i+1}(t)} - 1}\diff_{i+1}e^{-\wt{\mbf X}_{k+1}(t)} -  \sum_{m = 1}^{k-1} \Lapl_ie^{-\mbf X_m(t)}.
\end{align*}
Hence, in order to prove \eqref{eq:k+1_eq}, it suffices to show that 
\begin{align*}
&\f{e^{W_{k,i}(t)}}{e^{W_{k,i}(t)} - 1} \diff_ie^{-\mbf X_k(t)} - \f{e^{W_{k,i+1}(t)}}{e^{W_{k,i+1}(t)} - 1} \diff_{i+1}e^{-\mbf X_k(t)} \\
&\qquad \qquad + \Bigl(1 - \f{e^{W_{k,i}(t)}}{e^{W_{k,i}(t)} - 1}\Bigr)\diff_ie^{-\wt{\mbf X}_{k+1}(t)} +  \f{e^{W_{k,i+1}(t)}}{e^{W_{k,i+1}(t)} - 1}\diff_{i+1}e^{-\wt{\mbf X}_{k+1}(t)} \\
   &= \diff_ie^{-\mbf X_{k+1}(t)} - \Lapl_ie^{-\mbf X_k(t)}.
\end{align*}
This follows as a direct application of Lemma \ref{lem:nasty_algebra}.
\end{proof}

\begin{proposition} \label{prop:prod_invar}
For $N\in \N$, $\beta >0$ and $\theta_1,\ldots,\theta_k \in \R$, let $\bigl(\mbf X_1(t),\ldots,\mbf X_k(t)\bigr)_{t \ge 0}$ solve \eqref{DxR_gen} with initial condition $\bigl(\mbf X_1(0),\ldots,\mbf X_k(0)\bigr) \sim \nu_\beta^{(\theta_1,\ldots,\theta_k)}$. Then, for $t > 0$,
\[
\bigl(\mbf X_1(t),\ldots,\mbf X_k(t)\bigr) \deq \bigl(\mbf X_1(0),\ldots,\mbf X_k(0)\bigr).
\]
\end{proposition}
\begin{proof}
We verify this invariance directly by showing that the adjoint of the generator applied to the probability density of $\nu_\beta^{N,(\theta_1,\ldots,\theta_k)}$ is $0$. There are some technical details to be handled in Appendix \ref{appx:invariance}, which are pointed out below. It is important to note that the measure $\nu_\beta^{N,(\theta_1,\ldots,\theta_k)}$ does not have a density with respect to Lebesgue measure on $\R^{kN}$ because it is supported on $(\mbf X_1,\ldots,\mbf X_k) \in (\R^{\Z_N})^k$ with $\vecsum(\mbf X_r)=\sum_{i \in \Z_N} X_{r,i} = \theta_r$ for $1\leq r\leq k$ (in other words, it is supported on $\R^{\Z_N}_{\theta_1}\times \cdots \times \R^{\Z_N}_{\theta_k}$ in the notation of \eqref{RA2}).  Indeed, it follows directly  from \eqref{DxR_gen} that $\df{d}{dt}\vecsum(\mbf X_{r}(t)) = 0$ for $1 \le r \le k$. The initial condition $\bigl(\mbf X_1(0),\ldots,\mbf X_k(0)\bigr)$ satisfies $\vecsum(\mbf X_r(0)) = \theta_r$ a.s. for $1 \le r \le k$, so we may write the system \eqref{DxR_gen} as the $k(N-1)$-dimensional diffusion process in the variables
 $\bigl((X_{1,i}(t))_{1 \le i \le N-1},\ldots,(X_{k,i}(t))_{1 \le i \le N-1}\bigr)$, with
 \[
 X_{r,0}(t) := \theta_r - \sum_{i = 1}^{N-1} X_{r,i}(t),\quad 1 \le r \le k.
 \]
In these variables, $\nu_\beta^{N,(\theta_1,\ldots,\theta_k)}$ has a density with respect to Lebesgue measure on $\R^{k(N-1)}$.

With this motivation, we wish to reduce the number of Brownian motions in the system \eqref{DxR_gen} from $N$ to $N-1$. The diffusion term for the $i = 1$ terms in \eqref{DxR_gen} is $\beta dB_1(t) - \beta dB_0(t)$, which has the same law as $\sqrt 2 \beta dB_1(t)$.  Looking at the Brownian motions with $i = 1$ and $i = 2$, it is quickly seen from the pairwise covariance structure that
\[
\Biggl(\beta B_1(t) - \beta B_0(t),\beta B_2(t) - \beta B_1(t) \Biggr) \deq \Biggl(\sqrt 2 \beta B_1(t), -\f{1}{\sqrt 2}\beta B_1(t) +  \sqrt{\f{3}{2}}\beta B_2(t) \Biggr),
\]
and observe that this reduces the number of Brownian motions from $3$ to $2$. There are  many ways to make this reduction (all of which result in the same generator for the SDE), but this choice is particularly simple. In general, we can see that 
\[
(\beta B_{i} - \beta B_{i-1})_{1 \le i \le N-1}\deq (\sigma_{i,i-1} B_{i-1} + \sigma_{i,i} B_i  )_{1 \le i \le N-1}, 
\]
where, for $1 \le i \le N-1$, $\sigma_{i,i-1} = -\beta \sqrt{\f{i-1}{i}} 
\ind\{i > 1\}$, $\sigma_{i,i} = \beta \sqrt{\f{i+1}{i}}$, and $\sigma_{i,j} = 0$ otherwise. The vector on the right-hand side only involves $N-1$ Brownian motions since the coefficient of $B_0$ (when $i = 1$) is $0$. We consider now $\sigma$ as an $(N-1) \times (N-1)$ matrix. Hence, when fixing $\vecsum(\mbf X_r(t))$ for $1 \le r \le k$, the system of SDEs \eqref{DxR_gen} restricted to $1 \le i \le N-1$, is equal in law to the system
\be \label{eq:N_prod_SDE}
\begin{aligned}
    d X_{r,i}(t) &= \Bigl(\diff_ie^{-\mbf X_r(t)} - \sum_{m = 1}^{r-1}\Lapl_ie^{-\mbf X_m(t)} \Bigr)\,dt - \sigma_{i,i-1} dB_{i-1}(t) + \sigma_{i,i} dB_i(t), \quad 1 \le i \le N-1, 1 \le r \le k.
\end{aligned}
\ee
In the definition above, note that 
\[
\diff_1e^{-\mbf X_r} = e^{-X_{r,1}} - e^{-X_{r,0}} = e^{-X_{r,1}} - e^{-(\theta_r - X_{r,1} - \cdots - X_{r,N-1})},
\]
and we similarly replace $X_{r,0}$ with $\theta_r -\sum_{i=1}^{N-1} X_{r,i}$ when this terms shows up in $\Lapl_1e^{-\mbf X_m}$ and $\Lapl_{N-1}e^{-\mbf X_m}$. Because we use the same set of Brownian motions for each value of $r$, the diffusion matrix for this system is the $k(N-1) \times k(N-1)$ dimensional block matrix 
\[
\Sigma = \begin{pmatrix}
\sigma &0 &\cdots &0 \\
\sigma &0 &\cdots &0 \\
\vdots &\ddots &\ddots &\vdots \\
\sigma &0 &\cdots &0
\end{pmatrix},\qquad \text{and we define}\qquad \mbf D := \f{1}{2} \Sigma \Sigma^T = \f{1}{2}\begin{pmatrix}
\sigma \sigma^T &\sigma \sigma^T &\cdots &\sigma \sigma^T \\
\sigma \sigma^T &\sigma \sigma^T &\cdots &\sigma \sigma^T \\
\vdots &\ddots &\ddots &\vdots \\
\sigma \sigma^T &\sigma \sigma^T &\cdots &\sigma \sigma^T
\end{pmatrix}
\]
By a straightforward computation,  $\sigma \sigma^T$ is the matrix with diagonal entries $2\beta^2$, and nearest-neighbor entries $-\beta^2$. Note, we do not consider $1$ and $N-1$ as nearest neighbors. Written out, 
\[
\sigma \sigma^T  = \beta \begin{pmatrix} 2 &-1 &0 &\cdots &0 &0 &0 \\
-1 &2 &-1 &\cdots &0 &0 &0 \\
\vdots &\vdots &\vdots &\ddots &\vdots &\vdots &\vdots \\
0 &0 &0 &\cdots &-1 &2 &-1\\
0 &0 &0 &\cdots &0 &-1 &2
\end{pmatrix}
\]

The formal generator of the system of SDEs \eqref{eq:N_prod_SDE} is 
\begin{align} \label{eq:generator}
\mathcal L f  := \sum_{r = 1}^k \sum_{i = 1}^{N-1} \Bigl(\diff_ie^{-\mbf x_r} - \sum_{m = 1}^{r-1}\Lapl_ie^{-\mbf x_m} \Bigr) \partial_{x_{r,i}}f + \sum_{r_1 = 1}^k \sum_{r_2 = 1}^k \sum_{i_1 = 1}^{N-1} \sum_{i_2 = 1}^{N-1} D_{(r_1,i_1),(r_2,i_2)} \partial_{x_{r_1,i_1}}\partial_{x_{r_2,i_2}}f,  
\end{align}
where  $D_{(r_1,i_1),(r_2,i_2)}$ denotes the $(i_1,i_2)$ entry of the $(r_1,r_2)$ block of $\mbf D$. Here, $f$ is a function of $\mbf x_1,\ldots,\mbf x_k$, each of which are vectors in $\R^{\Z_N}$. The generator \eqref{eq:generator} may be written in the more compact form
\[
\mathcal L f = (\text{grad} f)^T \mbf A + (\text{grad} f)^T \mbf D (\text{grad} f),
\]
where $\text{grad} f$ is the gradient operator, and $\mbf A$ the $k(N-1)$-dimensional vector whose entries are the drift coefficients in \eqref{eq:N_prod_SDE}. The formal adjoint of $\mathcal L$ is
\be \label{L_star}
\mathcal L^\star f = -\sum_{r = 1}^k \sum_{i=1}^{N-1} \partial_{x_{r,i}}\Bigl[\Bigl(\diff_ie^{-\mbf x_r} - \sum_{m = 1}^{r-1}\Lapl_ie^{-\mbf x_m} \Bigr) f \Bigr]    
+ \sum_{r_1 = 1}^k \sum_{r_2 =1}^k  \sum_{i_1 = 1}^{N-1}\sum_{i_2 = 1}^{N-1} \partial_{x_{r_1,i_1}} \partial_{x_{r_2,i_2}} [ D_{(r_1,i_1),(r_2,i_2)} f].
\ee

For short, we use the notation $p(\mbf x_1,\ldots,\mbf x_k) = p_\beta^{N,(\theta_1,\ldots,\theta_k)}(\mbf x_1,\ldots,\mbf x_k)$ for the density in Definition \ref{def:p_meas}, where $\mbf x_r = (x_{r,i})_{1 \le i \le N-1}$ with the convention $x_{r,0} = \theta_r - \sum_{i=1}^{N-1} x_{r,i}$. From \eqref{p_dens1}, we have 
\be \label{eq:dens_big}
p(\mbf x_1,\ldots,\mbf x_k) := C \prod_{r = 1}^k \exp\Bigl(-\beta^{-2}e^{-(\theta_r - x_{r,1} - \cdots - x_{r,N-1})}\Bigr)\prod_{i = 1}^{N-1} \exp\Bigl(-\beta^{-2}e^{-x_{r,i}}\Bigr)
\ee
for a normalizing constant $C > 0$. 

It remains to show the that $p$ is an invariant probability density for the process \eqref{eq:N_prod_SDE}. We claim that
\begin{equation}\label{e.Lstarzero}
    \mathcal L^\star p = 0,
\end{equation} which by \eqref{L_star}, means that
\be \label{FPE_N}
\begin{aligned}
    &\quad \, \sum_{r = 1}^k \sum_{i=1}^{N-1} \partial_{x_{r,i}}\Bigl[\Bigl(\diff_ie^{-\mbf x_r} - \sum_{m = 1}^{r-1}\Lapl_ie^{-\mbf x_m} \Bigr) p(\mbf x_1,\ldots,\mbf x_k) \Bigr]    
    \\
    &= \sum_{r_1 = 1}^k \sum_{r_2 =1}^k  \sum_{i_1 = 1}^{N-1}\sum_{i_2 = 1}^{N-1} \partial_{x_{r_1,i_1}} \partial_{x_{r_2,i_2}} [ D_{(r_1,i_1),(r_2,i_2)} p(\mbf x_1,\ldots,\mbf x_k)],
\end{aligned}
\ee
We show this by direct calculation. First note from \eqref{eq:dens_big} that, for $1 \le r \le N$ and $1 \le i \le N- 1$, 
\[
\partial_{x_{r,i}} p(\mbf x_1,\ldots,\mbf x_k) = \f{1}{\beta^2}(e^{-x_{r,i}} - e^{-x_{r,0}})p(\mbf x_1,\ldots,\mbf x_k).
\]
By definition of $D$ and the tri-diagonal structure of $\sigma$, the right-hand-side of \eqref{FPE_N} becomes 
\begin{align*}
    &\quad \, \beta^2 \Biggl(\sum_{i = 1}^{N-1}\sum_{r_1 = 1}^k \sum_{r_2 = 1}^k  \partial_{x_{r_1,i}}\partial_{x_{r_2,i}}         - \sum_{i = 2}^{N-1} \sum_{r_1 = 1}^k \sum_{r_2 = 1}^k \partial_{x_{r_1,i}}\partial_{x_{r_2,i-1}}            \Biggr)p(\mbf x_1,\ldots,\mbf x_k) \\
    &= \beta^2 \Biggl(\sum_{i = 1}^{N-1} \Bigl(\sum_{r = 1}^k \partial_{x_r,i}^2 + \sum_{1 \le r_1 < r_2 \le k}2 \partial_{x_{r_2,i}} \partial_{x_{r_1,i}}    
  \Bigr) \\
  &\qquad\qquad \qquad - \sum_{i = 2}^{N-1} \Bigl(\sum_{r = 1}^k \partial_{x_{r,i}} \partial_{x_{r,i-1}} + \sum_{1 \le r_1 < r_2 \le k}( \partial_{x_{r_2,i}} \partial_{x_{r_1,i-1}} + \partial_{x_{r_2,i-1}} \partial_{x_{r_1,i}})  \Bigr)  \Biggr)p(\mbf x_1,\ldots,\mbf x_k) \\
    &= \sum_{i = 1}^{N-1} \biggl(\sum_{r = 1}^k \partial_{x_{r,i}}\big((e^{-x_{r,i}} - e^{-x_{r,0}})p(\mbf x_1,\ldots,\mbf x_k)\big) + \sum_{1 \le r_1 \le r_2 < k} 2\partial_{x_{r_2,i}}\big((e^{-x_{r_1,i}} - e^{-x_{r_1,0}})p(\mbf x_1,\ldots,\mbf x_k)\big)\biggr)  \\
    &\qquad - \sum_{i = 2}^{N-1}\biggl( \sum_{r = 1}^k \partial_{x_{r,i}}\big((e^{-x_{r,i-1}} - e^{-x_{r,0}})p(\mbf x_1,\ldots,\mbf x_k)\big) \\
    &\qquad\qquad + \sum_{1\le r_1 < r_2 \le k}\big((\partial_{x_{r_2,i}}(e^{-x_{r_1,i-1}} - e^{-x_{r_1,0}})  + \partial_{x_{r_2,i-1}}(e^{-r_1,i} - e^{-x_{r_1,0}})   )p(\mbf x_1,\ldots,\mbf x_k)\big)   \biggr) \\
    &= \sum_{i = 1}^{N-1} \sum_{r = 1}^k \partial_{x_{r,i}}\big((e^{-x_{r,i}} - e^{-x_{r,i-1}})p(\mbf x_1,\ldots,\mbf x_k)\big) \\
    &\qquad + \sum_{i = 1}^{N-1} \sum_{1 \le r_1 < r_2 \le k} \partial_{x_{r_2,i}}\big((2e^{-x_{r_1,i}} -  e^{-x_{r_1,i-1}} - e^{-x_{r_1,i+1}})p(\mbf x_1,\ldots,\mbf x_k)\big),
\end{align*}
and this is equal to the left-hand side of \eqref{FPE_N}. 

Having shown that $\mathcal L^\star p = 0$ we want to conclude that the density $p(\cdot)$ is invariant for the process \eqref{eq:N_prod_SDE}, e.g., by using the Echeverr\'ia criterion \cite{Echeverria-1982}. However, in the current setting it is unclear to us how to justify the assumptions there. For example it seems nontrivial to show that the solution to \eqref{eq:N_prod_SDE} is a Feller process, since the coefficients of the SDE are unbounded and non-Lipschitz (although they are locally Lipschitz). Furthermore, the matrix $\mbf D$ is degenerate. For these reasons, we prove by hand that the density $p(\cdot)$ is invariant, by an approximation argument utilizing the fact that $\mathcal L^\star p = 0$. In particular, by Proposition~\ref{p.invariantSDE} (proved in the appendix), we conclude invariance and complete our proof.
\end{proof}

We may now easily conclude with the proof of Theorem \ref{thm:OCY_joint}.

\begin{proof}[Proof of Theorem \ref{thm:OCY_joint}]
We first prove the theorem in the case when the $\theta_1,\ldots,\theta_k$ are all distinct. Let $\bigl(\mbf U_1(t),\ldots,\mbf U_k(t)\bigr)$ solve \eqref{eq:joint_OCY} with initial condition 
\be \label{U0}
\bigl(\mbf U_1(0),\ldots,\mbf U_k(0)\bigr) = \D^{N,k}\bigl(\mbf X_1',\ldots,\mbf X_k'\bigr),
\ee
where $\bigl(\mbf X_1',\ldots,\mbf X_k'\bigr) \sim \nu_\beta^{N,(\theta_1,\ldots,\theta_k)}$ for $\theta_1,\ldots,\theta_k$ distinct. By the bijection of  Proposition \ref{prop:transform}, with probability one, we have $\bigl(\mbf U_1(0),\ldots,\mbf U_k(0)\bigr) \in \mathcal R^{N,k}_{\theta_1,\ldots,\theta_k}$. Then, by Proposition \ref{prop:NewSDE} and since $\df{d}{dt} \vecsum(\mbf U_m(t)) = 0$ for $1 \le m \le k$, $\bigl(\mbf U_1(t),\ldots,\mbf U_k(t)\bigr) \in \mathcal R^{kN}_{\theta_1,\ldots,\theta_k}$ for all $t \ge 0$. By the same proposition, if, for $t \ge 0$, we define
\[
\bigl(\mbf X_1(t),\ldots,\mbf X_k(t)\bigr) = \mathcal J^{N,k}\bigl(\mbf U_1(t),\ldots,\mbf U_k(t)\bigr),
\]
then $\bigl(\mbf X_1(t),\ldots,\mbf X_k(t)\bigr)$ solves the dual system of SDEs \eqref{DxR_gen}. Applying \eqref{U0} and the bijection of Proposition \ref{prop:transform}, $\bigl(\mbf X_1(0),\ldots,\mbf X_k(0)\bigr) = (\mbf X_1',\ldots,\mbf X_k')$. Applying Proposition \ref{prop:transform} yet again along with the invariance in Proposition \ref{prop:prod_invar}, we obtain, for $t > 0$,
\[
\bigl(\mbf U_1(t),\ldots,\mbf U_k(t)\bigr) = \D^{N,k}\bigl(\mbf X_1(t),\ldots,\mbf X_k(t)\bigr) \deq \D^{N,k} \bigl(\mbf X_1(0),\ldots,\mbf X_k(0)\bigr) =\bigl(\mbf U_1(0),\ldots,\mbf U_k(0)\bigr).
\]

Now, assume the $\theta_1,\ldots,\theta_k$ are not distinct. Then, there exists an index $\ell \in \{1,\ldots,k-1\}$ and $\sigma \in \mathcal S(k)$ so that $\theta_{\sigma(1)},\ldots,\theta_{\sigma(\ell)}$ are distinct, and for all $r > \ell$, there exists $m \in \{1,\ldots,\ell\}$ so that $\sigma(r) = \sigma(m)$.  By Proposition \ref{prop:disc_consis}\ref{itm:mu_order}, if $(\mbf U_1,\ldots,\mbf U_k) \sim \mu_\beta^{N,(\theta_1,\ldots,\theta_k)}$, then for each $r > \ell$, there exists $m \le \ell$ so that $\mbf U_{\sigma(r)} = \mbf U_{\sigma(m)}$ almost surely. Hence, it suffices to show that the law of $(\mbf U_{\sigma(1)},\ldots,\mbf U_{\sigma(\ell)})$ is invariant for the coupled SDE \eqref{eq:joint_OCY}. By Proposition \ref{prop:disc_consis}\ref{itm:mu_perm}, $(\mbf U_{\sigma(1)},\ldots,\mbf U_{\sigma(\ell)}) \sim \mu_\beta^{N,(\theta_{\sigma(1)},\ldots,\theta_{\sigma(\ell)})}$, so the invariance now follows from the case when the $(\theta_m)_{1 \le m \le k}$ are all distinct. 
\end{proof}

\section{Convergence of the periodic O'Connell-Yor polymer to the stochastic heat equation} \label{sec:SHE_conv_section}

There are two main results of this section. The first is is Theorem \ref{L2_conv_main_theorem}, the $L^2$ convergence of the partition function for the periodic O'Connell-Yor polymer model to the solution of the periodic stochastic heat equation. The second main result is Theorem \ref{thm:SHE_convergence}, which deals with convergence of the model with continuous initial data.

For $\beta > 0$, let $\beta_N := \beta N^{-1/2}$. For $\theta  \in \R$, $X,Y \in [0,1]$, $S < T$, and $N \in \N$, let (recall $\OCYp_{\beta}$ from \eqref{eq:reg_conv})
\be \label{ZNdef}
\scOCYp(T,Y \viiva S,X;\theta ) := Ne^{-(T-S)N^2 - \f{\beta_N^2}{2}(T-S)N^2} \OCYp_{\beta_N}(TN^2, \lfloor TN^2 + Y N \rfloor \viiva SN^2, \lfloor SN^2 + XN \rfloor; \theta ).
\ee
We also note that $Z_\beta^{\text{per}}$ (which appears in the theorem) is precisely defined in
\eqref{eq:Zper_def} below.

\begin{theorem} \label{L2_conv_main_theorem}
Let $(\Omega,\F,\Pp)$ be a probability space on which a space-time noise $\xi_{\T}$ is defined. On this probability space, for each $N \in \N$, we may couple the process $\scOCYp$ so that, for each fixed $\theta  \in \R$, $X,Y \in [0,1]$ and $S < T$, and any sequences $T_N \to T,S_N \to S, X_N \to X, Y_N \to Y$, and $\theta_N \to \theta$, we have the convergence
\begin{align*}
\lim_{N\to \infty}\Ee\Big[\big|\scOCYp(T_N,Y_N \viiva S_N,X_N;\theta_N) -  Z_\beta^{\text{per}}(T,Y \viiva S,X;\theta)\big|^2\Big] \to 0.
\end{align*}
\end{theorem}
\begin{remark}
One may modify Theorem \ref{L2_conv_main_theorem},  to allow for arbitrary $\beta_N$ such that $\beta_N N^{1/2}$ converges to $\beta$. The only difference in the chaos series is an extra term $\Bigl(\f{\beta_N N^{1/2}}{\beta}\Bigr)^k$, which should converge to $1$ for each $k$ and is bounded by $C^k$. This is not necessary for our purposes and so we simplify things by taking $\beta_N = \beta N^{-1/2}$. 
\end{remark}

The proof of Theorem \ref{L2_conv_main_theorem} is through convergence of the discrete chaos series for $\scOCYp$ to the continuum chaos series for $Z_\beta^{\text{per}}$. The $L^2$ coupling comes from building the Brownian motions involved in the polymer model from the continuum white noise as in \cite{Nica-2021}. The chaos approach (not necessarily with $L^2$ convergence though) has been implemented several times for KPZ models in the full space 
\cite{Alberts-Khanin-Quastel-2014b,CSZ, Corwin2017, 10.1214/17-EJP32, Nica-2021,GRASS-23} and half-space \cite{Wu2020,10.1214/22-EJP775,10.1214/23-AOP1634}.This is the first instance where it has been done in a periodic domain. As is usually the case, the technical challenge in implementing this approach lies in developing precise control over the discrete heat kernel and multiple stochastic integrals involving it.

The rest of this section is as follows. 
Section \ref{sec:white_noise} defines periodic white noise and stochastic integrals. Section \ref{sec.l2converge} proves $L^2$ convergence of the chaos series for $Z_\beta^{\text{per}}$. Section \ref{sec:OYchaos} develops the chaos series for the O'Connell-Yor partition function. The proof of Theorem \ref{L2_conv_main_theorem} is contained in Section \ref{sec:L2_conv_final}. Finally, Section \ref{sec.convergeinitial} states and proves Theorem \ref{thm:SHE_convergence} which extends the convergence result to general initial data.

\subsection{White noise on the cylinder and its periodic extension} \label{sec:white_noise}
For $N \in \N$, we consider a space-time white noise $\xi_{N\T}$ on $\R \times N\T$. When $N = 1$, we shall simply write $\xi_\T$.
Precisely $\xi_{N\T}$ is defined by its action on test functions $f \in L^2\bigl(\R \times [0,N]\bigr)$, namely,  $\Bigl\{\xi_{N\T}(f): f \in L^2\bigl(\R \times [0,N]\bigr)\Big\}$ is a mean-zero Gaussian process satisfying the following: 
\begin{enumerate}[label=\textup{(\roman*)}]  
\item For $f,g \in L^2(\R \times [0,N])$ and $a,b \in \R$, $\xi_{N\T}(af + bg) = a\xi_{N\T}(f) + b\xi_\T(g)$.
\item For $f,g \in L^2(\R \times [0,N])$, $\Ee[\xi_{N\T}(f) \xi_{N\T}(g)] = \int_0^N \int_\R f(t,x)g(t,x)\,dt\,dx$.
\end{enumerate}
We will denote $\xi_{N\T}(f)$ as
\[
\xi_{N\T}(f) = \int_0^N \int_\R f(t,x)\xi_{N\T}(dt,dx). 
\]
We extend $\xi_{N\T}$ to a noise on $\R \times \R$, which we denote as $\xi_N$. When $N = 1$, we simply write $\xi$. This is constructed formally by the equality $\xi_{N\T}(t,x) = \xi_N(t,x)$ for $x \in [0,N]$, and $\xi_N(t,x) = \xi_N(t,x+N)$ for all $x \in\R$. To give a rigorous definition, for a function $f:\R \times \R \to \R$, if $t \in \R$ and $x \in [0,N]$, we define 
\[
f^{N,\text{per}}(t,x) := \sum_{j \in \Z} f(t,x+jN).
\]
Let $L^{2,N,\text{per}}(\R \times \R)$ be the set of functions $f:\R \times \R \to \R$ satisfying $f^{N,\text{per}} \in L^2\bigl(\R \times [0,N]\bigr)$. We consider equivalence classes in this space so that $f = g$ if $f^{N,\text{per}} = g^{N,\text{per}}$ a.e. The space $L^{2,N,\text{per}}(\R \times \R)$ is equipped with the inner product
\[
\langle f,g \rangle_{L^{2,N,\text{per}}(\R \times \R)} = \langle f^{N,\text{per}},g^{N,\text{per}} \rangle_{L^2(\R \times [0,N])},
\]
which makes $L^{2,N,\text{per}}(\R \times \R)$ a Hilbert space.
Then, we define $\xi_N$ to be the mean-zero Gaussian process
\[
\Bigl\{\xi_N(f): f \in L^{2,N,\text{per}}(\R \times \R)  \Bigr\}
\]
satisfying
$\xi_N(f) = \xi_{N\T}(f^{N,\text{per}})$ for all $f \in L^{2,N,\text{per}}(\R \times \R)$. Formally, we shall write
\[
\xi_N(f) = \int_\R \int_\R f(t,x)\,\xi_N(dt,dx).
\]
We will also work with multiple integrals, denoted by
\begin{align*}
I_{N,k}(f) :=\int_{[0,N]^k} \int_{\R^k} f(t_1,\ldots,t_k,x_1,\ldots,x_k) \prod_{i = 1}^k \xi_{N\T}(dt_i,dx_i),\qquad\text{for}\qquad f \in L^2(\R^k \times [0,N]^k), \qquad k \ge 1.
\end{align*}
These multiple integrals are defined, e.g. in \cite[Section 1.1]{Nua06}. We also define
\begin{multline} \label{eq:int_def}
\int_{\R^k} \int_{\R^k} f(t_1,\ldots,t_k,x_1,\ldots,x_k) \prod_{i = 1}^k \xi_N(dt_i,dx_i)  \\
:= \int_{[0,N]^k} \int_{\R^k} \sum_{j_1,\ldots,j_k \in \Z} f(t_1,\ldots,t_k,x_1 + j_1 N,\ldots,x_k + j_k N) \prod_{i = 1}^k \xi_{N\T}(dt_i,dx_i), 
\end{multline}
whenever the integrand on the right-hand side above lies in $L^2(\R^k \times [0,N]^k)$. If we replace $\R^k$ or $[0,N]^k$ in the integration limits above with some subset, it is equivalent to inserting the indicator function of that subset into the integrand.

A key fact is the following It\^o isometry for multiple integrals (see, for example \cite[Page 9-10]{Nua06}). Recall the definition of the simplex $\Delta^k$ from Section \ref{sec:not}. For $f \in L^2\bigl(\Delta^k \times [0,N]^k\bigr)$, we shall define
\[
I_{N,k}^\Delta(f) := \int_{[0,N]^k} \int_{\Delta^k} f(t_1,\ldots,t_k,x_1,\ldots,x_k) \prod_{i = 1}^k \xi_{N\T}(dt_i,dx_i),\qquad\text{for}\qquad f \in L^2(\R^k \times [0,N]^k)
\]
\begin{lemma} \label{lem:Ito_isometry}
For $k,m \in \N$, and $f \in L^2(\Delta^k \times [0,N]^k)$ and $g \in L^2(\Delta^m \times [0,N]^m)$,
\[
\Ee[I_{N,k}^\Delta(f) I_{N,m}^\Delta(g)]= \begin{cases}
\langle  f,g \rangle_{L^2(\Delta^k \times [0,N]^k)} &k = m \\
0 &k \neq m.
\end{cases}
\]
In particular, $\Ee[(I_{N,k}^\Delta(f))^2] = \|f\|_{L^2(\Delta^k \times [0,N]^k)}^2$. Furthermore, $\Ee[I_{N,k}^\Delta(f)] = 0$. 
\end{lemma}
\begin{remark}
As in \cite{Nua06}, one can also compute $\Ee[I_{N,k}(f)I_{N,k}(g)]$ for functions $f,g:\R^k \to [0,N]^k$, but one replaces $\langle  f,g \rangle_{L^2(\Delta^k \times [0,N]^k)}$ with $\langle  \wt f,\wt g \rangle_{L^2(\Delta^k \times [0,N]^k)}$, where $\wt f$ is the symmetrization of $f$. We do not need that more general fact in the present paper. 
\end{remark}

We also make use of the following scaling relations for white noise, which follows immediately from its definition. In particular, by using indicator functions, one can see that the periodicity is preserved under shear transforms, see Figure \ref{fig:shear}. The scaling factors are a result of Gaussian scaling from the It\^o isometry. 

\begin{figure}
    \centering
    \includegraphics[width=0.4\linewidth]{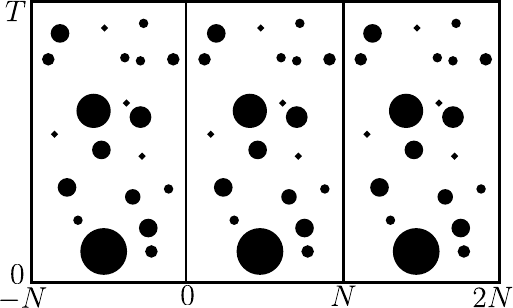} \hspace{1cm}
    \includegraphics[width=0.4\linewidth]{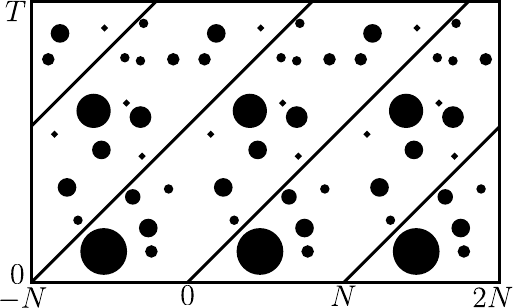}
    \caption{This figure represents shear invariance of the periodic noise $\xi_N$. On the left is a depiction of the noise $\xi_{N}$ for $(s,x) \in [-N,2N] \times [0,T]$ using discs of various sizes. We see the noise has been periodized by $N$ units of space. On the right is the same picture, but we have replaced the vertical lines $x = kN$ with $x = s + kN$ for $k \in \{-1,0,1,2\}$. We see that the noise is repeated between each of these diagonal lines. Under the transformation $y = x-s,t = s$, the region $\{(s,x): s + kN \le x \le s + (k+1)N \}$ is transformed to the region $\{(t,y): kN \le t \le (k+1)N\}$, and so the periodic structure is preserved.}
    \label{fig:shear}
\end{figure}
\begin{lemma} \label{lem:scaling_relations}
We have the following equalities in distribution under scaling.
\begin{enumerate} [label=\textup{(\roman*)}]
\item \textup{(Scaling)} For $N,m \in \N$, under the change of variable $y = \f{x}{N}$ and $t = \f{s}{m}$, we have 
\[
\xi_N(ds,dx) \deq (Nm)^{1/2} \xi(dt,dy),
\]
where we recall $\xi = \xi_1$. 
\item \textup{(Shear invariance)} For $N \in \N$, under the change of variable $y = x-s$ and $t = s$,
\[
\xi_N(ds,dx) \deq \xi_N(dt,dy).
\]
\end{enumerate}

\end{lemma}

\subsection{Chaos series for the periodic stochastic heat equation}\label{sec.l2converge}

Recall the following conventions from the Section \ref{sec:not}: for a set $\Lambda \subseteq \R$ and $x,y \in \R$, we write
\be \label{eq:Z_beta_ser}
\Lambda_{x,y}^{k} := \{
y_i \in \Lambda, 0 \le i \le k+1,
\text{ with set values } y_0 = x,\; \text{and}\; y_{k+1} = y \} \quad\text{and} \quad \Lambda_y^k = \Lambda_{0,y}^k.
\ee
We also define the notation
\[
\Delta^k(t \viiva s) := \{s= s_0 < s_1 < \cdots < s_k < s_{k + 1} = t: s_i \in \R\}\qquad\text{and for $t > 0$,}\qquad \Delta^k(t) = \Delta^k(t \viiva 0).
\]

 We define the following chaos series. We show later in Proposition \ref{prop:solve_SHE} that $Z_\beta(t,y \viiva s,x)$ is the Green's function of  \eqref{eq:SHE}; that is, it solves \eqref{eq:SHE} with initial condition $Z_\beta(s,y \viiva s,x) = \delta(y-x)$. 
 \be \label{eq:Z_b_series}
\begin{aligned}
Z_\beta(t,y \viiva s,x) &:= \sum_{k = 0}^\infty \beta^k I_k(t,y\viiva s,x)\qquad \text{where} \\  I_k(t,y\viiva s,x) &:= \int_{\R^k_{x,y}}\int_{\Delta^k(t \viiva s)} \prod_{i = 0}^k \rho(t_{i + 1} - t_i,x_{i +1} - x_i) \prod_{i = 1}^k \xi (dt_i,dx_i),
\end{aligned}
\ee
where we recall that $\rho(t,x)=(2\pi t)^{-1/2}\exp(-x^2/(2t))$ is the standard heat kernel. Here, notice that we only integrate over the variables $x_1,t_1,\ldots,x_k,t_k$, although we also refer to $x_0,t_0$ and  $x_{k+1},t_{k+1}$ which are set to values $x,s$ and $y,t$, respectively due to the definitions of $\R_{x,y}^k$ and $\Delta^k(t \viiva s)$.
Proposition \ref{prop:Z_conv} below shows that this series is convergent in $L^2(\Pp)$ for each $t > s$ and $x,y \in \R$. For $\theta \in \R$, we also define
\be \label{eq:Zper_def}
Z_\beta^{\text{per}}(t,y \viiva s,x;\theta) := \sum_{j \in \Z} e^{\theta j} Z_\beta(t,y \viiva s,x+j),
\ee
which will be used later for initial data whose logarithm has the slope $\theta$, see \eqref{eq:reg_conv} for its semi-discrete version.

\begin{proposition} \label{prop:Z_conv}
For $t > s$, and $\theta,x,y \in \R$, the following hold:
\begin{enumerate} [label=\textup{(\roman*)}]
\item \label{itm:s1} The series \eqref{eq:Z_b_series} converges in $L^2(\Pp)$. 
\item \label{itm:s2} The series \eqref{eq:Zper_def} converges in $L^2(\Pp)$. $Z_\beta^{\text{per}}(t,y \viiva s,x;\theta)$ has the following $L^2(\Pp)$ convergent chaos series:
\be \label{Jk_series}
\begin{aligned}
Z_\beta^{\text{per}}(t,y \viiva s,x;\theta) &= \sum_{k = 0}^\infty \beta^k J_k(t,y\viiva s,x;\theta),\qquad\text{where} \\
J_k(t,y\viiva s,x;\theta) &:= \int_{[0,1]^k_{x,y}} \int_{\Delta^k(t \viiva s)} \prod_{i = 0}^k  \sum_{r \in \Z} e^{-\theta r} \rho(t_{i+1} - t_i, y_{i+1} - y_i + r) \prod_{i = 1}^k \xi_{\T}(dt_i,dy_i).
\end{aligned}
\ee 
\end{enumerate}
\end{proposition}

The proof of Proposition \ref{prop:Z_conv} is mostly standard though laborious and requires some uniform control on the decay of the Gaussian heat kernel, see Lemma \ref{lem:rhobd} below. We start with this lemma as well as other intermediate lemmas regarding the heat kernel and its integrals. For integers $m \le n$, recall the notation $\lzb m,n \rzb = \{m,m+1,\ldots,n\}$.
\begin{lemma} \label{lem:rhobd}
For $T > 0$ and $\theta \in \R$, there exists a function $f:\Z \to \R$, summable over $\Z$ and depending only on $T$ and $\theta$, so that, whenever $0 < t-s < T$, $x,y \in [0,1]$, and $r \in \Z$,
\be \label{eq:rhobd}
e^{-\theta r} \rho(t - s,y - x + r) \le f(r)\max_{j \in \lzb -3,3 \rzb} \rho(t - s,y - x + j).
\ee
\end{lemma}
\begin{proof}
Start with the observation that for $y \in \R$, $r \in \Z$, and $ t > 0$, 
\be \label{rhobd1}
\rho(t,y + r) = \rho(t,y ) \exp\Bigl(-\f{r(2y + r)}{2t}\Bigr). 
\ee
If we assume that $0 < t-s < T$ and $x,y \in [0,1]$, then for all $|r| \geq  4$, $r(2(y - x) + r) \ge r^2/2$. Thus,
\[
\rho(t - s,y - x + r) \le \rho(t - s, y - x) \exp\Bigl(-\f{r^2}{4(t-s)}\Bigr).
\]
Multiplying both sides by $e^{-\theta r}$ we see that \eqref{eq:rhobd} holds if we let $f(r)=e^{-\theta r}$ for $r\in \lzb -3,3 \rzb$ and $f(r) = \exp\bigl(-\f{r^2}{4T} - \theta r\bigr)$ otherwise. Then, for $|r| > 3$, one can choose $j = 0$ in the right-hand side of \eqref{eq:rhobd}. Clearly, due to the decay of $-\f{r^2}{4T} - \theta r$, $f(r)$ is summable over $\Z$ as desired. 
\end{proof}

\begin{lemma} \label{lem:int_sum_fin}
For $\beta > 0, t > s,x,y \in [0,1]$, and $\theta \in \R$,
\be \label{eq:int_Series_fin}
\sum_{k = 0}^\infty \beta^{2k}\int_{[0,1]^k_{x,y}} \int_{\Delta^k(t \viiva s)} \bigg(\prod_{i = 0}^k \sum_{r \in \Z} e^{-\theta r} \rho(t_{i+1} - t_i, y_{i+1} - y_i + r) \bigg)^2 \prod_{i = 1}^k dt_i\,dy_i < \infty.
\ee
\end{lemma}
\begin{proof}
By Lemma \ref{lem:rhobd}, there exists a summable function $f:\Z \to \R$ and a constant $C$ (each depending on $\beta,\theta,t-s$ and $C$ is changing from line to line),  so that
\begin{align*}
&\quad \, \beta^{2k}\int_{[0,1]^k_{x,y}} \int_{\Delta^k(t \viiva s)} \bigg(\prod_{i = 0}^k \sum_{r \in \Z} e^{-\theta r} \rho(t_{i+1} - t_i, y_{i+1} - y_i + r) \bigg)^2 \prod_{i = 1}^k dt_i\,dy_i \\
&\le \beta^{2k} \int_{[0,1]^k_{x,y}} \int_{\Delta^k(t \viiva s)} \bigg(\prod_{i = 0}^k \sum_{r \in \Z} f(r) \max_{j \in \lzb - 3,3 \rzb} \rho(t_{i+1} - t_i, y_{i+1} - y_i + j)\bigg)^2 \prod_{i = 1}^k dt_i\,dy_i \\
&\le C^k  \int_{\R^k_{x,y}} \int_{\Delta^k(t \viiva s)} \prod_{i = 0}^k \max_{j \in \lzb -3, 3\rzb } \rho^2(t_{i+1} - t_i,y_{i+1} - y_i + j) \prod_{i = 1}^k dt_i\,dy_i \\
&\le C^k \sum_{j_0,\ldots,j_k \in \lzb -3,3\rzb }\int_{\R^k_{x,y}} \int_{\Delta^k(t \viiva s)} \prod_{i = 0}^k \rho^2(t_{i+1} - t_i,y_{i+1} - y_i + j_i) \prod_{i = 1}^k dt_i\,dy_i.  
\end{align*}
Now make the change of variables $x_i = y_i + j_{i - 1} + \cdots + j_0$ to obtain that the last line above is equal to
\begin{align*}
&=C^k \sum_{j_0,\ldots,j_k \in \{-3,\ldots, 3\}} \int_{\R_{x,y + j_0 + \cdots + j_k}^k} \int_{\Delta^k(t \viiva s)} \prod_{i = 0}^k \rho^2(t_{i+1} - t_i, x_{i+1} - x_i) \prod_{i = 1}^k dt_i\,dx_i\\
&\le C^k  \sum_{j \in \Z} \int_{\R^k_{x,y+j}} \int_{\Delta^k(t \viiva s)}  \prod_{i = 0}^k \rho^2(t_{i+1} - t_i, x_{i+1} - x_i) \prod_{i = 1}^k dt_i\,dx_i \\
&  C^k   \f{(t-s)^{(k+3)/2}\Gamma(1/2)^{k+1}}{\Gamma(\f{1}{2}(k+1))2^k \pi^{k/2}} \sum_{j \in \Z}\rho^2(t-s,y-x + j) \le  \f{C^k}{\Gamma(\f{1}{2}(k+1))},
\end{align*}
where the computation of the integral in the last line is Lemma \ref{lem:rho_int_comp}. Hence, the series \eqref{eq:int_Series_fin} is finite.
\end{proof}

\begin{lemma} \label{lem:ser_tail_to0}
For $k \in \N$, $t > s$ and $x,y,\theta \in \R$, the following holds
\[
\lim_{P \to \infty}  \int_{[0,1]^k_{x ,y}}\int_{\Delta^k(t \viiva s)} \Biggl(\sum_{\substack{r_0,\ldots,r_k \in \Z \\ |r_0 + \cdots + r_k| \ge P}}  \prod_{i = 0}^k e^{-\theta r_i} \rho(t_{i + 1} - t_i,x_{i +1}  - x_i + r_i) \Biggr)^2 \prod_{i = 1}^k \,dt_i\,dx_i = 0.
\]
\end{lemma}
\begin{proof}
By Lemma \ref{lem:rhobd}, there is a summable function $f:\Z \to \R$ such that
\begin{align*}
&\quad \, \int_{[0,1]^k_{x ,y}}\int_{\Delta^k(t \viiva s)} \Biggl(\sum_{\substack{r_0,\ldots,r_k \in \Z \\ |r_0 + \cdots + r_k| \ge P}}  \prod_{i = 0}^k e^{-\theta r_i} \rho(t_{i + 1} - t_i,x_{i +1}  - x_i + r_i) \Biggr)^2 \prod_{i = 1}^k \,dt_i\,dx_i \\
&\le \Biggl(\sum_{\substack{r_0,\ldots,r_k \in \Z \\ |r_0 + \cdots + r_k| \ge P}}   \prod_{i = 0}^k f(r_i) \Biggr)^2 \int_{[0,1]^k_{x ,y}}\int_{\Delta^k(t \viiva s)} \prod_{i = 0}^k \max_{j \in \lzb -3,3\rzb}  \rho^2(t_{i + 1} - t_i,x_{i +1}  - x_i + j) \prod_{i = 1}^k \,dt_i\,dx_i \\
&\le \Biggl(\sum_{\substack{r_0,\ldots,r_k \in \Z \\ |r_0 + \cdots + r_k| \ge P}}    \prod_{i = 0}^k f(r_i) \Biggr)^2\int_{[0,1]^k_{x ,y}}\int_{\Delta^k(t \viiva s)} 
 \sum_{j_0,\ldots,j_k \in \lzb -3,3\rzb} \prod_{i = 0}^k  \rho^2(t_{i + 1} - t_i,x_{i +1}  - x_i + j_i) \prod_{i = 1}^k \,dt_i\,dx_i. 
\end{align*}
The integral above is finite, as can be seen by expanding the bounds $[0,1]_{x,y}^k$ to $\R^k_{x,y}$ and computing directly via Lemma \ref{lem:rho_int_comp}. Further, since $f:\Z \to \R$ is summable, it follows that 
\[
\lim_{P \to \infty}\sum_{\substack{r_0,\ldots,r_k \in \Z \\ |r_0 + \cdots + r_k| \ge P}}   \prod_{i = 0}^k f(r_i) = 0,
\]
from which the claimed result of the lemma readily follows.
\end{proof}

With the previous lemmas, we are ready to prove Proposition~\ref{prop:Z_conv}:
\begin{proof}[Proof of Proposition \ref{prop:Z_conv}]

\medskip \noindent \textbf{Item \ref{itm:s1}:} Recalling \eqref{eq:Z_b_series}, we wish to show that the series 
\[
Z_\beta(t,y \viiva s,x):= \sum_{k = 0}^\infty \beta^k I_k(t,y \viiva s,x)
\]
converges in $L^2(\Pp)$. By the It\^o isometry, Lemma \ref{lem:Ito_isometry}, and the definition of the multiple integrals in \eqref{eq:int_def}, 
\begin{align*}
&\quad \,\Ee[Z_\beta^2(t,y \viiva s,x)] = \sum_{k = 0}^\infty \beta^{2k} \Ee[I_k^2(t,y\viiva s,x)] \\
&= \sum_{k = 0}^\infty \beta^{2k} \int_{[0,1]_{x,y}^k} \int_{\Delta^k(t \viiva s)} \Biggl(\sum_{\substack{j_1,\ldots,j_k \in \Z  \\
j_0 = j_{k+1}  = 0}} \prod_{i = 0}^k \rho(t_{i+1} - t_i,x_{i+1} + j_{i+1} - x_i - j_i)\Biggr)^2 \prod_{i = 1}^k dt_i\,dx_i \\
&\le \sum_{k = 0}^\infty \beta^{2k} \int_{[0,1]_{x,y}^k} \int_{\Delta^k(t \viiva s)} \Biggl(\sum_{\substack{j_0,j_1,\ldots,j_k \in \Z  \\
j_{k+1}  = 0}} \prod_{i = 0}^k \rho(t_{i+1} - t_i,x_{i+1} + j_{i+1} - x_i - j_i)\Biggr)^2 \prod_{i = 1}^k dt_i\,dx_i,
\end{align*}
where in the last line, we simply increased the sum by allowing $j_0$ to vary over $\Z$. Making the change of variables $r_i = j_{i+1} - j_i$, we can rewrite the above expression as
\begin{align*}
&\quad \,\sum_{k = 0}^\infty \beta^{2k} \int_{[0,1]_{x,y}^k} \int_{\Delta^k(t \viiva s)} \Biggl(\sum_{r_0,\ldots,r_k \in \Z} \prod_{i = 0}^k \rho(t_{i+1} - t_i,x_{i+1} - x_i + r_i)\Biggr)^2 \prod_{i = 1}^k dt_i\,dx_i \\
&= \sum_{k = 0}^\infty \beta^{2k}\int_{[0,1]^k_{x,y}} \int_{\Delta^k(t \viiva s)} \bigg(\prod_{i = 0}^k \sum_{r \in \Z} \rho(t_{i+1} - t_i, x_{i+1} - x_i + r) \bigg)^2 \prod_{i = 1}^k dt_i\,dx_i,
\end{align*}
and this sum is finite by the $\theta = 0$ case of Lemma \ref{lem:int_sum_fin}.

\medskip \noindent \textbf{Item \ref{itm:s2}:} There are three things to prove here.
\begin{enumerate} [label=\textup{(\alph*)}]
    \item \label{a} The series $\sum_{k = 0}^\infty \beta^k J_k(t,y\viiva s,x;\theta)$ is convergent in $L^2(\Pp)$, where we define
\begin{align*}
J_k(t,y\viiva s,x;\theta) := \int_{[0,1]^k_{x,y}} \int_{\Delta^k(t \viiva s)} \prod_{i = 0}^k  \sum_{r \in \Z} e^{-\theta r} \rho(t_{i+1} - t_i, y_{i+1} - y_i + r) \prod_{i = 1}^k \xi_{\T}(dt_i,dy_i).
\end{align*}
    \item \label{b}The following series is  $L^2(\Pp)$ convergent
    \[
    Z_\beta^{\text{per}}(t,y \viiva s,x;\theta) := \sum_{j \in \Z} e^{\theta j} Z_\beta(t,y \viiva s,x+j).
    \]
    \item \label{c} We have the almost sure equality
    \[
    Z_\beta^{\text{per}}(t,y \viiva s,x;\theta) = \sum_{k = 0}^\infty \beta^k J_k(t,y\viiva s,x;\theta).
    \]
\end{enumerate}
Observe that Item \ref{a} follows by the It\^o isometry of Lemma \ref{lem:Ito_isometry} and the computation in Lemma \ref{lem:int_sum_fin}. Before proving Items \ref{b} and \ref{c}, we first prove the following:
\begin{align} \label{701}
\sum_{j \in \Z} e^{\theta j} I_k(t,y \viiva s,x+j) &= J_k(t,y \viiva s,x;\theta).
\end{align}
and in particular, the series on the left-hand side converges in $L^2(\Pp)$. Recalling he definition of $J_k$ \eqref{Jk_series}, it suffices to prove the following more general equality, for which \eqref{701} is the $M = 0$ case: for integers $M  \ge 0$,
\be \label{701a}
\begin{aligned}
&\quad \, \sum_{j \in \Z \setminus (-M,M)} e^{\theta j} I_k(t,y \viiva s,x+j) \\
&= \int_{[0,1]^k_{x ,y}}\int_{\Delta^k(t \viiva s)} \sum_{\substack{r_0,\ldots,r_k \in \Z \\ |r_0 + \cdots + r_k| \ge M}}  \prod_{i = 0}^k e^{-\theta r_i} \rho(t_{i + 1} - t_i,x_{i +1}  - x_i + r_i) \prod_{i = 1}^k \xi_\T (dt_i,dx_i).
\end{aligned}
\ee

We start by rewriting the left-hand side of \eqref{701a} as 
\begin{align*}
    &\quad \, \sum_{j \in \Z \setminus (-M,M)} e^{\theta j} \int_{\R_{x+j,y}^k} \int_{\Delta^k(t \viiva s)} \prod_{i = 0}^k \rho(t_{i+1} - t_i,x_{i+1}-x_i) \prod_{i = 1}^k \xi(dt_i,dx_i) \\
 &=  \sum_{j \in \Z \setminus (-M,M)} e^{\theta j} \int_{[0,1]^k_{x + j,y}}\int_{\Delta^k(t \viiva s)} \sum_{\substack{j_1,\ldots,j_{k} \in \Z \\ j_0 = j_{k+1} = 0}} \prod_{i = 1}^k \rho(t_{i + 1} - t_i,x_{i +1} + j_{i+1} - x_i - j_i) \prod_{i = 1}^k \xi_\T (dt_i,dx_i) \\
 &= \lim_{P \to \infty} \sum_{M \le |j| \le P} e^{\theta j} \int_{[0,1]^k_{x + j,y}}\int_{\Delta^k(t \viiva s)} \sum_{\substack{j_1,\ldots,j_{k} \in \Z \\ j_0 = j_{k+1} = 0}} \prod_{i = 1}^k \rho(t_{i + 1} - t_i,x_{i +1} + j_{i+1} - x_i - j_i) \prod_{i = 1}^k \xi_\T (dt_i,dx_i) \\
 &= \lim_{P \to \infty}  \int_{[0,1]^k_{x + j,y}}\int_{\Delta^k(t \viiva s)} \sum_{\substack{j_0,j_1,\ldots,j_{k} \in \Z \\ M \le |j_0| \le P,\; j_{k+1} = 0}} e^{\theta j_0} \prod_{i = 1}^k \rho(t_{i + 1} - t_i,x_{i +1} + j_{i+1} - x_i - j_i) \prod_{i = 1}^k \xi_\T (dt_i,dx_i) \\
 &= \lim_{P \to \infty} \int_{[0,1]^k_{x + j,y}}\int_{\Delta^k(t \viiva s)} \sum_{\substack{r_0,\ldots,r_k \in \Z \\ M \le |r_0 + \cdots + r_k| \le P}} \prod_{i = 1}^k e^{-\theta r_i} \rho(t_{i + 1} - t_i,x_{i +1} - x_i + r_i) \prod_{i = 1}^k \xi_\T (dt_i,dx_i).
\end{align*}
The first line is from the definition of $I_k$ in \eqref{eq:Z_b_series}, the second comes from the definition of the multiple stochastic integral over $\R^k \times \R^k$ in \eqref{eq:int_def}, the third is just the definition of the sum over $j$, the fourth comes from interchanging that, now finite, sum in $j$ with the two integrals, and the fifth line comes from changing variables to $r_i=j_{i+1}-j_i$ for $0\leq i\leq k$.

To justify convergence of the series in \eqref{701a} to $J_k(t,y \viiva s,x;\theta)$, we therefore must show that the limit $P\to \infty$ can be taken inside the integral, with respect to $L^2(\Pp)$. This follows by the It\^o isometry in Lemma \ref{lem:Ito_isometry}, along with Lemma \ref{lem:ser_tail_to0}.

We now complete the proof of Items \ref{b} and \ref{c}. The equality to show is 
\be \label{702}
\sum_{j \in \Z}e^{\theta j} Z_\beta(t,y \viiva s,x+j) = \sum_{k = 0}^\infty \beta^k J_k(t,y \viiva s,x;\theta),
\ee
where we note that we have seen that the right-hand side of \eqref{702} converges in $L^2(\Pp)$ (Item \ref{a}). Thus, our task is to prove that the series on the left-hand side of \eqref{702} converges in $L^2(\Pp)$ (Item \ref{b}) and its limit is the right-hand side (Item \ref{c}). 

By definition, the left-hand side of \eqref{702} may be written as
\begin{align}
    \sum_{j \in \Z} e^{\theta j} \sum_{k = 0}^\infty \beta^k I_k(t,y \viiva s,x+j) &= \lim_{M \to \infty} \sum_{j \in \Z \cap (-M,M)} e^{\theta j} \sum_{k = 0}^\infty  \beta^k I_k(t,y \viiva s,x+j) \nonumber  \\
    &=\lim_{M \to \infty} \sum_{k = 0}^\infty \beta^k \sum_{j \in \Z \cap (-M,M)} e^{\theta j} I_k(t,y \viiva s,x+j). \label{eq:Mlim}
\end{align}
On the other hand, by \eqref{701}, the right-hand side of \eqref{702} is 
\be \label{eq:klim}
\sum_{k = 0}^\infty \beta^k \sum_{j \in \Z} e^{\theta j} I_k(t,y \viiva s,x+j).
\ee
Now, comparing the expressions in \eqref{eq:Mlim} and \eqref{eq:klim}, it suffices to show that  
\begin{align*}
    &\lim_{M \to \infty}\Biggl\|\sum_{k = 0}^\infty \beta^k \sum_{j \in \Z \setminus (-M,M)} e^{\theta j} I_k(t,y \viiva s,x+j)\Biggr\|_{L^2(\Pp)} = 0. 
    \end{align*}
By \eqref{701a}, the It\^o isometry,  and Lemma \ref{lem:ser_tail_to0} for any choice of $K \in \N$, 
\begin{align*}
\lim_{M \to \infty} \sum_{k = 0}^K \beta^k \Biggl\| \sum_{j \in \Z \setminus (-M,M)} e^{\theta j} I_k(t,y \viiva s,x+j)  \Biggr\|_{L^2(\Pp)} = 0.
\end{align*}
Thus, it suffices to show that for any $\ve > 0$, we may choose $K \in \N$ sufficiently large so that
\[
\Biggl\| \sum_{k = K+1}^\infty \beta^k \sum_{j \in \Z \setminus (-M,M)} e^{\theta j} I_k(t,y \viiva s,x+j)\Biggr\|_{L^2(\Pp)} \le \ve,
\]
for all $M \ge 0$.
By \eqref{701a} along with the It\^o isometry in Lemma \ref{lem:Ito_isometry}, we have 
\begin{align*}
&\quad \, \Biggl\| \sum_{k = K+1}^\infty \beta^k \sum_{j \in \Z \setminus (-M,M)} e^{\theta j} I_k(t,y \viiva s,x+j)\Biggr\|_{L^2(\Pp)}^2 \\
&= \sum_{k = K+1}^\infty \beta^{2k}\int_{[0,1]^k_{x + j,y}}\int_{\Delta^k(t \viiva s)} \Biggl(\sum_{\substack{r_0,\ldots,r_k \in \Z \\ |r_0 + \cdots + r_k| \ge M}} \prod_{i = 1}^k e^{-\theta r_i} \rho(t_{i + 1} - t_i,x_{i +1} - x_i + r_i)\Biggr)^2 \prod_{i = 1}^k dt_i\,dx_i  \\
&\le \sum_{k = K+1}^\infty \beta^{2k}\int_{[0,1]^k_{x + j,y}}\int_{\Delta^k(t \viiva s)} \Biggl( \prod_{i = 1}^k \sum_{r \in \Z} e^{-\theta r }\rho(t_{i + 1} - t_i,x_{i +1} - x_i + r_i)\Biggr)^2 \prod_{i = 1}^k dt_i\,dx_i.
\end{align*}
 Lemma \ref{lem:int_sum_fin} states that this last series is convergent, and hence we may choose the desired value of $K$.
\end{proof}

\subsection{Chaos series for the O'Connell-Yor polymer partition function}\label{sec:OYchaos}
In this section, we show how to write the scaled O'Connell-Yor polymer partition function as an explicit $L^2(\Pp)$-convergent chaos series. This result is Proposition \ref{prop:OCY_Chaos} below, and is the goal of the present subsection. 

The chaos series in the semi-discrete setting (see \eqref{JkNdef} below) is the same as the chaos series \eqref{Jk_series}, except that $\rho$ is replaced with the semi-discrete heat kernel, defined in terms of the Poisson heat kernel $q$ by
\begin{equation} \label{pN_def}
p_N(t,y \viiva s,x) := N q\big((t-s)N^2, \lfloor tN^2 + yN \rfloor - \lfloor sN^2 + xN \rfloor\big), \quad \textrm{where}\quad q(t,n) := e^{-t} \f{t^n}{n!} \ind\{(t,n)\in \R_{\ge 0} \times \Z_{\ge 0}\}.
\end{equation}

Before stating Proposition \ref{prop:OCY_Chaos}, we motivate the reasoning for how one can write the partition function for O'Connell-Yor polymer as a chaos series. First, we describe how to couple the white noise to a collection of independent Brownian motions. This is the same coupling used in the full-line setting in \cite{Nica-2021,GRASS-23}, which also appeared in an unpublished manuscript \cite{MFQR-note}. We point out that this formulation of the coupling may look different form that in \cite{Nica-2021} because, opposite the convention there, we first couple the Brownian motions to the white noise before scaling, then observe how the noise behaves under the change of variables that produces the SHE in the limit. 

For $N \in \N$,  let $\xi_{N \T}$ be a space-time white noise on  $\R \times [0,N]$, and $\xi_N$ is periodic extension, as defined in Section \ref{sec:white_noise}.  We may couple a field of Brownian motions $(B_r)_{r \in \Z}$ by defining 
\be \label{eq:BM_coupling}
B_r(t) = \int_r^{r+1}  \int_0^{t}  \xi_{N}(ds,dx).
\ee
It follows immediately from the definition of white noise and checking the covariance structure from the It\^o isometry \eqref{eq:Ito} that $(B_r)_{r \in \{0,\ldots,N-1\}}$ are i.i.d., and $B_{r} = B_j$ whenever $j \equiv r \mod N$.

Now, we define an intermediate quantity that is more convenient to work with. For integers $n \ge m$, real numbers $t \ge s$, and $\gamma >0$, set 
\be\label{Ydef}   Y_\gamma^N(t,n \viiva s,m) := e^{-(t - s) - \f{\gamma^2}{2} (t - s)}\OCY_\gamma(t,n \viiva s,m),\ee
with $Z_\gamma^N$ the semi-discrete O'Connell-Yor partition function defined in \eqref{OCYpart}. Note that this relates to the scaled object $\scOCYp$ from \eqref{ZNdef} is via the equality
\be \label{ZperY}
\scOCYp(t,y \viiva s,x;\theta) = N \sum_{j \in \Z} e^{\theta j} Y_{\beta_N}^N\bigl(tN^2,\lfloor tN^2 + yN \rfloor \viiva sN^2, \lfloor sN^2 + xN + jN \rfloor\bigr),
\ee
where $\beta_N = \beta N^{-1/2}$. Lemma \ref{lem:OCYp_finite} shows that this sum converges and is strictly positive a.s.  The following is the Feynman-Kac formula for the quantity $Y_\gamma$. It can be proved rigorously using It\^o calculus (for example, the  proof of \cite[Lemma 3.5]{GRASS-23} works the same in this setting). 
\begin{lemma} \label{lem:Yrep}
For  all integer $n \ge m$ and real $t \ge s$, $Y_\gamma^N(t,n \viiva s,m)$ satisfies the It\^o integral equation
\be \label{Yrep}
Y_\gamma^N(t,n \viiva s,m) = q(t-s,n-m) + \gamma \sum_{m \le k \le n} \int_s^t  q(t - u,n-k) \tspb Y_\gamma^N(u,k \viiva s,m)\,d B_k(u),
\ee
where $\{B_r\}_{r \in \Z}$ are the periodized  Brownian motions used in defining $\OCY_\gamma$. 
\end{lemma}

Starting from Lemma \ref{lem:Yrep}, by iterating \eqref{Yrep} and for $n \ge m$ and $t > s$, $Y_\gamma^{N}(t,n \viiva s,m)$ may be written as the following chaos series, once we show that the series indeed converges in $L^2(\Pp)$.
    \be \label{eq:Y_chaos}
\begin{aligned}
    Y_\gamma^{N}(t,n \viiva s,m) &= \sum_{k = 0}^\infty \gamma^k I_k^{N}(t,n \viiva s,m), \qquad \text{ where} \\
    I_k^{N}(t,n \viiva s,m) &:= \sum_{m = m_0 \le \cdots \le m_{k+1} = n}\int_{\Delta^k(t \viiva s)}  \prod_{i = 0}^k q(s_{i+1} - s_i, m_{i+1} - m_i) \prod_{i = 1}^k dB_{m_i}(s_i).
\end{aligned}
\ee
Note that we can remove the condition $m_i \le m_{i+1}$ in the sum above because $q(s_{i+1} - s_i,m_{i+1} - m_i) = 0$ if $m_i > m_{i+1}$ by definition. Then, using the coupling described in \eqref{eq:BM_coupling} and writing the sum as an integral by introducing the floor functions (see, e.g., \cite[Lemma 3.7]{GRASS-23}), we have 
\be \label{Ik_wNint}
I_k^{N}(t,n \viiva s,m) = \int_{\R^k_{m,n}} \int_{\Delta^k(t \viiva s)}  \prod_{i = 0}^k q(s_{i+1} - s_i, \lfloor x_{i+1} \rfloor - \lfloor x_i \rfloor) \prod_{i = 1}^k \xi_{N}(ds_i,d x_i).
\ee
After a scaling, the proof of Proposition \ref{itm:Ns2} will demonstrate how to write $\scOCYp$ as a chaos series over white noise integrals for the white noise $\xi_\T$ on $\R \times [0,1]$.

We now state the main result of this subsection, which is the semi-discrete analogue of Proposition \ref{prop:Z_conv}. It is proved at the end of Section \ref{sec:OYchaos}. 
\begin{proposition} \label{prop:OCY_Chaos} The following hold:
\begin{enumerate} [label=\textup{(\roman*)}]
    \item \label{itm:Ns1} For $t > s$ and $x,y \in \R$, the series in \eqref{eq:Y_chaos} converges in $L^2(\Pp)$.
    \item \label{itm:Ns2} For a space-time white noise $\xi_\T$ on the cylinder $\R \times \T$, there exists a coupling of $\tilde{Z}_{\beta}^{N,\text{per}}$ so that for all $t > s$ and $x,y \in \R$, we have 
\be \label{JkNdef}
\begin{aligned}
\,\scOCYp(t,y \viiva s,x;\theta) &= \sum_{k = 0}^\infty \beta^k J_k^N(t,y \viiva s,x;\theta), \qquad\text{where} \\
J_k^N(t,y \viiva s,x;\theta) &:=  \int_{[0,1]^k_{x,y}} \int_{\Delta^k(t \viiva s)} \prod_{i = 0}^k \sum_{r \in \Z} e^{-\theta r} p_N(t_{i+1},y_{i+1} + r \viiva t_i,y_i) \prod_{i = 1}^k \xi_{\T}(dt_i, dy_i),
\end{aligned}
\ee
and the series converges in $L^2(\Pp)$ for all sufficiently large $N$.
\end{enumerate}
\end{proposition}

The proof of Proposition \ref{prop:OCY_Chaos} follows a similar procedure as the proof of Proposition \ref{prop:Z_conv}, with some additional difficulties due to the discreteness. As such, several of the lemmas below have counterparts from the previous section.  We start by proving the semi-discrete analogue of Lemma \ref{lem:rhobd}.
\begin{lemma} \label{lem:pN_per_bd}
Assume that $(\theta_N)_{N \in \N}$ is a bounded sequence, and $T \in \R$. Then, there exists a function $h:\Z \to \R$ (depending on $T$ and the sequence $(\theta_N)_{N \in \N}$) that is summable over $\Z$, so that, for all sufficiently large $N$, whenever $0 < t-s < T$, $x,y \in [0,1]$, and $r \in \Z$,
\[
e^{-\theta_N r}p_N(t,y + r \viiva s,x) \le h(r) \max_{j \in \lzb -3,1 \rzb} p_N(t,y + j \viiva s,x).
\]
Furthermore, there exist constants $C,c > 0$ so that $h(r) \le C e^{-cr^2}$ for $r < 0$.  
\end{lemma}
\begin{proof}
We observe that, since $t-s > 0$ and $x,y \in [0,1]$,
\be \label{eq:tibd}
\lfloor t N^2 + y N \rfloor - \lfloor s N^2 + x N \rfloor + N + 1 \ge (t - s)N^2 - N - 1 + N + 1  = (t - s)N^2 > 0.
\ee
This implies that $p_N(t,y+1 \viiva s,x) > 0$ for all $t > s$ and $x,y \in [0,1]$. It suffices to consider two cases: $r > 1$ and $r < -3$, because we may take $h(r) = e^{\sup_N |\theta_N| r}$ for $r \in \lzb -3,1 \rzb$. Specifically, we show that for $r \notin \lzb -3,1\rzb$, $e^{-\theta_N r} p_N(t,y+r \viiva s,x) \le h(r) p_N(t,y+1 \viiva s,x)$ for a summable function $h:\Z \to \R$. 

\medskip \noindent \textbf{Case 1: $r > 1$.} 
Then, $\lfloor t N^2 + y N \rfloor - \lfloor s N^2 + x N \rfloor + rN > 0$, and
\begin{align}
&\quad \, \f{e^{-\theta_Nr} p_N(t,y + r \viiva s, x)}{p_N(t,y + 1 \viiva s,x)} \nonumber \\
&=  \f{e^{-\theta_N r} ((t - s)N^2)^{(r-1)N}}{\prod_{\ell = N+1}^{rN}\bigl((\lfloor t N^2 + y N \rfloor - \lfloor s N^2 + x N \rfloor + \ell \bigr) }  \nonumber \\
&= e^{-\theta_N r}  \prod_{j = 0}^{(r-1)N - 1} \f{(t - s)N^2  }{\lfloor t N^2 + y N \rfloor - \lfloor s N^2 + x N \rfloor + rN - j} \label{big_prod} \\
&\le e^{-\theta_N r}  \prod_{j = 0}^{\lfloor \f{(r-1)N - 1}{2}     \rfloor  }\f{(t - s)N^2  }{(t - s)N^2 -N + rN - j - 1} \label{drop_terms} \\
&\le  e^{-\theta_N r}  \Biggl(\f{(t - s)N^2}{(t - s)N^2 
 + \lceil \f{(r-1)N - 1}{2}     \rceil }\Biggr)^{\lceil \f{(r-1)N - 1}{2}     \rceil} \nonumber \\
 &\le e^{-\theta_N r}  \exp\Biggl(- \f{\lceil \f{(r-1)N - 1}{2}     \rceil^2}{(t - s)N^2 
 + \lceil \f{(r-1)N - 1}{2}     \rceil}  \Biggr)
 \label{r>0bdpre} 
 \end{align}
 To get \eqref{drop_terms}, we have used the assumption that $x,y \in [0,1]$ to bound the denominator, and we discarded several terms. This is justified because \eqref{eq:tibd} implies all terms in the product in \eqref{big_prod} are less than $1$. To get \eqref{r>0bdpre},  we simply used the inequality $1 + x \le e^x$. Next, Observe that, for all $r \ge 2$ and $N \ge 6$,
 \[
\f{1}{8}rN \le  \lceil \f{(r-1)N - 1}{2}     \rceil \le rN.
 \]
 Substituting these bounds in \eqref{r>0bdpre}, along with the assumption $t-s < T$, we obtain the upper bound
\begin{align*} \label{r>0bd}
 e^{-\theta_N r} \exp\biggl(-\f{(rN)^2}{64(TN^2 + rN)}\biggr) = e^{-\theta_N r} \exp\biggl(-\f{r^2}{64(T + \f{r}{N})}\biggr).
\end{align*}
Let  $C_1 = -\inf_N \theta_N$, and choose $C_2 > \max\{6,128 C_1\}$. Then, for all $N \ge C_2$,
\be \label{r>0bd}
\f{e^{-\theta_Nr} p_N(t,y + r \viiva s, x)}{p_N(t,y + 1 \viiva s,x)} \le \exp\biggl(C_1 r - \f{r^2}{64(T + \f{r}{C_2})}\biggr).
\ee
 We claim now that this bound is summable over $r > 1$. Note for $r \ge C_2 T$, we have $T + \f{r}{C_2} \le \f{2r}{C_2}$, and so
\[
\exp\biggl(C_1 r - \f{r^2}{64(T + \f{r}{C_2})}\biggr) \le \exp\biggl(C_1 r - \f{C_2 r}{128}\biggr),
\]
which is summable by assumption on $C_2$.

\medskip \noindent \textbf{Case 2:} $r < -3$.  We may assume that $\lfloor t N^2  + y N\rfloor-\lfloor s N^2  + x N\rfloor + rN \ge 0$; otherwise, the bound is trivial. This makes all terms in the products below positive. In the penultimate step below, we bound the geometric mean by the arithmetic mean, and in the ultimate step we again use $1+x\leq e^x$.
\begin{align}
    &\quad \, \f{e^{-\theta_N r} p_N(t,y + r \viiva s, x)}{p_N(t,y + 1 \viiva s,x)} \nonumber  \\
    &= e^{-\theta_N r}  ((t - s)N^2)^{(r - 1)N} \prod_{j = rN + 1}^N \bigl(\lfloor t N^2 + y N \rfloor - \lfloor s N^2 + x N \rfloor + \ell\bigr) \nonumber\\
    &\le e^{-\theta_N r}((t - s)N^2)^{(r-1)N} \prod_{j = rN + 1}^N ((t - s)N^2 + \ell + N + 1)\nonumber \\
    &\le e^{-\theta_N r} \Bigl(\f{2(t - s)N^2 + (r+3)N + 3  }{2(t - s)N^2}\Bigr)^{(1-r)N}  \nonumber \\
    &\le e^{-\theta_N r} \exp\Bigl(  \f{((r+3)N +3)(1-r)N}{2(t - s)N^2}\Bigr).\nonumber 
\end{align}
Because $r < -3$, for $N > 3$,  this is bounded above by 
\be \label{r<0bd}
e^{-\theta_N r} \exp\Bigl( \f{\bigl((r+3)N + 3\bigr)(1-r)N}{2T N^2}\Bigr) \le C e^{-cr^2}
\ee
for $C,c > 0$ (noting $\theta_N$ is bounded by assumption).
Combining \eqref{r>0bd} and \eqref{r<0bd} completes the proof. 
\end{proof}

We prove two semi-discrete analogues of Lemma \ref{lem:int_sum_fin}; one for the unscaled Poisson kernel, and the other for the scaled Poisson kernel. These are Lemmas \ref{lem:disc_ser_converge} and \ref{lem:JkN_unif_bd} below.  The first lemma is used to prove Item \ref{itm:Ns1} of Proposition \ref{prop:OCY_Chaos}.  Then, in the proof of Proposition \ref{prop:OCY_Chaos}, we perform a change of variable for the scaled object, and will need an $L^2(\Pp)$ bound for the scaled series at the end of the proof of Proposition \ref{prop:OCY_Chaos}. While we could actually derive this latter bound from Lemma \ref{lem:disc_ser_converge} (after making some changes of variables), we need the uniform-in-$N$ bounds from Lemma \ref{lem:JkN_unif_bd} later in the proof of Theorem \ref{L2_conv_main_theorem}. Thus, it is more convenient to prove Lemma \ref{lem:JkN_unif_bd} in this section.   
\begin{lemma} \label{lem:disc_ser_converge}
For $t \ge s, m,n \in \Z$, $\beta > 0$, and $\theta \in \R$,
 \be \label{eq:disc_int}
   \sum_{k = 0}^\infty \beta^{2k} \int_{[0,N]^k_{m,n}} \int_{\Delta^k(t \viiva s)} \Biggl( \prod_{i = 0}^k \sum_{r \in \Z} e^{-\theta r} q(s_{i+1} - s_i, \lfloor x_{i+1} \rfloor - \lfloor x_i \rfloor + r N)\Biggr)^2 \prod_{i = 1}^k ds_i\,dx_i < \infty.
\ee
\end{lemma}
\begin{proof}
Define $\mathfrak q_2(t,y,\ell,r) := q(t,y + \ell)q(t,y + r)$. Calling the $k$-th term in  \eqref{eq:disc_int} $L_k$ we have
\be \label{eq:sum_int}
\begin{aligned}
    L_k&=\beta^{2k}\int_{[0,N]^k_{m,n}} \int_{\Delta^k(t \viiva s)} \sum_{\substack{r_0,\ldots,r_k \in \Z \\ \ell_0,\ldots,\ell_k \in \Z}} \prod_{i=0}^k e^{-\theta (r_i + \ell_i)} \mathfrak q_2(s_{i+1} -s_i, \lfloor x_{i+1} \rfloor - \lfloor x_i \rfloor, \ell_i N,r_i N) \prod_{i = 1}^k ds_i\,dx_i \\
    &= \beta^{2k}\sum_{\substack{ m_1,\ldots,m_k \in \Z_N\\ m_0 = m,m_{k+1} = n}}\int_{\Delta^k(t \viiva s)} \sum_{\substack{r_0,\ldots,r_k \in \Z \\ \ell_0,\ldots,\ell_k \in \Z}} \prod_{i=0}^k e^{-\theta (r_i + \ell_i)} \mathfrak q_2(s_{i+1} -s_i, m_{i+1} - m_i, \ell_i N,r_i N) \prod_{i = 1}^k ds_i.
    \end{aligned}
    \ee
    Now, observe that the summands above are zero if $m_{i+1} - m_i + r_i N < 0$ or $m_{i+1} - m_i + \ell_i N > 0$, 
    and by definition, when $m_{i+1} - m_i + r_i N \ge 0$ and $m_{i+1} - m_i + \ell_i N \ge 0$, we have 
    \[
    \mathfrak q_2(s_{i+1} -s_i, m_{i+1} - m_i, \ell_i N,r_i N) = \f{(s_{i+1} -s_i)^{2(m_{i+1} - m_i) + (\ell_i + r_i)N}}{(m_{i+1} - m_i + r_i N)! (m_{i+1} - m_i + \ell_i N)!},
    \]
    so by Tonelli's theorem and the the computation of a Dirichlet integral, we obtain from \eqref{eq:sum_int} that 
    \begin{align*}
    &L_k=\beta^{2k}\f{(t-s)^k}{e^{2(t-s)}} \sum_{\substack{ m_1,\ldots,m_k \in \Z_N\\ m_0 = m,m_{k+1} = n}}\sum_{\substack{\ell_0,\ldots,\ell_k \in \Z \\
    r_0,\ldots, r_k \in \Z \\
    m_{i+1} - m_i + \ell_i N \ge 0 \\
    m_{i+1} - m_i + r_i N \ge 0}} \f{(t-s)^{2(n-m) + \sum_{i = 0}^k (\ell_i + r_i)N}}{\bigl(2(n-m) + \sum_{i = 0}^k (\ell_i + r_i)N + k\bigr)!} \\
&\qquad\qquad\qquad\qquad\qquad\qquad\qquad\qquad\qquad\qquad\qquad\qquad  \times \prod_{i = 0}^k   \f{e^{-\theta(\ell_i + r_i)}\big(2(m_{i+1} - m_i) + (\ell_i + r_i)N\big)!}{(m_{i+1} - m_i + r_i N)! (m_{i+1} - m_i + \ell_i N)!}. \\
    \end{align*}
 Since the factorial satisfies $(a+b)!\geq a! b!$ for $a,b\in \N$, it follows that
\[
\Bigl(2(n-m) + \sum_{i = 0}^k (\ell_i + r_i)N + k\Bigr)! \ge  k! \Bigl(2(n-m) + \sum_{i = 0}^k (\ell_i + r_i)N\Bigr)! \ge k! \prod_{i = 0}^k \Bigl(2(m_{i+1} - m_i) + (\ell_i + r_i)N\Bigr)!.
\]
From this we obtain the bound 
\begin{align*}
I_k&\le \f{\beta^{2k}(t-s)^{k + 2(n-m)}}{k! e^{2(t-s)}} \sum_{\substack{ m_1,\ldots,m_k \in \Z_N\\ m_0 = m,m_{k+1} = n}}\sum_{\substack{\ell_0,\ldots,\ell_N \in \Z \\
    r_0,\ldots, r_N \in \Z \\
    m_{i+1} - m_i + \ell_i N > 0 \\
    m_{i+1} - m_i + r_i N > 0}} \prod_{i = 0}^k  \f{e^{-\theta(\ell_i + r_i)}(t-s)^{\ell_i + r_i}}{(m_{i+1} - m_i + r_i N)! (m_{i+1} - m_i + \ell_i N)!} \\
    &\le \f{\beta^{2k} (t-s)^{k + 2(n-m)} N^k}{k! e^{2(t-s)}} \sum_{\substack{\ell_0,\ldots,\ell_N \in \Z_{\ge 0} \\ r_0,\ldots,r_N \in \Z_{\ge 0}}  } \prod_{i =0}^k \f{e^{-\theta(\ell_i + r_i)} (t-s)^{\ell_i + r_i}}{(r_i N)! (\ell_i N)!} \\
    &= \f{\beta^{2k}(t-s)^{k + 2(n-m)} N^k}{k! e^{2(t-s)}}  \Bigl(\sum_{r \ge 0} \f{e^{-\theta r}(t-s)^r}{(rN)!}\Bigr)^{2(k+1)}, 
\end{align*}
which   is summable in $k$. Hence, the series in \eqref{eq:disc_int} converges. 
\end{proof}

Next is the second semi-discrete analogue of Lemma \ref{lem:int_sum_fin}. 
\begin{lemma} \label{lem:JkN_unif_bd}
Let $\beta > 0$, and let $X_N,Y_N,S_N,T_N$ be sequences converging to $X,Y \in \R$ and $S < T$, respectively. Then, there exist constants $C,c > 0$ (depending on the sequences $X_N,Y_N,S_N,T_N)$ so that, for all sufficiently large $N$ and all $k \in \N$, 
\[
\beta^{2k} \int_{[0,1]^k_{X_N,Y_N}}  \int_{\Delta^k(T_N \viiva S_N)} \Bigl(\prod_{i = 0}^k \sum_{r \in \Z} e^{-\theta_N r} p_{N}(t_{i+1},y_{i+1} + r, t_i,y_i) \Bigr)^2 \prod_{i = 1}^k dt_i\, dy_i \le  C e^{-c k}.
\]
\end{lemma}
Lemma \ref{lem:JkN_unif_bd} is proved immediately after the next two intermediate lemmas. 
\begin{lemma} \label{lem:p_n_ineq}
Let $(\theta_N)_{N \in \N}$ be a bounded sequence. Then, for $T > 0$, there exists a constant $C > 0$ so that, for all sufficiently large $N$, all $k \in \N$, $t_0 < \cdots < t_{k+1}$ with $t_{k+1} - t_0 < T$, and $y_0,\ldots,y_{k+1} \in [0,1]$,
\begin{multline*}
\prod_{i = 0}^k \Biggl(\sum_{r \in \Z} e^{-\theta_N r} p_N(t_{i+1},y_{i+1} + r \viiva t_i,y_i)\Biggr)^2
\le C^k\!\!\!\!\!\!\!\!\!\! \sum_{\substack{j_0,\ldots,j_k \in \Z, j_{k+1} = 0 \\ j_{i+1} - j_i \in \lzb -3,1 \rzb,\, 0 \le i \le k}} \prod_{i = 0}^k p_N^2\big(t_{i+1},y_{i+1} + j_{i+1} \viiva t_i,y_i + j_i\big). 
\end{multline*}
\end{lemma}
\begin{proof}
Using Lemma \ref{lem:pN_per_bd}, we obtain, for a summable function $h:\Z \to \R$, and a constant $C$ changing from line to line, 
\begin{align*}
    &\quad \, \prod_{i = 0}^k \Biggl(\sum_{r \in \Z} e^{-\theta_N r} p_N(t_{i+1},y_{i+1} + r \viiva t_i,y_i)\Biggr)^2   \\
    &\le \prod_{i = 0}^k \max_{j \in \lzb -3,1 \rzb} p_N^2(t_{i+1},y_{i+1} + j \viiva t_i,y_i)\Bigl(\sum_{r \in \Z} h(r)\Bigr)^2 \\
    &\le C^k \prod_{i = 0}^k \sum_{j \in \lzb -3,1 \rzb} p_N^2(t_{i+1},y_{i+1} + j \viiva t_i,y_i) \\
    &= C^k \sum_{r_0,\ldots,r_k \in \lzb -3,1 \rzb } \prod_{i = 0}^k p_N^2(t_{i+1},y_{i+1} + r_i \viiva t_i,y_i)  \nonumber \\
    &= C^k \sum_{\substack{j_0,\ldots,j_k \in \Z, j_{k+1} = 0 \\ j_{i+1} - j_i \in \lzb -3,1 \rzb,\, 0 \le i \le k}} \prod_{i = 0}^k p_N^2(t_{i+1},y_{i+1} + j_{i+1} \viiva t_i,y_i + j_i).
    \end{align*}
    The last line used the change of variables $r_i = j_{i+1} - j_i$ and the shift invariance of Lemma \ref{lem:intshift}. 
\end{proof}

Before introducing the next lemma, we define the following quantity: for $j,k \in \N$, set
\be \label{Djk}
D_{j,k} = \#\{(j_1,\ldots,j_k) \in \Z: j_{i+1} - j_i \in \lzb -3,1 \rzb \text{ for } 0 \le i \le k, \; j_0 = j,\text{ and }j_{k+1} = 0\}.
\ee
\begin{lemma} \label{lem:Djk_integral}
   Let $X_N,Y_N,S_N,T_N$ be sequences converging to $X,Y,S,T\in \R$, respectively, with $S < T$. Then, there exist sequences $T_N',Y_N'$ converging to $T-S$ and $Y-X$, respectively, so that, for each $N \in \N$,
   \be \label{i1}
   \begin{aligned}
   &\int_{\R^k_{X_N,Y_N}} \int_{\Delta^k(T_N \viiva S_N)} \sum_{\substack{j_0,\ldots,j_k \in \Z, j_{k+1} = 0 \\ j_{i+1} - j_i \in \lzb -3,1 \rzb,\, 0 \le i \le k}}\prod_{i = 0}^k p_{N}^2\big(t_{i+1},y_{i+1} + j_{i+1} \viiva t_i,y_i + j_i\big)  \prod_{i = 1}^k dt_i \,dy_i \\
   &\qquad\qquad\qquad =  \sum_{j \in \lzb -(k+1),3(k+1) \rzb 
 }     D_{j,k} \int_{\R^k_{Y_N' - j}} \int_{\Delta^k(T_N')} \prod_{i = 0}^k p_{N}^2\big(t_{i+1},y_{i+1} \viiva t_i,y_i\big) \prod_{i = 1}^k dt_i \,dy_i.
   \end{aligned}
   \ee
Additionally, the following analogue holds for $\rho$ in place of $p_N$:
   \be \label{i2}
   \begin{aligned}
   \int_{\R^k_{X,Y}} \int_{\Delta^k(T \viiva S)} \sum_{\substack{j_0,\ldots,j_k \in \Z, j_{k+1} = 0 \\ j_{i+1} - j_i \in \lzb -3,1 \rzb,\, 0 \le i \le k}}\prod_{i = 0}^k \rho^2\big(t_{i+1},y_{i+1} + j_{i+1} \viiva t_i,y_i + j_i\big)  \prod_{i = 1}^k dt_i \,dy_i \\
   =  \sum_{j \in \lzb -(k+1),3(k+1) \rzb 
 }     D_{j,k} \int_{\R^k_{Y - X - j}} \int_{\Delta^k(T)} \prod_{i = 0}^k \rho^2\big(t_{i+1},y_{i+1} \viiva t_i,y_i\big) \prod_{i = 1}^k dt_i \,dy_i.
   \end{aligned}
   \ee
\end{lemma} 
\begin{proof}
We prove \eqref{i1}, and \eqref{i2} follows by a symmetric (and simpler) proof. Note that if $j_0,\ldots,j_k \in \Z$, $j_{k+1} = 0$, and $j_{i+1} - j_i \in \lzb -3,1 \rzb$ for $0 \le i \le k$, then we necessarily have $j_0 \in \lzb -(k+1),3(k+1) \rzb$. Using this notation we may write 
\be \label{1242}
\begin{aligned}
&\quad \, \int_{\R^k_{X_N,Y_N}} \int_{\Delta^k(T_N \viiva S_N)} \sum_{\substack{j_0,\ldots,j_k \in \Z, j_{k+1} = 0 \\ j_{i+1} - j_i \in \lzb -3,1 \rzb,\, 0 \le i \le k}}\prod_{i = 0}^k p_{N}^2\big(t_{i+1},y_{i+1} + j_{i+1} \viiva t_i,y_i + j_i\big)  \prod_{i = 1}^k dt_i \,dy_i,\\
&= \sum_{j_0 \in \lzb -(k+1),3(k+1) \rzb} \sum_{\substack{j_1,\ldots,j_k \in \Z, j_{k+1} = 0 \\ j_{i+1} - j_i \in \lzb -3,1 \rzb,\, 0 \le i \le k}} \int_{\R^k_{X_N + j_0,Y_N}} \int_{\Delta^k(T_N \viiva S_N)} \prod_{i = 0}^k p_{N}^2\big(t_{i+1},x_{i+1} \viiva t_i,x_i \big)  \prod_{i = 1}^k dt_i \,dx_i \\
&= \sum_{j \in \lzb -(k+1),3(k+1) \rzb 
 }     D_{j,k} \int_{\R^k_{X_N+j,Y_N}} \int_{\Delta^k(T_N \viiva S_N)} \prod_{i = 0}^k p_{N}^2\big(t_{i+1},x_{i+1} \viiva t_i,x_i\big) \prod_{i = 1}^k dt_i \,dx_i,
 \end{aligned}
 \ee
 where in the second line, we changed the order of the (finite) sum and integral and made the change of variable $x_i = y_i + j_i$. We complete the proof by making an additional change of variable, which we now describe. Choose $W_N$ so that 
 \be \label{WN_choice}
 S_N N^2 + W_N N = \lfloor S_N N^2 + X_N N \rfloor,
 \ee
 and make the change of variable $u_i = t_i - S_N$ and $w_i = x_i - W_N - j$ for $1 \le i \le k$. Define $T_N' = T_N - S_N$, and choose $Y_N' \to Y-X$ so that 
 \be \label{YN_choice}
 T_N' N^2 + Y_N' N = \lfloor T_N N^2 + Y_N N \rfloor - \lfloor S_N N^2 + X_N N \rfloor.
 \ee
  To obtain the right-hand side of \eqref{i1}, we must show that, for $1 \le i \le k$, 
 \be \label{pNtxuw}
 \begin{aligned}
    &p_N(t_{i+1},x_{i+1} \viiva t_i,x_i) = p_N(u_{i+1},w_{i+1} \viiva u_i,w_i), \quad \text{where}\\
    &t_0 = S_N, x_0 = X_N + j,t_{k+1} = T_N, x_{k+1} = Y_N,  \quad \text{and} \quad u_0 = w_0 = 0, u_{k+1} = T_N', w_{k+1} = Y_N' - j.
    \end{aligned}
    \ee
 To see this, first recall from the definition \eqref{pN_def} that 
    \be \label{eq:pN_recall}
    p_N(t_{i+1},x_{i+1} \viiva t_i,x_i) = N e^{-(t_{i+1} - t_i)N^2} \f{\bigl((t_{i+1} - t_i)N^2\bigr)^{(\lfloor t_{i+1} N^2 + x_{i+1} N \rfloor - \lfloor t_i N^2 + x_i N \rfloor)}}{ (\lfloor t_{i+1} N^2 + x_{i+1} N \rfloor - \lfloor t_i N^2 + x_i N \rfloor)!}.
    \ee
Note that $\lfloor a + c\rfloor - \lfloor b + c \rfloor \neq \lfloor a \rfloor - \lfloor b \rfloor$ in general, but equality holds whenever $c \in \Z$. Then, since $s_N N^2 + W_N N  + jN \in \Z$ by \eqref{WN_choice}, for $2 \le i \le k-1$,
    \begin{multline*}
     \lfloor t_{i+1} N^2 + x_{i+1} N \rfloor - \lfloor t_i N^2 + x_i N \rfloor \\
    =  \lfloor u_{i+1} N^2 + w_{i+1} N + S_N N^2 + X_N N + jN \rfloor - \lfloor u_i N^2 + w_i N +  S_N N^2 + X_N N + jN  \rfloor \\
    = \lfloor u_{i+1} N^2 + w_{i+1} N\rfloor - \lfloor u_i N^2 + w_i N   \rfloor,
    \end{multline*}
    and by definition of $Y_N'$ the same holds (excluding the middle equality) for $i = 1$ and $i= k$. Combining this observation with \eqref{eq:pN_recall} gives \eqref{pNtxuw}.
\end{proof}

\begin{proof}[Proof of Lemma \ref{lem:JkN_unif_bd}]
Apply Lemma \ref{lem:p_n_ineq}, then Lemma \ref{lem:Djk_integral} to get
\begin{align}
&\quad \, \beta^{2k} \int_{[0,1]^k_{X_N,Y_N}}  \int_{\Delta^k(T_N \viiva S_N)} \Bigl(\prod_{i = 0}^k \sum_{r \in \Z} e^{-\theta_N r} p_{N}(t_{i+1},y_{i+1} + r, t_i,y_i) \Bigr)^2 \prod_{i = 1}^k dt_i\, dy_i \nonumber  \\
 &\le C^k \sum_{\substack{j_0,\ldots,j_k \in \Z, j_{k+1} = 0 \\ j_{i+1} - j_i \in \lzb -3,1 \rzb}}\int_{\R^k_{X_N,Y_N}} \int_{\Delta^k(T_N \viiva S_N)} \prod_{i= 0}^k p_{N}^2(t_{i+1},y_{i+1} + j_{i+1} \viiva t_i,y_i + j_i) \prod_{i = 1}^k\, dt_i\,dy_i \nonumber \\
 &=   \sum_{j \in \lzb -(k+1),3(k+1) \rzb} C^k D_{j,k}\int_{\R^k_{Y_N' - j}}\int_{\Delta^k(T_N')} \prod_{i= 0}^k p_{N}^2(t_{i+1},x_{i+1}\viiva t_i,x_i) \prod_{i = 1}^k\, dt_i\,dx_i, \label{eq:SumKj}
 \end{align}
for a constant $C > 0$, and $D_{j,k}$ is as defined in \eqref{Djk}, and $T_N',Y_N'$ are sequences converging to $T-S$,$Y-X$, respectively. Observe from the definition of $p_N$ \eqref{pN_def} that the integrals in \eqref{eq:SumKj} are nonzero if and only if $\lfloor T_N' N^2 + Y_N N\rfloor - jN \ge 0$. Hence, some of the $j$ terms in the sum \eqref{eq:SumKj} may be $0$. Assuming $j \in \lzb -(k+1), 3(k+1)\rzb$, we bound the terms of the sum  via two cases.

 \medskip \noindent \textbf{Case 1:}  $j < \f{T_N'}{2}N$. Then,   observing that $D_{j,k} \le 5^k$ for all $j$, 
\begin{align*}
&\quad \,C^k D_{j,k} \int_{\R^k_{Y_N' - j}}\int_{\Delta^k(T_N')} \prod_{i= 0}^k p_{N}^2(t_{i+1},x_{i+1}\viiva t_i,x_i) \prod_{i = 1}^k\, dt_i\,dx_i \\
&\le  \f{C^k  p_N^2(T_N',Y_N' - j)}{{\Gamma((k+1)/2)}}   \sqrt{\f{\lfloor T_N' N^2 + Y_N' N \rfloor - jN}{N^2} }  \Bigl(\f{N^2}{\lfloor T_N' N^2 + Y_N' N \rfloor - jN }\Bigr)^{(k+1)/2} \\
&\le \f{C^k  p_N^2(T_N',Y_N' - j)}{{\Gamma((k+1)/2)}}   \sqrt{\f{\lfloor T_N' N^2 + Y_N' N \rfloor + (k+1)N }{N^2} }  \Biggl(\f{N^2}{\f{T_N'}{2} N^2 + Y_N' N -1} \Biggr)^{(k+1)/2}\\
&\le   \f{C^k  p_N^2(T_N',Y_N' - j)}{{\Gamma((k+1)/2)}} \le \f{C^k}{\Gamma((k+1)/2)} \max_{\ell \in \lzb -3, 1\rzb } p_N^2(T_N',Y_N' + \ell)  \le \f{C^k}{\Gamma((k+1)/2)},
\end{align*}
where the constant $C$ changes from line to line, the first inequality is by \eqref{eq:full_int_bd} of Lemma \ref{lem:pNint_conv}, the second inequality holds by the assumption $j \ge -(k+1)$, the penultimate inequality holds by Lemma \ref{lem:pN_per_bd}, and the last inequality holds by convergence of $p_N^2(T_N',Y_N' + \ell)$ to $\rho(T-S,Y-X +\ell)$ (Lemma \ref{lem:pNto_p}).

\medskip \noindent \textbf{Case 2:} $j \ge \f{T_N'}{2}N$.  In this case, we use the bound \eqref{eq:fib2} of Lemma \ref{lem:pNint_conv}: 
\begin{align*}
    &\quad\, C^k D_{j,k}\int_{\R^k_{Y_N' - j}}\int_{\Delta^k(T_N')} \prod_{i= 0}^k p_{N}^2(t_{i+1},x_{i+1}\viiva t_i,x_i) \prod_{i = 1}^k\, dt_i\,dx_i \\
&\overset{\eqref{eq:fib2}}{\le}   \f{C^k N^{k+1}}{k!} p_N^2(T_N',Y_N'-j) \sqrt{\f{\lfloor T_N' N^2 + Y_N' N - jN \rfloor + 1}{N^2} } \\
&\le  \f{C^k N^{k+1}}{k!} p_N^2(T_N',Y_N'-j) \\
&\le \f{C^k N^{k+1}}{k!} e^{-cj^2} \max_{\ell \in \lzb -3,1 \rzb} p_N^2(T_N',Y_N'+\ell) \qquad \text{(Using Lemma  \ref{lem:pN_per_bd})} \\
& \le  \f{C_1 ^k N^{k+1}}{k!} e^{-c_1N^2} 
\end{align*}
where the constants $c,C$ change from line to line, but the constants $C_1$ and $c_1$ are fixed for the remainder of the proof. The last line above follows by the convergence of $p_N(T_N',Y_N'+\ell)$ to $\rho(T,Y+\ell)$. We consider two further subcases, where $C_2$ is a fixed constant chosen to be larger than $C_1 e$.

\medskip \noindent \textbf{Case 2.1:} $N \ge \f{k}{C_2}$. We obtain the upper bound
\begin{align*}
&\quad \, \f{C_1^k N^{k+1}}{k!} e^{-c_1 N^2}  \le \f{ C_1^k N^{k+1} }{k!} e^{-c_1Nk/C_2}  \le  \f{C_1^k }{k!}  \Bigl(\f{N^2}{e^{c_1N/C_2}}\Bigr)^k \le \f{\wt C^k}{k!} ,
\end{align*}
for a constant $\wt C$ depending on $C_1,c_1$, and $C_2$.

\medskip \noindent \textbf{Case 2.2:} $N < \f{k}{C_2}$. Then, we have the bound 
\begin{align*}
\f{C_1^k N^{k+1}}{k!} e^{-c_1 N^2} \le \f{(C_1k)^k}{k! C_2^k} N e^{-c_1 N^2} \le \wt C e^{-\wt c k}
\end{align*}
 and $\wt C$ and $\wt c$ are positive constants. The last inequality holds by Stirling's formula since we took $C_2 > C_1e$.

In each case, we have shown that, for each $C > 0$, there exist constants $\wt C,\wt c$ so that, for $k,N \ge 1$,
\[
C^k D_{j,k} \int_{\R^k_{Y_N' - j}}\int_{\Delta^k(T_N')} \prod_{i= 0}^k p_{N}^2(t_{i+1},x_{i+1}\viiva t_i,x_i) \prod_{i = 1}^k\, dt_i\,dx_i \le \wt C e^{-\wt ck}.
\]
Then, by \eqref{eq:SumKj}, we have
\[
\beta^{2k} \int_{[0,1]^k_{X_N,Y_N}}  \int_{\Delta^k(T_N \viiva S_N)} \Bigl(\prod_{i = 0}^k \sum_{r \in \Z} e^{-\theta_N r} p_{N}(t_{i+1},y_{i+1} + r, t_i,y_i) \Bigr)^2 \prod_{i = 1}^k dt_i\, dy_i \le \sum_{j \in \lzb -(k+1),3(k+1) \rzb} \wt C e^{-\wt ck}.
\]
As there are $O(k)$ terms in the sum on the right-hand side, this is bounded by $Ce^{-ck}$, for new $C,c > 0$. 
\end{proof}

Lastly, we record the semi-discrete analogue of Lemma \ref{lem:ser_tail_to0}.
\begin{lemma} \label{lem:pN_ser_tail_to0}
For sufficiently large $N$, all $k \in \N$, $t > s$ and $x,y,\theta \in \R$, the following holds
\[
\lim_{P \to \infty}  \int_{[0,1]^k_{x ,y}}\int_{\Delta^k(t \viiva s)} \Biggl(\sum_{\substack{r_0,\ldots,r_k \in \Z \\ |r_0 + \cdots + r_k| \ge P}}  \prod_{i = 0}^k e^{-\theta r_i} p_N(t_{i + 1},x_{i+1} + r_i \viiva t_i,x_i) \Biggr)^2 \prod_{i = 1}^k \,dt_i\,dx_i = 0.
\]
\end{lemma}
\begin{proof}
This follows an identical proof to Lemma \ref{lem:ser_tail_to0}, replacing the use of Lemma \ref{lem:rhobd} with Lemma \ref{lem:pN_per_bd} and replacing Lemma \ref{lem:rho_int_comp} with \ref{eq:full_int_bd} of Lemma \ref{lem:pNint_conv}.
\end{proof}

We are now prepared to prove Proposition \ref{prop:OCY_Chaos}. 

\begin{proof}[Proof of Proposition \ref{prop:OCY_Chaos}]
\textbf{Item \ref{itm:Ns1}:} This follows an identical proof to Proposition \ref{prop:Z_conv}\ref{itm:s1}, using the It\^o isometry from Lemma \ref{lem:Ito_isometry} applied to \eqref{eq:Y_chaos}, and replacing the use of Lemma \ref{lem:int_sum_fin} with Lemma \ref{lem:disc_ser_converge}.  

\medskip \noindent \textbf{Item \ref{itm:Ns1}:}
By the chaos series representation for $Y_\gamma$ in \eqref{eq:Y_chaos}, along with  \eqref{ZperY},
\begin{align} \label{eq:Sc_Ser}
\scOCYp(t,y \viiva s,x;\theta) = \sum_{j \in \Z} e^{\theta j} \sum_{k = 0}^\infty \beta^k N^{1-\f{k}{2}} I_k^N(tN^2,\lfloor tN^2 + yN \rfloor \viiva sN^2,\lfloor sN^2 + xN \rfloor + jN).
\end{align}
By the white-noise integral representation of $I_k$ in \eqref{Ik_wNint}, $I_k^N(tN^2, \lfloor tN^2 + yN \rfloor \viiva sN^2, \lfloor sN^2 + xN \rfloor)$ equals
\be \label{eq:long_int}
\int_{\R^k_{\lfloor sN^2 + xN \rfloor, \lfloor tN^2 + yN \rfloor}}  \int_{\Delta^k(t N^2 \viiva s N^2)}  \prod_{i = 0}^k q(s_{i+1} - s_i, \lfloor x_{i+1} \rfloor - \lfloor x_i \rfloor) \prod_{i = 1}^k \xi_{N}(ds_i,d x_i),
\ee
where we recall that $\xi_{N}$ is a white noise on $\R \times [0,N]$ that has been periodically extended by shifts of $N$ units of space. 
The rest of the proof is to show that this series expression \eqref{eq:Sc_Ser} may be written as \eqref{JkNdef}. For this, we need to make a change of variables in the integral \eqref{eq:long_int} and then justify changes of the order of summation and integration in \eqref{eq:Sc_Ser}.

Consider the change of variables $t_i = \f{s_i}{N^2}$ and $y_i = \f{x_i - s_i}{N}$ in the integral \eqref{eq:long_int} so that $s_i = t_i N^2$ and $x_i = t_i N^2 + y_i N$. This is a composition of two changes of variable, namely $w_i = x_i - s_i, u_i = s_i$, followed by $y_i = \f{w_i}{N}$ and $t_i = \f{u_i}{N^2}$. By shear invariance and scaling properties of white noise (Lemma \ref{lem:scaling_relations}) we have
\be \label{noise_shift}
\xi_{N}(ds_i,dx_i) = N^{\f{3}{2}}\xi(dt_i,dy_i),
\ee
where $\xi$ is a new noise that is a space-time white noise on $\R \times [0,1]$, and periodically extended by shifts of one unit of space, as in Section \ref{sec:white_noise}. The restriction of $\xi$ to $\R \times [0,1]$, denoted $\xi_\T$ is the white noise used in the coupling of the proposition.  To be precise, the coupling to the process $\scOCYp$ comes by reversing this transformation, building the Brownian motions from $\xi_{N\T}$ by \eqref{eq:BM_coupling}, then transforming back. 

By \eqref{noise_shift}, we have
\be \label{eq:white}
\begin{aligned}
    &\quad \, I_k\big(tN^2, \lfloor tN^2 + yN \rfloor \viiva sN^2 , \lfloor sN^2 + xN \rfloor\big) \\
    &= N^{\f{3k}{2}}\int_{\R^k_{x,y}} \int_{\Delta^k(t \viiva s)} \prod_{i = 0}^k q\big((t_{i+1} - t_i)N^2, \lfloor  t_{i+1} N^2 + y_{i+1} N \rfloor - \lfloor t_i N^2  + y_i N \rfloor\big) \prod_{i = 1}^k \xi(dt_i, dy_i)\\
    &= N^{\f{k}{2}-1} \int_{\R^k_{x,y}} \int_{\Delta^k(t \viiva s)} \prod_{i = 0}^k p_N\big(t_{i+1},y_{i+1} \viiva t_i,y_i\big) \prod_{i = 1}^k \xi(dt_i, dy_i) \\
    &= N^{\f{k}{2}-1} \int_{[0,1]^k_{x,y}} \int_{\Delta^k(t \viiva s)} \sum_{\substack{j_1,\ldots,j_k \in \Z, \\j_0 = j_{k+1} = 0}} \prod_{i = 0}^k p_N\big(t_{i+1},y_{i+1} + j_{i+1} \viiva t_i,y_i + j_i\big) \prod_{i = 1}^k \xi_\T(dt_i, dy_i),
\end{aligned}
\ee
where the second equality follows from the definition of $p_N$ \eqref{pN_def}, and the last equality follows from the definition of integrals over $\xi$ \eqref{eq:white}. Substituting this into \eqref{eq:Sc_Ser} we obtain that $\scOCYp(t,y \viiva s,x;\theta)$ equals
\be \label{888}
\sum_{j \in \Z} e^{\theta j} \sum_{k = 0}^\infty \beta^k  \int_{[0,1]^k_{x,y}} \int_{\Delta^k(t \viiva s)} \sum_{\substack{j_1,\ldots,j_k \in \Z, \\j_0 = j_{k+1} = 0}} \prod_{i = 0}^k p_N\big(t_{i+1},y_{i+1} + j_{i+1} \viiva t_i,y_i + j_i\big) \prod_{i = 1}^k \xi_\T(dt_i, dy_i),
\ee
and we recall from earlier in the proof that the sum over $k$ converges in $L^2(\Pp)$, and the the sum over $j$ converges almost surely by Lemma \ref{lem:OCYp_finite}. Recalling \eqref{JkNdef}, the claim of the proposition is that the expression in \eqref{888} is equal to 
\be \label{889}
 \sum_{k = 0}^\infty \beta^k \int_{[0,1]^k_{x,y}} \int_{\Delta^k(t \viiva s)} \prod_{i = 0}^k \sum_{r \in \Z} e^{-\theta r} p_N(t_{i+1},y_{i+1} + r \viiva t_i,y_i) \prod_{i = 1}^k \xi_{\T}(dt_i, dy_i).
\ee
We first note that by the It\^o isometry and Lemma \ref{lem:JkN_unif_bd} that the series in \eqref{889} converges in $L^2(\Pp)$ for each $N$. Expand the product in \eqref{889}, perform the change of variables $r_i = j_{i+1} - j_i, j_{k+1} = 0$, and use the shift invariance of Lemma \ref{lem:intshift}  to obtain that the expression in \eqref{889} equals
\begin{align*}
 &\quad \, \sum_{k = 0}^\infty \beta^k \int_{[0,1]^k_{x,y}} \int_{\Delta^k(t \viiva s)} \sum_{r_0,\ldots,r_k \in \Z}\prod_{i = 0}^k  e^{-\theta r_i} p_N(t_{i+1},y_{i+1} + r_i \viiva t_i,y_i) \prod_{i = 1}^k \xi_{\T}(dt_i, dy_i) \\
 &= \sum_{k = 0}^\infty \beta^k \int_{[0,1]^k_{x,y}} \int_{\Delta^k(t \viiva s)} \sum_{\substack{ j_0,\ldots,j_k \in \Z \\ j_{k+1} = 0}} e^{\theta j_0}\prod_{i = 0}^k   p_N(t_{i+1},y_{i+1} + j_{i+1} \viiva t_i,y_i + j_i) \prod_{i = 1}^k \xi_{\T}(dt_i, dy_i).
\end{align*}
Hence, it remains to show that, in the expression \eqref{888}, the sum over $j \in \Z$ may be brought inside the sum over $k$, then inside the integral.  This follows the exact proof as Proposition \ref{prop:Z_conv}\ref{itm:s2}, replacing the use of Lemmas \ref{lem:int_sum_fin} and \ref{lem:ser_tail_to0} with their semi-discrete analogues, Lemmas \ref{lem:JkN_unif_bd} and \ref{lem:pN_ser_tail_to0}, respectively. 
\end{proof}

\subsection{Proof of Theorem \ref{L2_conv_main_theorem}: $L^2$ convergence to the SHE} \label{sec:L2_conv_final} This section completes the proof of Theorem \ref{L2_conv_main_theorem}. Propositions \ref{prop:Z_conv}\ref{itm:s2} and \ref{prop:OCY_Chaos}\ref{itm:Ns2} establish the chaos series representations for the limiting and prelimiting objects, respectively.  The first ingredient of the proof is to show the $L^2(\Pp)$ convergence of each term in the chaos series.  
The proof of this result relies on the following:
\begin{lemma}[Generalized Dominated Convergence Theorem]\label{gdct}
On a general measure space, if $(f_n)_{n\in \N}$, $f$, $(g_n)_{n\in \N}$ and $g$ are real valued functions that satisfy the conditions that $f_N \to f$ a.e., $|f_N| \le g_N\to g$ a.e., and $\int g_N \to \int g < \infty$, then $\int f_N \to \int f$.
\end{lemma}

\begin{proposition} \label{prop:ChaosL2conv}
For each $k \in \Z_{\ge 0}$, $\theta\in \R$, $X,Y \in [0,1]$, and real $S < T$, whenever $\theta_N,X_N,Y_N,S_N,T_N$ are sequences converging to $\theta,X,Y,S,T$, respectively, we have the convergence 
\[
\lim_{N \to \infty} \Ee\Big[\big(J_k^{N}(T_N,Y_N \viiva S_N,X_N;\theta_N) - J_k(T,Y \viiva S,X;\theta )\big)^2\Big] = 0.
\]
\end{proposition}
\begin{proof}
Recall $J_k$ and $J_k^N$ from \eqref{Jk_series} and \eqref{JkNdef}, respectively.  By the It\^o isometry (Lemma \ref{lem:Ito_isometry}),  
    \be \label{L2dist}
\begin{aligned} 
    &\Ee\Big[\big(J_k^{N}(T_N,Y_N \viiva S_N,X_N;\theta_N) - J_k(T,Y \viiva S,X;\theta )\big)^2\Big] \\
    &= \beta^{2k} \int_{[0,1]^k} \int_{\R^k} \Biggl[\ind\bigl\{[0,1]_{X_N,Y_N}^k \times \Delta^k(T_N \viiva S_N)\bigr\}\prod_{i = 0}^k \sum_{r \in \Z} e^{-\theta_N r} p_N(t_{i+1},y_{i+1} + r \viiva t_i,y_i) \\
    &\qquad\qquad\qquad\qquad - \ind\bigl\{[0,1]_{X,Y}^k \times \Delta^k(T \viiva S)\bigr\}\prod_{i = 0}^k \sum_{r \in \Z} e^{-\theta r} \rho(t_{i+1} - t_i,y_{i+1} - y_i + r)  \Biggr]^2 \prod_{i = 1}^k dt_i\,dy_i,
\end{aligned}
\ee
    where we have used the shorthand notation
    \[
\ind\bigl\{([0,1]_{X,Y}^k \times \Delta^k(T \viiva S)\bigr\} = \ind\bigl\{(y_0,y_1,\ldots,y_{k},y_{k+1}) \in [0,1]^k_{X,Y}, (t_0,\ldots,t_{k+1}) \in \Delta^k(T \viiva S)\bigr\}.
\]
Observe our convention under this notation that the endpoints $t_0,t_{k+1},y_0,y_{k+1}$ differ between the two terms in \eqref{L2dist}.
This proof of convergence follows in two steps:
\begin{enumerate} [label=\textup{(\roman*)}]
    \item \label{step1} Show that  we have the Lebesgue-a.e. convergence 
    \be \label{sum_conv}
    \begin{aligned}
    &\ind\bigl\{[0,1]_{X_N,Y_N}^k \times \Delta^k(T_N \viiva S_N)\bigr\} \prod_{i = 0}^k \sum_{r \in \Z} e^{-\theta_N r} p_N(t_{i+1},y_{i+1} + r \viiva t_i,y_i)  \\
    &\qquad\qquad\qquad\qquad \to \ind\bigl\{[0,1]_{X,Y}^k \times \Delta^k(T \viiva S)\bigr\}\prod_{i = 0}^k \sum_{r \in \Z} e^{-\theta r} \rho(t_{i+1} - t_i,y_{i+1} - y_i + r).
    \end{aligned}
    \ee
    \item \label{step2} 
    Apply the generalized dominated convergence theorem to show convergence of the integral in \eqref{L2dist}. 
\end{enumerate}

 \medskip \noindent \textbf{Step \ref{step1}:}  Since the products in \eqref{sum_conv} are finite, it suffices to show that, for $1 \le i \le k$, we have the Lebesgue-a.e. convergence  
\be \label{eq:es_conv}
 \begin{aligned}
 & \sum_{r \in \Z} e^{-\theta_N r} p_N(t_{i+1},y_{i+1} + r \viiva t_i,y_i)\ind\bigl\{[0,1]_{X_N,Y_N}^k \times \Delta^k(T_N \viiva S_N)\bigr\}  \\
    &\qquad\qquad\qquad\qquad \to \sum_{r \in \Z} e^{-\theta r} \rho(t_{i+1} - t_i,y_{i+1} - y_i + r)\ind\bigl\{[0,1]_{X,Y}^k \times \Delta^k(T \viiva S)\bigr\}.
    \end{aligned}
\ee
Specifically, we will show convergence everywhere except at the boundary of the set $\{[0,1]_{X,Y}^k \times \Delta^k(T \viiva S)\bigr\}$.

 By Lemma \ref{lem:pNto_p}, we have that  $p_N(t_{i+1},y_{i+1} + r \viiva t_i,y_i) \to \rho(t_{i+1} - t_i, y_{i+1} - y_i + r)$ pointwise. Since $T_N - S_N \to T-S$, the sequence $T_N - S_N$ is bounded, and by Lemma \ref{lem:pN_per_bd}, there exists a function $h:\Z \to \R$ (depending only on $\theta ,T-S$) that is summable over $r \in \Z$ so that, for sufficiently large $N$, $0 < t_{i+1} - t_i < T_N - S_N$, and $y_i,y_{i+1} \in [0,1]$,
\be \label{eq:maxbd}
\begin{aligned}
&\quad \, e^{-\theta_N r} p_N(t_{i+1},y_{i+1} + r \viiva t_i,y_i) \le h(r) \max_{j \in \lzb -3,1 \rzb} p_N(t_{i+1},y_{i+1} + j \viiva t_i,y_i) 
\end{aligned}
\ee   
 By the pointwise convergence of $p_N$ to $\rho$, the term $\max_{j \in \lzb -3,1 \rzb} p_N(t_{i+1},y_{i+1} + j \viiva t_i,y_i)$ is bounded in $N$. Then, for each choice of $t_{i},t_{i+1},y_i,y_{i+1}$, the summand on the left-hand side in \eqref{eq:es_conv} is bounded by $C h(r)$ for a constant $C$, and the dominated convergence theorem (applied to the sum) completes the proof of \eqref{eq:es_conv}.

\medskip \noindent \textbf{Step \ref{step2}:} We use the generalized dominated convergence theorem, Lemma \ref{gdct}. The measure space is $\R^k\times \R^k$ with Lebesgue measure, and we choose 
\be \label{fN}
\begin{aligned}
f_N(y_1,\ldots, y_k,t_1,\ldots ,t_k)&=\Biggl[\ind\bigl\{[0,1]_{X_N,Y_N}^k \times \Delta^k(T_N \viiva S_N)\bigr\}\prod_{i = 0}^k \sum_{r \in \Z} e^{-\theta_N r} p_N(t_{i+1},y_{i+1} + r \viiva t_i,y_i) \\
    &\qquad\quad - \ind\bigl\{[0,1]_{X,Y}^k \times \Delta^k(T \viiva S)\bigr\}\prod_{i = 0}^k \sum_{r \in \Z} e^{-\theta r} \rho(t_{i+1} - t_i,y_{i+1} - y_i + r)  \Biggr]^2.
    \end{aligned}
\ee
By Step \ref{step1}, $f_N$ converges Lebesgue a.e. to $0$ so set $f\equiv 0$ in the application of Lemma \eqref{gdct}. We wish to show that the integral of $f_N$ converges to $0$, so we identify an appropriate dominating function $g_N$. 

Using Lemma \ref{lem:p_n_ineq}, our choice for $g_N$ in this context is 
\be \label{gn}
\begin{aligned}
&C^k \Biggl[\sum_{\substack{j_0,\ldots,j_k \in \Z, j_{k+1} = 0 \\ j_{i+1} - j_i \in \lzb-3,1\rzb,\, 0 \le i \le k}} \ind\bigl\{\R_{X_N,Y_N}^k \times \Delta^k(T_N \viiva S_N)\bigr\}\prod_{i = 0}^k p_N^2(t_{i+1},y_{i+1} + j_{i+1} \viiva t_i,y_i + j_i)  \\
&\qquad\qquad\qquad\qquad\qquad\qquad+ \ind\bigl\{[0,1]_{X,Y}^k \times \Delta^k(T \viiva S)\bigr\}\prod_{i = 0}^k \Biggl(\sum_{r \in \Z} e^{-\theta r} \rho(t_{i+1} - t_i,y_{i+1} - y_i)\Biggr)^2  \Biggr] ,
\end{aligned}
\ee
for an appropriate constant $C$. Note that in the first term in \eqref{gn}, we have replaced the indicator $\bigl\{[0,1]_{X_N,Y_N}^k \times \Delta^k(T_N \viiva S_N)\bigr\}$ with $\bigl\{\R_{X_N,Y_N}^k \times \Delta^k(T_N \viiva S_N)\bigr\}$  at the cost of an inequality. This will make it possible to explicitly compute the integral. 
Since the term in the second line of \eqref{gn} does not depend on $N$ and is integrable (which follows by Lemma \ref{lem:int_sum_fin}), it suffices to show that the term on the first line converges pointwise to some function, and that the integrals converge as well. 
Since the sum in \eqref{gn} is finite, Lemma \ref{lem:pNto_p}  shows that we have the Lebesgue a.e.  convergence of \eqref{gn} to the function $g$ defined as
\be \label{pnpt}
\begin{aligned}
&C^k \Biggl[\sum_{\substack{j_0,\ldots,j_k \in \Z, j_{k+1} = 0 \\ j_{i+1} - j_i \in \lzb -3,1 \rzb,\, 0 \le i \le k}} \ind\bigl\{\R_{X,Y}^k \times \Delta^k(T \viiva S)\bigr\}\prod_{i = 0}^k \rho^2(t_{i+1} - t_i,y_{i+1} + j_{i+1} - y_i - j_i), \\
&\qquad\qquad\qquad\qquad\qquad\qquad+ \ind\bigl\{[0,1]_{X,Y}^k \times \Delta^k(T \viiva S)\bigr\}\prod_{i = 0}^k \Biggl(\sum_{r \in \Z} e^{-\theta r} \rho(t_{i+1} - t_i,y_{i+1} - y_i)\Biggr)^2  \Biggr] ,
\end{aligned}
\ee
where $C$ in \eqref{gn} and \eqref{pnpt} match. It remains to show that $\int g_n\to \int g < \infty$. By \eqref{i1} of Lemma \ref{lem:Djk_integral}, 
  \[
   \begin{aligned}
   &\int_{\R^k_{X_N,Y_N}} \int_{\Delta^k(T_N \viiva S_N)} \sum_{\substack{j_0,\ldots,j_k \in \Z, j_{k+1} = 0 \\ j_{i+1} - j_i \in \lzb -3,1 \rzb,\, 0 \le i \le k}}\prod_{i = 0}^k p_{N}^2\big(t_{i+1},y_{i+1} + j_{i+1} \viiva t_i,y_i + j_i\big)  \prod_{i = 1}^k dt_i \,dy_i \\
   &\qquad\qquad\qquad =  \sum_{j \in \lzb -(k+1),3(k+1) \rzb 
 }     D_{j,k} \int_{\R^k_{Y_N' - j}} \int_{\Delta^k(T_N')} \prod_{i = 0}^k p_{N}^2\big(t_{i+1},y_{i+1} \viiva t_i,y_i\big) \prod_{i = 1}^k dt_i \,dy_i.
   \end{aligned}
   \]
where $D_{j,k}$ is defined in \eqref{Djk}. By Lemma \ref{lem:pNint_conv}, this converges, as $N \to \infty$, to 
\[
\sum_{j \in \lzb -(k+1),3(k+1) \rzb 
 }     D_{j,k} \int_{\R^k_{Y - X -j}} \int_{\Delta^k(T-S)} \prod_{i = 0}^k \rho^2(t_{i+1} - t_i,y_{i+1} - y_i) \prod_{i = 1}^k dt_i \,dy_i,
\]
and by \eqref{i2} in Lemma \ref{lem:Djk_integral},  this is equal to 
\[
\int_{\R^k_{X,Y}} \int_{\Delta^k(T \viiva S)} \sum_{\substack{j_0,\ldots,j_k \in \Z, j_{k+1} = 0 \\ j_{i+1} - j_i \in \lzb -3,1 \rzb}} \prod_{i = 0}^k \rho^2(t_{i+1} - t_i,y_{i+1} - y_i + j_{i+1})  \prod_{i = 1}^k dt_i \,dy_i.
\] 
By Lemma \ref{lem:rho_int_comp}, this integral is finite.
We have thus shown convergence of the integrals of the functions in \eqref{gn} to the integral of \eqref{pnpt}, completing the proof.
\end{proof}

\begin{proof}[Proof of Theorem \ref{L2_conv_main_theorem}]
Recall that the statement to be proved is \begin{align} \label{L2dyadic_conv}
\lim_{N \to \infty}\Ee\Bigl[\bigl(\scOCYp(T_N,Y_N \viiva S_N,X_N;\theta_N) -  Z_\beta^{\text{per}}(T,Y \viiva S,X;\theta)\bigr)^2\Bigr] = 0.
\end{align}
Using the chaos expansions from \eqref{Jk_series} and Proposition \ref{prop:OCY_Chaos}, for any integer $K \ge 0$,
\be \label{eq:Chaos_triang}
\begin{aligned}
&\quad \, \|\scOCYp(T_N,Y_N \viiva S_N,X_N;\theta_N) - Z_\beta^{\text{per}}(T,Y \viiva S,X;\theta)\|_{L^2(\Pp)} \\
  &\le \sum_{k = 0}^K \beta^{k} \|J_k^{N}(T_N,Y_N\viiva S_N,X_N;\theta_N) - J_k(T,Y \viiva S,X;\theta) \|_{L^2(\Pp)} \\
  &\qquad\qquad +  \Biggl\| \sum_{k = K+1}^\infty \beta^k J_k^{N}(T_N,Y_N \viiva S_N,X_N;\theta_N) 
 \Biggr \|_{L^2(\Pp)} + \Biggl\|\sum_{k = K+1}^\infty \beta^k J_k(T,Y \viiva S,X;\theta) 
  \Biggr\|_{L^2(\Pp)}.
\end{aligned}
\ee
There are three terms on the right-hand side of \eqref{eq:Chaos_triang}, and we analyze them separately below.

By the It\^o isometry (Lemma \ref{lem:Ito_isometry}), 
\begin{align*}
     &\quad \, \Biggl\| \sum_{k = K+1}^\infty \beta^k J_k^{N}(T_N,Y_N \viiva S_N,X_N;\theta_N) 
 \Biggr \|_{L^2(\Pp)}^2  = \sum_{k = K + 1}^\infty \beta^{2k} \|J_k^{N}(T_N,Y_N \viiva S_N,X_N;\theta_N)\|_{L^2(\Pp)}^2 \nonumber \\
 &=\sum_{k = K+1}^\infty \beta^{2k} \int_{[0,1]^k_{X_N,Y_N}}  \int_{\Delta^k(T_N \viiva S_N)} \Bigl(\prod_{i = 0}^k \sum_{r \in \Z} e^{-\theta_N r} p_{N}(t_{i+1},y_{i+1} + r, t_i,y_i) \Bigr)^2 \prod_{i = 1}^k dt_i\, dy_i \,\, \le \sum_{k = K+1}^\infty C e^{-ck},
\end{align*}
for constants $C,c > 0$, where the last inequality is from Lemma \ref{lem:JkN_unif_bd}. Since $\sum_{k = 0}^\infty \beta^k J_k(T,Y \viiva S,X;\theta)$ is convergent in $L^2(\Pp)$ (Proposition \ref{prop:Z_conv}), for any $\ve > 0$ we may choose $K$ large enough so that, for all $N \in \N$,
\be \label{WNbd1}
\Biggl\| \sum_{k = K+1}^\infty \beta^k J_k^{N}(T_N,Y_N \viiva S_N,X_N;\theta_N) 
 \Biggr \|_{L^2(\Pp)} \le \f{\ve}{3},\qquad\text{and}\qquad \Biggl\|\sum_{k = K+1}^\infty \beta^k J_k(T,Y \viiva S,X;\theta) 
  \Biggr\|_{L^2(\Pp)} \le \f{\ve}{3}.
\ee
 For this choice of $K$, Proposition \ref{prop:ChaosL2conv} implies that for all sufficiently large $N$,
\[
\beta^k \|J_k^{N}(T_N,Y_N\viiva S_N,X_N;\theta_N) - J_k(T,Y \viiva S,X;\theta) \|_{L^2(\Pp)} \le \f{\ve}{3(K+1)}, \text{ for } 0 \le k \le K+1.
\]
Substituting this bound along with \eqref{WNbd1} into \eqref{eq:Chaos_triang} completes the proof.  
\end{proof}

\subsection{Solving the SHE and convergence of the model with initial data}\label{sec.convergeinitial}

For general initial data $F:[0,1] \to \R_{>0}$, we extend $F$ to a function $\R \to \R_{>0}$ by the condition $\f{F(x+1)}{F(x)} = \f{F(1)}{F(0)}$ for all $x \in \R$. For $t > s$ and $y \in \R$, define 
\[
Z_\beta(t,y \viiva s, F) = \int_\R Z_\beta(t,y \viiva s,x)F(x)\,dx.
\]
We now state the following key results on the solution to the stochastic heat equation. 
\begin{proposition} \label{prop:solve_SHE}
For some $\theta > 0$, let $F:[0,1] \to \R_{>0}$ be a continuous function satisfying $\f{F(1)}{F(0)} = e^\theta$. Then, with probability one:
\begin{enumerate} [label=\textup{(\roman*)}]
\item \label{itm:Z_conv} For all real $t > s$ and $y \in \R$,
\[
Z_\beta(t,y \viiva s,F) = \int_0^1 Z_\beta^{\text{per}}(t,y \viiva s,x;\theta)F(x)\,dx.
\]
\item \label{itm:SHE_soln} For all real $t > s$ and $y \in \R$, $Z_\beta(t,y \viiva s,F)\in (0,\infty)$ and $(t,y) \mapsto Z_\beta(t,y \viiva s,F)$, is the unique mild solution to \eqref{eq:SHE} with initial data $F$ started at time $s$. 
\item \label{itm:per_pres}  For all real $t > s$ and $y \in \R$, $\f{Z_\beta(t,y+1 \viiva s,F) }{Z_\beta(t,y \viiva s,F)} = e^\theta$.
\end{enumerate}
\end{proposition} 
\begin{proof}
\medskip \noindent \textbf{Item \ref{itm:Z_conv}:} 
Note that $\f{F(x+j)}{F(x)} = e^{\theta j}$ for $x \in \R$ and $j \in \Z$. Then, we may write 
\begin{align*}
Z_\beta(t,y \viiva s, F) &= \sum_{j \in \Z} \int_{j}^{j+1} Z_\beta(t,y \viiva s,x) F(x) \,dx \\
&= \sum_{j \in \Z} \int_0^1 Z_\beta(t,y \viiva s,x+j) F(x+j)\,dx   = \int_0^1 \Bigl(\sum_{j \in \Z} e^{\theta j} Z_\beta(t,y \viiva s,x+j)\Bigr) F(x)\,dx \\
&= \int_0^1 \OCYp_\beta(t,y \viiva s,x)F(x)\,dx,
\end{align*}
where the last line follows by definition of $\OCYp_\beta$ in \eqref{eq:Zper_def}. The interchange of sum and integral is justified by Tonelli's theorem because $Z_\beta(t,y\viiva s,x)$ and $F(x)$ are nonnegative. The nonnegativity of $Z_\beta$  can be seen, for example, by adapting the proof of Theorem \ref{L2_conv_main_theorem} to realize $Z_\beta(t,y \viiva s,x)$ is a scaling limit of $\OCY_\beta(t,y \viiva s,x)$ which is non-negative. 

\medskip \noindent \textbf{Item \ref{itm:SHE_soln}}: 
By the definition of $Z_\beta(t,y\viiva s,x)$ through the chaos expansion, see the beginning of Section~\ref{sec.l2converge}, it is straightforward to check that $Z_\beta(t,y\viiva s,F)$ satisfies the mild formulation 
\[
Z_\beta(t,y\viiva s,F)=\int_{\R} \rho(t-s,y-x)F(x)dx+\beta\int_s^t\int_{\R} \rho(t-u,y-x)Z_\beta(u,x\viiva s,F) \xi(du,dx).
\]
By a standard argument which  can be found e.g. in \cite[Chapter 3]{walsh1986introduction}, the mild solution is unique. For the proof of the strict positivity, we need the negative moment estimates combined with a standard continuity argument. For the negative moment bound, one can either follow the argument in \cite{Mueller91} or \cite[Section 4]{HL22}. Note that in either case we can  first perform a shear transformation as in the proof of Lemma~\ref{l.equallaw} below to reduce the problem to the case of $\theta=0$.

\medskip \noindent \textbf{Item \ref{itm:per_pres}:} By the periodicity of $\xi$, we have that $Z_\beta(t,y + n \viiva s,x+n) = Z_\beta(t,y \viiva s,x)$ for $n \in \Z$. Then,
\begin{align*}
\f{Z_\beta(t,y + 1 \viiva s,F)}{Z_\beta(t,y \viiva s,F)} &= \f{\int_\R Z_\beta(t,y + 1 \viiva s,x) F(x)\,dx }{\int_\R Z_\beta(t,y \viiva s,x)F(x)\,dx}  = \f{\int_\R Z_\beta(t,y \viiva s,x -1) F(x)\,dx }{\int_\R Z_\beta(t,y \viiva s,x)F(x)\,dx} \\
&= \f{\int_\R Z_\beta(t,y \viiva s,x) F(x + 1)\,dx }{\int_\R Z_\beta(t,y \viiva s,x)F(x)\,dx} 
= \f{\int_\R Z_\beta(t,y \viiva s,x) e^\theta F(x)\,dx }{\int_\R Z_\beta(t,y \viiva s,x)F(x)\,dx} = e^\theta. \qedhere
\end{align*}
\end{proof}

\begin{theorem} \label{thm:SHE_convergence}
Let $\beta > 0$ and $\theta \in \R$. Let $\theta_N$ be a sequence with $\theta_N \to \theta$, and let $\beta_N = \beta N^{-1/2}$.  Let $F:[0,1] \to \R_{>0}$ be a random continuous function satisfying $F(0) = 1$ and $F(1) = e^\theta$ almost surely, and let $F^N:\{0,\ldots,N\} \to \R_{>0}$ be a sequence of random functions satisfying $F^N(0) = 1$ and $F^N(N) = e^{\theta_N}$ almost surely.
Assume that $(F^N,F)$ is independent of $\{Z_{N,\beta}^{\text{per}}(t,y \viiva 0,x;\theta_N), Z_\beta^{\text{per}}(t,y \viiva 0,x;\theta): t >0,x,y \in \R\}$ for $N \ge 1$, and assume further that 
\be \label{eq:FntoF}
\lim_{N \to \infty} \sup_{x \in [0,1]} \Ee\Bigl[\bigl|F^N(\lfloor xN \rfloor) - F(x)\bigr|\Bigr]  = 0, \qquad \text{and}\qquad \sup_{x \in [0,1]} \Ee\bigl[F(x)\bigr] < \infty.
\ee
Then, for  any sequences $T_N \to T > 0$ and $Y_N \to Y \in [0,1]$, 
\begin{align*}
    \lim_{N \to \infty} \Ee \Bigl[ \Bigl|e^{-T_N N - \f{\beta_{N}^2}{2}T_N N^2 
 }\OCY_{\beta_{N}}(T_N N^2, \lfloor T_N N^2 + Y_N N \rfloor \viiva F^N) - Z_\beta(T,Y \viiva F)\Bigr|  \Bigr] = 0.
\end{align*}
\end{theorem}
\begin{proof}
 By \eqref{per_conv}, for a function $G:\{0,\ldots,N\} \to \R$, with $G(1) = 1$ and $G(N) = \theta'$, 
 we have
\begin{align*}
&\quad \,e^{-t N - \f{\beta_{N}^2}{2}t N^2 
 }\OCY_{\beta_{N}}(tN^2, \lfloor tN^2 + y N \rfloor \viiva G) \\
 &=e^{-t N^2 - \f{\beta_{N}^2}{2}t N^2 
 } \sum_{m  = 0}^{N-1} \OCYp_{\beta_N}(tN^2, \lfloor tN^2 + yN \rfloor \viiva 0,m;\theta') G(m) \\
 &= e^{-t N^2 - \f{\beta_{N}^2}{2}t N^2 
 } \int_0^N \OCYp_{\beta_N}(tN^2, \lfloor tN^2 + yN \rfloor \viiva 0,\lfloor u \rfloor;\theta') G(\lfloor u \rfloor)\,du \\
 &= N e^{-t N^2 - \f{\beta_{N}^2}{2}t N^2}\int_0^1 \scOCYp(tN^2, \lfloor tN^2 + yN \rfloor \viiva 0,\lfloor xN \rfloor;\theta') G(\lfloor xN \rfloor)\,du \\
 &= \int_0^1 \scOCYp(t,y \viiva 0,x;\theta') G(\lfloor xN \rfloor)\,dx.
\end{align*}
Combined with Proposition \ref{prop:Z_conv}\ref{itm:Z_conv}, we have
\begin{align}
&\quad \Ee\bigg[\Big|e^{-T_N N - \f{\beta_{N}^2}{2}T_N N^2 
 }\OCY_{\beta_{N}}(T_N N^2, \lfloor T_N N^2 + Y_N N \rfloor \viiva F^N) - Z_\beta(T,Y \viiva F)\Big|\bigg] \nonumber  \\
 &=\Ee \bigg[\Big| \int_0^1 \Bigl(\scOCYp(T_N,Y_N \viiva 0,x;\theta_N) F^N(\lfloor x N \rfloor) - Z_\beta^{\text{per}}(T,Y \viiva 0,x;\theta) F(x)\Bigr)\,dx\Big|\bigg] \nonumber  \\
  &\le \int_0^1 \Ee\bigg[ \Big| \scOCYp(T_N,Y_N \viiva 0,x;\theta_N) - Z_\beta^{\text{per}}(T,Y \viiva 0,x;\theta)    \Big|\bigg] \Ee[F^N(\lfloor xN\rfloor)]\,dx \label{int1}\\
  &\qquad\qquad+ \int_0^1 \Ee\big[Z_\beta^{\text{per}}(T,Y \viiva 0,x;\theta)\big] \Ee\big[|F^N(\lfloor xN \rfloor) - F(x)|\big]\,dx, \label{int2}
\end{align}
where we used the triangle inequality and independence in the last step. We argue that the integrals in \eqref{int1} and \eqref{int2} converge to $0$. By the assumption \eqref{eq:FntoF}, we have $\sup_{N \ge 1, x \in [0,1]} \Ee[F^N(\lfloor x N \rfloor)] < \infty$, and the integral in \eqref{int1} is bounded above by 
\[
C \int_0^1 \Ee\bigg[ \Big| \scOCYp(T_N,Y_N \viiva 0,x;\theta_N) - Z_\beta^{\text{per}}(T,Y \viiva 0,x;\theta)    \Big|\bigg]\,dx
\]
for a constant $C$. By Theorem \ref{L2_conv_main_theorem}, $\Ee\bigg[ \Big| \scOCYp(T_N,Y_N \viiva 0,x;\theta_N) - Z_\beta^{\text{per}}(T,Y \viiva 0,x;\theta)    \Big|\bigg] \to 0$ for each $x \in [0,1]$ (bounding the $L^1$ norm by the $L^2$ norm). Furthermore, the integrand is bounded above by
\begin{align*}
&\quad \, \Ee\big[\scOCYp(T_N,Y_N \viiva 0,x;\theta_N)\big] + \Ee\big[Z_\beta^{\text{per}}(T,Y \viiva 0,x;\theta)\big] \\
&= \sum_{r \in \Z} e^{-\theta_N r} p_{N}(T_N,Y_N + r \viiva 0,x) + \sum_{r \in \Z} e^{-\theta_N r} \rho(T,Y + r - x),
\end{align*}
where the computation of the expectation is by the chaos series representations in  Propositions \ref{prop:Z_conv} and \ref{prop:OCY_Chaos} (noting that the $k = 0$ term in the chaos expansions are the only ones with nonzero expectation by Lemma \ref{lem:Ito_isometry}). The second sum is independent of $N$, and by \eqref{rhobd1} from Lemma \ref{lem:rhobd}, it is bounded for $x \in [0,1]$. Using Lemma \ref{lem:pN_per_bd}, we also get the bound 
\[
\sum_{r \in \Z} e^{-\theta_N r} p_N(T_N,Y_N + r \viiva 0,x) \le C \max_{j \in \lzb -3,1 \rzb} p_N(T_N,Y_N + j \viiva 0,x),
\]
for a constant $C > 0$. As $T_N$ and $T$ are bounded away from $0$ and $x,Y \in [0,1]$, Lemma \ref{lem:pn_ubd} implies that there exists a constant $C$ so that $p_N(T_N,Y_N \viiva 0,x) \le C$ for all $x \in [0,1]$ and sufficiently large $N$.  The bounded convergence theorem completes the proof that \eqref{int1} converges to $0$. 

To handle \eqref{int2}, we have already seen that $\Ee\big[Z_\beta^{\text{per}}(T,Y \viiva 0,x;\theta)\big]$ is bounded by a constant for fixed $T > 0$ and $x,Y \in [0,1]$. Hence, we wish to show that 
\[
\lim_{N \to \infty} \int_0^1 \Ee\big[|F^N(\lfloor x N \rfloor) - F(x)|\big]\,dx = 0, 
\]
which follows by the assumption \eqref{eq:FntoF}.
\end{proof}

\section{Proof of invariance of the periodic KPZ horizon and its limits}\label{sec:proofmaintthm}

\subsection{Convergence of random walk bridges to sloped Brownian bridge}
Before getting to the proof of the main theorems, we prove an intermediate lemma that shows convergence of the measures $\mu_\beta^{N,(\theta_1,\ldots,\theta_k)}$ (recall Definition \ref{def:mu_meas}) to the limiting measures $\Pm_\beta^{(\theta_1,\ldots,\theta_k)}$ at the level of a coupling. This is Lemma \ref{lem:initial_data_conv}.
For this, we recall the definition of the map $\D^{N,k}$ in \eqref{DNk_intro}.

Before proving Lemma \ref{lem:initial_data_conv}, we give a lemma that allows us to rewrite the output of $\D^{N,k}$ as an expression involving multiple integrals. In the following, for a function $f:\{0,\ldots,N\} \to \R$ and $0 \le i,j \le N$, we define 
    \be \label{eq:f_bracket}
    f\ls i,j\rs   := f(j) - f(i) + \ind\{i > j\}\bigl(f(N) - f(0)).
    \ee
The notation $f[i,j]$ depends on the domain $\{0,\ldots,N\}$, but to avoid unnecessarily complicated notation, we do not specify the dependence on $N$, as the domain will always be clear from context. This is a discrete analogue of the notation $f(x,y) = f(y) - f(x) + \ind(x > y)(f(1) - f(0))$ for functions $f;[0,1] \to \R$.
\begin{lemma} \label{lem:GR_walk}
For $N,k \in \N$, and $(\mbf X_1,\ldots,\mbf X_k) \in (\R^{\Z_N})^k$, let $(\mbf U_1,\ldots,\mbf U_k) = \D^{N,k}(\mbf X_1,\ldots,\mbf X_k)$. For $m \in \{1,\ldots,k\}$ and $i \in \{0,\ldots,N\}$, define 
    \[
    f_m(i) = \begin{cases}
    0 &i = 0 \\
    X_{m,[1,i]} &1 \le i \le N.
    \end{cases}
    \]
    Recall $X_{m,[1,i]}=\sum_{\ell = 1}^i X_{m,\ell}$ (taken in cyclic order on $\Z_N$), and we have $X_{m,N} = X_{m,0}$ so that $f_m(N) = \sum_{i \in \Z_N} X_{m,i}$. 
    Define $g_m(i)$ as $f_m(i)$ with $U$ replacing $X$. 
    Then, for $2 \le m \le k$, $g_m(0) = 0$, and $x,y \in [0,1]$, 
    \be \label{eq:gr_fr}
    \begin{aligned}
    &\quad \, g_m \Bigl \ls \lfloor xN \rfloor, \lfloor yN  \rfloor \Bigr \rs \\
    &= f_m\Bigl \ls \lfloor xN \rfloor, \lfloor yN \rfloor \Bigr \rs + \log \Biggl(\f{\int_{[0,1]^{m-1}, x_0 = y} \prod_{r = 1}^{m-1} \,dx_r\, e^{f_m \ls\lfloor x_{r-1}N \rfloor ,\lfloor x_r N \rfloor\rs - f_r^N\ls \lfloor x_{r-1}N \rfloor ,\lfloor x_r N \rfloor \rs   
  }}{\int_{[0,1]^{m-1}, x_0 = x} \prod_{i = 1}^{r-1} \,dx_i\, e^{f_r^N \ls\lfloor x_{i-1}N \rfloor ,\lfloor x_i N \rfloor\rs - f_i^N\ls \lfloor x_{i-1}N \rfloor ,\lfloor x_i N \rfloor \rs   
  }}\Biggr).
    \end{aligned}
    \ee
\end{lemma}
\begin{proof}
Since $f_m[i,j] = X_{m,(i,j]}$, the result follows from Lemma \ref{lem:Dnm_alt} and writing the sum as an integral. 
\end{proof}

\begin{lemma}[Coupling] \label{lem:initial_data_conv}
Let $(\theta_1,\ldots,\theta_k) \in \R^k$, $\beta > 0$, and consider any $N \in \N$. Define $\beta_N = \beta N^{-1/2}$, let $\gamma_N$ be chosen so that $\psi(\gamma_N) = \log \beta_N^{-2}$, and for $1 \le m \le k$, set $\theta_m^N =  \f{\theta_m}{\beta} \sqrt{N \psi_1(\gamma_N)}$.
Let $(\mbf U_1^N,\ldots,\mbf U_k^N) \sim \mu_{\beta_N}^{N,(\theta_1^N,\ldots,\theta_k^N)}$ and let $g_m^N(i)$, for $m\in \{1,\ldots, k\}$ and $i\in\{0,\ldots,N\}$, be defined as in Lemma \ref{lem:GR_walk} with $(\mbf U_1^N,\ldots,\mbf U_k^N)$ in place of $(\mbf U_1,\ldots,\mbf U_k)$.
Then for each $N\in \N$, there exists a probability space $(\Omega_N,\Ff_N,\Pp_N)$ on which $(g_1^N,\ldots,g_k^N)$ and $(g_1,\ldots,g_k) \sim \Pm_\beta^{(\theta_1^N,\ldots,\theta_k^N)}$ can be commonly defined such that for $1 \le m \le k$,
\begin{equation}\label{eq:couplinglimit}
    \lim_{N \to \infty}\Ee_N\bigg[\sup_{x \in [0,1]}\Big| e^{g_m^N(\lfloor x N \rfloor )}-e^{g_m(x)}\Big|\bigg] = 0.
\end{equation}
\end{lemma}
\begin{remark}
It is shown in the proof of Lemma \ref{lem:BB_coup} that $\theta_m^N \to \theta_m$ as $N \to \infty$. This particular choice of $\theta_m^N$ is made only for convenience in applying a result from \cite{Dmitrov-Wu-2021} (Theorem \ref{thm:KMTBB} herein). 
\end{remark}
\begin{proof}
The basic premise of this proof is that the laws $\mu_{\beta_N}^{N,(\theta_1^N,\ldots,\theta_k^N)}$ and $\Pm_\beta^{(\theta_1^N,\ldots,\theta_k^N)}$ are both the push forwards under the same map of certain tuples of conditioned random walks / Brownian bridges. We use the KMT coupling (Lemma \ref{lem:BB_coup}) to couple those underlying random processes and then utilize certain regularity estimates about the map to prove our lemma. 

Let $(\mbf U_1^N,\ldots,\mbf U_k^N) := \D^{N,k}(\mbf X_1,\ldots,\mbf X_k)$ where  $(\mbf X_1^N,\ldots,\mbf X_k^N) \sim \nu_{\beta_N}^{N,(\theta_1^N,\ldots,\theta_k^N)}$.
Let $f_m^N(i)$, for $m\in \{1,\ldots, k\}$ and $i\in\{0,\ldots,N\}$, be defined as in Lemma \ref{lem:GR_walk} with $(\mbf X_1^N,\ldots,\mbf X_k^N)$ in place of $(\mbf X_1,\ldots,\mbf X_k)$, and note that by the independence of the $\mbf X_m^N$ for $1\le m\le k$, the functions $(f^N_1,\ldots,f^N_k)$ are also independent. Let $(f_1,\ldots,f_k) \in C([0,1])^k$ be independent Brownian bridges such that for $1 \le m \le k$, $f_m \deq \B_{\beta,\theta_m}$.

By Lemma \ref{lem:BB_coup} and independence of the $f_m$ and $f^N_m$, for each $N\in \N$ there exists a probability space $(\Omega_N,\Ff_N,\Pp_N)$
on which $(f_1^N,\ldots, f_k^N)$ and $(f_1,\ldots,f_k)$ are commonly defined and on which for all $u \ge 0$,
  \be \label{eq:PN_dif}
  \Pp_N\big(\sup_{0 \le x \le 1} \big|f_m(x) - f_m^N(\lfloor xN \rfloor)\big| \ge u\bigg) \le C N^{\alpha' - \f{a u \sqrt N}{\log N}}
  \ee
  for constants $C,\alpha',a > 0$. Define $(g_1,\ldots,g_m) = \FcDm^k(f_1,\ldots,f_k)$ (recall \eqref{Psi_map}). By \eqref{eq:alternate_rep}, for $1 \le m \le k$,
\be \label{gm_exp}
g_m(x) = f_m(x) + \log \Biggl(\f{\int_{[0,1]^{m-1}, x_0 = x} \prod_{r = 1}^{m-1} \,dx_r\, e^{f_m(x_{r-1},x_r ) - f_r( x_{r-1}, x_r)   
  }}{\int_{[0,1]^{m-1}, x_0 = 0} \prod_{r = 1}^{m-1} \,dx_r\, e^{f_m(x_{r-1},x_r ) - f_r( x_{r-1}, x_r) 
  }}\Biggr),
\ee
where we recall the definition $f_r(x,y) = f_r(y) - f_r(x) + \ind\{x > y\} (f_r(1) - f_r(0))$. Compare this formula for $g_m$ in terms of $f_m$ to \eqref{eq:gr_fr}. By Cauchy-Schwartz, we can bound
\be \label{eq:Holder}
\Ee_N\bigg[\sup_{x \in [0,1]}\big|e^{g_m^N(\lfloor xN\rfloor )} - e^{g_m(x)}\big|\bigg] \le \Ee_N\bigg[\sup_{x \in [0,1]} e^{2g_m(x)}\bigg]^{1/2} \Ee_N\bigg[\sup_{x \in [0,1]}\big|e^{g_m^N(\lfloor xN\rfloor) - g_m(x)} - 1\big|^2\bigg]^{1/2}.
\ee
Next, note that $f_r(0) = 0$ and 
\be \label{fr_bound}
|f_r(x_{r-1},x_r)| = |f_r(x_{r-1}) - f_r(x_r) + \ind\{x > y\}f_r(1)| \le 3 \sup_{x \in [0,1]}|f_r(x)|.
\ee
The same bound holds with $f_m$ in place of $f_r$. Hence, from \eqref{gm_exp}, for $1 \le m \le k$, we have 
\be \label{gm_bd}
|g_m(x)| \le 7m \sup_{x \in [0,1]}|f_m(x)| +   6 \sum_{r = 1}^{m-1} \sup_{x \in [0,1]} |f_r(x)|, 
\ee
because \eqref{fr_bound} gives $1 + 6m \le 7m$ terms of $\sup_{x \in[0,1]} |f_m(x)|$ and $6$ terms of $\sup_{x \in [0,1]}|f_r(x)|$ for each $r$.

Since $f_m(x) \deq \B_{\beta,\theta_r}(x) \deq \beta \B_{1,0}(x) + \theta_m x$ for $1 \le r \le k$, the tail bounds  of Lemma \ref{lem:BB_max} guarantee  $\Ee_N[\sup_{x \in [0,1]} e^{2g_m(x)}]^{1/2} < \infty$. Consequently, in order to prove \eqref{eq:couplinglimit}, it suffices to show that 
\begin{equation}\label{eq:couplinglimit2}
\lim_{N\to\infty} \Ee_N\bigg[\sup_{x \in [0,1]}\big|e^{g_m^N(\lfloor xN\rfloor) - g_m(x)} - 1\big|^2\bigg]=0.
\end{equation}

By Lemma \ref{lem:GR_walk}, $g_m^N(\lfloor xN\rfloor ) - g_m(x) = f_m^N(\lfloor xN \rfloor) - f_m(x) + \log\big(A_m^N(x) + 1\big) + \log\big(B_m^N + 1\big)$,
where
\begin{align*}
&A_m^N(x) := \f{1}{\int_{[0,1]^{m-1}, x_0 = x} \prod_{r = 1}^{m-1} \,dx_r\, e^{f_m(x_{r-1},x_r ) - f_r( x_{r-1}, x_r) }} \\
  &\times \int_{[0,1]^{m-1}, x_0 = x} \Bigl(\prod_{r = 1}^{m-1} e^{f_m^N[\lfloor x_{r-1}N \rfloor ,\lfloor x_r N \rfloor] - f_r^N[\lfloor x_{r-1}N \rfloor ,\lfloor x_r N \rfloor]} -  \prod_{r = 1}^{m-1} e^{f_m(x_{r-1},x_r ) - f_r( x_{r-1}, x_r)  
  }\Bigr)\prod_{r = 1}^{m-1}\,dx_r, \\
  &B_m^N := \f{1}{\int_{[0,1]^{m-1}, x_0 = 0} \prod_{r = 1}^{m-1} \,dx_r \, e^{f_m^N[\lfloor x_{r-1}N \rfloor ,\lfloor x_r N \rfloor] - f_r^N[\lfloor x_{r-1}N \rfloor ,\lfloor x_r N \rfloor]   
  }} \\
  &\times \int_{[0,1]^{m-1}, x_0 = 0} \Bigl( \prod_{r = 1}^{m-1}  e^{f_m(x_{r-1},x_r ) - f_r( x_{r-1}, x_r)   
  } - \prod_{r = 1}^{m-1} e^{f_m^N[\lfloor x_{r-1}N \rfloor ,\lfloor x_r N \rfloor] - f_r^N[\lfloor x_{r-1}N \rfloor ,\lfloor x_r N \rfloor]   
  }\Bigr)\prod_{r = 1}^{m-1}\,dx_r,
\end{align*}
and we have used the shorthand notation $f[i,j]$ from \eqref{eq:f_bracket}.
Therefore,
\be \label{eq:egr_exp}
e^{g_m^N(\lfloor xN\rfloor ) - g_m(x)} - 1 
= e^{f_m^N(\lfloor xN \rfloor) -  f_m(x)} - 1 + e^{f_m^N(\lfloor xN \rfloor) -  f_m(x)}   
 \bigl(A_m^N(x) B_m^N + A_m^N(x) + B_m^N\bigr).
 \ee
From this and H\"older's inequality, to prove \eqref{eq:couplinglimit2} it suffices to show the following:
 \begin{enumerate} [label=\textup{(\roman*)}]
      \item \label{itm:frHolder} For each $b > 0$, $\lim_{N \to \infty} \Ee_N[\sup_{x \in [0,1]}|e^{f_m^N(\lfloor xN \rfloor) -  f_m(x)} - 1|^b] = 0$. In particular,  \\$\Ee_N[\sup_{x \in [0,1]}e^{b(f_m^N(\lfloor xN \rfloor) -  f_m(x))}]$ is bounded, uniformly in $N$. 
     \item \label{itm:frHolder2} For each $b > 0$, $\lim_{N \to \infty}\Ee_N[\sup_{x \in [0,1]}|A_m^N(x)|^b] = \lim_{N \to \infty} \Ee_N[|B_m^N|^b] = 0$. 
 \end{enumerate}
To accomplish this, we will compare $f_r^N\bigl[\lfloor wN \rfloor, \lfloor y N \rfloor\bigr]$ with $f_r (w,y)$ for $w,y \in [0,1]$ and $1 \le r \le k$. When $w \le y$, then also $\lfloor w N \rfloor \le \lfloor y N \rfloor$, and so
\be \label{eq:fcomp1}
f_r^N\bigl[\lfloor wN \rfloor, \lfloor y N \rfloor\bigr] - f_r(w,y) = f_r^N(\lfloor y N \rfloor) - f_r(y)  + f_r(w) - f_r^N(\lfloor w N \rfloor).
\ee
Next, if $\lfloor wN \rfloor > \lfloor y N \rfloor$, then also $w > y$, so 
\be \label{eq:fcomp2}
f_r^N\bigl[\lfloor wN \rfloor, \lfloor y N \rfloor\bigr] - f_r(w,y) = f_r^N(\lfloor y N \rfloor) - f_r(y)  + f_r(w) - f_r^N(\lfloor w N \rfloor) + f_r^N(N) - f_r(1).
\ee
In both of these cases, we may apply Lemma \ref{lem:BB_coup} to control the difference. However, in the case of $w > y$, but $\lfloor w N \rfloor = \lfloor y N \rfloor$, we have 
\be \label{eq:fcomp3}
f_r^N\bigl[\lfloor wN \rfloor, \lfloor y N \rfloor\bigr] - f_r(w,y) = f_r^N(\lfloor y N \rfloor) - f_r(y)  + f_r(w) - f_r^N(\lfloor wN \rfloor) - f_r(1),
\ee
and the additional term $f_r(1)$ must be handled. However, the Lebesgue measure, $\mathrm m(N)$, of the subset of $[0,1]^{m-1}$ where $x_{r-1} > x_r$, but $\lfloor x_{r-1} N \rfloor = \lfloor x_r N \rfloor$ for some $1 \le r \le m-1$ is vanishing as $N \to \infty$. We next claim the following bounds:  
\[
    \sup_{x \in [0,1]}|A_m^N(x)| \le A_m^{N,1} (A_m^{N,2} + \mathrm m(N) A_m^{N,3})\qquad\text{and}\qquad |B_m^N| \le B_m^{N,1} (A_m^{N,2} + \mathrm m(N) A_m^{N,3}),\]
where
\begin{align*}
    A_m^{N,1} &= \prod_{r = 1}^m \sup_{y \in [0,1]} e^{6|f_r(y)| }, \qquad\qquad\qquad\qquad\qquad\;\;
    A_m^{N,2} = \bigg(\prod_{r = 1}^m \sup_{y \in [0,1]} e^{3|f_r^N(\lfloor  y N \rfloor) -f_r(y)|} \bigg) - 1, \\
    A_m^{N,3} &=  \prod_{r = 1}^m \sup_{y \in  [0,1]} e^{3 |f_r(y)|} \sup_{y \in [0,1]}e^{2|f_r^N(\lfloor y N \rfloor)|}, \qquad \text{and}\qquad  B_m^{N,1} = \prod_{r = 1}^m \sup_{y \in [0,1]} e^{6|f_r^N(\lfloor y N \rfloor )| }.
\end{align*}
To see this, we write the integrands as $e^{w} - e^z = e^z(e^{w-z} - 1)$ and use the fact that $|e^x - 1| \le e^{|x|} - 1$ for all $x\in \R$. The terms $A_m^{N,1}$ and $B_m^{N,1}$ control the denominators of $A_m^N(x)$ and $B_m^N$,  as well as the $e^w$ term. Hence, the factor of $6$ comes from applying \eqref{fr_bound} twice: once for the denominators and once for the term $e^w$. The factor of $3$ in the term $A_{m}^{N,2}$ comes because in \eqref{eq:fcomp1} and \eqref{eq:fcomp2}, we are comparing at most 3 differences.  The term $A_m^{N,3}$ controls the integrands on the set $\mathrm m(N)$, using \eqref{eq:fcomp3}.

 Repeated applications of H\"older's inequality and the triangle inequality implies that it further suffices to show that, for $1 \le m \le k$ and $b,c > 0$,
\begin{align}
&\quad \,\Ee_N\bigg[\prod_{r = 1}^m \sup_{y \in [0,1]} e^{b|f_r(y)|}    \bigg] < \infty,\qquad \text{and} \label{eq:fisup1}\\
&\lim_{N \to \infty}\Ee_N \Bigg[\Biggl(\bigg[\prod_{r = 1}^m \sup_{y \in [0,1]} e^{c |f_r^N(\lfloor  y N  \rfloor) -f_r(y)|} \bigg] - 1\Biggr)^b\Bigg] = 0. \label{eq:fisup2}
\end{align}
Indeed, \eqref{eq:fisup1} and \eqref{eq:fisup2} imply \ref{itm:frHolder2}, that all moments of $\sup_{x \in [0,1]}|A_m^N(x)|$ and $|B_m^N|$ converge to $0$ as $N \to \infty$, and the $m = 1$ case of \eqref{eq:fisup2} (along with $|e^x - 1| \le e^{|x|}  - 1$) implies \ref{itm:frHolder}. We start with proving \eqref{eq:fisup1}: For sufficiently large $z$, using the bounds on the max of Brownian bridge in Lemma \ref{lem:BB_max},
\begin{align*}
    \Pp_N\Bigl(\prod_{r = 1}^m \sup_{x \in [0,1]} e^{|f_r(x)|} > z  \Bigr)
    \le \sum_{r = 1}^m \Pp_N\Big(\sup_{y \in [0,1]} |f_r(y)| > \f{\log(z)}{ m} \Big)  \le 2 \sum_{r = 1}^m e^{ - 2 (\f{\log z}{m} - \theta_r)^2},
\end{align*}
and so $\prod_{r = 1}^m \sup_{y \in [0,1]} e^{|f_r(y)|}$ has finite moments of all orders, showing  \eqref{eq:fisup1}. For \eqref{eq:fisup2}, first observe that
  \be \label{eq:conv_bound}
  u > \f{2 \alpha' \log N}{a \sqrt N} \Longrightarrow \alpha' - \f{au\sqrt N}{\log N} < -\f{au\sqrt N}{2 \log N}.
  \ee
  Thus, for $z > \ell_N := e^{\f{2\alpha' cr \log N}{a \sqrt N}} - 1$, we use \eqref{eq:PN_dif} and \eqref{eq:conv_bound} to show that 
\begin{align*}
    &\quad \,\Pp_N\Biggl(\Bigl[\prod_{r = 1}^m \sup_{y \in [0,1]} e^{c \sum_{r = 1}^{m}|f_r^N(\lfloor y N  \rfloor) -f_r(y)|} \Bigr] - 1 > z   \Biggr) \\
    &\le \sum_{r = 1}^m \Pp_N\Bigl(\sup_{y \in [0,1]} |f_r^N(\lfloor y N \rfloor) - f_r(y)| > \f{\log(z+1)}{cm}   \Bigr) \\
    &\le Cm  N^{-\f{a \log(z+1) \sqrt N}{2 \log N}} = \f{Cm}{(z+1)^{\f{a}{2c} \sqrt N}}.
\end{align*}
Then, for $b \ge 1$ (bounding tail probabilities by $1$ for $z \le \ell_N$) 
\begin{align*}
    &\quad \, \Ee_N\Biggl[\Biggl(\Bigl[\prod_{r = 1}^m \sup_{y \in [0,1]} e^{c \sum_{r = 1}^{m}|f_r^N(\lfloor y N \rfloor) -f_r(y)|} \Bigr] - 1\Biggr)^b\Biggr] \le b\int_0^{\ell_N} z^{b-1} + C m b \int_{\ell_N}^\infty \f{z^{b-1}}{(z+1)^{\f{a}{2c}\sqrt N}},
\end{align*}
and the right-hand side converges to $0$ by the dominated convergence theorem, thus proving \eqref{eq:fisup2}.
\end{proof}

\subsection{Proof of the invariance and 1F1S principle for the periodic KPZ horizon}\label{sec:proof234}
Here, we prove prove Theorems \ref{thm:KPZ_invar_main} and \ref{thm:1f1s}.

\begin{proof}[Proof of Theorem \ref{thm:KPZ_invar_main}]
    Let $(\theta_1,\ldots,\theta_k) \in \R^k$, $\beta > 0$, and for $N \in \N$,  let $(\mbf U_{1}^N,\ldots,\mbf U_k^N) \sim \mu_{\beta}^{N,(\theta_1^N,\ldots,\theta_k^N)}$ where $\theta_m^N = \f{\theta_m}{\beta}\sqrt{N \log (\beta^{-2}N)}$. For $1 \le m \le N$, let $G^N_m(0) = 1$, and for $1 \le i \le N$, set 
    \[
    G_m^N(i) = e^{g_m^N(i)},\qquad\text{where}\qquad g_m^N(i) = \sum_{j = 1}^i U_{m,j}.
    \]
   Let $T > 0$, and for $N \in \N$,  define $T_N = \f{\lfloor TN^2 \rfloor}{N^2}$ so that $T_N \to T$ and $T_N N^2 \in \Z$. Then, by Lemma \ref{lem:sd_SDE}\ref{itm:SDE}, the invariance property proved in Theorem \ref{thm:OCY_joint}, and the shift invariance from Corollary \ref{cor:shift}, for $N \in \N$, we have the equality in distribution 
   \begin{align*}
   \Biggl(\f{Z_{\beta_N}^N(T_N N^2, T_N N^2 + i \viiva G_m^N)}{Z_{\beta_N}^N(T_N N^2, T_N N^2 \viiva G_m^N)}: 0 \le i \le N, 1 \le m \le k \Biggr) \deq \Bigl(G^N_m(i): 0 \le i \le N, 1 \le m \le k\Bigr),
   \end{align*}
   where we recall that $Z_{\beta_N}^N$ was defined in \eqref{Zfdef} and $\beta_N = \beta N^{-1/2}$.
   This implies the following distributional equality as processes on the product of Skorokhod spaces $D([0,1],\R)^k$:
   \be \label{eq:process_eq}
   \Biggl(\f{Z_{\beta_N}^N(T_N N^2,  T_N N^2 + \lfloor y N \rfloor \viiva G_m^N)}{Z_{\beta_N}^N(T_N N^2, T_N N^2 \viiva G_m^N)}: y \in [0,1],1 \le m \le k \Biggr) \deq \bigl(G^N_m(\lfloor yN \rfloor ): y \in [0,1],1 \le m \le k\bigr).
   \ee
   Here, $D[0,1]$ is the Skorokhod space of cadlag  functions $[0,1] \to \R$, together with the Skorokhod topology. $D[0,1]$ contains $C[0,1]$ as a subspace, and uniform convergence implies convergence in $D[0,1]$ (see, e.g., \cite[Section 12]{billing}).
   The proof of the theorem follows by taking limits on both sides of \eqref{eq:process_eq}. In particular, let $G_m(y) = e^{g_m(y)}$ where the process $\bigl(g_m(y):y \in [0,1],1 \le m \le k\bigr)\sim \Pm_\beta^{(\theta_1,\ldots,\theta_k)}$. To complete the proof of the main theorem, we must show that
   \be \label{eq:limt_equality}
   \Biggl(\f{Z_\beta(T,y \viiva G_m)}{Z_\beta(T,0 \viiva G_m)}: y \in [0,1] \Biggr)_{1 \le m \le k} \deq \bigl(G_m(y): y \in [0,1]\bigr)_{1 \le m \le k}.
   \ee
   Note that by Proposition \ref{prop:solve_SHE}\ref{itm:SHE_soln}, $Z_\beta(T,0 \viiva G_m) > 0$ with probability one, so the left-hand side is well-defined. 

   We first show uniform convergence of the right-hand side of \eqref{eq:process_eq} to  that of  \eqref{eq:limt_equality}. By Lemma \ref{lem:initial_data_conv}, there exists for each $N\in \N$ a common probability space  $(\Omega_N,\Ff_N,\Pp_N)$ on which $\bigl(g_m^N(\lfloor yN \rfloor: y \in [0,1], 1 \le m \le k \bigr)$ and $\bigl(g_m(y):y \in [0,1],1 \le m \le k\bigr)$ are supported such that we have the desired convergence: for $1 \le m \le k$,
   \be \label{eq:gn_conv}
   \lim_{N \to \infty}\Ee_N\bigg[\sup_{y \in [0,1]} \big|G_m^N(\lfloor yN \rfloor) - G_m(y) \big|\bigg] = 0.
   \ee

   We turn now to convergence of the left-hand sides. The convergence in distribution of the right-hand side of \eqref{eq:process_eq} to the right-hand side of \eqref{eq:limt_equality} implies that the process on the left-hand side of \eqref{eq:process_eq} is tight as $N \to \infty$. It suffices to show that, for all subsequential limits, the finite dimensional distributions agree with those on the left-hand side of \eqref{eq:limt_equality}. Let $y_{1,1},\ldots,y_{1,r_1},\ldots,y_{k,1},\ldots,y_{k,r_k} \in \R$, and set $y_{m,0} = 0$ for $1 \le m \le k$. Using \eqref{eq:gn_conv} along with standard bounds on the maximum of Brownian bridge, Theorem \ref{thm:SHE_convergence} implies that, for $1 \le m \le k$ and $0 \le \ell \le r_m$,
   \[
   \lim_{N \to \infty}  \wt \Ee_N \Bigl[ \Bigl|e^{-T_N N - \f{\beta_{N}^2}{2}T_N N^2 
 }\OCY_{\beta_{N}}(T_N N^2, \lfloor T_N N^2 + y_{m,\ell} N \rfloor \viiva G_m^N) - Z_\beta(T,y_{m,\ell} \viiva G_m)\Bigr|  \Bigr] = 0,
   \]
   where $\wt \Ee_N$ denotes the expectation with respect to the product space of $(\Omega_N,\Ff_N,\Pp_N)$ and the noise.
   This convergence holds for all these points in the same coupling. Using that $Z_\beta(T,0 \viiva G_m) > 0$ with probability one (Proposition \ref{prop:solve_SHE}\ref{itm:SHE_soln}), after canceling the prefactor $e^{-T_N N - \f{\beta_N^2}{2}T_N N^2}$, this implies distributional convergence
   \begin{multline*}
   \Biggl(\f{Z_{\beta_N}^N(T_N N^2,  T_N N^2 +\lfloor y_{m,\ell} N \rfloor \viiva G_m^N)}{Z_{\beta_N}^N(T_N N^2, T_N N^2 \viiva G_m^N)}: 1 \le m \le k, 1 \le \ell \le r_m \Biggr)\\
   \overset{N \to \infty}{\Longrightarrow} \Biggl(\f{Z_\beta(T,y_{m,\ell} \viiva G_m)}{Z_\beta(T,0 \viiva G_m)}:1 \le m \le k, 1 \le \ell \le r_m \Biggr).
   \end{multline*}
   Hence, \eqref{eq:limt_equality} holds by taking process-level distributional limits on both sides of \eqref{eq:process_eq}.
\end{proof}
 
Next, we prove the 1F1S principle. The following shear invariance of the KPZ equation reduces this to the 1F1S principle when $\theta=0$, i.e., on the torus without slope. 

\begin{lemma}\label{l.reductiontheta0}
For any $\theta\in\R$ and $f\in C[0,1]$ with $f(1)-f(0)=\theta$, define $g(x)=f(x)-\theta x$. Then, 
\[
h_\beta(0,y\viiva -t,f)=\tilde{h}_\beta(0,y\viiva -t,\tilde{g}_t)+\theta y+\frac12\theta^2 t
\]
where $\tilde{h}_\beta$ solves the KPZ equation with noise $\xi_\theta(s,y):=\xi(s,y-\theta s)$, and initial data $\tilde{g}_t(x):=g(x+\theta t)$.
\end{lemma}

\begin{proof}
    First, assume that the noise $\xi$ is only white in time but smooth in the spatial variable with the spatial covariance function $r(x)$. Then by the Feynman-Kac formula \cite{Bertini1995}, one can write $h_\beta$ as 
    \[
    h_\beta(0,y\viiva -t,f)=\log \Ee_B e^{\beta\int_0^t \xi(-s,y+B_s)ds-\frac12\beta^2r(0)t} e^{\theta (y+B_t)+g(y+B_t)},
    \]
    where $B$ is a standard Brownian motion starting from the origin and $\Ee_B$ is the expectation on $B$. By the Cameron-Martin theorem, we rewrite the above expression as 
    \[
    \begin{aligned}
    h_\beta(0,y\viiva -t,f)&=\log \Ee_B e^{\beta\int_0^t \xi(-s,y+B_s)ds-\frac12\beta^2r(0)t} e^{\theta B_t-\frac12\theta^2 t} e^{g(y+B_t)}+\theta y+\frac12\theta^2 t\\
    &=\log \Ee_B e^{\beta\int_0^t \xi(-s,y+B_s+\theta s)ds-\frac12\beta^2r(0)t}  e^{g(y+B_t+\theta t)}+\theta y+\frac12\theta^2 t\\
    &=\log \Ee_B e^{\beta\int_0^t \xi_\theta(-s,y+B_s)ds-\frac12\beta^2r(0)t}  e^{\tilde{g}_t(y+B_t)}+\theta y+\frac12\theta^2 t.
    \end{aligned}
    \]
    Note that the last expression is precisely $\tilde{h}_\beta(0,y\viiva -t,\tilde{g}_t)+\theta y+\frac12\theta^2 t$, so this proves the lemma for the noise that is smooth in the spatial variable. Then it suffices to apply a standard approximation argument, e.g. in \cite{Bertini1995}, to derive the same result for the white noise, thus completing the proof.
\end{proof}

The following 1F1S principle for the periodic KPZ equation was shown by rather different means in \cite{TR22}. We record another proof using results from \cite{GK23}.
\begin{lemma}\label{l.ofos}
    For any $\theta\in\R$ and $f\in C[0,1]$ satisfying $f(1)-f(0)=\theta$, we have  that
    \[
    \Bigl(h_\beta(0,x \viiva -t,f)-h_\beta(0,0 \viiva -t,f): x \in [0,1] \Bigr)
    \]
    converges to a limit in probability in $\Cpin[0,1]$, as $t\to\infty$. The limit depends only on $\theta$ and not on the particular choice of $f$. In particular, one can choose $f$ depending on $t$.
\end{lemma}

\begin{proof}
    By Lemma~\ref{l.reductiontheta0}, it is enough to consider the case of $\theta=0$ so that one can pose the problem on a torus. Rather than   the solution to the KPZ equation, we consider a related object, the endpoint density of the directed polymer on a cylinder: define 
    \[
    \rho_\beta(0,x \viiva -t,f):= \frac{e^{h_\beta(0,x \viiva -t,f)}}{\int_0^1 e^{h_\beta(0,y \viiva -t,f)} dy},
    \]
    which can be interpreted as the polymer  density at time $0$, with the initial distribution at time $-t$ given by $\f{e^{f(x)}}{\int_0^1 e^{f(y)}dy}$. The proposition is equivalent to showing that $\log \rho_\beta(0,\cdot \viiva -t,f)-\log \rho_\beta(0,0 \viiva -t,f)$ converges in probability in $\Cpin[0,1]$, with the limit independent of the choice of $f$. This is a consequence of the following  result: for any $p\in [1,\infty)$, and any functions $f_1,f_2 \in C[0,1]$ such that $f_i(1)=f_i(0),i=1,2$, 
    \begin{equation}\label{e.contractionpolymer}
    \Ee\bigg[ \sup_{x\in[0,1]} \Big|\log\big(\rho_\beta(0,x \viiva -t_1,f_1)\big)-\log\big(\rho_\beta(0,x \viiva -t_2,f_2)\big)\Big|^p\bigg] \to0,
    \end{equation}
    as $t_1,t_2\to\infty$. To show \eqref{e.contractionpolymer}, first, by the elementary inequality 
    \[
    |\log x_1-\log x_2| \leq (x_1^{-1}+x_2^{-1})|x_1-x_2|,
    \]
    we estimate the difference as
\[
\begin{aligned}
  &\sup_{x\in[0,1]} \Big|\log\big(\rho_\beta(0,x \viiva -t_1,f_1)\big)-\log\big(\rho_\beta(0,x \viiva -t_2,f_2)\big)\Big|^p\\
  &\leq  C\Big(\sup_{x\in[0,1]} \rho_\beta(0,x\viiva -t_1,f_1)^{-p}+\sup_{x\in[0,1]} \rho_\beta(0,x\viiva -t_2,f_2)^{-p}\Big)
   \sup_{x\in[0,1]} \Big|\rho_\beta(0,x \viiva -t_1,f_1)-\rho_\beta(0,x \viiva -t_2,f_2)\Big|^p.
  \end{aligned}
    \]
    The right-hand side converges to $0$ as $t_1,t_2 \to \infty$ by \cite[Proposition A.1, equations (A.3) and (A.4)]{GK23}.
\end{proof}

\begin{proof}[Proof of Theorem \ref{thm:1f1s}]
 Consider the $(\Cpin[0,1])^k$ valued random variable,
 \[
 \mathbf{Y}_t:=\bigl(h_\beta(0,y \viiva -t, f_1) - h_\beta(0,0 \viiva  -t,f_1),\ldots,h_\beta(0,y \viiva  -t,f_k) - h_\beta(0,0 \viiva -t,f_k) \bigr).
 \]
 The goal is to show it converges in probability as $t\to\infty$ and the limit has the law given by $\mathcal{P}_\beta^{(\theta_1,\ldots,\theta_k)}$. 
 The convergence directly comes from Lemma~\ref{l.ofos}, applied to each component of $\mathbf{Y}_t$. Since the limit does not depend on the choice of $f_m$, one can sample $\{f_m\}_{m=1}^k$ from $\mathcal{P}_\beta^{\theta_1,\ldots,\theta_k}$ so that  the law of $\mathbf{Y}_t$ is unchanged as $t$ varies, which implies that the limit of $\mathbf{Y}_t$ has the law $\mathcal{P}_\beta^{\theta_1,\ldots,\theta_k}$.
\end{proof}

\subsection{Proof of the properties of the periodic KPZ horizon (Proposition \ref{prop:g_prop_intro})} \label{sec:proof567}
We first need a lemma that gives an alternate description of the map $\cDm^2$, which will be be more convenient in many applications. 
\begin{lemma} \label{lem:Phiout}
For $f_1,f_2 \in (\Cpin[0,1])$, and $x \in [0,1]$,
\begin{align*}
\cDm^2(f_1,f_2)(x) &= f_1(x) + \log \Biggl(\f{e^{f_2(1) - f_1(1)}\int_0^x e^{f_2(y) - f_1(y)}\,dy + \int_x^1 e^{f_2(y) - f_1(y)}\,dy   }{\int_0^1 e^{f_2(y) - f_1(y)}\,dy}    \Biggr) \\
&= f_1(x) + \log\Biggl(1 + \f{(e^{f_2(1) - f_1(1)} - 1)\int_0^x e^{f_2(y) - f_1(y)}\,dy}{\int_0^1 e^{f_2(y) - f_1(y)}\,dy}\Biggr)
\end{align*}
\end{lemma}
\begin{proof}
With $f_1(0) = f_2(0) = 0$, we have
\begin{align*}
\cDm^2(f_1,f_2)(x) &= f_2(x) + \log \Biggl(\f{\int_0^1 e^{f_2(x,y) - f_1(x,y)}\,dy}{\int_0^1 e^{f_2(0,y)- f_1(0,y)}}\Biggr) \\
&= f_2(x) + \log \Biggl(\f{\int_0^x e^{f_2(y) - f_1(x) + f_2(1) - (f_1(y) - f_1(x) + f_1(1) )}\,dy + \int_x^1 e^{f_2(y) - f_2(x) - (f_1(y) - f_1(x))}\,dy  }{\int_0^1 e^{f_2(y) - f_1(y)}\,dy}\Biggr) \\
&= f_1(x) + \log \Biggl(\f{e^{f_2(1) - f_1(1)}\int_0^x e^{f_2(y) - f_1(y)}\,dy + \int_x^1 e^{f_2(y) - f_1(y)}\,dy   }{\int_0^1 e^{f_2(y) - f_1(y)}\,dy}    \Biggr). \qedhere
\end{align*} 
\end{proof}

The final lemma before the proof of Proposition \ref{prop:g_prop_intro} is the continuum analogue of Lemma \ref{lem:sum_v_order}. 
\begin{lemma} \label{lem:Phi_ord}
Let $f_1,f_2 \in (\Cpin[0,1])$.
\begin{enumerate} [label=\textup{(\roman*)}]
\item If $f_1(1) = f_2(1)$, then $\cDm^2(f_1,f_2)(x) = f_1(x)$ for all $x \in [0,1]$.
\item If $f_1(1) < f_2(1)$, then $f_1(x,y) < \cDm^2(f_1,f_2)(x,y)$ for $x,y \in [0,1]$. 
\item If $f_1(1) > f_2(1)$, then $f_1(x,y) > \cDm^2(f_1,f_2)(x,y)$ for $x,y \in [0,1]$.
\end{enumerate}
\end{lemma}
\begin{proof}
    Since, for $f:[0,1] \to \R$,
    \[
    f(x,y) = \begin{cases}
    f(y) - f(x) &y \ge x,\\
    f(1) - f(x) + f(y) - f(0) &y < x,
    \end{cases}
    \]
    it suffices to take $0 \le x < y \le 1$ in the statements above. Furthermore, for $g_1,g_2:[0,1] \to \R$, by writing 
    \be \label{eq:g2inc_mont}
    g_2(y) - g_2(x) - \bigl(g_1(y) - g_1(x)\bigr) = g_2(y) - g_1(y) - \bigl(g_2(x) - g_1(x)\bigr),
    \ee
    we see that the condition that $g_1(x,y) < g_2(x,y)$ for all $x < y$ is equivalent to the condition that $x \mapsto g_2(x) - g_1(x)$ is strictly increasing. 
    From the second equality of Lemma \ref{lem:Phiout}, we see that $\cDm^2(f_1,f_2)(x) - f_1(x)$ is $0$ if $f_2(1) = f_1(1)$, strictly increasing if $f_2(1) > f_1(1)$ and strictly decreasing if $f_2(1) < f_1(1)$.
\end{proof}

\begin{proof}[Proof of Proposition \ref{prop:g_prop_intro}]
\medskip \noindent \textbf{Items \ref{itm:g_perm_invar}-\ref{itm:g_consis}:} These follow immediately from their discrete versions in Proposition \ref{prop:disc_consis}\ref{itm:mu_perm}-\ref{itm:mu_consis}, and passing to the limit via Lemma \ref{lem:initial_data_conv}.

\medskip \noindent \textbf{Item \ref{itm:g_mont}:} The statement for $r,m \in \{1,2\}$ follows from Lemma \ref{lem:Phi_ord}. Then, the statement for general $r,m \in \{1,\ldots,k\}$ follows from the permutation invariance in Item \ref{itm:g_perm_invar}. The fact that $x \mapsto g_r(x)- g_m(x)$ is increasing comes from \eqref{eq:g2inc_mont}.
 
\medskip \noindent \textbf{Item \ref{itm:Pm_cont}:} If $(\theta_1^{(n)},\ldots,\theta_k^{(n)}) \to (\theta_1,\ldots,\theta_k)$, and $(f_1^{(n)},\ldots,f_k^{(n)})$ are independent, with $f_m^{(n)} \deq \B_{\beta,\theta_m^{(n)}}$, then clearly, $(f_1^{(n)},\ldots,f_k^{(n)})$ converges in distribution to $(f_1,\ldots,f_k)$, where  $(f_1,\ldots,f_k)$ are independent, and $f_m \deq \B_{\beta,\theta_m}$. It then suffices to show that the function $\FcDm^k:(\Cpin[0,1])^k \to (\Cpin[0,1])^k$ is continuous with respect to the sup norm on functions. The function $\cDm^2: \Cpin[0,1] \times \Cpin[0,1] \to \Cpin[0,1]$ is readily seen to be continuous, and since $\FcDm^k$ is an iterative composition of $\cDm^2$, the continuity follows. 
\end{proof}

\subsection{Proof of limits of the periodic KPZ horizon in $\beta$ (Proposition \ref{prop:betalim})}

\begin{proof}[Proof of Proposition \ref{prop:betalim}] \label{sec:proof8910}
When $\ICH_\beta \sim \Pm_\beta$, recall the definition $\wt \ICH_\beta = (\wt g_{\beta,\theta})_{\theta \in \R}$, where  $\wt g_{\beta,\theta}(x) = \f{1}{\beta} g_{\beta,\beta \theta}(x)$. By the monotonicity of Proposition \ref{prop:g_prop_intro}\ref{itm:g_mont}, for $\beta > 0$ and $\theta_1,\theta_2 \in \R$,
\[
\sup_{x \in [0,1]} |\wt g_{\beta,\theta_1}(x) - \wt g_{\beta,\theta_2}(x)| \le |\wt g_{\beta,\theta_1}(1) - \wt g_{\beta,\theta_2}(1)| = |\theta_2 - \theta_1|.
\]
Furthermore, by Proposition \ref{prop:g_prop_intro}\ref{itm:g_consis}, and rescaling, for each $\theta \in \R$, $\wt g_{\beta,\theta}$ has the law of $\B_{1,\theta}$.
Therefore, the family of random functions $(\theta \mapsto \wt g_{\beta,\theta})_{\beta > 0}$ is tight in $C\bigl(\R,\Cpin[0,1]\bigr)$. It then suffices, separately for the cases $\beta \searrow 0$ and $\beta \to \infty$, to characterize the subsequential limits.

\medskip \noindent \textbf{$\beta \searrow 0$:} It suffices to show that, for any $x \in [0,1]$ and any pair $\theta_1,\theta_2 \in \R$, $\wt g_{\beta,\theta_2}(x) - \wt g_{\beta,\theta_1}(x)$ converges in probability to $(\theta_2 - \theta_1)x$. 

Let $f_1,f_2$ be independent, with $f_m \deq \B_{1,\theta_m}$ for $m \in \{1,2\}$, and for $\beta > 0$, let $f_{\beta,m}(x) = \beta f_m(x)$. Then, $(f_{\beta,1},f_{\beta,2})$ are independent, with $f_{\beta,m} \deq \B_{\beta,\beta \theta_m}$ for $m \in \{1,2\}$, so 
\[
\wt g_{\beta,\theta_2}(x) - \wt g_{\beta,\theta_1}(x) \deq \f{\cDm^2(f_{\beta,1},f_{\beta,2})(x) - f_{\beta,1}(x)}{\beta}.
\]
By Lemma \ref{lem:Phiout}, the right-hand side is
\[
\f{1}{\beta}\log \Biggl(1 + \f{(e^{\beta(\theta_2 - \theta_1)} - 1)\int_0^x e^{\beta (f_2(y) - f_1(y))}\,dy   }{\int_0^1 e^{\beta(f_2(y) - f_1(y))}\,dy}\Biggr),
\]
and, in this coupling, we have the almost sure limit
\begin{align*}
&\quad \, \lim_{\beta \searrow 0} \f{1}{\beta}\log \Biggl(1 + \f{(e^{\beta(\theta_2 - \theta_1)} - 1)\int_0^x e^{\beta (f_2(y) - f_1(y))}\,dy   }{\int_0^1 e^{\beta(f_2(y) - f_1(y))}\,dy}\Biggr)  \\
&= \lim_{\beta \searrow 0} (\theta_2 - \theta_1) \f{\int_0^x e^{\beta (f_2(y) - f_1(y))}\,dy   }{\int_0^1 e^{\beta(f_2(y) - f_1(y))}\,dy} = (\theta_2 - \theta_1)x.
\end{align*}

\medskip \noindent \textbf{$\beta \to \infty$:} Let $(\theta_1,\ldots,\theta_k) \in \R^k$, and let $(f_1,\ldots,f_k)$ be independent, with $f_m \deq \B_{1,\theta_m}$ for $1 \le m \le k$. Define $f_{\beta,m} = \beta f_m$ so that $f_{\beta,m} \deq \B_{\beta,\beta \theta_m}$. Then, 
\be \label{gf_eq}
(\wt g_{\beta,\theta_1},\ldots,\wt g_{\beta,\theta_k}) \deq \f{1}{\beta} \FcDm^k(f_{\beta,1},\ldots,f_{\beta,k}) = \Bigl(\f{1}{\beta}\cDm^1(f_{\beta,1}),\ldots,\f{1}{\beta}\cDm^k(f_{\beta,1},\ldots,f_{\beta,k})\Bigr).
\ee
We have seen that the right-hand side,   indexed by $\beta$, is tight in $(\Cpin[0,1])^k$ since the law of each marginal does not depend on $\beta$. By \eqref{eq:alternate_rep}, for $1 \le m \le k$,
\begin{align*}
\f{1}{\beta}\cDm^m(f_{\beta,1},\ldots,f_{\beta,m})(x) = f_m(x) +\f{1}{\beta}\log \int_{[0,1]^{m-1}, x_0 = x} \prod_{r = 1}^{m-1} \,dx_r\, e^{\beta f_{m}(x_{r-1},x_r ) - \beta f_r( x_{r-1}, x_r)}  \\
- \f{1}{\beta}\log \int_{[0,1]^{m-1}, x_0 = 0} \prod_{r = 1}^{m-1} \,dx_r\, e^{\beta f_{m}(x_{r-1},x_r ) - \beta f_r( x_{r-1}, x_r)} \\
\overset{\beta \to \infty}{\longrightarrow} f_m(x) + \sup_{x_1,\ldots,x_{m-1} \in [0,1],x_0 = x} \Bigl\{\sum_{r = 1}^{m-1}[f_m(x_{r-1},x_r) - f_r(x_{r-1},x_r)]  \Bigr\} \\
\qquad \qquad\qquad \qquad - \sup_{x_1,\ldots,x_{m-1} \in [0,1],x_0 = 0} \Bigl\{\sum_{r = 1}^{m-1}[f_m(x_{r-1},x_r) - f_r(x_{r-1},x_r)]  \Bigr\},
\end{align*}
where, in this coupling, the limit holds with probability one for every $x \in [0,1]$. Since we know the law of the process in \eqref{gf_eq} is tight, this characterizes the law of the limit. 
\end{proof}

\subsection{Proof of limit to the full-line KPZ horizon (Proposition \ref{prop:width_lim})} \label{sec:proof111213}

Proposition \ref{prop:width_lim} states the convergence of the periodic KPZ horizon to the full-line KPZ horizon, when the periodicity is sent to $+\infty$. Recall that the periodic KPZ horizon is described in terms of maps of independent Brownian bridges, while the full-line KPZ horizon is described in terms of maps of independent Brownian motions. For this, we need a result that connects convergence of scaled Brownian bridge to Brownian motion to convergence of the outputs of these maps. This is done in Lemma \ref{lem:Phi_close}. Before then, we prove Lemma \ref{lem:Phi_pres_slope}, which is the continuum analogue of Lemma \ref{lem:Dsum_pres}. Afterwards, we prove Lemma \ref{lem:Phi2_ScaleL}, which gives a form of the output of $\cDm^2$ under scaling that is useful for taking limits. 

\begin{lemma} \label{lem:Phi_pres_slope}
Let $(f_1,f_2) \in \Cpin[0,1]$. Then, $\cDm^2(f_1,f_2)(0) = 0$ and $\cDm^2(f_1,f_2)(1) = f_2(1)$.
\end{lemma}
\begin{proof}
This follows immediately from the second equation in Lemma \ref{lem:Phiout}. 
\end{proof}

\begin{lemma} \label{lem:Phi2_ScaleL}
Let $\theta_1,\theta_2 \in \R$, and let $f_{L,1},f_{L,2}$ be functions $[0,1] \to \R$ satisfying $f_{L,1}(0) = f_{L,2}(0) = 0$ and $f_{L,m}(L) = \theta_m L$ for $m \in \{1,2\}$. Extend $f_{L,1}$ and $f_{L,2}$ to  functions $\R \to \R$ by setting $f_{L,m}(x+1)-f_{L,m}(x)=f_{L,m}(1)-f_{L,m}(0)=\theta_m L$ for $m\in \{1,2\}$ and $x\in \R$, and similarly for $\cDm^2(f_{L,1},f_{L,2})$. For $L > 0$, $m \in \{1,2\}$, and $x \in \R$, define $\wt f_{L,m}(x) = f_m\bigl(\f{x}{L}\bigr)$. Then, for $x\in[0,L]$,
\be \label{601}
\cDm^2(f_{L,1},f_{L,2})\Bigl( \f{x}{L}\bigr) = \wt f_{L,1}(x) + \log \Biggl(1 + \f{(e^{(\theta_2 - \theta_1)L} - 1) \int_0^x e^{\wt f_{L,2}(y) - \wt f_{L,1}(y)}\,dy}{\int_{0}^L e^{\wt f_{L,2}(y) - \wt f_{L,1}(y)}\,dy}\Biggr),
\ee
while for $x\in[-L,0]$,
\be \label{602}
\cDm^2(f_{L,1},f_{L,2})\Bigl( \f{x}{L}\bigr) = \wt f_{L,1}(x) + \log \Biggl(1 + \f{(e^{(\theta_1 - \theta_2) L}-1)\int_{0}^{x} e^{\wt f_{L,2}(y) - \wt f_{L,1}(y)}\,dy}{\int_{-L}^{0} e^{\wt f_{L,2}(y) - \wt f_{L,1}(y)}\,dy}  \Biggr).
\ee
\end{lemma}
\begin{proof}
For $x\in[0,L]$, we start from the second equality of Lemma \ref{lem:Phiout}:
\begin{align*}
\cDm^2(f_{L,1},f_{L,2})\Bigl( \f{x}{L}\Bigr) &= f_{L,1}\Bigl(\f{x}{L}\Bigr) + \log\Biggl(1 +  \f{(e^{(\theta_2 - \theta_1)L} - 1) \int_0^{\f{x}{L}} e^{f_{2,L}(y) - f_{1,L}(y)  }\,dy}{\int_0^1 e^{f_{2,L}(y) - f_{1,L}(y)  }\,dy }\Biggr) \\
&=f_{L,1}\Bigl(\f{x}{L}\Bigr) + \log\Biggl(1 + (e^{(\theta_2 - \theta_1)L} - 1) \f{\int_0^{x} e^{f_{2,L}(\f{y}{L}) - f_{1,L}(\f{y}{L})  }\,dy}{\int_0^L e^{f_{2,L}(\f{y}{L}) - f_{1,L}(\f{y}{L})  }\,dy }\Biggr),
\end{align*}
and this is equal to the right-hand side of \eqref{601}. By Lemma \ref{lem:Phi_pres_slope}, $\cDm^2(f_{L,1},f_{L,2})(1) = f_{L,2}(1) = \theta_2 L$. Then, for $x\in[-L,0]$, starting from the first equality of Lemma \ref{lem:Phiout},
\begin{align*}
    &\quad \, \cDm^2(f_{L,1},f_{L,2})\Bigl(\f{x}{L}\Bigr) =  \cDm^2(f_{L,1},f_{L,2})\Bigl(\f{x + L}{L}\Bigr)  - \theta_2 L  \\
    &= f_{L,1}\Bigl(\f{L+x}{L}\Bigr) - \theta_2 L +   \log \Biggl(\f{e^{\theta_2 L - \theta_1 L}\int_0^{\f{L+x}{L}} e^{f_{L,2}(y) - f_{L,1}(y)}\,dy + \int_{\f{L+x}{L}}^1 e^{f_{L,2}(y) - f_{L,1}(y)}\,dy   }{\int_0^1 e^{f_{L,2}(y) - f_{L,1}(y)}\,dy}    \Biggr) \\
    &= f_{L,1}\Bigl(\f{L+x}{L}\Bigr) - \theta_1 L + \log \Biggl(\f{\int_0^{\f{L+x}{L}} e^{f_{L,2}(y) - f_{L,1}(y)}\,dy + e^{(\theta_1 - \theta_2)L} \int_{\f{L+x}{L}}^1 e^{f_{L,2}(y) - f_{L,1}(y)}\,dy   }{\int_0^1 e^{f_{L,2}(y) - f_{L,1}(y)}\,dy}   \Biggr) \\
    &= f_{L,1}\Bigl(\f{x}{L}\Bigr) - \theta_1 L + \log \Biggl(1 + \f{(e^{(\theta_1 - \theta_2)L}-1)\int_{\f{L+x}{L}}^1 e^{f_{L,2}(y) - f_{L,1}(y)}\,dy   }{\int_0^1 e^{f_{L,2}(y) - f_{L,1}(y)}\,dy} \Biggr) \\
    &= f_{L,1}\Bigl(\f{x}{L}\Bigr) - \theta_1 L + \log \Biggl(1 + \f{(e^{(\theta_1 - \theta_2)L}-1)\int_{L+x}^{L} e^{f_{L,2}(\f{y}{L}) - f_{L,1}(\f{y}{L})}\,dy   }{\int_0^L e^{f_{L,2}(\f{y}{L}) - f_{L,1}(\f{y}{L})}\,dy} \Biggr) \\
    &= f_{L,1}\Bigl(\f{x}{L}\Bigr) - \theta_1 L + \log \Biggl(1 + \f{(e^{(\theta_1 - \theta_2)L}-1)\int_{0}^{x} e^{f_{L,2}(\f{L+y}{L}) - \theta_2 L - f_{L,1}(\f{L+y}{L}) + \theta_1 L}\,dy   }{\int_{-L}^0 e^{f_{L,2}(\f{L+y}{L}) - \theta_2 L - f_{L,1}(\f{L+y}{L}) + \theta_1 L}\,dy} \Biggr).
\end{align*}
Since $\wt f_{L,m}(x) = f_{L,m}\bigl(\f{L+x}{L}\bigr) - \theta_m L$ for $x\in[-L,0]$ and $m \in \{1,2\}$, we obtain the right-hand side of \eqref{602}.
\end{proof}

The following lemma is the key input needed for the proof of Proposition \ref{prop:width_lim}. It connects convergence of scaled Brownian bridge to Brownian motion to convergence of the scaled output of $\cDm^2$ to $\cDm_\infty^2$.

\begin{lemma} \label{lem:Phi_close}
Let $\theta_1,\theta_2 \in \R$,  $\beta > 0$, and  $(L_n)_{n \ge 0}$ be a sequence converging to $+\infty$. Let  $\bigl(f_{1}^{(n)},f_{2}^{(n)}\bigr)_{n \ge 1}$ be a coupling of processes on a probability space $(\Omega,\Ff,\Pp)$, taking values in $\Cpin[0,1] \times \Cpin[0,1]$, together with two random functions $(f_1,f_2)$ taking values in $C(\R) \times C(\R)$ such that the following hold:
\begin{enumerate} [label=\textup{(\roman*)}]
    \item For each $n$, $f_{1}^{(n)}$ and $f_{2}^{(n)}$ are independent, with $f_{m}^{(n)} \deq \B_{\beta \sqrt L_n,\theta_m}$ for $m = 1,2$.
    \item For each $\ve > 0$ and each compact set $K \subseteq \R$, 
    \be \label{close_assumption}
    \limsup_{n \to \infty} \Pp\Biggl(\sup_{\substack{x \in K \\ m \in \{1,2\}}} \Bigl|f_{m}^{(n)}\Bigl(\f{x}{L_n}\Bigr) - f_m(x)   \Bigr| > \ve \Biggr) = 0,
    \ee
    where we extend $f_m^{(n)}$ to functions on $\R$ by setting $f_m^{(n)}(x+1) - f_m^{(n)}(x) =f_m^{(n)}(1) - f_m^{(n)}(0)$ for all $x \in \R$.
\end{enumerate}
Then,  for each compact set $K \subseteq \R$ and $\ve > 0$,
\[
\limsup_{n \to \infty} \Pp\Biggl(\sup_{x \in K } \Bigl|\cDm^2(f_1^{(n)},f_2^{(n)})\Bigl(\f{x}{L_n}\Bigr) 
  - \cDm_\infty^2(f_1,f_2)(x)   \Bigr| > \ve \Biggr) = 0,
\]
where $\cDm^2(f_1^{(n)},f_2^{(n)}):[0,1] \to \R$ is extended to a function on $\R$ in the same manner as $f_1^{(n)},f_2^{(n)}$. 
\end{lemma}
\begin{proof}
By independence of $f_1^{(n)}$ and $f_1^{(n)}$ and the distributional convergence in Lemma \ref{lem:BB_to_BM}, $(f_1,f_2)$ are independent, and $f_m \deq B_{\beta,\theta_m}$ for $m \in \{1,2\}$. In particular, for $m \in \{1,2\}$, we have the almost sure limit
\[
\lim_{|x| \to \infty} \f{f_m(x)}{x} = \theta_m.
\]
We consider the case $\theta_2 > \theta_1$, and the case $\theta_2 < \theta_1$ follows by an analogous proof. The case $\theta_1 = \theta_2$ is tautological since $\cDm^2(f_1^{(n)},f_2^{(n)}) = f_1^{(n)}$ (Lemma \ref{lem:Phi_ord}), and $\cDm^2(f_1,f_2) = f_1$ by definition \eqref{eq:Phi_inf_def}. 

It further suffices to consider compact sets $K$ of the form $[0,M]$ and $[-M,0]$ for $M > 0$. Adopt the shorthand notation $\wt f_m^{(n)}(x) = f_m^{(n)}\bigl(\f{x}{L_n}\bigr)$. In the case $\theta_2 > \theta_1$, we have (recall \eqref{eq:Phi_inf_def}) 
\be \label{603}
\cDm_\infty^2(f_1,f_2)(x) = f_1(x) + \log \Biggl(\f{\int_{-\infty}^x e^{f_2(y) - f_1(y)}\,dy}{\int_{-\infty}^0 e^{f_2(y) - f_1(y)}\,dy}\Biggr)
\ee

\medskip \noindent \textbf{Case 1: $K = [0,M]$:} Assume that $L_n \ge M$. 
For $x \in [0,M]$, we use the equality \eqref{601}:
\be \label{604}
\cDm^2(f_1^{(n)},f_2^{(n)})\Bigl( \f{x}{L_n}\bigr) = \wt f_1^{(n)}(x) + \log \Biggl(1 + \f{(e^{(\theta_2 - \theta_1)L_n} - 1) \int_0^x e^{\wt f_2^{(n)}(y) - \wt f_1^{(n)}(y)}\,dy}{\int_{0}^L e^{\wt f_2^{(n)}(y) - \wt f_1^{(n)}(y)}\,dy}\Biggr),
\ee
and write \eqref{603} as
\be \label{604a}
\cDm_\infty^2(f_1,f_2)(x) = f_1(x) + \log \Biggl(1 + \f{\int_{0}^x e^{f_2(y) - f_1(y)}\,dy}{\int_{-\infty}^0 e^{f_2(y) - f_1(y)}\,dy}\Biggr).
\ee
Comparing \eqref{604} with \eqref{604a}, by uniform continuity of the function $x \mapsto \log(1+x)$ over $x \in [0,\infty)$, it suffices to show that, for each $M,\ve > 0$,
\be \label{605}
\limsup_{n \to \infty} \Pp\Biggl(\sup_{x \in [0,M]}\Biggl|\f{(e^{(\theta_2 - \theta_1)L_n} - 1) \int_0^x e^{\wt f_2^{(n)}(y) - \wt f_1^{(n)}(y)}\,dy}{\int_{0}^{L_n} e^{\wt f_2^{(n)}(y) - \wt f_1^{(n)}(y)}\,dy} - \f{\int_{0}^x e^{f_2(y) - f_1(y)}\,dy}{\int_{-\infty}^0 e^{f_2(y) - f_1(y)}\,dy}\Biggr| > \ve  \Biggr) = 0.
\ee
To give a sketch of the proof below, it follows by the assumption \eqref{close_assumption}  that the difference
\[
\int_0^x e^{\wt f_2^{(n)}(y) - \wt f_1^{(n)}(y)}\,dy - \int_0^x e^{\wt f_2^{(n)} - \wt f_1^{(n)}(y)}\,dy
\]
is small with high probability. Furthermore, since $\theta_2 > \theta_1$. for large $n$,
\[
\f{(e^{(\theta_2 - \theta_1)L_n} - 1)}{\int_{0}^{L_n} e^{\wt f_2^{(n)}(y) - \wt f_1^{(n)}(y)}\,dy} \approx \f{1}{\int_{0}^{L_n} e^{\wt f_2^{(n)}(y) - \theta_2 L_n - \wt f_1^{(n)}(y) + \theta_1 L_n}\,dy},
\]
and by a change of variables, the denominator in the right-hand side above is $\int_{-L_n}^0 e^{\wt f_2^{(n)}(y) - \wt f_1^{(n)}(y)}\,dy$. This heuristically justifies \eqref{605}. We now make this rigorous as follows:

For $x \in [0,M]$, write 
\[
\begin{aligned}
&\quad \, \f{(e^{(\theta_2 - \theta_1)L_n} - 1) \int_0^x e^{\wt f_2^{(n)}(y) - \wt f_1^{(n)}(y)}\,dy}{\int_{0}^{L_n} e^{\wt f_2^{(n)}(y) - \wt f_1^{(n)}(y)}\,dy} = \f{e^{(\theta_2 - \theta_1)L_n} - 1}{e^{(\theta_2 - \theta_1)L_n}} \f{  \int_0^x e^{\wt f_2^{(n)}(y) - \wt f_1^{(n)}(y)}\,dy}{\int_{0}^{L_n} e^{\wt f_2^{(n)}(y) - \theta_2 L_n - \wt f_1^{(n)}(y) + \theta_1 L_n}\,dy} \\
&= \f{e^{(\theta_2 - \theta_1)L_n} - 1}{e^{(\theta_2 - \theta_1)L_n}} \f{  \int_0^x e^{\wt f_2^{(n)}(y) - \wt f_1^{(n)}(y)}\,dy}{\int_{-L_n}^0 e^{\wt f_2^{(n)}(L_n+y) - \theta_2 L_n - \wt f_1^{(n)}(L_n+y) + \theta_1 L_n}\,dy} =   \f{e^{(\theta_2 - \theta_1)L_n} - 1}{e^{(\theta_2 - \theta_1)L_n}} \f{  \int_0^x e^{\wt f_2^{(n)}(y) - \wt f_1^{(n)}(y)}\,dy}{\int_{-L_n}^0 e^{\wt f_2^{(n)}(y)  - \wt f_1^{(n)}(y)}\,dy}.
\end{aligned}
\]
Then, 
\begin{align}
&\quad \, \f{(e^{(\theta_2 - \theta_1)L_n} - 1) \int_0^x e^{\wt f_2^{(n)}(y) - \wt f_1^{(n)}(y)}\,dy}{\int_{0}^{L_n} e^{\wt f_2^{(n)}(y) - \wt f_1^{(n)}(y)}\,dy} - \f{\int_{0}^x e^{f_2(y) - f_1(y)}\,dy}{\int_{-\infty}^0 e^{f_2(y) - f_1(y)}\,dy} \nonumber  \\
&=\f{e^{(\theta_2 - \theta_1)L_n} - 1}{e^{(\theta_2 - \theta_1)L_n}\int_{-L_n}^0 e^{\wt f_2^{(n)}(y) - \wt f_1^{(n)}(y)}\,dy}\int_0^x \Bigl(e^{\wt f_2^{(n)}(y) - \wt f_1^{(n)}(y)} - e^{f_2(y) - f_1(y)}\Bigr)\,dy \nonumber \\
&\qquad + \Biggl(\f{e^{(\theta_2 - \theta_1)L_n} - 1}{e^{(\theta_2 - \theta_1) L_n}\int_{-L_n}^0 e^{\wt f_2^{(n)}(y) - \wt f_1^{(n)}(y)}\,dy} - \f{1}{\int_{-\infty}^0 e^{f_2(y) - f_1(y)}\,dy}\Biggr)\int_{0}^x e^{f_2(y) - f_1(y)}\,dy. \nonumber  \\
&= \f{e^{(\theta_2 - \theta_1)L_n} - 1}{e^{(\theta_2 - \theta_1)L_n}\int_{-L_n}^0 e^{\wt f_2^{(n)}(y) - \wt f_1^{(n)}(y)}\,dy}\int_0^x e^{f_2(y) - f_1(y)}\Bigl(e^{\wt f_2^{(n)}(y) - f_2(y) + f_1(y) - \wt f_1^{(n)}(y)} - 1\Bigr)\,dy \label{606}\\
&\qquad + \f{e^{(\theta_2 - \theta_1)L_n }- 1}{e^{(\theta_2 - \theta_1)L_n}}\Biggl(\f{\int_{-\infty}^0 e^{f_2(y) - f_1(y)}\,dy - \int_{-L_n}^0 e^{\wt f_2^{(n)}(y) - \wt f_1^{(n)}(y)}\,dy}{\int_{-\infty}^0 e^{f_2(y) - f_1(y)}\,dy \int_{-L_n}^0 e^{\wt f_2^{(n)}(y) - \wt f_1^{(n)}(y)}\,dy }\Biggr)\int_0^x e^{f_2(y) - f_1(y)}\,dy \label{607}\\
&\qquad + \Biggl(\f{e^{(\theta_2 - \theta_1)L_n} - 1}{e^{(\theta_2 - \theta_1)L_n}} - 1\Biggr) \f{\int_0^x e^{f_2(y) - f_1(y)}\,dy}{\int_{-\infty}^0 e^{f_2(y) - f_1(y)}\,dy} \label{608}
\end{align}
Write the terms in \eqref{606},\eqref{607}, and \eqref{608}, as $A_{L_n,1}(x),A_{L_n,2}(x)$, and $A_{L_n,3}(x)$, respectively. It suffices to show that for $m \in \{1,2,3\}$ and every $\ve > 0$, $\limsup_{n \to \infty} \Pp(\sup_{x \in [0,M]} |A_{L_n,m}(x)| > \ve) = 0$. We show this for $A_{L,1}$, and the other terms are handled using the same estimates.
\begin{align}
&\quad \, \Pp(\sup_{x \in [0,M]} |A_{L_n,m}(x)| > \ve)  \\
&\le \Pp \biggl(\f{e^{(\theta_2 - \theta_1)L_n}\int_{-L_n}^0 e^{\wt f_2^{(n)}(y) - \wt f_1^{(n)}(y)}\,dy}{e^{(\theta_2 - \theta_1)L_n} - 1}   < \delta \Biggr) + \Biggl(\sup_{0 \le x \le M}e^{f_2(y) - f_1(y)} > D\biggr) \label{609} \\
&+ \Pp\Biggl(\sup_{0 \le x \le M}\Bigl|e^{\wt f_2^{(n)}(y) - f_2(y) + f_1(y) - \wt f_1^{(n)}(y)} - 1\Bigl| > \f{\delta \ve}{D M} \Biggr). \label{610}
\end{align}
For any choice of $\delta,\ve,D,M > 0$, as $n \to \infty$, the term \eqref{610} vanishes by the assumption \eqref{close_assumption}. Take the $\limsup$ of the right-hand side above as $n \to \infty$, then limits as $\delta \searrow 0$ and $D \to \infty$. By the convergence in probability $\int_{-L_n}^0 e^{\wt f_2^{(n)}(y) - \wt f_1^{(n)}(y)}\,dy \overset{p}{\longrightarrow} \int_{-\infty}^0 e^{f_2(y) - f_1(y)}\,dy$ (Lemma \ref{lem:BB_tail}) and since $f_1$ and $f_2$ are Brownian motions with drift, $\limsup_{n \to \infty}\Pp(\sup_{x \in [0,M]} |A_{n,m}(x)| > \ve)  = 0$.   

\medskip \noindent \textbf{Case 2: $K = [-M,0]$:} We assume $L_n > M$, and for $x \in [-M,0]$, we use the equality \eqref{602}:
\be \label{611}
\cDm^2(f_1^{(n)},f_2^{(n)})\Bigl( \f{x}{L_n}\bigr) = \wt f_{1}^{(n)}(x) + \log \Biggl(1 + \f{(e^{(\theta_1 - \theta_2) L_n}-1)\int_{0}^{x} e^{\wt f_{2}^{(n)}(y) - \wt f_{1}^{(n)}(y)}\,dy}{\int_{-L_n}^{0} e^{\wt f_{2}^{(n)}(y) - \wt f_{1}^{(n)}(y)}\,dy}  \Biggr).
\ee
Next, write \eqref{603} as
\be \label{612}
\cDm_\infty^2(f_1,f_2)(x) = f_1(x) + \log\Biggl(1 - \f{\int_0^x e^{f_2(y) - f_1(y)}\,dy}{\int_{-\infty}^0 e^{f_2(y) - f_1(y)}\,dy} 
 \Biggr).
\ee
Since $\theta_1 < \theta_2$, $e^{(\theta_1 - \theta_2)L_n} \to 0$ as $n \to \infty$. Thus, we formally see that \eqref{611} should be close to \eqref{612} with high probability as $n \to \infty$. The proof follows essentially the same proof as in the $K = [0,M]$ case. The only additional difficulty here is that we are dealing with the function $\log(1 + y)$ for $y > -1$, which is not uniformly continuous. However, $y \mapsto \log(1+ y)$ is uniformly continuous on the domain $[-1 + \delta,\infty)$ for $\delta > 0$. By continuity of the process \eqref{612} in $x$, for any $\ve ' > 0$, there exists $\delta > 0$ sufficiently small so that
\[
\Pp\Biggl(\inf_{x \in [-M,0]}\Bigl[1 - \f{\int_0^x e^{f_2(y) - f_1(y)}\,dy}{\int_{-\infty}^0 e^{f_2(y) - f_1(y)}\,dy}\Bigr] < \delta \Biggr) \le \ve'.
\]
By taking a limit as $\delta \searrow 0$, it suffices to show that, for every $\ve > 0$,
\[
\limsup_{n \to \infty} \Pp\Biggl(\sup_{x \in [-M,0]}\Bigl|\f{(e^{(\theta_1 - \theta_2) L_n}-1)\int_{0}^{x} e^{\wt f_{2}^{(n)}(y) - \wt f_{1}^{(n)}(y)}\,dy}{\int_{-L_n}^{0} e^{\wt f_{2}^{(n)}(y) - \wt f_{1}^{(n)}(y)}\,dy} - \f{\int_0^x e^{f_2(y) - f_1(y)}\,dy}{\int_{-\infty}^0 e^{f_2(y) - f_1(y)}\,dy}\Bigr| > \ve \Biggr) = 0.
\]
and this follows similarly as in the $K = [0,M]$ case.
\end{proof}

We now complete the proof of Proposition \ref{prop:width_lim}.
\begin{proof}[Proof of Proposition \ref{prop:width_lim}]
Let $\beta > 0$ and $(\theta_1,\ldots,\theta_k) \in \R^k$. Let $(L_n)_{n \in \N}$ be a sequence going to $\infty$ as $n \to \infty$. For $n \in \N$, let $(f_1^{(n)},\ldots,f_k^{(n)})$ be independent, with $f_m^{(n)} \deq \B_{\beta \sqrt L_n,L\theta_m}$ for $1 \le m \le k$, and define
\[
(g_1^{(n)},\ldots,g_k^{(n)}) := \FcDm^k(f_1^{(n)},\ldots,f_k^{(n)})
\]
so that $(g_1^{(n)},\ldots,g_k^{(n)}) \sim \Pp_{\beta\sqrt L_n}^{(L\theta_1,\ldots,L\theta_k)}$. For $1 \le m \le k$, extend $f_m^{(n)}$ and $g_{m}^{(n)}$ to functions on $\R$ by the condition that $f(x+1) - f(x) = f(1) - f(0)$ for all $x \in \R$. Then, for $1 \le m \le k$, define $\wt g_m^{(n)}(x) = g_m^{(n)}\bigl(\f{x}{L_n}\bigr)$ and  $\wt f_m^{(n)}(x) =f_m^{(n)}\bigl(\f{x}{L_n}\bigr)$. By Lemma \ref{lem:BB_to_BM}, $\wt f_m^{(n)}$ converges in distribution, with respect to the topology of uniform convergence on compact sets, to $B_{\beta,\theta_m}$ (two-sided Brownian motion with diffusivity $\beta$ and drift $\theta_m$). By Skorokhod representation (see, e.g., \cite[Thm.~11.7.2]{dudl}), on some probability space, we may couple $(\wt f_1^{(n)},\ldots,\wt f_k^{(n)})$ and a process $(f_1,\ldots,f_k) \in C(\R)^k$ of independent functions with marginal distributions $f_m \deq B_{\beta,\theta_m}$, so that, for each $1 \le m \le k$, with probability one, $\wt f_m^{(n)}$ converges uniformly on compact sets to $f_m$. In particular, for each $\ve > 0$, $1 \le m \le k$, and compact set $K\subseteq \R$,
\be \label{eq:fclose}
\limsup_{n \to \infty} \Pp\Big( \sup_{x \in K}| \wt f_m^{(n)}(x) - f_m(x)| > \ve\Big) = 0.
\ee
Define $(\wt g_1,\ldots,\wt g_k) = \FcDm_\infty(f_1,\ldots,f_k)$ so that $(\wt g_1,\ldots,\wt g_k) \deq \Pm_{\beta,\infty}^{(\theta_1,\ldots,\theta_k)}$. By definition of the maps $\FcDm$ and $\FcDm_\infty$, for $1 \le m \le k$,
\[
g_m^{(n)} = \cDm^m(f_1^{(n)},\ldots,f_m^{(n)}),\qquad \text{and}\qquad f_m = \cDm^m_\infty(f_1,\ldots,f_m),
\]
and we recall the inductive definitions $\cDm^1(f_1) = f_1$, $\cDm^m(f_1,\ldots,f_m) = \cDm^2\bigl(f_1,\cDm^2(f_2,\ldots,f_m)\bigr)$ and similarly for $\cDm_\infty^2$. Then, by inductively applying \eqref{eq:fclose} and Lemma \ref{lem:Phi_close}, we have, for $1 \le m \le k$, $\ve > 0$, and any compact set $K \subseteq \R$,
\[
\limsup_{n \to \infty} \Pp\Big( \sup_{x \in K}| \wt g_m^{(n)}(x) - \wt g_m(x)| > \ve\Big) = 0.
\]
In particular, $(\wt g_1^{(n)},\ldots,\wt g_k^{(n)})$ converges in distribution, on the space $C(\R)^k$, to $(\wt g_1,\ldots,\wt g_k)$, as desired. 
\end{proof}

\section{Gaussian process from height fluctuations}
\label{s.applicationSection}

In this section, we present the proof of Theorem~\ref{t.conG}. The only needed inputs from the previous sections is the existence of jointly invariant measures for the KPZ equation from Theorem \ref{thm:KPZ_invar_main} and some notation and results for the Green's function of the SHE from Section \ref{sec:SHE_conv_section}. In particular, in this section, we do not make use of the particular  form of the jointly invariant measures described in Theorem \ref{thm:KPZ_invar_main}. Hence, this section may be read independently from the previous sections of the paper. Recall from \eqref{e.defGamma} the process
\[
\mathcal{X}_\beta^{(\theta)}(t):=\frac{\log Z_\beta^{(\theta)}(t)-\gamma_\beta^{(\theta)} t}{\sqrt{t}}.
\]
Our goal is to show the convergence, as $t\to\infty$, of $\{\mathcal{X}_\beta^{(\theta)}(t):\theta\in\R\}$ to a limiting Gaussian process with covariance expressed in terms of $\mathcal{P}_\beta^{(0,-\theta)}$ as in \eqref{e.defA}. The idea is to approximate  $\log Z_\beta^{(\theta)}(t )-\gamma_\beta^{(\theta)}t $ by a martingale, written in the form of an It\^o integral, then to apply the martingale central limit theorem. It turns out that the stochastic integrand in the It\^o integral can be expressed in terms of the solution of the stochastic Burgers' equation with mean $-\theta$. Therefore, the formula for limiting correlation relies on the joint distribution of the stationary KPZ equation with different slopes.

By  shear invariance of $\xi$ (see Lemma \ref{l.reductiontheta0}), $\log Z_\beta^{(\theta)}(t)\deq\log Z_\beta^{(0)}(t)+\frac12\theta^2t$. Combined with \cite[Equation (5.3)]{ADYGTK22} and \cite[Proposition 4.1]{GK211}, we have as $t\to\infty$
\be \label{eq:exp_comp}
\Ee \log Z_\beta^{(\theta)}(t)=\Ee \log Z_\beta^{(0)}(t)+\frac12\theta^2 t=\gamma_\beta^{(\theta)} t+O(1),
\ee
so it is enough to show that Theorem~\ref{t.conG} holds for $\frac{1}{\sqrt{t}} (\log Z_\beta^{(\theta)}(t)-\Ee \log Z_\beta^{(\theta)}(t))$. 

The rest of the  proof consists of several steps, which we distribute to the following three sections.
 
\subsection{Stochastic integral representation}  To write $\log Z_\beta^{(\theta)}(t)-\Ee \log Z_\beta^{(\theta)}(t)$  as an It\^o integral, we resort to the Clark-Ocone formula \cite[Proposition 6.3]{CKNP21}, which applies to $\log Z_\beta^{(\theta)}(t)$ and gives 
\be \label{eq:CO}
\log Z_\beta^{(\theta)}(t)-\Ee \log Z_\beta^{(\theta)}(t)=\int_0^t\int_0^1 \Ee[\D_{s,y} \log Z_\beta^{(\theta)}(t)|\F_s] \xi(ds,dy).
\ee
Here $\D_{s,y}$ is the Malliavin derivative operator associated with the white noise $\xi$ (see \cite[Chapter 1.2.1]{Nua06}), and $\{\F_s\}_{s\geq0}$ is the natural filtration generated by $\xi$. In this subsection, we prove Proposition \ref{p.stoIn}, which gives an explicit formula for  $\Ee[\D_{s,y} \log Z_\beta^{(\theta)}(t)|\F_s]$.

We recall some notations:  $Z_\beta, Z_\beta^{\text{per}}$  were defined in Section~\ref{sec.l2converge}, see \eqref{eq:Z_b_series} and \eqref{eq:Zper_def}, and $Z_\beta$ is the Green's function of the SHE on $ \R$, see Section \ref{sec.convergeinitial}. To ease the notation, we let 
\begin{equation}\label{e.defGbeta}
Z_\beta^{\text{per}}(t,y\viiva s,x):=Z_\beta^{\text{per}}(t,y \viiva s,x;0)=\sum_{n\in\Z}  Z_\beta(t,y\viiva s,x+n),
\end{equation}
which, by Proposition \ref{prop:solve_SHE}\ref{itm:Z_conv}, may be viewed  as the Green's function of the SHE on the unit torus. To write \eqref{eq:CO} more explicitly in the case $\theta=0$, a direct calculation combined with the equation satisfied by $Z_\beta$ and Tonelli's theorem gives
\[
\begin{aligned}
\D_{s,y}\log Z_\beta^{(0)}(t)=\frac{\D_{s,y}Z_\beta^{(0)}(t)}{Z_\beta^{(0)}(t)}&=\beta\frac{\sum_{n\in\Z}\int_{\R} Z_\beta(t,x\viiva s,y+n)dx\cdot Z_\beta (s,y+n \viiva 0,0)}{\int_{\R} Z_\beta (t,x\viiva 0,0)dx}\\
&=\beta\frac{\int_0^1 Z_\beta^{\text{per}} (t,x \viiva s,y)dx \cdot Z_\beta^{\text{per}} (s,y\viiva 0,0)}{\int_0^1\int_0^1  Z_\beta^{\text{per}}(t,x \viiva s,y') Z_\beta^{\text{per}} (s,y'\viiva 0,0)dxdy'}.
\end{aligned}
\]
To obtain the second line, we used the fact that $\int_{\R} Z_\beta(t,x|s,y+n)dx=\int_0^1 Z_\beta^{\text{per}}(t,x|s,y)dx$, which comes from the periodicity of the noise and the following calculation
\[
\begin{aligned}
&\int_{\R} Z_\beta(t,x|s,y+n)dx=\int_{\R} Z_\beta(t,x-n|s,y)dx=\int_{\R} Z_\beta(t,x|s,y)dx\\
&=\sum_{k\in\Z} \int_0^1 Z_\beta(t,x+k|s,y)dx=\int_0^1 Z_\beta^{\text{per}}(t,x|s,y)dx.
\end{aligned}
\]
From the expression and the strict positivity of $Z_\beta^{\text{per}}$ (which can be derived following the argument of \cite{HL22}, see the proof of Proposition \ref{prop:solve_SHE}\ref{itm:SHE_soln}),
we conclude that $\D_{s,y}\log Z_\beta^{(0)}(t)>0$ and 
$\int_0^1 \D_{s,y}\log Z_\beta^{(0)}(t) dy=\beta$. As a matter of fact, $\beta^{-1}\D_{s,y} \log Z_\beta^{(0)}(t)$ can be interpreted as the quenched midpoint density of the (point-to-line) directed polymer evaluated at $(s,y)$, if we view the polymer path as lying on a cylinder.  For non-zero $\theta$, the calculation of $\D_{s,y}\log Z_\beta^{(\theta)}(t)$ is more involved. To state the result,  define 
\begin{equation}\label{e.defhu}
\begin{aligned}
h^{(\theta)}(s,y)&:= \log Z_\beta^{\text{per}}(s,y \viiva 0,0;\theta) = \log \sum_{n\in \Z} Z_\beta(s,y \viiva 0,n)e^{\theta n},\qquad\text{and}\\
u^{(\theta)}(s,y)&=\partial_y h^{(\theta)}(s,y),
\end{aligned}
\end{equation}   
where $\partial_y h^{(\theta)}(s,y)$ is taken in the distributional sense. Since $ Z_\beta^{\text{per}}(s,y \viiva 0,0;\theta)$ solves \eqref{eq:SHE}, one may view $h^{(\theta)}$ as the solution to \eqref{eq:KPZ},  
and $u^{(\theta)}$ as the solution to \eqref{eq:SBE}. Note that we have kept the dependence of $h^{(\theta)},u^{(\theta)}$ on the parameter $\beta$ implicit. 

The following is the main result of this subsection.
\begin{proposition}\label{p.stoIn}
We have 
\begin{equation}\label{e.co}
\log Z_\beta^{(\theta)}(t)- \Ee \log Z_\beta^{(\theta)}(t)=\beta \int_0^t\int_0^1 \Ee  [f^{(\theta)}(t,s,y)|\F_s]\xi(ds,dy),
\end{equation}
with
\begin{equation}\label{e.deff}
f^{(\theta)}(t,s,y):=\frac{ \int_0^1 Z_\beta^{\text{per}} (t,x\viiva s,y)dx \cdot \exp(\theta y+ h^{(-\theta)}(s,y))}{\int_0^1(\int_0^1 Z_\beta^{\text{per}}(t,x \viiva s,y')dx) \cdot \exp(\theta y'+ h^{(-\theta)}(s,y'))dy'}.
\end{equation}
\end{proposition}
\begin{proof}
By \cite[Proposition 2.2]{GK23}, we have $\Ee[\D_{s,y} \log Z_\beta^{(\theta)}(t)|\F_s] = \Ee[f^{(\theta)}(t,s,y)|\F_s]$, where
\[
f^{(\theta)}(t,s,y):= \frac{\int_0^1 Z_\beta^{\text{per}}(t,x \viiva s,y)dx \sum_{n\in\Z} e^{\theta(y+n)}Z_\beta (s,y+n \viiva 0,0)}{\int_{\R} (\int_0^1 Z_\beta^{\text{per}} (t,x \viiva s,y')dx)e^{\theta y'}Z_\beta (s,y'\viiva 0,0)dy'} .
\]
By \eqref{eq:CO}, it suffices to show that $f^{(\theta)}$ can be written as in \eqref{e.deff}. 
First, one can check that $\int_0^1 f^{(\theta)}(t,s,y)dy=1$, so it can be viewed as a (random) probability density on $\bT$. Then one can  use the $h^{(\theta)}$ defined in \eqref{e.defhu} to rewrite the numerator in the above expression, and using the fact that the denominator is the integral over $\bT$ of the numerator, we obtain \eqref{e.deff}, completing the proof. 
 \end{proof}
 
 \begin{remark}
On a formal level, we can view $h^{(\theta)}$ as the solution to the KPZ equation with the initial data $h^{(\theta)}(0,y)$ given by $\log \sum_{n\in\Z} \delta(y-n)e^{\theta n}$, which, roughly speaking, can be viewed as the interface with an ``average slope'' $\theta$. Therefore, the corresponding Burgers' solution $u^{(\theta)}$  has the ``average'' $\theta$. For different values of $\theta$, the process $ h^{(\theta)}$ solves \eqref{eq:KPZ} driven by the same noise $\beta \xi$, but with different initial conditions. Thus, in the long time limit of $s\to\infty$, we expect that $\{h^{(\theta)}(s,\cdot) - h^{(0)}(s,0)\}_\theta$ converges to the jointly invariant measures obtained in Theorem~\ref{thm:KPZ_invar_main}. 
 \end{remark}

 \begin{remark}\label{r.proofconjecture}
     To prove Conjecture~\ref{conj1},  one needs to take the (high order) derivatives of $\log Z_\beta^{(\theta)}(t)$ with respect to $\theta$. A starting point could be the above Clark-Ocone representation in \eqref{e.co}. In other words, we differentiate with respect to $\theta$ on both sides of \eqref{e.co}, and proceed from there. From the expression of $f^{(\theta)}$ in \eqref{e.deff}, it is clear that one needs to understand how the long time behavior of the recentered version of $h^{(\theta)}$ depends on  $\theta$. This leads us to study the jointly invariant measures in  Theorem~\ref{thm:KPZ_invar_main}. In the proof below, we use elements of the proof of the 1F1S principle in Theorem \ref{thm:1f1s}, adapted to this specific situation where we have a sum of delta initial conditions. 
 \end{remark}

\subsection{Approximation} 
In this section, we construct an approximation of the stochastic integrand $f^{(\theta)}(t,s,y)$, which appeared in the Clark-Ocone representation \eqref{e.co}. In its expression, one can view 
\[
\frac{\int_0^1 Z_\beta^{\text{per}}(t,x\viiva s,y)dx}{\int_0^1\int_0^1 Z_\beta^{\text{per}}(t,x\viiva s,y')dxdy'}
\]
as the endpoint density  of the directed polymer of length $t-s$, running backwards in time on the cylinder, with the starting point sampled from the uniform distribution at time $t$. For $t-s\gg1$, the endpoint density evaluated at $(s,y)$ can be approximated by $\exp(\beta \B(y))/\int_0^1 \exp(\beta \B(y'))dy'$, where $\B$ is a standard Brownian bridge connecting $(0,0)$ and $(1,0)$. For the other factor $\exp\bigl(\theta y + h^{(-\theta)}(s,y)\bigr)$, we expect that for large $s\gg1$, $\{h^{(-\theta)}(s,\cdot) - h^{(-\theta)}(s,0)\}_{\theta}$  is close in law to the periodic KPZ horizon $\mathcal{P}_\beta$.

We introduce a few notations   on the endpoint densities
of the ``forward'' and ``backward'' polymer path on a cylinder. Let $\mathcal{M}_1(\mathbb{T})$ be the set of Borel probability
measures  on $\bT$. 
For
any $\nu\in{\mathcal M}_1(\bT)$ and $t>s$, let
\begin{equation}\label{e.forwardbackward}
\begin{aligned}
&\rho_{\mathrm{f}}(t,x \viiva s,\nu):=\frac{\int_0^1 Z_\beta^{\text{per}}(t,x \viiva s,y)\nu(dy)}{\int_0^1\int_0^1 Z_\beta^{\text{per}}  (t,x'\viiva s,y)\nu(dy)dx'}, \quad x\in\bT,\\
&\rho_{\mathrm{b}}(t,\nu\viiva s,y):=\frac{\int_0^1Z_\beta^{\text{per}} (t,x\viiva s,y)\nu(dx)}{\int_0^1\int_0^1 Z_\beta^{\text{per}} (t,x \viiva s,y')\nu(dx)dy'},\quad y\in\bT.
\end{aligned}
\end{equation}
In other words, for $\rho_{\mathrm{f}}$, $s$ is the initial time and $t$ is the terminal time, while for $\rho_{\mathrm{b}}$, it is the other way around.
Here the subscripts ``$\mathrm{f}, \mathrm{b}$'' represent ``forward'' and ``backward'' respectively. 
Using the new notation, we can rewrite $f^{(\theta)}$ from \eqref{e.deff} as follows, where ``$\mathfrak{m}$'' is the Lebesgue measure on $\bT$:
\begin{equation}\label{e.fthetanew}
f^{(\theta)}(t,s,y)=\frac{ \rho_{\mathrm{b}}(t,\mathfrak{m} \viiva s,y) \cdot \exp(\theta y+h^{(-\theta)}(s,y)-h^{(-\theta)}(s,0))}{\int_0^1\rho_{\mathrm{b}}(t,\mathfrak{m} \viiva s,y')\cdot \exp(\theta y'+h^{(-\theta)}(s,y')-h^{(-\theta)}(s,0))dy'}.
\end{equation}

To construct the approximation of $f^{(\theta)}(t,s,y)$, we let $\{H^{(\theta)}(t,x), t\geq0, x\in\bT\}_{\theta}$ be jointly stationary solutions  to the KPZ equation: Let $\{H^{(\theta)}(0,\cdot)\}_\theta$ be sampled from the measure $\mathcal P_\beta$ defined in Corollary \ref{cor:ICH_process}, independent from $\xi$, and let $H^{(\theta)}$ solve \eqref{eq:KPZ} with this initial condition. In other words, $h^{(\theta)}$ and $H^{(\theta)}$ solve the   equation with the same noise, but with different initial data.

We will use $\Ee_{\mathcal P_\beta}$ to denote the expectation only with respect to the initial data $\{H^{(\theta)}(0,\cdot)\}_\theta$, while $\Ee$ remains the expectation on the noise $\xi$ only.

Assume $\B$ is a Brownian bridge connecting $(0,0)$ and $(1,0)$, independent of $\{H^{(\theta)}(0,\cdot)\}_\theta$ and $\xi$, and let $\nu_\B$ be the (random) probability measure on $\bT$ such that 
\[
\nu_\B(dy)=\frac{\exp(\beta \B(y))}{\int_0^1 \exp(\beta \B(y'))dy'}dy.
\] 

Using $\nu_\B$ as the initial distribution of the backward polymer and replacing $h^{(-\theta)}$ by $H^{(-\theta)}$, we define the approximation of $f^{(\theta)}$   as   
\begin{equation}\label{e.deftildef}
F^{(\theta)}(t,s,y)= \frac{ \rho_{\mathrm{b}}(t,\nu_\B \viiva s,y) \cdot \exp(\theta y+H^{(-\theta)}(s,y)-H^{(-\theta)}(s,0))}{\int_0^1\rho_{\mathrm{b}}(t,\nu_\B \viiva s,y')\cdot \exp(\theta y'+H^{(-\theta)}(s,y')-H^{(-\theta)}(s,0))dy'},
\end{equation}
which we will show to be ``close'' to $f^{(\theta)}(t,s,y)$, in the regime of $t-s\gg1, s\gg1$. 

For any given $\theta\in\R$ and $t>s$, using the invariance of the Brownian bridge under the KPZ evolution, one can check that 
\begin{equation}\label{e.flaw}
\{F^{(\theta)}(t,s,y)\}_{y\in\bT}\deq \left\{\frac{\exp(\beta \B(y))\exp(\beta \tilde{\B}(y))}{\int_0^1 \exp(\beta \B(y'))\exp(\beta \tilde{\B}(y')dy'}\right\}_{y\in\bT},
\end{equation}
where $\tilde{\B}$ is an independent copy of $\B$. As a matter of fact, this comes from (i) for any $\theta\in\R$ and $s\geq0$, the process $\{\theta y +H^{(-\theta)}(s,z)-H^{(-\theta)}(s,0)\}_{z\in[0,1]}$  has the same distribution as $\beta \tilde{\B}$; (ii) for any $t\geq s$, we have 
\[
\{\rho_{\mathrm{b}}(t,\nu_\B\viiva s,y)\}_{y\in\bT}\deq\left\{\frac{\exp(\beta \B(y))}{\int_0^1 \exp(\beta \B(y'))dy'}\right\}_{y\in\bT}.
\]
Further, define 
\begin{equation}\label{e.defM}
M^{(\theta)}(t)=\beta \int_0^t\int_0^1 \Ee_{\B}\Ee[F^{(\theta)}(t,s,y)|\F_s]\xi(s,y)dyds,
\end{equation}
where $\Ee_\B$ is the expectation on $\B$ only. It is worth noting the multiple sources of randomnesses in the above expression: the Brownian bridge $\B$, the noise $\xi$, and the initial data $H^{(\theta)}(0,\cdot)$, so here we stress that  the It\^o integral is defined for almost  every realization of $H^{(\theta)}(0,\cdot)$.

Dividing by $t$ in the following result shows that $\f{M^{(\theta)}(t)}{t}$ approximates $\f{\log Z_\beta^{(\theta)}(t)-\Ee \log Z_\beta^{(\theta)}(t)}{t}$ for large $t$. It relies on estimates from \cite{GK21} showing mixing of the polymer endpoints from different initial conditions.

\begin{proposition}\label{p.maapp}
There exists a constant $C>0$ independent of $t\geq1$ such that 
\[
\Ee_{\mathcal P_\beta}\Ee \left|\log Z_\beta^{(\theta)}(t)-\Ee \log Z_\beta^{(\theta)}(t)-M^{(\theta)}(t)\right|^2 \leq C.
\]
\end{proposition}
\begin{proof} 
 Define 
\begin{align*}
\tilde{f}^{(\theta)}(t,s,y)&:=\frac{ \rho_{\mathrm{b}}(t,\mathfrak{m} \viiva s,y) \cdot \exp(\theta y+\tilde{h}^{(-\theta)}(s,y)-\tilde{h}^{(-\theta)}(s,0))}{\int_0^1\rho_{\mathrm{b}}(t,\mathfrak{m} \viiva s,y')\cdot \exp(\theta y'+\tilde{h}^{(-\theta)}(s,y')-\tilde{h}^{(-\theta)}(s,0))dy'},\\
\tilde{F}^{(\theta)}(t,s,y)&:=\frac{ \rho_{\mathrm{b}}(t,\nu_\B \viiva s,y)  \exp(\theta y+\tilde{H}^{(-\theta)}(s,y)-\tilde{H}^{(-\theta)}(s,0))}{\int_0^1 \rho_{\mathrm{b}}(t,\nu_\B\viiva s,y')  \exp(\theta y'+\tilde{H}^{(-\theta)}(s,y')-\tilde{H}^{(-\theta)}(s,0))dy'},
\end{align*}
where 
\begin{equation}\label{e.deftildeh}
\begin{aligned}
\tilde{h}^{(\theta)}(t,x)&:=\theta x+\frac12\theta^2 t+\log \sum_{n\in\Z} Z_\beta(t,x\viiva 0,n-\theta t), \qquad\text{and}\\
\tilde{H}^{(\theta)}(t,x)&:=\theta x+\frac12\theta^2 t+\log \int_{\R} Z_\beta (t,x \viiva 0,y)\exp\bigl(H^{(\theta)}(0,y+\theta t)-\theta(y+\theta t)\bigr) dy.
\end{aligned}
\end{equation}
By Lemma~\ref{l.equallaw} below,  for any fixed $\theta$ and $t>s$, we have
\begin{equation}\label{e.flawnew}
\bigg(\{f^{(\theta)}(t,s,y)\}_{y\in\bT}, \{F^{(\theta)}(t,s,y)\}_{y\in\bT}\bigg)\deq\bigg(\{\tilde{f}^{(\theta)}(t,s,y)\}_{y\in\bT}, \{\tilde{F}^{(\theta)}(t,s,y)\}_{y\in\bT}\bigg).
\end{equation}
Thus, taking the difference between \eqref{e.co} and \eqref{e.defM} and  applying the It\^o isometry, we obtain
\begin{equation}\label{e.itoisometry}
\begin{aligned}
&\Ee_{\mathcal P_\beta}\Ee \left|\log Z_\beta^{(\theta)}(t)-\Ee \log Z_\beta^{(\theta)}(t)-M^{(\theta)}(t)\right|^2\\
&=\beta^2\int_0^t\int_0^1 \Ee_{\mathcal P_\beta}\Ee\big[ |\Ee[\tilde{f}^{(\theta)}(t,s,y)|\F_s]-\Ee_{\B}\Ee [\tilde{F}^{(\theta)}(t,s,y)|\F_s]|^2\big] dyds.
\end{aligned}
\end{equation}

The rest of the proof consists of two steps. 

\medskip
  \emph{Step 1, rewrite $\tilde{f}^{(\theta)}$ and $\tilde{F}^{(\theta)}$ using the polymer endpoint density.} To estimate the right-hand side of \eqref{e.itoisometry}, we will   rewrite the factors $\exp(\theta y+\tilde{h}^{(-\theta)}(s,y)-\tilde{h}^{(-\theta)}(s,0))$ and $\exp(\theta y+\tilde{H}^{(-\theta)}(s,y)-\tilde{H}^{(-\theta)}(s,0))$ in terms of the forward polymer endpoint. By definition, we have  
\[
\theta y+\tilde{h}^{(-\theta)}(s,y)-\tilde{h}^{(-\theta)}(s,0)=\log \sum_{n\in\Z} Z_\beta (s,y\viiva 0,n+\theta s)-\log \sum_{n\in\Z} Z_\beta (s,0\viiva 0,n+\theta s),
\]
and
\[
\begin{aligned}
\theta y+\tilde{H}^{(-\theta)}(s,y)-\tilde{H}^{(-\theta)}(s,0)=&\log \int_{\R} Z_\beta(s,y\viiva 0,y')\exp\bigl(H^{(-\theta)}(0,y'-\theta s)+\theta(y'-\theta s)\bigr) dy'\\
&-\log \int_{\R} Z_\beta (s,0\viiva 0,y')\exp\bigl(H^{(-\theta)}(0,y'-\theta s)+\theta(y'-\theta s)\bigr) dy'.
\end{aligned}
\]
One can further rewrite the summation as
\[
\sum_{n\in\Z} Z_\beta(s,y \viiva 0,n+\theta s)=\sum_{n\in\Z} Z_\beta (s,y+n \viiva 0,\theta s-\lfloor \theta s\rfloor)=Z_\beta^{\text{per}}(s,y\viiva 0,\theta s-\lfloor \theta s\rfloor).
\]
Similarly, since $H^{(-\theta)}(0,y)+\theta y$ is a periodic function, we have 
\[
\begin{aligned}
&\int_{\R} Z_\beta (s,y \viiva 0,y')\exp\bigl(H^{(-\theta)}(0,y'-\theta s)+\theta(y'-\theta s)\bigr) dy'\\
&=\int_0^1 Z_\beta^{\text{per}}(s,y \viiva 0,y')\exp\bigl(H^{(-\theta)}(0,y'-\theta s)+\theta(y'-\theta s)\bigr) dy'.
\end{aligned}
\]
For any $\theta$ and $s$, define the following two probability measures on $\bT$: 
\[
\mu_{\theta,s}(dy)=\delta(y-(\theta s-\lfloor \theta s\rfloor)), \quad\text{and}\quad \pi_{\theta,s}(dy)=\frac{\exp(H^{(-\theta)}(0,y-\theta s)+\theta(y-\theta s))}{\int_0^1 \exp(H^{(-\theta)}(0,y'-\theta s)+\theta(y'-\theta s)) dy'}dy.
\]
Then, we can rewrite $\tilde{f}^{(\theta)},\tilde{F}^{(\theta)}$ as 
\begin{equation}\label{e.newexfF}
\begin{aligned}
\tilde{f}^{(\theta)}(t,s,y)=\frac{ \rho_{\mathrm{b}}(t,\mathfrak{m} \viiva s,y) \cdot \rho_{\mathrm{f}}(s,y\viiva 0,\mu_{\theta,s})}{\int_0^1 \rho_{\mathrm{b}}(t,\mathfrak{m}\viiva s,y') \cdot \rho_{\mathrm{f}}(s,y'\viiva 0,\mu_{\theta,s})dy'},\\
\tilde{F}^{(\theta)}(t,s,y)=\frac{ \rho_{\mathrm{b}}(t,\nu_{\B}\viiva s,y) \cdot \rho_{\mathrm{f}}(s,y\viiva 0,\pi_{\theta,s})}{\int_0^1 \rho_{\mathrm{b}}(t,\nu_{\B}\viiva s,y') \cdot \rho_{\mathrm{f}}(s,y'\viiva 0,\pi_{\theta,s})dy'}.
\end{aligned}
\end{equation}

\emph{Step 2, error estimates in \eqref{e.itoisometry}.} Compare the expressions of $\tilde{f}^{(\theta)}$ and $\tilde{F}^{(\theta)}$ in \eqref{e.newexfF}, we see that the only differences come from the initial distributions in $\rho_{\mathrm{b}}$ and $\rho_{\mathrm{f}}$. Since the polymer endpoint mixes exponentially fast, we expect that the error $\tilde{f}^{(\theta)}(t,s,y)-\tilde{F}^{(\theta)}(t,s,y)$ to be exponentially small in $s\gg1$ and $t-s\gg1$. To make the heuristics precise, we make use of the following results obtained in \cite{GK21}, summarized in \cite[Proposition A.1]{GK23}: (i) for any $\nu\in \mathcal{M}_1(\mathbb{T})$ and $t>s$, we have 
\begin{equation}\label{e.rhofb}
\{\rho_{\mathrm{b}}(t,\nu\viiva s,y)\}_{y\in\bT}\deq\{\rho_{\mathrm{f}}(t,y\viiva s,\nu)\}_{y\in\bT},
\end{equation}
(ii) for any $p\geq 1$, there exist $C,\lambda>0$ such that for all $t\geq1$, 
\begin{equation}
  \label{051705-23}
\Ee \sup_{\mu,\nu\in\mathcal{M}_1(\mathbb{T})} \sup_{x\in\mathbb{T}} |\rho_{\mathrm{f}}(t,x\viiva 0,\mu)-\rho_{\mathrm{f}}(t,x\viiva 0,\nu)|^p  \leq Ce^{-\lambda t}, 
\end{equation}
and
\begin{equation}\label{e.mmbdrho}
\Ee \sup_{\nu\in \mathcal{M}_1(\mathbb{T})}\sup_{x\in\mathbb{T}}\, \{\rho_{\mathrm{f}}(t,x\viiva 0,\nu)^p
+\rho_{\mathrm{f}}(t,x\viiva 0,\nu)^{-p}\} \leq C.
\end{equation} 
Without loss of generality, we assume $t\geq10$ and decompose the integral on the right-hand side of \eqref{e.itoisometry} into two parts:
\[
\begin{aligned}
&\beta^{-2}\Ee_{\mathcal P_\beta}\Ee \left|\log Z_\beta^{(\theta)}(t)-\Ee \log Z_\beta^{(\theta)}(t)-M^{(\theta)}(t)\right|^2\\
&=\left(\int_1^{t-1}+\int_{[0,t]\setminus[1,t-1]}\right)\int_0^1\Ee_{\mathcal P_\beta}\Ee\big[ |\Ee[\tilde{f}^{(\theta)}(t,s,y)|\F_s]-\Ee_{\B}\Ee [\tilde{F}^{(\theta)}(t,s,y)|\F_s]|^2\big]dyds=:I_{1,t}+I_{2,t}.
\end{aligned}
\]
To complete the proof of lemma, it suffices to show that $I_{1,t}$ and $I_{2,t}$ are each bounded by a constant $C$, independent of $t$. 

For $I_{1,t}$, by \eqref{e.rhofb}, \eqref{051705-23} and \eqref{e.mmbdrho}, it is straightforward to derive that, for every realization of $\nu_{\B}$ and $\pi_{\theta,s}$ (note that they are random probability measures), we have 
\[
\Ee \,\sup_{y\in\bT}|\tilde{f}^{(\theta)}(t,s,y)-\tilde{F}^{(\theta)}(t,s,y)|^2 \leq C(e^{-\lambda(t-s)}+e^{-\lambda s}), \quad\quad s\in[1,t-1],
\]
and this implies (with an application of Jensen's inequality for the conditional expectation)
\[
I_{1,t} \leq C\int_1^{t-1}(e^{-\lambda(t-s)}+e^{-\lambda s}) ds\leq C/\lambda.
\]

For $I_{2,t}$, we estimate the two terms associated with $\tilde{f}^{(\theta)}$ and $\tilde{F}^{(\theta)}$ separately. For $\tilde{F}^{(\theta)}$, by \eqref{e.flawnew} and an application of Jensen's inequality, we have 
\[
\int_{[0,t]\setminus[1,t-1]}\int_0^1 \Ee_{\mathcal P_\beta}\EE |\Ee_{\B}\EE [\tilde{F}^{(\theta)}(t,s,y)|\F_s]|^2 dyds \leq \int_{[0,t]\setminus[1,t-1]}\int_0^1\Ee_{\mathcal P_\beta}\EE\Ee_{\B} |F^{(\theta)}(t,s,y)|^2 dyds.
\]
Further applying \eqref{e.flaw}, we conclude that the above term is bounded independent of $t$. 
For $\tilde{f}^{(\theta)}$, there is  a singularity at $s$ near $0$, so we deal with the integration domain $s\in[0,1]$ and $s\in[t-1,t]$ separately:

(i) For $ s\in[t-1,t]$, we use the expression in \eqref{e.newexfF}, and apply \eqref{e.mmbdrho} and H\"older inequality to obtain   
\[
\EE |\tilde{f}^{(\theta)}(t,s,y)|^2 \leq C\sqrt{\EE | \rho_{\mathrm{b}}(t,\mathfrak{m}\viiva s,y)|^4} \leq C,
\]
where the second inequality comes from the standard positive and negative moment estimates of \\ $\int_0^1Z_\beta^{\text{per}} (t,x\viiva s,y)dx$, which can be found e.g. in \cite[Lemma B.2]{GK21}. 

(ii) For $s\in[0,1]$, by Jensen's inequality (noting that $\rho_{\mathrm{f}}(s,y'\viiva 0,\mu_{\theta,s})dy'$ is a probability measure), we have
\[
\begin{aligned}
 \Bigl|\int_0^1 \rho_{\mathrm{b}}(t,\mathfrak{m}\viiva s,y') \cdot \rho_{\mathrm{f}}(s,y'\viiva 0,\mu_{\theta,s})dy'\Bigr|^{-p} &\leq  \int_0^1 |\rho_{\mathrm{b}}(t,\mathfrak{m}\viiva s,y')|^{-p}\rho_{\mathrm{f}}(s,y'\viiva 0,\mu_{\theta,s})dy'\\
&\leq  \sup_{y'\in\bT} |\rho_{\mathrm{b}}(t,\mathfrak{m}\viiva s,y')|^{-p},
\end{aligned}
\]
which is bounded after taking the expectation, in light of \eqref{e.mmbdrho}. Next, from \eqref{e.newexfF}, we apply H\"older's inequality and conclude that
\[
\|\tilde{f}^{(\theta)}(t,s,y)\|_2 \leq C\|\rho_{\mathrm{f}}(s,y\viiva 0,\mu_{\theta,s})\|_4,
\]
where $\| \cdot \|_2$ and $\|\cdot \|_4$ denote the $L^2$ and $L^4$ norms with respect to all sources of randomness. 
By \cite[Equation (3.22)]{GK23}, the right-hand side is bounded from above by $\sum_n \rho(s,n+y-(\theta s-\lfloor \theta s\rfloor))$, where $\rho(t,x)$ is the standard Gaussian kernel. Thus, we have 
\[
\int_0^1\int_0^1  \EE |\tilde{f}^{(\theta)}(t,s,y)|^2 dyds \leq C\int_0^1 \int_0^1 \Bigl|\sum_n \rho(s,n+y-(\theta s-\lfloor \theta s\rfloor))\Bigr|^2 dyds \leq C.
\]
The proof is complete.
\end{proof}

\begin{lemma}\label{l.equallaw}
Fix $\theta\in\R$ and $t>0$, we have 
\[
\bigg(\{h^{(\theta)}(t,x)\}_{x\in[0,1]}, \{H^{(\theta)}(t,x)\}_{x\in[0,1]}\bigg)\deq\bigg(\{\tilde{h}^{(\theta)}(t,x)\}_{x\in[0,1]}, \{\tilde{H}^{(\theta)}(t,x)\}_{x\in[0,1]}\bigg),
\]
where $\tilde h^{(\theta)}$ and $\tilde H^{(\theta)}$ are defined as in \eqref{e.deftildeh}.
\end{lemma}

\begin{proof}
Recall from \eqref{e.defhu} that $h^{(\theta)}(t,x)=\log \sum_{n\in \Z} Z_\beta(t,x\viiva 0,n)e^{\theta n}$ and $H^{(\theta)}$ solves the KPZ equation with the initial data $H^{(\theta)}(0,\cdot)$ sampled from $\mathcal{P}_\beta$. The proof is done by  approximation. Let $Z_\beta^\ve$ be the Green's function of SHE with $\xi$ replaced by $\xi_\ve$, which is a spatial mollification of $\xi$: consider the  mollifier $\phi_\ve(x)=\ve^{-1}\phi(x/\ve)$ with an even $\phi\in C_c^\infty(\R)$, define $\xi_\ve(t,x)=\int_{\R} \phi_\ve(x-y)\xi(t,y)dy$, so its covariance function is 
\[
\Ee\big[ \xi_\ve(t,x)\xi_\ve(s,y)\big]=\delta(t-s)r_\ve(x-y):=\delta(t-s)\sum_{n\in\Z} \phi_\ve\star\phi_\ve(x-y+n).
\]
In the following, we define the approximations of $h^{(\theta)}, H^{(\theta)}$. 

(i) Define $h^{(\theta)}_{\ve}(t,x)=\log \sum_{n\in \Z} Z_\beta^\ve(t,x\viiva 0,n)e^{\theta n}$. By the Feynman-Kac formula, we can write it as 
\[
h^{(\theta)}_{\ve}(t,x)= \log \sum_{n\in\Z} \Ee_B \left[\exp(\beta \int_0^t \xi_\ve(t-s, x+B_s)ds-\frac12\beta^2r_\ve(0)t)\delta(x+B_t-n)e^{\theta n}\right],
\]
where $\Ee_B$ is only on the Brownian motion $B$. By the Cameron-Martin theorem, the above expression can be rewritten as 
\begin{equation}\label{e.hthetaeps}
\begin{aligned}
h_\ve^{(\theta)}(t,x)&= \log \sum_{n\in\Z} \Ee_B \left[\exp(\beta \int_0^t \xi_\ve(t-s, x+B_s)ds-\frac12\beta^2r_\ve(0)t)\delta(x+B_t-n)e^{\theta (x+B_t)}\right]\\
&=\theta x+\frac12\theta^2 t+\log \sum_{n\in\Z} \Ee_B \left[\exp(\beta\int_0^t \xi_\ve(t-s, x+\theta s+B_s)ds-\frac12\beta^2r_\ve(0)t)\delta(x+\theta t+B_t-n) \right].
\end{aligned}
\end{equation}

(ii) To approximate $H^{(\theta)}(t,x)$, we define $H^{(\theta)}_{\ve}(t,x)=\log \int_{\R} Z_\beta^\ve(t,x \viiva 0,y)e^{H^{(\theta)}(0,y)}dy$. Repeating the above argument, we derive
\begin{equation}\label{e.Hthetaeps}
\begin{aligned}
&H^{(\theta)}_{\ve}(t,x)=\theta x+\frac12\theta^2 t\\
&+\log \Ee_B\left[\exp(\beta \int_0^t \xi_\ve(t-s,x+\theta s+B_s)ds-\frac12\beta^2r_\ve(0)t)\exp(H^{(\theta)}(0,x+\theta t+B_t)-\theta(x+\theta t+B_t))\right].
\end{aligned}\end{equation}
For any fixed $t>0$, $\ve>0$ and $\theta\in\R$, the following two Gaussian processes have the same law:
\[
\{\xi_\ve(t-s,\theta s+y)\}_{s\in\R,y\in\R}\deq\{\xi_\ve(t-s,y)\}_{s\in\R,y\in\R},
\]
which can be checked e.g. by comparing the covariance functions. 
Therefore, we can replace the  $\{\xi_\ve(t-s,\theta s+y)\}_{s\in\R,y\in\R}$ by 
$\{\xi_\ve(t-s,y)\}$ in \eqref{e.hthetaeps} and \eqref{e.Hthetaeps}, without changing the joint law of $h^{(\theta)}_{\ve}$ and $H^{(\theta)}_{\ve}$, then it suffices to pass to the limit of $\ve\to0$ to complete the proof.
 \end{proof}

\subsection{Convergence of finite dimensional distributions and tightness} In this section, we  complete the proof of Theorem \ref{t.conG} by showing the convergence, as $t \to \infty$, of the $C(\R)$-valued process $\{\mathcal X_\beta^{(\theta)}:\theta \in \R\}$. This follows by convergence of finite-dimensional distributions (Lemma \ref{lem:fdd}), and by tightness (Lemma \ref{lem:tight}). By \eqref{eq:exp_comp} and by dividing by $t$ in \eqref{p.maapp}, for the finite-dimensional convergence, it suffices to consider the process $\bigl\{\f{M^{(\theta)}(t)}{t}:\theta \in \R\bigr\}$.

Before proving Lemma \ref{lem:fdd}, we make a few observations. From the expression in \eqref{e.defM}, it is not   clear if $M^{(\theta)}(t)$ is a martingale, since the stochastic integrand depends on $t$. However, because 
\[
\{\rho_{\mathrm{b}}(t,\nu_\B;s,y)\}_{y\in\bT}\deq \left\{\frac{\exp(\beta \B(y))}{\int_0^1 \exp(\beta\B(y'))dy'}\right\}_{y\in\bT},
\]
we conclude immediately that
\begin{equation} \label{eq:fths}
\Ee_{\B}\EE[F^{(\theta)}(t,s,y)|\F_s]=\Ee_{\B}\left[\frac{ \exp(\beta\B(y)) \cdot \exp(\theta y+H^{(-\theta)}(s,y)-H^{(-\theta)}(s,0))}{\int_0^1\exp(\beta\B(y'))\cdot \exp(\theta y'+H^{(-\theta)}(s,y')-H^{(-\theta)}(s,0))dy'}\right]=:\mathsf{f}^{(\theta)}(s,y).
\end{equation}
By our choice of $\{H^{(\theta)}(0,\cdot)\}_{\theta}$, we have, for any $\theta_1,\theta_2\in\R$,
\be \label{eq:statH}
\begin{aligned}
\quad \, \Bigl(H^{(\theta_1)}(s,y) -H^{(\theta_1)}(s,0),H^{(\theta_2)}(s,y) -H^{(\theta_2)}(s,0): y \in [0,1] \Bigr)_{s\geq0} 
\end{aligned}
\ee
is a stationary process. 
Hence, $\{\mathsf{f}^{(\theta)}(s,y):y\in[0,1]\}_{s\geq0}$ is stationary in the $s-$variable. Therefore, $M^{(\theta)}(t)$ can be rewritten as 
\begin{equation}\label{e.newexM}
M^{(\theta)}(t)=\beta\int_0^t\int_0^1 \mathsf{f}^{(\theta)}(s,y) \xi(ds,dy).
\end{equation}
which is a martingale with stationary increments. This allows us to use a result from \cite[Theorem 2.1]{TKCLSO12} in the proof below to show that we obtain a Gaussian process in the limit.
\begin{lemma} \label{lem:fdd}
For any $n\in\Z_+$ and $\theta_1,\ldots,\theta_n\in\R$, we have 
\begin{equation}\label{e.conma}
\bigl(\mathcal X_\beta^{(\theta_1)}(t),\ldots,  \mathcal X_\beta^{(\theta_n)}(t) \bigr)\Longrightarrow \bigl(\mathcal{A}_\beta^{(\theta_1)},\ldots, \mathcal{A}_\beta^{(\theta_n)}\bigr)
\end{equation}
in distribution as $t\to\infty$, where $\mathcal{A}_\beta^{(\theta)}$ is the centered Gaussian process with the covariance function $R_\beta^{(\theta)}$ given in \eqref{e.defA}.
\end{lemma}

\begin{proof}
As remarked above, it suffices to instead show convergence 
 of \[
 \frac{1}{\sqrt{t}}\bigl(M^{(\theta_1)}(t),\ldots,M^{(\theta_n)}(t)\bigr)
 \]
 in distribution to the centered Gaussian random vector with the desired covariance matrix.
 From  \eqref{e.newexM} and the It\^o isometry, we write the quadratic variation of $M^{(\theta_i)}$ as
\[
\langle M^{(\theta_i)}\rangle (t)=\beta^2\int_0^t \int_0^1 |\mathsf{f}^{(\theta_i)}(s,y)|^2 dyds,
\]
and the covariation of $M^{(\theta_i)}, M^{(\theta_j)}$ is given by 
\[
\langle M^{(\theta_i)},M^{(\theta_j)}\rangle(t)=\beta^2\int_0^t \int_0^1 \mathsf{f}^{(\theta_i)}(s,y)\mathsf{f}^{(\theta_j)}(s,y)dyds.
\]
Note that both terms were written as the time integral of a stationary process. To prove the convergence of finite dimensional distributions to a Gaussian process, we consider any linear combination of $\frac{1}{\sqrt{t}}M^{(\theta_i)}(t)$. By \cite[Theorem 2.1]{TKCLSO12}, it suffices to show that 
\begin{equation}\label{e.871}
\frac{1}{t} \langle M^{(\theta_i)}\rangle(t)\to \Ee |\mathcal{A}_\beta^{(\theta_i)}|^2, \quad\text{and}\quad \frac{1}{t} \langle M^{(\theta_i)},M^{(\theta_j)}\rangle(t)\to \Ee \mathcal{A}_\beta^{(\theta_i)}\mathcal{A}_\beta^{(\theta_j)}
\end{equation}
in $L^1(\Omega)$, as $t\to\infty$, with
\[
\begin{aligned}
 &\Ee |\mathcal{A}_\beta^{(\theta_i)}|^2=\beta^2  \int_0^1 |\mathsf{f}^{(\theta_i)}(0,y)|^2 dy,\\
&\Ee \mathcal{A}_\beta^{(\theta_i)}\mathcal{A}_\beta^{(\theta_j)}=\beta^2\int_0^1 \mathsf{f}^{(\theta_i)}(0,y)\mathsf{f}^{(\theta_j)}(0,y)dy.
\end{aligned}
\]

We will  only prove the convergence of $\frac{1}{t} \langle M^{(\theta_i)}\rangle(t)$ below. The other case can be treated in the same way. 

First, by the Birkhoff Ergodic theorem, we know that the convergence of $\frac{1}{t} \langle M^{(\theta_i)}\rangle(t)\to \Ee |\mathcal{A}_\beta^{(\theta_i)}|^2$ holds almost surely. Next, for any $p\in[1,\infty)$, we apply the triangle inequality to obtain that 
\[
\|\frac{1}{t} \langle M^{(\theta_i)}\rangle(t)\|_p\leq \beta^2\frac{1}{t} \int_0^t\int_0^1 \|f^{(\theta_i)}(s,y)\|_{2p}^2 dyds \leq \beta^2 C,
\]
with the constant $C>0$ only depending on $p,\beta$. The last step comes from the fact that for any fixed $\theta\in\R$ and $s\geq0$, we have
\[
\{f^{(\theta)}(s,y)\}_{y\in[0,1]}\deq \left\{ \Ee_{\B}\left[\frac{ \exp(\beta\B(y)+\beta \tilde{\B}(y))  }{\int_0^1\exp(\beta\B(y')+\beta \tilde{\B}(y')) dy'}\right]\right\}_{y\in[0,1]}
\]
where $\tilde{\B}$ is a standard Brownian bridge independent of $\B$. This implies the uniform integrability hence the $L^1(\Omega)$ convergence. The proof is complete. 
\end{proof}
To complete the proof of Theorem~\ref{t.conG}, it remains to prove the tightness of the process $\{\mathcal{X}_\beta^{(\theta)}(t):\theta\in\R\}_{t\geq 1}$. This is done in  the following proposition:
\begin{lemma} \label{lem:tight}
$\{\mathcal{X}_\beta^{(\theta)}(t):\theta\in\R\}_{t\geq 1}$ is tight in $C(\R)$.
\end{lemma}

\begin{proof}
By the convergence in distribution of $\mathcal{X}_\beta^{(0)}(t)$, it is enough to show that  
\[
\|\mathcal{X}_\beta^{(\theta_1)}(t)-\mathcal{X}_\beta^{(\theta_2)}(t)\|_2 \leq C |\theta_1-\theta_2|,
\]
for some constant $C$ independent of $t\geq1$. Again, $\|\cdot\|_2$ is the $L^2$ norm of the probability space of all sources of randomnesses. By definition, we have 
\[
\begin{aligned}
\partial_\theta \mathcal{X}_\beta^{(\theta)}(t)=\frac{1}{\sqrt{t}} \left(\frac{\partial_\theta Z_\beta^{(\theta)}(t)}{Z_\beta^{(\theta)}(t)}-\theta t\right).
\end{aligned}
\]
Since $\mathcal{X}_\beta^{(\theta)}(t)$ is stationary in $\theta$, we know $\partial_\theta \mathcal{X}_\beta^{(\theta)}(t)$ is also stationary in $\theta$, thus, we have 
\[
\Ee |\partial_\theta \mathcal{X}_\beta^{(\theta)}(t)|^2=\frac{1}{t}\Ee \left|\frac{\partial_\theta Z_\beta^{(\theta)}(t)}{Z_\beta^{(\theta)}(t)}\right|^2\bigg|_{\theta=0}.
\]
Note that $\frac{\partial_\theta Z_\beta^{(\theta)}(t)}{Z_\beta^{(\theta)}(t)}\Big|_{\theta=0}$ is the quenched mean of the endpoint of the point-to-line directed polymer. It was shown in \cite{YGTK22} that the polymer endpoint is diffusive, in particular, by \cite[Propositions 3.7 and 5.1]{YGTK22}, we know that 
\[
\frac{1}{t}\Ee \left|\frac{\partial_\theta Z_\beta^{(\theta)}(t)}{Z_\beta^{(\theta)}(t)}\right|^2\bigg|_{\theta=0} \leq C
\]
for some constant $C>0$ independent of $t\geq1$. This implies that, for any $\theta_1>\theta_2$, we have 
\[
\begin{aligned}
\|\mathcal{X}_\beta^{(\theta_1)}(t)-\mathcal{X}_\beta^{(\theta_2)}(t)\|_2&=\left\|\int_{\theta_2}^{\theta_1} \partial_\theta \mathcal{X}_\beta^{(\theta)}(t) d\theta\right\|_2\\
&\leq \int_{\theta_2}^{\theta_1}\|\partial_\theta \mathcal{X}_\beta^{(\theta)}(t)\|_2 d\theta \leq C(\theta_1-\theta_2),
\end{aligned}
\]
which completes the proof.
\end{proof}

\appendix
\section{Basic properties of the periodic O'Connell-Yor and inverse-gamma polymers} \label{appx:OCY_SDE_proofs}

\subsection{The O'Connell-Yor polymer}
In this section, we prove some routine lemmas stated in Section \ref{sec:OCY_def}.
\begin{proof}[Proof of Lemma \ref{lem:OCYp_finite}]
Strict positivity follows because we are adding nonnegative terms, not all of which are $0$. We turn to proving finiteness. The expectation of $\OCY_\beta(t,n \viiva s,m + jN)$ can be computed explicitly. Use Tonelli's theorem and independence of the disjoint increments of Brownian motion, we obtain
\[
e^{\theta j}\Ee[\OCY_\beta( t,n \viiva s,m + jN)] = e^{\theta j}e^{\beta^2 (t-s)/2} \f{(t-s)^{n - m - jN}}{(n-m-jN)!} \ind\{n-m - jN \ge 0\}.
\]
Hence,
$
\Ee\Bigl[\OCYp
_\beta (t,n \viiva s,m + jN)\Bigr] < \infty,
$
so there exists a full probability event on which $\OCYp_\beta(r_2,n \viiva r_1,m + jN;\theta) < \infty$ for all pairs of integers $r_1 < r_2$ and integers $\theta$. For an integer $\theta$ and arbitrary $s < t$, choose integers $r_1,r_2$ so that $r_1 < s < t < r_2$. Then, for all $n \ge m$, we rewrite $\OCY_\beta(t,n \viiva s,m)$ as
\begin{align*}
 &\quad \, e^{\beta(B_n(t) - B_m(s))} \int_{\pathsp_{(s,m),(t,n)}} \exp\Bigl(\beta(B_m(s_m) - B_n(s_{n-1})) +\beta \sum_{r = j+1}^{i-1} \big(B_r(s_r) - B_r(s_{r-1})\big)  \Bigr)\,d\mbf s_{m:n-1} \\
 &\le e^{\beta(B_n(t) - B_m(s))} \int_{\pathsp_{(r_1,m),(r_2,n)}} \exp\Bigl(\beta(B_m(s_m) - B_n(s_{n-1})) +\beta \sum_{r = j+1}^{i-1} \big(B_r(s_r) - B_r(s_{r-1})\big)   \Bigr)\,d\mbf s_{m:n-1} \\
&= e^{\beta(B_n(t) - B_n(r_2) - B_m(s) + B_m(r_1))} \OCY_\beta(r_2,n \viiva r_1,m).  
\end{align*}
Thus, since $B_{r + jN} = B_{r}$ for all integers $j$, we have 
\begin{align*}
&\quad \OCYp_\beta(t,n \viiva s,m; \theta) = \sum_{j \in\Z} e^{\theta j} \OCY_\beta(t,n \viiva s,m + jN)  \\
&\le e^{\beta(B_n(t) - B_n(r_2) - B_m(s) + B_m(r_1))} \sum_{j \in \Z} e^{\theta j} \OCY_\beta(r_2 ,n \viiva r_1, m + jN)\\
&= e^{\beta(B_n(t) - B_n(r_2) - B_m(s) + B_m(r_1))}\sum_{j \in \Z} e^{\theta j}\OCY_\beta(r_2 ,n \viiva r_1, m + jN) \\
&=e^{\beta(B_n(t) - B_n(r_2) - B_m(s) + B_m(r_1))} \OCYp_\beta(r_2,n \viiva r_1, m; \theta ) < \infty. 
\end{align*}
The extension to all $\theta  \in \R$ follows by immediately by monotonicity (noting the sum is finite if and only if the sum over negative indices is finite). 
\end{proof}

\begin{proof}[Proof of Lemma \ref{lem:sd_SDE}]
For ease of notation, we prove these items in the case $k = 1$; the extension to general $k$ follows immediately. To remind the reader of the notation, we let $\mbf u = (u_i)_{i \in \Z_N} \in\R^{\Z_N}$ satisfy $\vecsum(\mbf u) = \theta$, and we extend $\mbf u$ to $(u_i)_{i \in \Z}$ by $u_i = u_j$ if $i \equiv j \mod N$. Define $F:\Z \to \R_{>0}$ so that $F(0) = 1$ and $\f{F(j)}{F(j-1)} = e^{u_j}$ for $j \in \Z$. Then, we define $\mbf U(t) = \bigl(U_i(t) \bigr)_{i \in \Z}$ by
\[
U_{i}(t) = U_\beta^N(t,i \viiva s,\mbf u).
\]

\medskip \noindent \textbf{Item \ref{itm:cont_lim}:} By definition \eqref{Zfdef},
\[
\OCY_\beta(t,i \viiva s,F) = \sum_{j \in \Z} F(j) \OCY_\beta(t,i \viiva s,j).
\]
By definition of $\OCY_\beta$, each term $\OCY_\beta(t,i \viiva s,j)$ is a continuous function for $t \ge s$. Since $\f{F(j)}{F(j-1)} = e^{u_j}$, and $(u_j)_{j \in \Z}$ is periodic in shifts of $N$, $F$ grows exponentially. Then, a similar procedure as in the proof of Lemma \ref{lem:OCYp_finite} allows us to use the dominated convergence theorem to conclude continuity of $t \mapsto \OCY_\beta(t,i \viiva s,F)$. Recalling that $\OCY_\beta(s,i \viiva s,j) = \ind\{i = j\}$, we see that $\OCY_\beta(s,i \viiva s,F) = F(i)$.   Then, 
\[
t \mapsto U_i(t) = \log \f{\OCY_\beta(s,i \viiva s,F)}{\OCY_\beta(s,i-1 \viiva s,F)}
\]
is continuous for $t \ge s$. In particular, $\lim_{t \searrow s} U_i(t) = U_i(s) = \f{F(i)}{F(i-1)} = e^{u_i}$.

\medskip \noindent \textbf{Item \ref{itm:Zper}:} We first note that $\OCY_\beta(t,i + N \viiva s, j+N) = \OCY_\beta(t,i \viiva s,j)$ for any $i \ge j$ and $s < t$ by the periodicity of the $B_j$. Observe also that 
\[
\f{F(j+N)}{F(j)} = \prod_{i = 0}^{N-1} \f{F(j + i + 1)}{F(j+i) } = \prod_{i = 1}^{N} e^{u_i} = e^\theta.
\]
Then, 
\begin{align*}
\OCY_\beta(t,i \viiva s,F) &= \sum_{j \le i} F(j) \OCY_\beta(t,i \viiva s,j) = \sum_{j \le i} e^{-\theta } F(j+N) \OCY_\beta(t,i +N\viiva s,j+N)\\
&= e^{-\theta } \sum_{j \le i + N}F(j) \OCY_\beta(r,i+N \viiva s,j) = e^{-\theta } \OCY_\beta(t,i+ N \viiva s,F). 
\end{align*}
Then, we see that 
\[
U_{i+N}(t) = \log \f{\OCY_\beta(t,i +N \viiva s,F)}{\OCY_\beta(t,i +N - 1 \viiva s,F)} = \log \f{\OCY_\beta(t,i \viiva s,F)}{\OCY_\beta(t,i - 1 \viiva s,F)} = U_i(t).
\]

\medskip \noindent \textbf{Item \ref{itm:prod_pres}}: Observe that 
\[
\prod_{i = 0}^{N-1} e^{U_j(t)} = \f{\OCY_\beta(t,N-1 \viiva s, F)}{\OCY_\beta(t,-1 \viiva s,F)} = e^\theta,
\]
the last equality following by Item \ref{itm:Zper}. 

\medskip \noindent \textbf{Item \ref{itm:SDE}:} For $j < i$, set
\[
I_\beta(t,i \viiva s,j) = \int_{\pathsp_{(s,j),(t,i)}} \exp\Bigl(-\beta B_i(s_{i-1}) + \beta \sum_{r = j}^{i-1} \big(B_r(s_r) - B_r(s_{r-1})\big)   \Bigr)\,d\mbf s_{j:i-1},
\]
and for $i = j$, set $I_\beta(t,i \viiva s,i) = e^{-\beta B_i(s)}$
so that $\OCY_\beta (t,i \viiva s,j) = e^{ \beta B_i(t)}I_\beta (t,i \viiva s,j)$ for $i \le j$. When $i > j$, observe that 
\[
I_\beta (t,i \viiva s,j) = \int_s^t e^{-\beta B_{i}(u)}\OCY_\beta (u,i-1 \viiva s,j)\,du ,
\]
and so, letting $d$ denote differentiation in the $t$ variable with $s$ fixed, 
\be \label{Idt}
d I_\beta(t,i \viiva s,j) = e^{-\beta B_i(t)}\OCY_\beta (t,i-1 \viiva s,j)\,dt.
\ee
This also holds in the case $i = j$, as both sides of the equality are $0$. Similar estimates used in the proof of Lemma \ref{lem:OCYp_finite} justify the interchange of differentiation and summation. Then, we have
\[
U_i(t) = \beta B_{i}(t) - \beta B_{i-1}(t) + \log \sum_{j \le i} f(j) I_\beta (t,i \viiva s,j) - \log \sum_{j \le i-1}f(j) I_\beta (t,i-1 \viiva s,j),
\]
and so, by \eqref{Idt}, 
\begin{align*}
    dU_i(t) &= \beta\bigl(dB_{i}(t) - dB_{i-1}(t)\bigr)\\
    &\qquad+ \biggl(\f{e^{-\beta B_i(t)}\sum_{j \le i - 1} F(j)\OCY(t,i-1 \viiva s,j) }{\sum_{j \le i}F(j) I_\beta(t,i\viiva s,j )} - \f{e^{-\beta B_{i-1}(t)}\sum_{j \le i-1}F(j)\OCY(t,i-2 \viiva s,j) }{\sum_{j \le i-1}F(j) I_\beta(t,i-1 \viiva s,j)}\biggr)\,dt \\
    &= \beta\bigl(dB_{i}(t) - dB_{i-1}(t)\bigr) + \biggl(\f{\sum_{j \le i - 1} F(j)\OCY_\beta(t,i-1 \viiva s,j) }{\sum_{j \le i}F(j) \OCY_\beta(t,i\viiva s,j )} - \f{\sum_{j \le i-1}F(j)\OCY_\beta(t,i-2 \viiva s,j) }{\sum_{j \le i-1}F(j) \OCY_\beta(t,i-1 \viiva s,j)}\biggr)\,dt\\
    &= \beta\bigl(dB_{i}(t) - dB_{i-1}(t)\bigr) + (e^{-U_i(t)} - e^{-U_{i-1}(t)})\,dt.
\end{align*}
Because the coefficients of the equation are smooth and therefore locally bounded and locally Lipschitz, \cite[Theorem 5.3.7, page 297]{ethi-kurt} implies pathwise uniqueness of solutions to the system of SDEs. Then \cite[theorem 5.3.6]{ethi-kurt} implies uniqueness of solutions in law.
\end{proof}

The following is used in the proof of Proposition \ref{prop:NewSDE} to show  a strong comparison principle. 
\begin{lemma} \label{lem:SDE_strong_comp}
Let $(\mbf U(t))_{t \ge 0} = \bigl((U_i(t))_{i \in \Z_N}\bigr)_{t \ge 0}$ and $(\mbf V(t))_{t \ge 0} = \bigl((V_i(t)_{i \in \Z_N}\bigr)_{t \ge 0}$ be continuous functions of $t$ such that, for $i \in \Z_N$, $t \mapsto V_i(t) - U_i(t)$ is differentiable, with
\be \label{eq:VUdif}
\df{d}{dt}\bigl(V_i(t) - U_i(t)\bigr) = e^{-V_i(t)} - e^{-U_i(t)} + e^{-U_{i-1}(t)} - e^{-V_{i-1}(t)}. 
\ee
If $\mbf U(0) < \mbf V(0)$, then for all $t > 0$, $\mbf U(t) < \mbf V(t)$. 
\end{lemma}
\begin{proof}
Set $
\tau := \inf\{t > 0: U_i(t) = V_i(t) \text{ for some }i \in \Z_N\}.
$
By continuity of sample paths, the lemma will follow if we can show that almost surely $\tau=+\infty$. We will assume the contrary that $\tau < \infty$ and reach a contradiction. \eqref{eq:VUdif} implies the conservation law $\df{d}{dt} \vecsum\big(\mbf V(t) - \mbf U(t)\big) = 0$ and thus $\vecsum\big(\mbf V(t) - \mbf U(t)\big)>0$ for all $t\geq 0$ (since it holds for $t=0$). This implies that at the stopping time $\tau$ (assumed presently to be finite), there must exist some $i \in \Z_N$ such that $V_i(\tau) - U_i(\tau) = 0$ and $V_{i-1}(\tau) - U_{i-1}(\tau) > 0$ (recall that the indices are elements of $\Z_N$ with cyclic subtraction).
Now, with our choice of $i$, for each $\ve \in (0,V_i(0) - U_i(0))$, set
\[
t_0(\ve) := \sup\{t \in (0,\tau): V_i(t) - U_i(t) = \ve\}.
\]
By continuity, $t_0(\ve)\in (0,\tau)$ exists, $V_i(t_0(\ve)) - U_i(t_0(\ve)) = \ve$, and $V_i(t) - U_i(t)\in [0,\ve]$ for $t \in [t_0(\ve),\tau]$. Set
\[
C_1 := \inf_{0 \le t \le \tau} e^{-V_{i-1}(t)}\Bigl(e^{V_{i-1}(t) - U_{i-1}(t)} - 1\Bigr), \qquad \text{and}\qquad C_2 = \sup_{0 \le t \le \tau} e^{-V_i(t)},
\]
and observe that $C_1,C_2 \in (0,\infty)$ by continuity and because $V_{i-1}(t) - U_{i-1}(t) > 0$ for all $t \in [0,\tau]$ by choice of $\tau$ and $i$. Now, choose $\ve \in (0,V_i(0) - U_i(0))$ sufficiently small so that $C_1 - C_2(e^{\ve} - 1) > 0$. 
Then, by \eqref{eq:VUdif},
\begin{align*}
0 = V_i(\tau)- U_i(\tau) &= \ve + \int_{t_0(\ve)}^\tau\Bigl(e^{-V_i(t)} - e^{-U_i(t)} + e^{-U_{i-1}(t)} - e^{-V_{i-1}(t)}\Bigr)\,dt  \\\
&=\ve + \int_{t_0(\ve)}^\tau\Biggl[ e^{-V_{i-1}(t)}\Bigl(e^{V_{i-1}(t) - U_{i-1}(t)} - 1\Bigr) -  e^{-V_i(t)} \Bigl(e^{V_i(t) - U_i(t)} - 1\Bigr)\Biggr]\,dt  \\
&\ge  \ve + (\tau - t_0(\ve))(C_1 - C_2(e^\ve - 1)) > \ve,
\end{align*}
giving a contradiction.
\end{proof}

\subsection{Justifying the polymer interpretation from Section \ref{sec:polymerinterpretation}}
Here, we present a lemma that connects the polymer model defied in Section \ref{sec:polymerinterpretation} to the discrete-time Markov chain in \eqref{eq:coupled_disc_MC}.

\begin{lemma} \label{lem:inv_gamma}
Let $\mbf U^{(0)} \in \R^{\Z_N}$, and $(\mbf W_m)_{m \ge 1}$ be a sequence of vectors in $\R^{\Z_N}$. Extend these to sequences in $\R^\Z$ by $W_{m,i} = W_{m,j}$ for $i \equiv j \mod N$, and define $Z^N(m,i \viiva F)$ as in \eqref{eq:disc_poly_init}. Set $F_m(i) = Z^N(m,i|F)$ for $m \ge 1$ and $i \in \Z$, and define $F_0(i) = F(i)$. Then,
\begin{enumerate} [label=\textup{(\roman*)}]
    \item \label{itm:finsumcond} For $i \in \Z$ and $m \ge 1$, $F_m(i) < \infty$ if and only if \be \label{eq:sum_order_cond}
    \vecsum(\mbf U^{(0)}) > \vecsum(\mbf W_r)\qquad\text{for}\qquad 1 \le r \le m.
    \ee
\end{enumerate}
If for some $m \ge 1$, \eqref{eq:sum_order_cond} holds, for $i \in \Z$, define $U^{(m)}_i = \log \f{F_m(i)}{F_m(i-1)}$ as in \eqref{eq:poly_part}. Then, the following hold.
\begin{enumerate}[label=\textup{(\roman*)},resume]
\item \label{itm:Xper} $U^{(m)}_i = U^{(m)}_j$ whenever $i \equiv j \mod N$. Hence, we may consider $\mbf U^{(m)}$ as a sequence in $\R^{\Z_N}$.
\item \label{itm:DN2map} $\mbf U^{(m)} = D^{N,2}(\mbf W_m,\mbf U^{(m-1)})$. 
\end{enumerate}
\end{lemma}
\begin{proof}
For $m \ge 1$, we first quickly prove Item \ref{itm:Xper} under the assumption that $F_m(i) < \infty$ for all $i \in \Z$. Recall that we define $F:\Z \to \R$ by extending $\mbf U^{(0)}$ periodically, then set $F(0) = 1$ and $\f{F(i)}{F(i-1)} = e^{U^{(0)}_i}$ for $i \in \Z$. Since the initial condition $\mbf U^{(0)}$ and the environment are periodic, we have
\begin{align*}
Z^N(m,i \viiva F) &= \sum_{j \le i - m} F(j) Z^N(m,i \viiva 1,j+1) = e^{-\vecsum(\mbf U^{(0)})} \sum_{j \le i-m} F(j+N) Z^N(m,i+N \viiva j + N + 1) \\
&= e^{-\vecsum(\mbf U^{(0)})} \sum_{j \le i + N -m} F(j) Z^N(m,i+N \viiva j) = e^{-\vecsum(\mbf U^{(0)})}Z^N(m,i + N \viiva F). 
\end{align*}
Item \ref{itm:Xper} now follows by taking ratios.

Next, we prove Items \ref{itm:finsumcond} and \ref{itm:DN2map}. First, we prove the following: for $k \ge 0$, if $F_k(j) < \infty$ for all $j \in \Z$, then
\begin{align} \label{eq:inf_sum}
\f{F_{k+1}(i)}{F_k(i-1)} = e^{U_i^{(k)}}\sum_{j = i - N}^{i-1} \prod_{\ell = j+1}^{i} e^{W_{k+1,\ell} - U^{(k)}_\ell} \sum_{r = -\infty}^0 e^{-r(\vecsum(\mbf W_{k+1}) - \vecsum(\mbf U^{(k)}))}
\end{align}
Recalling the definitions in \eqref{eq:poly_part} and \eqref{eq:Zdis}, the dynamic programming principle gives us 
\be \label{eq:F_conv}
F_{k+1}(i) = \sum_{j \le i-1} F_{k}(j) \prod_{\ell = j+1}^i e^{W_{k+1,\ell}}.
\ee
 Then, we observe that
\begin{align*}
\f{F_{k+1}(i)}{F_k(i - 1)} &= \sum_{j \le i-1} \f{F_{k}(j)}{F_{k}(i-1)} \prod_{\ell = j+1}^{i} e^{W_{k+1,\ell}} = e^{U_i^{(k)}}\sum_{j \le i-1} \prod_{\ell = j+1}^i e^{W_{k+1,\ell} - U^{(k)}_\ell}  \\
&= e^{U_i^{(k)}}\sum_{j = i - N}^{i-1} \sum_{r = -\infty}^0 \prod_{\ell = j+rN + 1}^i e^{W_{k+1,\ell} - U^{(k)}_\ell} \\
&=  e^{U_i^{(k)}}\sum_{j = i - N}^{i-1} \prod_{\ell = j+1}^{i} e^{W_{k+1,\ell} - U^{(k)}_\ell}   \sum_{r = -\infty}^0 \prod_{\ell = j+rN + 1}^j e^{W_{k+1,\ell} - U^{(k)}_\ell}  \\
&=  e^{U_i^{(k)}}\sum_{j = i - N}^{i-1} \prod_{\ell = j+1}^{i} e^{W_{k+1,\ell} - U^{(k)}_\ell} \sum_{r = -\infty}^0 e^{-r(\vecsum(\mbf W_{k+1}) - \vecsum(\mbf U^{(k)}))}. \label{eq:inf_sum}
\end{align*}

 We now prove Items \ref{itm:finsumcond} and \ref{itm:DN2map} together by strong induction. To start with the base case, we have, by definition, $F(i) = F_0(i) < \infty$ for all $i \in \Z$. Then, \eqref{eq:inf_sum} implies that $F_1(i) < \infty$ if and only if $\vecsum(\mbf U^{(0)}) > \vecsum(\mbf W_{1})$. If this holds, then using \eqref{eq:inf_sum}, we have, for $i \in \Z$,
 \begin{align*}
     U_i^{(1)} = \log \f{F_1(i)}{F_1(i-1)} =   \log \Biggl(\f{F_1(i)}{F_0(i-1)} \f{F_0(i-2)}{F_1(i-1)}e^{U^{(0)}_{i-1}}\Biggr).
 \end{align*}
 Then, applying the $k = 0$ case of \eqref{eq:inf_sum} for both $i$ and $i-1$, we obtain
 \begin{align*}
 U_i^{(1)} &= U_i^{(0)}  +\log\Biggl(\f{\sum_{j = i - N}^{i-1} \prod_{\ell = j+1}^{i} e^{W_{1,\ell} - U^{(0)}_\ell}}{\sum_{j = i-1 - N}^{i-2} \prod_{\ell = j+1}^{i-1} e^{W_{1,\ell} - U^{(0)}_\ell}}    \Biggr) \\
&= U^{(0)}_i + \log \Biggl(\f{\prod_{\ell \in \Z_N} e^{W_{1,\ell} - U^{(0)}_\ell} +
  \sum_{j \in \Z_N \setminus \{i\} } e^{W_{1,(j,i]} - U^{(0)}_{(j,i]}}   }{\prod_{\ell \in \Z_N} e^{W_{1,\ell} - U^{(0)}_\ell} +
  \sum_{j \in \Z_N \setminus \{i-1\} } e^{W_{1,(j,i-1]} - U^{(0)}_{(j,i-1]}}   }\Biggr)  \\
  &= U^{(0)}_i + \log \Biggl(\f{
  \sum_{j \in \Z_N } e^{U^{(0)}_{(i,j]} - W_{1,(i,j]} }   }{
  \sum_{j\in \Z_N  } e^{U^{(0)}_{(i-1,j]} - W_{1,(i-1,j]} }   }\Biggr) = D_i^{N,2}(\mbf W_{1},\mbf U^{(0)}),
 \end{align*}
where in the penultimate step, we multiplied by $\prod_{\ell \in \Z_N} e^{W_{1,\ell} - U^{(0)}_\ell}$ in both the numerator and denominator, noting that, for a sequence $\mbf Y \in \R^{\Z_N}$, $Y_{(j,i]} + Y_{(i,j]} = \sum_{\ell \in \Z_N} Y_\ell$ for $j \neq i$, and $Y_{(i,i]} = 0$. This completes the proof of the $m =1$ case of Item \ref{itm:DN2map}.

Now, we assume by induction, that, for some $k \ge 1$, for all $1 \le m \le k$, $F_m(i) < \infty$ if and only if $\vecsum(\mbf U^{(0)}) > \vecsum(\mbf W_r)$ for $1 \le r \le m$, and whenever this is the case, we have $\mbf U^{(m)} = D^{N,2}(\mbf W_m,\mbf U^{(m-1)})$.

Now, assume that $F_{k+1}(i) < \infty$ for some $i \in \Z$. From \eqref{eq:F_conv}, it follows that $F_k(j) < \infty$ for $j \le i -1$. By Item \ref{itm:finsumcond} of the inductive assumption, $\vecsum(\mbf U^{(0)}) > \vecsum(\mbf W_r)$ for $1 \le r \le k$.  By Item \ref{itm:DN2map} of the inductive assumption, $\mbf U^{(m)} = D^{N,2}(\mbf W_m,\mbf U^{(m-1)})$ for $1 \le m \le k$. Then, by iterating Lemma \ref{lem:Dsum_pres},
\be \label{eq:Xmsumpres}
\vecsum(\mbf U^{(k)}) = \vecsum(\mbf U^{(0)})
\ee

Thus, by \eqref{eq:Xmsumpres} and \eqref{eq:inf_sum}, since we assumed $F_{k+1}(i) < \infty$, we must have $\vecsum(\mbf W_{k+1}) < \vecsum(\mbf U^{(k)}) = \vecsum(\mbf U^{(0)})$. This completes one implication of Item \ref{itm:finsumcond} for $m = k +1$. If, on the other hand, $F_{k+1}(i) = \infty$ for some $i \in \Z$, then we consider two cases. If $F_k(j) = \infty$ for some $j \in \Z$, then Item \ref{itm:finsumcond} of the inductive assumption implies $\vecsum(\mbf U^{(0)}) \le \vecsum(\mbf W_r)$ for some $1 \le r \le k$.  If, on the other hand, $\vecsum(\mbf U^{(0)}) >\vecsum(\mbf W_r)$  for all $1 \le r \le k$, then $F_k(j) < \infty$ for all $j \in \Z$, so \eqref{eq:Xmsumpres} and \eqref{eq:inf_sum} show that $\vecsum(\mbf W_{k+1}) \ge \vecsum(\mbf U^{(k)}) = \vecsum(\mbf U^{(0)})$. This proves Item \ref{itm:finsumcond} for $m = k+1$. Item \ref{itm:DN2map} for $m = k+1$ is then proved from \eqref{eq:inf_sum} similarly as in the base case.
\end{proof}

\section{Further calculations involving the $D$ and $J$ bijection} \label{sec:consist_technical}
The following gives an alternate description of the output of the map $D^{N,m}$ for $m \ge 2$.
\begin{lemma} \label{lem:Dnm_alt}
For any $N,m\in \N$, $(\mbf X_1,\ldots,\mbf X_m) \in \R^{\Z_N}$ and $i \in \Z_N$,
\be \label{eq:Qnm_gen}
\begin{aligned}
D^{N,m}_i(\mbf X_1,\ldots,\mbf X_m) &= X_{m,i} + Q_i^{N,m}(\mbf X_1,\ldots,\mbf X_m) - Q_{i-1}^{N,m}(\mbf X_1,\ldots,\mbf X_m) \quad \text{where}\\
Q_i^{N,m}(\mbf X_1,\ldots,\mbf X_m) &= \log \sum_{\substack{j_1,\ldots,j_{m-1} \in \Z_N \\ j_0 = i}} \prod_{r = 1}^{m-1} e^{X_{m,(j_{r - 1},j_r]}-X_{r,(j_{r - 1},j_r]}}.
\end{aligned}
\ee
\end{lemma}
\begin{proof}
We prove this by induction on $m$. The $m = 2$ case is by definition. Assume that \eqref{eq:Qnm_gen} holds for some $m \ge 2$. We apply the induction assumption to $(\mbf X_2,\ldots,\mbf X_{m+1})$ to see that
\begin{align*}
D^{N,m}_i(\mbf X_2,\ldots,\mbf X_{m+1}) &= X_{m+1,i} + Q_i^{N,m}(\mbf X_2,\ldots,\mbf X_{m+1}) -  Q_{i-1}^{N,m}(\mbf X_2,\ldots,\mbf X_{m+1}),\\
D^{N,m}_{(i,j]}(\mbf X_2,\ldots,\mbf X_{m+1}) &= X_{m+1,(i,j]} + Q_j^{N,m}(\mbf X_2,\ldots,\mbf X_{m+1}) - Q_i^{N,m}(\mbf X_2,\ldots,\mbf X_{m+1}),
\end{align*}
Since $D^{N,m+1}(\mbf X_1,\ldots,\mbf X_{m+1}) = D^{N,2}\bigl(\mbf X_1,D^{N,m}(\mbf X_2,\ldots,\mbf X_{m+1})\bigr)$ by definition, it follows from above that 
\begin{align*}
D_i^{N,m+1}(\mbf X_1,\ldots,\mbf X_{m+1}) &= D_i^{N,m}(\mbf X_2,\ldots,\mbf X_{m+1}) + \log \Biggl(\f{\sum_{j \in \Z_N} e^{D_{(i,j]}^{N,m}(\mbf X_2,\ldots,\mbf X_{m+1}) - X_{1,(i,j]} }}{\sum_{j \in \Z_N} e^{D_{(i-1,j]}^{N,m}(\mbf X_2,\ldots,\mbf X_{m+1}) - X_{1,(i-1,j]} }} \Biggr) \\
&=  X_{m+1,i} + Q_i^{N,m}(\mbf X_2,\ldots,\mbf X_{m+1}) -  Q_{i-1}^{N,m}(\mbf X_2,\ldots,\mbf X_{m+1})  \\
&+ \log\Biggl(\f{\sum_{j \in \Z_N} e^{X_{m+1,(i,j]} + Q_j^{N,m}(\mbf X_2,\ldots,\mbf X_{m+1}) - Q_i^{N,m}(\mbf X_2,\ldots,\mbf X_{m+1}) - X_{1,(i,j]}}   }{\sum_{j \in \Z_N} e^{X_{m+1,(i-1,j]} + Q_j^{N,m}(\mbf X_2,\ldots,\mbf X_{m+1}) - Q_{i-1}^{N,m}(\mbf X_2,\ldots,\mbf X_{m+1}) - X_{1,(i-1,j]}} }\Biggr) \\
&= X_{m+1,i} + \log\Biggl(\f{\sum_{j \in \Z_N} e^{X_{m+1,(i,j]} - X_{1,(i,j]} + Q_j^{N,m}(\mbf X_2,\ldots,\mbf X_{m+1}) }   }{\sum_{j \in \Z_N} e^{X_{m+1,(i-1,j]}- X_{1,(i,j]} + Q_j^{N,m}(\mbf X_2,\ldots,\mbf X_{m+1})} }\Biggr). 
\end{align*}
The proof is complete upon observing that 
\[
\begin{aligned}
&\log \sum_{j \in \Z_N} e^{X_{m+1,(i,j]} - X_{1,(i,j]} + Q_j^{N,m}(\mbf X_2,\ldots,\mbf X_{m+1}) } \\
&= \log \sum_{j \in \Z_N} e^{X_{m+1,(i,j]} - X_{1,(i,j]}}\sum_{\substack{j_2,\ldots,j_{m} \in \Z_N \\ j_1 = j}} \prod_{r = 2}^{m} e^{X_{m+1,(j_{r - 1},j_r]}-X_{r,(j_{r - 1},j_r]}}  = Q_i^{N,m+1}(\mbf X_1,\ldots,\mbf X_{m+1}). \qedhere
\end{aligned}
\]
\end{proof}

We now prove some additional technical inputs that are used in the main body.
For the next lemma, recall that, for $\mbf X = ( X_i)_{i \in \Z} \in \R^{\Z_N}$, we defined $\diff_ie^{-\mbf X} = e^{-X_i} - e^{-X_{i-1}}$ and $\Lapl_ie^{-\mbf X} = e^{-X_{i-1}} + e^{-X_{i+1}} - 2e^{-X_i}$.

\begin{lemma} \label{lem:nasty_algebra}
Recall the definition of $\mathcal R^{N,2}$ in \eqref{RA2} and $\mathcal J^{N,2}$ in  \eqref{Jimap}. Let $(\mbf U_1,\mbf U_2) \in \mathcal R^{N,2}$, let $W_i = U_{2,i} - U_{1,i}$ for $i \in \Z_N$, and let $(\mbf X_1,\mbf X_2) =(\mbf U_1, J^{N,2}(\mbf U_1,\mbf U_2))$.  Then, for $i \in \Z_N$,
\begin{equation}
\begin{aligned}\label{eq:lemmab3}
    &\quad \, \f{e^{W_i}}{e^{W_i} - 1} \diff_ie^{-\mbf U_1} - \f{e^{W_{i+1}}}{e^{W_{i+1}} - 1} \diff_{i+1}e^{-\mbf U_1} + \Bigl(1 - \f{e^{W_i}}{e^{W_i} - 1}\Bigr)\diff_ie^{-\mbf U_2}+ \f{e^{W_{i+1}}}{e^{W_{i+1}} - 1}\diff_{i+1}e^{-\mbf U_2}  \\
   &= \diff_ie^{-\mbf X_2}- \Lapl_{i}e^{-\mbf X_1}.
\end{aligned}
\end{equation}
\end{lemma}
\begin{proof}
The left-hand side  of \eqref{eq:lemmab3} equals
\be \label{eq:drift1}
\begin{aligned}
   &\quad \, \f{1}{e^{W_i} - 1}(e^{W_i - U_{1,i}} - e^{W_i - U_{1,i-1}} + e^{-U_{2,i-1}} - e^{-U_{2,i}}  ) \\
   &\qquad\qquad + \f{1}{e^{W_{i+1}} - 1}(e^{-U_{1,i+1}} - e^{W_{i+1}-U_{2,i}}  -e^{W_{i+1}-U_{1,i+1}} + e^{W_{i+1}-U_{1,i}} )  \\
   &=\f{1}{(e^{W_{i+1}}-1)(e^{W_i} - 1)}\Bigl(-e^{W_i - U_{1,i}} + e^{-U_{2,i}}+ e^{W_{i+1} - U_{2,i-1}}- e^{-U_{2,i-1}} + 2e^{W_{i+1} + W_i - U_{1,i}} - 2e^{W_{i+1} - U_{1,i}}  \\
&\qquad\qquad- e^{W_{i+1} + W_i - U_{1,i-1}} - e^{W_i - U_{1,i-1}} - e^{W_{i+1} + W_i - U_{1,i+1}} + e^{W_i - U_{1,i+1}} + e^{W_{i+1} - U_{1,i+1}} - e^{-U_{1,i+1}}\Bigr).
\end{aligned}
\ee

Since $\mbf X_1 = \mbf U_1$ and $\mbf X_2 = J^{N,2}(\mbf U_1,\mbf U_2)$, by definition of $J^{N,2}$ \eqref{Jimap}, the right-hand side of \eqref{eq:lemmab3} equals
\begin{align*}
&\quad \, e^{-U_{2,i}}\f{e^{W_i}-1}{e^{W_{i+1}}-1} - e^{-U_{2,i-1}} \f{e^{W_{i-1}} - 1}{e^{W_i} - 1} + 2e^{-U_{1,i}} - e^{-U_{1,i+1}}- e^{-U_{1,i-1}} \\
&= \f{e^{-U_{1,i}} - e^{-U_{2,i}}}{e^{W_{i+1}} - 1} - \f{e^{-U_{1,i-1}} - e^{-U_{2,i-1}}}{e^{W_i}-1}+2e^{-U_{1,i}} - e^{-U_{1,i+1}}  - e^{-U_{1,i-1}} \\
&= \f{e^{W_i - U_{1,i}} - 2e^{-U_{1,i}} + e^{-U_{2,i}} - e^{W_{i+1} - U_{1,i-1}} + e^{W_{i+1} - U_{2,i-1}} + e^{-U_{1,i-1}} - e^{-U_{2,i-1}} }{(e^{W_{i+1}}-1)(e^{W_i} - 1)} \\
&+\f{1}{(e^{W_{i+1}}-1)(e^{W_i} - 1)}\Bigl(2e^{W_{i+1} + W_i - U_{1,i}} - 2e^{W_i -U_{1,i}} -2 e^{W_{i+1} - U_{1,i}} + 2e^{-U_{1,i}} - e^{W_{i+1} + W_i - U_{1,i-1}}\\
&- e^{W_i - U_{1,i-1}}  + e^{W_{i+1} - U_{1,i-1}} - e^{-U_{1,i-1}} - e^{W_{i+1} + W_i - U_{1,i+1}} + e^{W_i - U_{1,i+1}} + e^{W_{i+1} - U_{1,i+1}} - e^{-U_{1,i+1}} \Bigr) \\
&= \f{1}{(e^{W_{i+1}}-1)(e^{W_i} - 1)}\Bigl(-e^{W_i - U_{1,i}} + e^{-U_{2,i}}+ e^{W_{i+1} - U_{2,i-1}}- e^{-U_{2,i-1}} + 2e^{W_{i+1} + W_i - U_{1,i}}  \\
& - 2e^{W_{i+1} - U_{1,i}} - e^{W_{i+1} + W_i - U_{1,i-1}} - e^{W_i - U_{1,i-1}} - e^{W_{i+1} + W_i - U_{1,i+1}} + e^{W_i - U_{1,i+1}} + e^{W_{i+1} - U_{1,i+1}} - e^{-U_{1,i+1}}\Bigr),
\end{align*}
which we can see matches  \eqref{eq:drift1}.
\end{proof}

The following shows an intertwining of the map $\D^{N,k}$ with the shift operator $\tau_1: \R^{\Z_N} \to \R^{\Z_N}$ defined by $(\tau_1 \mbf X)_i = X_{i+1}$. 
Below, we note that the operator $\tau_1$ is applied separately to each sequence. This is a different type of shift than that in Proposition \ref{prop:disc_consis}\ref{itm:mu_perm}, where the order of the sequences is permuted. 
\begin{lemma} \label{lem:shift}
Let $(\mbf X_1,\ldots,\mbf X_k) \in (\R^{\Z_N})^k$, and define $(\mbf U_1,\ldots,\mbf U_k) = \D^{N,k}(\mbf X_1,\ldots,\mbf X_k)$. Then,
\[
(\tau_1 \mbf U_1,\ldots,\tau_1 \mbf U_k) = \D^{N,k}(\tau_1 \mbf X_1,\ldots\tau_1 \mbf X_k).
\]
\end{lemma}
\begin{proof}
By the inductive definition of $\D^{N,k}$, it suffices to show that $\tau_1 \mbf U_2 = \D^{N,2}(\tau_1 \mbf X_1,\tau_1 \mbf X_2)$. For $\ell \in \Z_N$, define $Y_\ell = X_{2,\ell} - X_{1,\ell}$ and $Y_\ell' = X_{2,\ell + 1} - X_{1,\ell + 1}$. Then, observe that
\begin{align*}
D^{N,2}_i(\tau_1 \mbf X_1,\tau_1 \mbf X_2) &= X_{2,i+1} + \log\Biggl(\f{\sum_{j \in \Z_N} e^{Y'_{(i,j]}}     }{\sum_{j \in \Z_N} e^{Y'_{(i-1,j]}  }}\Biggr) =  X_{2,i+1} + \log\Biggl(\f{\sum_{j \in \Z_N} e^{Y_{(i+1,j+1]}}     }{\sum_{j \in \Z_N} e^{Y'_{(i,j+1]}  }}\Biggr)  \\
&=  X_{2,i+1} + \log\Biggl(\f{\sum_{j \in \Z_N} e^{Y_{(i+1,j]}}     }{\sum_{j \in \Z_N} e^{Y'_{(i,j]}  }}\Biggr) = D_{i+1}^{N,2}(\mbf X_1,\mbf X_2) = \bigl[\tau_1 \D^{N,2}(\mbf X_1,\mbf X_2)\bigr]_i. \qedhere 
\end{align*}
\end{proof}

\begin{corollary} \label{cor:shift}
For $(\theta_1,\ldots,\theta_k) \in \R^k$, let $(\mbf U_1,\ldots,\mbf U_k) \sim \mu_\beta^{N,(\theta_1,\ldots,\theta_k)}$. Then, we also have
\be \label{eq:U_shift}
(\tau_1 \mbf U_1,\ldots,\tau_1 \mbf U_k) \sim \mu_\beta^{N,(\theta_1,\ldots,\theta_k)}.
\ee
\end{corollary}
\begin{proof}
It is immediate that if $(\mbf X_1,\ldots,\mbf X_k) \sim \nu_\beta^{N,(\theta_1,\ldots,\theta_k)}$, then $(\tau_1 \mbf X_1,\ldots,\tau_1 \mbf X_k) \sim \nu_\beta^{N,(\theta_1,\ldots,\theta_k)}$ as well. Then, the shift invariance in \eqref{eq:U_shift}  follows from Lemma \ref{lem:shift}.
\end{proof}

The following was referenced in Section \ref{sec:approach}
\begin{lemma} \label{lem:Pitman_reflect}
Let $W:\R^{\Z_N} \times \R^{\Z_N} \to \R^{\Z_N} \times \R^{\Z_N} $ be the periodic Pitman transform defined by 
\[
W(\mbf X_1,\mbf X_2) = \bigl(T^{N,2}(\mbf X_1,\mbf X_2),D^{N,2}(\mbf X_1,\mbf X_2)\bigr)
\]
Let $r:\R^{\Z_N} \to \R^{\Z_N}$ be the reflection operator defined by $(r \mbf X)_i = - X_{N-i}$, and extend this to an operator $(\R^{\Z_N})^2 \to (\R^{\Z_N})^2$ by $r(\mbf X_1,\mbf X_2) = (r\mbf X_1,r\mbf X_2)$. Then, $rWr(\mbf X_1,\mbf X_2) = \bigl(\wt{\mbf X}_1,\wt{\mbf X}_2\bigr)$,
where, for $i \in \Z_N$,
\[
\wt X_{1,i} =  X_{2,i} - X_{2,i-1} + T_{i-1}^{N,2}(\mbf X_1,\mbf X_2),\quad\text{and}\quad \wt X_{2,i} = X_{1,i} - X_{1,i+1} + D_{i+1}^{N,2}(\mbf X_1,\mbf X_2).
\]
\end{lemma}
\begin{proof}
Let $(\mbf X_1,\mbf X_2)  \in \R^{\Z_N} \times \R^{\Z_N}$, and define $(\mbf X_1',\mbf X_2') = r(\mbf X_1,\mbf X_2)$. Then, for $i \in \Z_N$, by definition \eqref{Rdef},
\begin{align*}
 T_i^{N,2}(\mbf X_1',\mbf X_2') &= X_{1,i}' + \log\Biggl(\f{\sum_{j \in \Z_N} e^{X_{2,[i,j]}' - X_{1,[i,j]}'}  }{\sum_{j \in \Z_N} e^{X_{2,[i+1,j]}' - X_{1,[i+1,j]}'}}\Biggr) \\
&= -X_{1,N-i} + \log\Biggl(\f{\sum_{j \in \Z_N} e^{X_{1,[N-j,N-i]} - X_{2,[N-j,N-i]}}  }{\sum_{j \in \Z_N} e^{X_{1,[N-j,N - (i+1)]} - X_{2,[N-j,N-(i+1)]}}}\Biggr).
\end{align*}
Then, we have
\begin{align*}
\wt X_{1,i} = -T^{N,2}(\mbf X_1,\mbf X_2)_{N-i} &= X_{1,i} + \log\Biggl(\f{\sum_{j \in \Z_N} e^{X_{1,[N-j,i-1]} - X_{2,[N-j,i-1]}}}{\sum_{j \in \Z_N} e^{X_{1,[N-j,i]} - X_{2,[N-j,i]}}  }    \Biggr) \\
&= X_{1,i} + \log\Biggl(\f{\sum_{\ell \in \Z_N} e^{X_{2,(i-1,\ell]} - X_{1,(i-1,\ell]}}    }{\sum_{\ell \in \Z_N} e^{X_{2,(i,\ell]} - X_{1,(i,\ell]}}} \Biggr) \\
&= X_{2,i} - X_{2,i-1} + X_{1,i-1} + \log\Biggl(\f{\sum_{\ell \in \Z_N} e^{X_{2,[i-1,\ell]} - X_{1,[i-1,\ell]}}    }{\sum_{\ell \in \Z_N} e^{X_{2,[i,\ell]} - X_{1,[i,\ell]}}} \Biggr),
\end{align*}
and this is readily seen to be equal to $ X_{2,i} - X_{2,i-1} + T_{i-1}^{N,2}(\mbf X_1,\mbf X_2),$ as desired. To get the expression for $\wt X_{2,i}$, we use the relation in Lemma \ref{lem:DT_Add}:
\be \label{eq:DT_add}
D_i^{N,2}(\mbf X_1,\mbf X_2) + T_{i-1}^{N,2}(\mbf X_1,\mbf X_2) = X_{1,i} + X_{2,i-1}.
\ee
Then, we have 
\begin{align*}
    \wt X_{2,i} = - D^{N,2}_{N-i}(\mbf X_1',\mbf X_2') &= X_{1,i} + X_{2,i+1} - T_{N - (i+1)}^{N,2}(\mbf X_1',\mbf X_2') \\
    &=X_{1,i} + X_{2,i} - T_i^{N,2}(\mbf X_1,\mbf X_2).
\end{align*}
Where the second equality follows by the equality for $T^{N,2}(\mbf X_1',\mbf X_2')$ we previously proved. By another application of \eqref{eq:DT_add}, this is equal to $X_{1,i} - X_{1,i+1} + D_{i+1}^{N,2}(\mbf X_1,\mbf X_2)$, as desired. 
\end{proof}

\section{Invariant measure for the dual system of SDEs} \label{appx:invariance}
The goal of this section is to show the following result.
\begin{proposition}\label{p.invariantSDE}
    The density $p(\cdot)$ defined in \eqref{eq:dens_big} is invariant for the SDE \eqref{eq:N_prod_SDE}.
\end{proposition}
To simplify the notations, let us rewrite the SDE \eqref{eq:N_prod_SDE}  as follows:
\begin{equation}\label{e.simplesde}
d\mbf X(t) =\mbf A\bigl(X(t)\bigr)dt+\Sigma d\mbf B(t),
\end{equation}
where $\mbf B$ is a $d-$dimensional standard Brownian motion, $\Sigma$ is a constant matrix, and the vector-valued drift $\mbf A(\cdot)$ contains linear combinations of functions of the form $e^{-x_{r,i}}, e^{-(\theta_r-x_{r,1}-\ldots-x_{r,N-1})}$, with $r=1,\ldots,k,i=1,\ldots,N-1$ and the parameters  $\theta_1,\ldots,\theta_k\in\R$ are fixed in this section. We note that $d=k(N-1)$.

The challenge with the SDE \eqref{e.simplesde} is that the drift is unbounded and non-Lipschitz, so one can not apply standard theory. Additionally, the matrix $\Sigma$ is degenerate. Nevertheless, in our case by using the explicit solution and applying the It\^o formula, we have shown in Proposition \ref{prop:NewSDE} that a solution defined for all $t \ge 0$ exists, and that is unique both pathwise and in law. 

To apply standard semigroup and generator theory, we define an approximation. For any $n\in \N$, define
\[
S_n:=\Big\{\mbf x=(\bfx_1,\ldots,\bfx_k)\in\R^d:  \theta_r-\sum_{i=1}^{N-1} x_{r,i}> -n, x_{r,i} > -n \mbox{ for all } r\in \{1,\ldots, k\}, i\in \{1,\ldots,N-1\}\Big\}.
\]
In this region $b(\cdot)$ is bounded and Lipschitz. Define $\mbf A_n(\mbf x)=\mbf A(\mbf x)$ for those $\mbf x\in S_n$. Outside $S_n$ we define $\mbf A_n$ so that it is a $C_c^\infty(\R^d)$ function (continuous functions with compact support) with
\begin{equation}\label{e.bdbn}
\sup_{\mbf x\in\R^d} (|\mbf A_n(\mbf x)|+|\text{grad} \mbf A_n(\mbf x)|) \leq C \exp(Cn),
\end{equation}
for some constant $C>0$. Let $\mbf W(t)$ be another $d-$dimensional standard Brownian motion that is independent of $\mbf B$, and consider the approximation of \eqref{e.simplesde}
\begin{equation}\label{e.simplesdeapp}
d \mbf X^{(n)}(t)=\mbf A_n\bigl(\mbf X^{(n)}(t)\bigr)dt+\Sigma d\mbf B(t)+\alpha_n^{-1} d\mbf W(t),
\end{equation}
where $\alpha_n=\exp(\exp(n^2))$. In this way, we come up with a new SDE with nice coefficients, and the corresponding generator is uniformly elliptic which will imply smooth solutions to the backwards Kolmogorov equation below and allow us to perform integration by parts in  \eqref{e.ibpLn}. Note, we did not assume that $\Sigma$ had rank $d$ and in our application this is not the case. This is why we include $\mbf W(t)$. By \cite[Theorem 3.10, Theorem 3.7]{ethi-kurt}, for each initial condition $\mbf x \in \R^d$, there exists a pathwise-unique solution $\mbf X^{(n)}$ to \eqref{e.simplesde} satisfying $\mbf X^{(n)}(0) = x$. The following lemma shows that $\mbf X^{(n)}$ is a pathwise approximation of $\mbf X$.
\begin{lemma}
    For any $(t,\mbf x)\in \R_{\geq 0}\times \R^d$, if $\mbf X^{(n)}(0)=\mbf X(0)=\mbf x$,  then  $\mbf X^{(n)}(t) \to \mbf X(t)$  almost surely as $n\to \infty$.
\end{lemma}
\begin{proof}
By the integral formulation, we have
\[
\mbf X^{(n)}(t)-\mbf X(t)=\int_0^t \bigl(\mbf A_n(\mbf X^{(n)}(t))-\mbf A\bigl(\mbf X(s))\bigr)ds+\frac{1}{\alpha_n}\mbf W(t), \quad t \ge 0.
\]
Fix $k \in \N$, let $\tau_k = \inf\{t > 0: \mbf X(t) \notin S_k\}$. Then, since $\mbf A = \mbf A_n$ on $S_n$, for $n \ge k$ and $0 \le t \le \tau_k$, we have 
\[
\mbf X^{(n)}(t)-\mbf X(t) =\int_0^{t} \bigl(\mbf A_n(\mbf X^{(n)}(s))-\mbf A_n(\mbf X(s))\bigr)ds+\frac{1}{\alpha_n}\mbf W(t),
\]
which implies 
\[
|\mbf X^{(n)}(t )-\mbf X(t )| \leq C_n \int_0^{t} |\mbf X^{(n)}(s)-\mbf X(s)| ds+\frac{1}{\alpha_n}\max_{s\in[0,t]}|\mbf W(s)|,
\] 
where $C_n=C\exp(Cn)$ (see \eqref{e.bdbn}). The reason for replacing $|\mbf W(t)|$ with $\sup_{0 \le s \le t} |\mbf \mbf W(s)|$ is because the latter is a nondecreasing function of $t$, and so by Gronwall's inequality, we conclude
\[
|\mbf X^{(n)}(t)-\mbf X(t)|\leq \frac{1}{\alpha_n}e^{C_n t}\max_{s\in[0,t]}|\mbf W(s)|, \quad\quad 0\le t\leq \tau_k.
\]
Sending $n\to\infty$ and using the fact that $\alpha_n\gg e^{C_nt}$, we see that $\lim_{n \to \infty} \mbf X^{(n)}(t) = \mbf X(t)$ for $0 \le t \le \tau_k$. Sending $k \to \infty$ completes the proof. 
\end{proof}

Denote the semigroups associated with \eqref{e.simplesde} and \eqref{e.simplesdeapp} by $P_t$ and $P^{(n)}_t$, respectively.
\begin{corollary}\label{c.conPn}
For any $f\in C_c^\infty(\R^d)$ and $(t,\mbf x)\in \R_{\geq 0}\times \R^d$, we have $P_t^{(n)}f(\mbf x)\to P_tf(\mbf x)$ as $n\to\infty$.
\end{corollary}

\begin{proof}
We write $P_t^{(n)}f(\mbf x)-P_tf(\mbf x)=\Ee[f(X_t^{(n)})-f(X_t)|X_0^{(n)}=X_0=\mbf x]$, then apply bounded convergence theorem and the previous lemma.
\end{proof}

In the following, we use $\la\cdot,\cdot\ra$ to represent the $L^2(\R^d)$ inner product. Here is another elementary lemma.
\begin{lemma}\label{l.inv}
If a density $p(\cdot)$ satisfies
\begin{equation}\label{e.invPt}
\la P_tf,p\ra=\la f,p\ra
\end{equation} for any $f\in C_c^\infty(\R^d)$, then $p(\cdot)$ is an invariant density for $\mbf X$.
\end{lemma}

\begin{proof}
We claim \eqref{e.invPt} holds for indicator functions $f$, which would then imply that $p(\cdot)$ is an invariant density. Fix any indicator function $f$, and let $f_n\in C_c^\infty(\R^d)$ be such that $|f_n|\leq f$ for all $n$ and $f_n\to f$ Lebesgue-a.e. By assumption, $\la P_t f_n, p\ra=\la f_n,p\ra$. Using the bound $|P_tf_n(x)|\leq \sup_{x} |f_n(x)| \leq 1$ and the dominated convergence theorem, Corollary \ref{c.conPn} completes the proof.
\end{proof}

\begin{proof}[Proof of Proposition \ref{p.invariantSDE}]
We use the above deductions as well as some specific properties of the proposed invariant density $p$ (see Lemmas \ref{l.Lngzero} and \ref{l.integrability} below).

From now on, we fix $f\in C_c^\infty(\R^d)$. Let $L^{(n)}$ be the generator of $P_t^{(n)}$.  By \cite[Theorem 8.7]{LeGall-book}, we know that  $P_t^{(n)}$ is a Feller semigroup and $f$ lives in the domain of  $L^{(n)}$. Thus, 
\[
P_t^{(n)} f=f+\int_0^t L^{(n)} P_s^{(n)} f ds,
\]
where the above identity holds in $C_0(\R^d)$, the space of continuous functions that vanish at infinity (which, of course, contains $C_c^\infty(\R^d)$ as a subspace). Taking the inner product with $p(\cdot)$ on both sides, we obtain
\be \label{eq.pLint}
\la P_t^{(n)} f,p\ra=\la f,p\ra+ \int_0^t \la p, L^{(n)} P_s^{(n)} f\ra ds,
\ee
where above, we used Fubini's Theorem, noting that $L^{(n)} P_s^{(n)} f = P_s^{(n)} L^{(n)} f$, and so \[
\sup_{\mbf x \in \R^d} |L^{(n)} P_s^{(n)} f(\mbf x)| \le  \|L^{(n)} f\|_\infty,
\]
and this is finite because the coefficients and their derivatives in \eqref{e.simplesdeapp} are bounded.  Since \\ $\sup_{\mbf x \in \R^d} |P_t^{(n)}f(\mbf x)| \leq \sup_{\mbf x \in \R^d} |f(\mbf x)|$, Corollary \ref{c.conPn}  and the dominated convergence theorem imply that  
\[
\la P_t^{(n)}f,p\ra\to \la P_tf,p\ra.
\]
Then, taking limits on both sides of \eqref{eq.pLint}, by Lemma \ref{l.inv}, it suffices to show that as $n\to\infty$,
\begin{equation}\label{e.endes}
\int_0^t \la p, L^{(n)} P_s^{(n)} f\ra ds =\int_0^t\int_{\R^d} p(\mbf x) L^{(n)} P_s^{(n)} f(\mbf x)d\mbf x ds \to 0.
\end{equation}
Considering the integral in $\mbf x$, we claim that for any $s\in[0,t]$, 
\begin{equation}\label{e.ibpLn}
\int_{\R^d} p(\mbf x) L^{(n)} P_s^{(n)}f(\mbf x)d\mbf x=\int_{\R^d} (L^{(n)})^*p(\mbf x) P_s^{(n)}f(\mbf x)d\mbf x,
\end{equation}
where $(L^{(n)})^*p$ is the formal adjoint of $L^{(n)}$ applied to $p(\cdot)$. To justify the integration by parts, we note that, since $L^{(n)}$ is a uniform elliptic operator with smooth and bounded coefficients, by the classical PDE theory \cite[Chapter 1]{friedman2008partial} and the probabilistic representation of the solutions, see e.g. \cite[Chapter 5, Theorem 7.6 and Remark 7.8]{Karatzas_Shreve}, the function $P_s^n f(\mbf x)$ is the classical solution to the PDE $\partial_s P_s^{(n)} f(\mbf x)=L^{(n)} P_s^{(n)}f(\mbf x)$, where $L^{(n)}$ is the differential operator acting on $P_s^{(n)}f(\mbf x)$. In addition, the solution $P_s^{(n)}f(\mbf x)$ and its derivatives vanish at infinity.

To show \eqref{e.endes} we observe that for all $s\geq 0$
\begin{equation}\label{e.gotozero}
\bigg|\int_{\R^d} (L^{(n)})^*p(\mbf x) P_s^{(n)}f(\mbf x)d\mbf x\bigg| \leq \sup_{\mbf x\in\R^d}\big|f(\mbf x)\big| \int_{\R^d} \big|(L^{(n)})^*p(\mbf x)\big|d\mbf x.
\end{equation}
Applying Lemma \ref{l.Lngzero} below completes the proof of \eqref{e.endes} and the proposition.
\end{proof}
\begin{lemma}\label{l.Lngzero}
$\int_{\R^d} \big|(L^{(n)})^*p(\mbf x)\big|d\mbf x\to0$ as $n\to\infty$.
\end{lemma}
Before the proof of Lemma~\ref{l.Lngzero}, we recall that $p(\cdot)$ takes the form (see \eqref{eq:dens_big})
\[
p(\mbf x_1,\ldots,\mbf x_k) = C \prod_{r = 1}^k\Biggl[ \exp\Bigl(-\beta^{-2}e^{-(\theta_r - x_{r,1} - \cdots - x_{r,N-1})}\Bigr)\prod_{i = 1}^{N-1} \exp\Bigl(-\beta^{-2}e^{-x_{r,i}}\Bigr)\Biggr]=:C\prod_{r=1}^k \tilde{p}_r(\bfx_r),
\]
where $C>0$. The  function $\tilde{p}_r$ is defined on $\R^{N-1}$, and the following lemma  comes from a direct calculation.
\begin{lemma}\label{l.integrability}
For any $r=1,\ldots,k$ and $i,j=1,\ldots,N-1$, we have $\partial_i\tilde{p}_r, \partial_{ij} \tilde{p}_r\in L^1(\R^{N-1})$. 
\end{lemma}
\begin{proof}
By symmetry, it suffices to consider $\partial_1 \tilde{p}_r, \partial_{11}\tilde{p}_r$ and $\partial_{12}\tilde{p}_r$. With $\mbf y=(y_1,\ldots,y_{N-1})$, we have 
\[
\begin{aligned}
&\partial_1\tilde{p}_r(\mbf y)=\tilde{p}_r(\mbf y )\left(-\beta^{-2}e^{-(\theta_r - y_{1} - \cdots - y_{N-1})}+\beta^{-2}e^{-y_1}\right),\\
&\partial_{11} \tilde{p}_r(\mbf y)=\tilde{p}_r(\mbf y)\left(-\beta^{-2}e^{-(\theta_r - y_{1} - \cdots - y_{N-1})}+\beta^{-2}e^{-y_1}\right)^2+\tilde{p}_r(\mbf y )\left(-\beta^{-2}e^{-(\theta_r - y_{1} - \cdots - y_{N-1})}-\beta^{-2}e^{-y_1}\right),
\end{aligned}
\]
and 
\[
\begin{aligned}
\partial_{12}\tilde{p}_r(\mbf y)&=\tilde{p}_r(\mbf y)\left(-\beta^{-2}e^{-(\theta_r - y_{1} - \cdots - y_{N-1})}+\beta^{-2}e^{-y_1}\right)\left(-\beta^{-2}e^{-(\theta_r - y_{1} - \cdots - y_{N-1})}+\beta^{-2}e^{-y_2}\right)
\\
&-\tilde{p}_r(\mbf y)\beta^{-2}e^{-(\theta_r - y_{1} - \cdots - y_{N-1})}.
\end{aligned}
\]
We treat $\partial_1 \tilde{p}_r$ here (the other two can be treated analogously). By a change of variable $z_i=e^{-y_i}$, we have 
\begin{equation}\label{e.7171}
\begin{aligned}
&\int_{\R^{N-1}}|\partial_1 \tilde{p}_r(\mbf y)|d \mbf y\\
&\leq \int_{\R^{N-1}_{>0}}\exp\Big(-\frac{\beta^{-2}e^{-\theta_r}}{z_1\ldots z_{N-1}}-\beta^{-2}(z_1+\ldots z_{N-1})\Big)\left(\frac{\beta^{-2}e^{-\theta_r}}{z_1\ldots z_{N-1}}+ \beta^{-2}z_1 \right) \frac{1}{z_1\ldots z_{N-1}} dz_1\ldots dz_{N-1} \\
&= \int_{\R^{N-1}_{>0}}\exp\Big(-\beta^{-2}(z_1+\ldots z_{N-1})\Big)(C_1 + C_2 z_1) dz_1\ldots dz_{N-1} < \infty,
\end{aligned}
\end{equation}
where $C_1,C_2 > 0$ are constants, and in the last line, we used the fact that  $\sup_{x>0}e^{-\lambda x}x^\alpha < \infty$ for any $\lambda,\alpha >0$. 
  \end{proof}
  
\begin{proof}[Proof of Lemma~\ref{l.Lngzero}]
Since $L^{(n)}$ is the generator for $X^{(n)}$ which solves the SDE~\eqref{e.simplesdeapp}, we can write it as 
  \[
  L^{(n)}=\cL^{(n)}+\frac{1}{2\alpha_n^2}\Delta,
  \]
  with $\cL^{(n)}$ corresponding to the part $\mbf A_n(X^{(n)}(t))dt+\Sigma d\mbf B(t)$ of \eqref{e.simplesdeapp}. We treat the two parts separately.
  
  First, we have (owing to $\alpha_n\to\infty$ and $\Delta p(\cdot) \in L^1(\R^d)$ in light of Lemma~\ref{l.integrability}) that as $n\to \infty$, 
  \[
  \frac{1}{2\alpha_n^2} \int_{\R^d} |\Delta p(\mbf x)|d\mbf x \to0.
  \]

  Next, to deal with $\cL^{(n)}$, first, since $\mbf A_n(\mbf x)=\mbf A(\mbf x)$ for $\mbf x\in S_n$, we have $(\cL^{(n)})^*p(\mbf x)=\cL^*p(\mbf x)$ for $\mbf x\in S_n$, where $\cL^*$ was  defined explicitly in \eqref{L_star}. Since $\cL^* p=0$, see \eqref{e.Lstarzero} and the explicit calculation therein, we only need to estimate
  \begin{align}
  \int_{\R^d}| (\cL^{(n)})^*p(\mbf x) |d\mbf x&=\int_{\R^d}1\{\mbf x\in S_n^c\} |(\cL^{(n)})^*p(\mbf x)|d\mbf x,\qquad \textrm{where}\nonumber \\
  (\cL^{(n)})^*p(\mbf x)&=-\text{grad}\cdot (\mbf A_n(\mbf x)p(\mbf x))+\frac12\sum_{i,j=1}^{d}D_{ij}\partial_{ij}p(\mbf x),  \label{eq:Lnstar}
  \end{align}
and $D_{ij}=\sum_{k=1}^d \Sigma_{ik}\Sigma_{jk}$. 
For the term involving the second derivative, using the fact  that $\partial_{ij}p\in L^1(\R^d)$ (Lemma~\ref{l.integrability}), we obtain 
\[
\frac12\sum_{i,j=1}^d |D_{ij}|\int_{\R^d} 1\{\mbf x\in S_n^c\} |\partial_{ij}p(\mbf x)|d\mbf x\to0
\]
as $n\to\infty$. The other term in \eqref{eq:Lnstar} takes the form 
\[
-\sum_{i=1}^d \partial_i (A_{n,i}(\mbf x)p(\mbf x))=-\sum_{i=1}^d \big(\partial_iA_{n,i}(\mbf x) p(\mbf x)+A_{n,i}(\mbf x)\partial_ip(\mbf x)\big),
\]
where $A_{n,i}$ is the $i-$th component of $\mbf A_n$. By \eqref{e.bdbn}, for each $1 \le i \le N-1$, we have 
\begin{equation}\label{e.driftSn}
\int_{\R^d} 1\{\mbf x\in S_n^c\}|\partial_i (A_{n,i}(\mbf x)p(\mbf x))|d\mbf x \leq C\exp(Cn)\int_{\R^d} 1\{\mbf x\in S_n^c\} [p(\mbf x )+|\partial_i p(\mbf x)|]d\mbf x.
\end{equation}
For those $\mbf x\in S_n^c$, there are two cases: (i) there exists some $r,i$ so that $x_{r,i}\leq -n$ (ii) there exists some $r$ so that  $\theta_r-\sum_{i=1}^{N-1}x_{r,i}\leq -n$. In both cases, we bound the density $p(\bfx)$ as follows:
\[
\begin{aligned}
p(\bfx)&=C \prod_{r = 1}^k \Biggl[\exp\Bigl(-\beta^{-2}e^{-(\theta_r - x_{r,1} - \cdots - x_{r,N-1})}\Bigr)\prod_{i = 1}^{N-1} \exp\Bigl(-\beta^{-2}e^{-x_{r,i}}\Bigr)\Biggr]\\
&\leq C \exp\Bigl(-\frac12\beta^{-2}e^n\Bigr)\prod_{r = 1}^k \Biggl[ \exp\Bigl(-\frac12\beta^{-2}e^{-(\theta_r - x_{r,1} - \cdots - x_{r,N-1})}\Bigr)\prod_{i = 1}^{N-1} \exp\Bigl(-\frac12\beta^{-2}e^{-x_{r,i}}\Bigr)\Biggr].
\end{aligned}
\]
A similar bound can be obtained for $|\partial_i p(\mbf x)|$.  Since $\exp(Cn)\exp\Bigl(-\frac12\beta^{-2} e^{n}\Bigr)\to0$, we conclude that the right-hand side of \eqref{e.driftSn} converges to zero as $n\to\infty$, which completes the proof.
  \end{proof}

\section{Chaos series bounds and limits} \label{appx:Chaos}
We first prove some basic results about the discrete heat kernel, defined in \eqref{pN_def}. 
The following results follow straightforwardly from Stirling's approximation; we give the proof of the second.
\begin{lemma} \label{lem:pNto_p}
Let $t > s$ and $x,y \in \R$, and let $s_N,t_N,x_N,y_N$ be sequences converging to $s,t,x,y$, respectively. Then, $p_N(t_N,y_N \viiva s_N,x_N) \to \rho(t-s,y-x)$ (the full-space Gaussian heat kernel) as $N \to \infty$.
\end{lemma}

\begin{lemma} \label{lem:pn_ubd}
For each $\ve,T,K > 0$, there exists a constant $C = C(\ve,T,K)$ such that, for all sufficiently large $N$, $p_N(t,y\viiva s,x) \le C$ uniformly over $t-s \in [\ve,T]$ and $x,y \in [-K,K]$.
\end{lemma}
\begin{proof}
We use the definition of $p_N$ and Stirling's approximation to obtain, for a constant $C$, changing from line to line and depending on $\ve,K$,
\begin{align*}
    &\quad \, p_N(t,y \viiva s,x)
    = N e^{-(t-s)N^2} \f{\bigl((t-s)N^2\bigr)^{\lfloor tN^2 + yN \rfloor - \lfloor sN^2 + xN \rfloor}}{\bigl(\lfloor tN^2 + yN \rfloor - \lfloor sN^2 + xN \rfloor\bigr)!} \\
    &\le \f{CNe^{-(t-s)N^2} \bigl((t-s)N^2\bigr)^{\lfloor tN^2 + yN \rfloor - \lfloor sN^2 + xN \rfloor} }{\sqrt{\lfloor tN^2 + yN \rfloor - \lfloor sN^2 + xN \rfloor }}\Bigl(\f{e}{\lfloor tN^2 + yN \rfloor - \lfloor sN^2 + xN \rfloor}\Bigr)^{\lfloor tN^2 + yN \rfloor - \lfloor sN^2 + xN \rfloor}  \\
    &\le \f{CN}{\sqrt{\ve N^2 - 2KN - 1}} e^{\lfloor tN^2 + yN \rfloor - \lfloor sN^2 + xN \rfloor - (t-s)N^2} \Biggl(\f{(t-s)N^2}{\lfloor tN^2 + yN \rfloor - \lfloor sN^2 + xN \rfloor}\Biggr)^{\lfloor tN^2 + yN \rfloor - \lfloor sN^2 + xN \rfloor} \\
    &\le C e^{\lfloor tN^2 + yN \rfloor - \lfloor sN^2 + xN \rfloor - (t-s)N^2} \Biggl(1 + \f{(t-s)N^2 -\lfloor tN^2 + yN \rfloor + \lfloor sN^2 + xN \rfloor }{\lfloor tN^2 + yN \rfloor - \lfloor sN^2 + xN \rfloor}\Biggr)^{\lfloor tN^2 + yN \rfloor - \lfloor sN^2 + xN \rfloor}.
\end{align*}
The assumption that $t-s \in [\ve,T]$ and $x,y \in [-K,K]$, implies that
\[
\f{C_1}{N} \le \f{(t-s)N^2 -\lfloor tN^2 + yN \rfloor + \lfloor sN^2 + xN \rfloor }{\lfloor tN^2 + yN \rfloor - \lfloor sN^2 + xN \rfloor} \le \f{C_2}{N}
\]
and that $\lfloor tN^2 + yN \rfloor + \lfloor sN^2 + xN \rfloor  \ge cN^2$
for constants $c,C_1,C_2 > 0$ and sufficiently large $N$. Then, a Taylor expansion of the function $\log(1+x)$ shows that 
\[
p_N(t,y \viiva s,x) \le Ce^{\lfloor tN^2 + yN \rfloor - \lfloor sN^2 + xN \rfloor - (t-s)N^2} e^{(t-s)N^2 -\lfloor tN^2 + yN \rfloor + \lfloor sN^2 + xN \rfloor + C_3},
\]
for another constant $C_3$. This completes the proof, after a cancellation of terms. 
\end{proof}

\begin{lemma} \label{lem:intshift}
    Let $N \in \N$, $t > s$ and $x,y \in \R$. Then, if $w\in\R$ is such that $wN \in \Z$,
    \be \label{pNshift}
    p_N(t,y + w\viiva s,x) = p_N(t,y \viiva s,x-w).
    \ee
\end{lemma}
\begin{proof}
     By definition,
     \begin{align*}
         p_N(t,y + w\viiva s,x) &= Nq((t-s)N^2,\lfloor tN^2 + (y+w)N \rfloor - \lfloor sN^2 + xN \rfloor) \\
         &=   Nq((t-s)N^2,\lfloor tN^2 + yN \rfloor - \lfloor sN^2 + (x-w)N \rfloor) = p_N(t,y \viiva s,x-w).
     \end{align*}
     where we have used the fact that $\lfloor x + j \rfloor = \lfloor x \rfloor + j$ for integers $j$. 
\end{proof}

\begin{lemma} \label{lem:rho_int_comp}
For $X,Y \in \R$ and $T > S$,
\[
     \int_{\R^k_{X,Y}} \int_{\Delta^k(T \viiva S)} \prod_{i = 0}^k \rho^2(t_{i+1} - t_i, y_{i+1}-y_i) \prod_{i = 1}^k dt_i dy_i 
     = \f{(T-S)^{(k+3)/2}\Gamma(1/2)^{k+1}}{\Gamma(\f{1}{2}(k+1))2^k \pi^{k/2}}\rho^2(T-S,Y-X).
\]
\end{lemma} 

\begin{proof}
      By shift-invariance, it suffices to prove the $S = X = 0$ case. Rewriting the integrand as 
     \[
        \f{1}{2^{k+1}\pi^{(k+1)/2}}\prod_{i = 0}^k \f{1}{\sqrt{t_{i+1} - t_i}} \prod_{i = 0}^k \f{1}{\sqrt{\pi(t_{i+1} - t_i)}}e^{-\f{(y_{i+1} - y_i)^2}{t_{i+1} - t_i}},
     \]
we notice that the second product is the transition probability for a variance $1/2$ Brownian motion.  Hence, 
     \[
     \int_{\R^k_{Y}} \int_{\Delta^k(T)} \prod_{i = 0}^k \rho^2(t_{i+1} - t_i, y_{i+1}-y_i) \prod_{i = 1}^k dt_i dy_i  = \f{e^{-Y^2/T}}{2^{k+1}\pi^{(k+1)/2}\sqrt{\pi T}}  \int_{\R^k}\prod_{i = 0}^k \f{\ind\{t_{i+1} > t_i\}}{\sqrt{t_{i+1} - t_i}} \prod_{i = 1}^k\,dt_i.  
     \]
     To complete the proof, change variables and compute the Dirichlet integral. 
\end{proof}

The following is a more general version of \cite[Lemma 3.11]{GRASS-23}.  The proof here is different, somewhat streamlined, and allows for converging sequences $T_N,Y_N$ in \eqref{InkDef} below (instead of just $T$ and $Y$).  \begin{lemma} \label{lem:Dir_int_conv}
    Fix $T > 0, Y \in \R$, and sequences $T_N,Y_N$ converging to $T,Y$ respectively. For $k \in \Z_{>0}$, define
    \begin{align} \label{InkDef}
    I_{N,k}&:= \int_{\R^k}\Biggl( \prod_{i = 1}^k\,dt_i \sqrt{\f{N^2\ind\{t_i > 0\} }{\max(\lfloor t_i N^2 \rfloor, 1)}}\Biggr)  \sqrt{\f{N^2\ind\bigl\{\sum_{i = 1}^k \lfloor t_i N^2 \rfloor \le \lfloor T_N N^2 + Y_N N \rfloor\bigr\}}{\max\Big(\big(\lfloor T_N N^2 + Y_N N\rfloor - \sum_{i = 1}^k \lfloor t_i N^2 \rfloor\big),1\Big)}},\\
  \nonumber  I_k &:= \int_{\R^k} \prod_{i = 1}^{k+1} \f{\ind\{s_i > 0\}}{\sqrt s_i} \prod_{i = 1}^k \,ds_i = T^{(k-1)/2} \f{\Gamma(1/2)^k}{\Gamma((k+1)/2)} ,
    \end{align}
    with the convention $s_{k+1} = T - \sum_{i =1}^k s_i $.  Then, for each $k \ge 1$, $I_{N,k} \to I_k$ as $N \to \infty$. 
    Also, for all $N,k\in \N$,
    \begin{align} 
    \label{IMkbd} I_{N,k} &\le  I_k \Bigl(\f{\lfloor T_N N^2 + Y_N N \rfloor + k}{TN^2}\Bigr)^k \Bigl(\f{TN^2}{\lfloor T_N N^2 + Y_N N \rfloor }\Bigr)^{(k+1)/2}2^{k/2}\sqrt{k+2}  \quad \textrm{if}\quad\lfloor T_N N^2 + Y_N N \rfloor \ge 1,\\
     \label{IMkbd_alt}
  I_{N,k} &\le \f{N^{k+1}}{k!}\Bigl(\f{\lfloor T_N N^2 + Y_N N \rfloor + k}{N^2}\Bigr)^k \quad \textrm{if}\quad \lfloor T_N N^2 + Y_N N \rfloor \ge 0.
    \end{align} 
\end{lemma}
\begin{proof}
Making the change of variables $s_i = t_i \Bigl(\f{TN^2}{\lfloor T_N N^2 + Y_N N \rfloor + k}\Bigr)$ in \eqref{InkDef} yields
\be \label{eq:int1}
\begin{aligned}
&\quad \, \f{I_{N,k}}{\Bigl(\f{\lfloor T_N N^2 + Y_N N \rfloor + k}{TN^2}\Bigr)^k}  \\
&=  \int_{\R^k}\Biggl( \prod_{i = 1}^k\,ds_i \sqrt{\f{N^2 \ind\{s_i > 0\} }{\max\Big(\Bigl \lfloor s_i  \Bigl(\f{\lfloor T_N N^2 + Y_N N \rfloor + k}{T}\Bigr) \Bigr \rfloor,1\Big)}}\Biggr)\ind\Biggl\{\sum_{i = 1}^k \Bigl \lfloor s_i  \Bigl(\f{\lfloor T_N N^2 + Y_N N \rfloor + k}{T}\Bigr) \Bigr \rfloor \le \lfloor T_N N^2 + Y_N N \rfloor\Biggr\} \\
&\qquad\qquad   \times \sqrt{\f{N^2}{\max\bigg(\Big(\lfloor T_N N^2  + Y_N N\rfloor - \sum_{i = 1}^k \Bigl \lfloor s_i  \Bigl(\f{\lfloor T_N N^2 + Y_N N \rfloor + k}{T}\Bigr) \Bigr \rfloor\Big),1\bigg)}}.
\end{aligned}
\ee
For each $k \in \Z$, the Jacobian $\Bigl(\f{\lfloor T_N N^2 + Y_N N \rfloor + k}{TN^2}\Bigr)^k\to 1$ as $N \to \infty$, and the integrand converges to
$
\prod_{i = 1}^{k+1} \f{\ind\{s_i > 0\}}{\sqrt s_i}.
$
We use the dominated convergence theorem to show the integral converges. For $s_i > 0$,
\be \label{eq:xm1h}
    \max\Big(\Bigl \lfloor s_i  \Bigl(\f{\lfloor T_N N^2 + Y_N N \rfloor + k}{T}\Bigr) \Bigr \rfloor, 1\Big)  \ge \max\Big(\Bigl(s_i  \f{\lfloor T_N N^2 + Y_N N \rfloor}{T} - 1\Bigr), 1\Big) \ge s_i \f{\lfloor T_N N^2 + Y_N N \rfloor }{2T}, 
\ee
where the last inequality uses (with $\ell=1$) the inequality that for any $\ell>0$ and all $x\in \R$
\begin{equation}\label{lem:veebd}
    \max(x-\ell,1)\geq \frac{x}{\ell+1}.
\end{equation}
Observe that for $t_1,\ldots,t_k \in \R$, $\sum_{i = 1}^k \lfloor t_i N^2 \rfloor \ge \sum_{i = 1}^k  t_i N^2 - k$,  so if $ \lfloor T_N N^2 + Y_N N \rfloor\geq \sum_{i = 1}^k \lfloor t_i N^2 \rfloor$, then 
\[
\sum_{i = 1}^k t_i \le \f{\lfloor T_N N^2 + Y_N N \rfloor + k}{N^2}.
\]
In terms of the $s_i$ variables, this conclusion is precisely that $\sum_{i = 1}^k s_i \le T$. On this event
\be\label{eq:xmkh2}
\begin{aligned}
&\quad \, \max\bigg(\Bigl(\lfloor T_N N^2 + Y_N N \rfloor - \sum_{i = 1}^k \Bigl\lfloor s_i  \Bigl(\f{\lfloor T_N N^2 + Y_N N \rfloor + k}{T}\Bigr)\Bigr \rfloor\Bigr), 1\bigg) \\
&\ge \max\bigg(\Bigl(\f{\lfloor T_N N^2 + Y_N N \rfloor}{TN^2}(TN^2 - \sum_{i = 1}^k s_i N^2)  - k\Bigr) , 1\bigg) \ge \f{\lfloor T_N N^2 + Y_N N \rfloor}{TN^2}\f{(TN^2 - \sum_{i = 1}^k s_i N^2)}{k + 1},
\end{aligned}
\ee
where we applied \eqref{lem:veebd} with $\ell=k$ in the last step. Combining \eqref{eq:xm1h} and \eqref{eq:xmkh2}, we arrive at the domination
\[
\textrm{Integrand in }\eqref{eq:int1} \leq 
 \Bigl(\f{TN^2}{\lfloor T_N N^2 + Y_N N \rfloor }\Bigr)^{(k+1)/2}2^{k/2}\sqrt{k+1} \prod_{i = 1}^{k+1} \f{\ind\{s_i > 0\}}{\sqrt s_i}.
\]
As $I_k<\infty$ the dominated convergence theorem shows $I_{N,k} \to I_k$. Integrating the above bound yields \eqref{IMkbd}. 

To get the bound \eqref{IMkbd_alt}, we substitute the bound $\sqrt{\f{N^2}{ \max(\lfloor t_i N^2 \rfloor, 1)}} \le N$ into \eqref{InkDef} to obtain
\begin{align*}
    I_{N,k} &\le N^{k+1} \int_{\R^k}\ind\Biggl\{\sum_{i = 1}^k \lfloor t_i N^2 \rfloor \le \lfloor T_N N^2 + Y_N N \rfloor\Biggr\} \prod_{i = 1}^k\,dt_i \ind\{t_i > 0\}   \\
    &\le N^{k+1} \int_{\R^k}\ind\Biggl\{\sum_{i = 1}^k t_i \le \f{\lfloor T_N N^2 + Y_N N \rfloor + k}{N^2}\Biggr\} \prod_{i = 1}^k\,dt_i \ind\{t_i > 0\}.
\end{align*}
\eqref{IMkbd_alt} follows from the changing to  $s_i$ variables and computing the Dirichlet integral.
\end{proof}

In the following lemma, we recall the shorthand notation $\Delta^k(t) = \Delta^k(t \viiva 0)$ and $\R_y^k = \R_{0,y}^k$. Note that the independence of the constant $C$ below with respect to $Y$ and $Y_N$ is important because we apply this result uniformly for $Y_N = Y_N' - j$, where $Y_N' \in [0,1]$ and $j \in \Z$.

\begin{lemma} \label{lem:pNint_conv}
 Let $T_N,Y_N$ be sequences converging to $T,Y$, respectively, where $Y \in \R$ and $T > 0$. Then, 
\begin{multline} \label{eq:pN_full_conv}
\lim_{N \to \infty} \int\limits_{\R^k_{Y_N}} \int\limits_{\Delta^k(T_N)} \prod_{i = 0}^k p_{N}^2(t_{i+1},y_{i+1} \viiva t_i,y_i) \prod_{i = 1}^k dt_i \,dy_i = \int\limits_{\R^k_{Y}} \int\limits_{\Delta^k(T)} \prod_{i = 0}^k \rho^2(t_{i+1} - t_i,y_{i+1} - y_i) \prod_{i = 1}^k dt_i \,dy_i.
\end{multline}
Also, there exists $C>0$ depending on $T$ and the sequence $T_N$ (but not on $Y$ or $Y_N$) so that, for all $N\in \N$,
\be \label{eq:full_int_bd}
\begin{aligned}
 &\int\limits_{\R^k_{Y_N}} \int\limits_{\Delta^k(T_N)} \prod_{i = 0}^k p_{N}^2(t_{i+1},y_{i+1} \viiva t_i,y_i) \prod_{i = 1}^k dt_i \,dy_i \\
&\qquad \le \f{C^k}{\Gamma((k+1)/2)} p_N^2(T_N,Y_N) \sqrt{\f{\lfloor T_N N^2 + Y_N N \rfloor + 1}{N^2} }  \Bigl(\f{N^2}{\lfloor T_N N^2 + Y_N N \rfloor }\Bigr)^{(k+1)/2}
\end{aligned}
\ee
if $\lfloor T_N N^2 + Y_N N \rfloor \ge 1$, and so that if $\lfloor T_N N^2 + Y_N N \rfloor \ge 0$,
\be \label{eq:fib2}
\int\limits_{\R^k_{Y_N}} \int\limits_{\Delta^k(T_N)} \prod_{i = 0}^k p_{N}^2(t_{i+1},y_{i+1} \viiva t_i,y_i) \prod_{i = 1}^k dt_i \,dy_i  \le  \f{C^k N^{k+1}}{k!} p_N^2(T_N,Y_N) \sqrt{\f{\lfloor T_N N^2 + Y_N N \rfloor + 1}{N^2} }.
\ee
\end{lemma}
\begin{proof}
For $\lfloor t_{i+1} N^2 + y_{i+1} N \rfloor \ge \lfloor t_{i} N^2 + y_{i} N \rfloor$, by definition we have 
\[
p_{N}^2(t_{i+1},y_{i+1} \viiva t_i,y_i) = N^2 e^{-2(t_{i+1} - t_i)N^2} \f{\big((t_{i+1} - t_i)N^2\big)^{2\big(\lfloor t_{i+1} N^2 + y_{i+1} N \rfloor - \lfloor t_{i} N^2 + y_{i} N \rfloor\big)}}{\Big(\big(\lfloor t_{i+1} N^2 + y_{i+1} N \rfloor - \lfloor t_{i} N^2 + y_{i} N \rfloor\big)!\Big)^2},
\]
while $p_N(t_{i+1},y_{i+1} \viiva t_i,y_i) = 0$ when $\lfloor t_{i+1} N^2 + y_{i+1} N \rfloor < \lfloor t_{i} N^2 + y_{i} N \rfloor$. Stirling's approximation yields
\begin{multline*}
f_N := \ind\{\R_{Y_N}^k \times \Delta^k(T_N)\}\prod_{i = 0}^k p_{N}^2(t_{i+1},y_{i+1} \viiva t_i,y_i)
\to \ind\bigl\{\R_{Y}^k \times \Delta^k(T )\bigr\}\prod_{i = 0}^k \rho^2(t_{i+1} - t_i,y_{i+1} - y_i) =: f,
\end{multline*}
    where we have used the shorthand notation
    \[
\ind\{\R_{Y}^k \times \Delta^k(T)\} := \ind\bigl\{(y_0,\ldots,y_{k+1}) \in \R_{Y}^k, (t_0,\ldots,t_{k+1}) \in \Delta^k(T)\bigr\}.
\]
We use the generalized dominated convergence theorem, Lemma \ref{gdct} with $f_N$ as above to show the integrals also converge.  By Stirling's approximation
\be \label{eq:Stir1}
(n!)^2 \sim \f{(2n)! \sqrt{\pi n}}{2^{2n}}.
\ee
Here $a_n\sim b_n$ means that $a_n/b_n\to 1$ as $n \to \infty$. Thus, there exists $C > 0$ so that for $n \ge 0$, 
\be \label{eq:Stir2}
\f{1}{(n!)^2} \le C \f{2^{2n}}{(2n)!\sqrt{\pi \max(n,1)}}.
\ee
Stirling's approximation also implies that for $k\geq 0$ fixed, as $n \to \infty$
\be \label{eq:Stir3}
\f{(n!)^2 2^{2n}}{(2n + k)!} \sim \f{\sqrt{\pi n}}{2^k n^k},
\ee
and that there exist $C_1,C_2 > 0$ so that for all $n,k \ge 0$,  
\be \label{eq:Stir4}
\f{(n!)^2 2^{2n}}{(2n + k)!} \le C_1 e^{k} \sqrt n \Bigl(\f{n}{n + k/2}\Bigr)^{2n} \f{1}{(2n + k)^k} \le C_2^k \f{\sqrt{n+1}}{(2n + k)^k}
\ee
Inserting the bound \eqref{eq:Stir2} into the definition of $p_N$, our choice of dominating function $g_N$ is 
\begin{align*}
g_N &:= C^{k+1} N^{2(k+1)}  \f{(2N^2)^{2(\lfloor T_N N^2 + Y_N N \rfloor )}}{\pi^{(k+1)/2}e^{2 T_N N^2}}  \ind\{\R_{Y_N}^k \times \Delta^k(T_N)\} \\
&\qquad\times \prod_{i = 0}^k \f{(t_{i+1} - t_i)^{2(\lfloor t_{i+1} N^2 + y_{i+1} N \rfloor - \lfloor t_{i} N^2 + y_{i} N \rfloor)}\ind\bigl\{\lfloor t_{i+1} N^2 + y_{i+1} N \rfloor - \lfloor t_{i} N^2 + y_{i} N \rfloor \ge 0 \bigr\}}{\Big(2\big(\lfloor t_{i+1} N^2 + y_{i+1} N \rfloor - \lfloor t_{i} N^2 + y_{i} N \rfloor\big)\Big)!\sqrt{\max\Big((\lfloor t_{i+1} N^2 + y_{i+1} N \rfloor - \lfloor t_{i} N^2 + y_{i} N \rfloor) ,1\Big)}}
\end{align*}
where $C$ is from \eqref{eq:Stir2}. Just as $f_N\to f$ we see now that $g_N\to g:=C^{k+1}f$. Lemma \ref{lem:rho_int_comp} implies that 
\be \label{rhointcomp}
     \begin{aligned}
     \int g &= \quad \; C^{k+1} \int\limits_{\R^k_{Y}} \int\limits_{\Delta^k(T)} \prod_{i = 0}^k \rho^2(t_{i+1} - t_i, y_{i+1}-y_i) \prod_{i = 1}^k dt_i dy_i \\
     &= C^{k+1} \f{\sqrt{T}}{2^{k}\pi^{k/2}} \rho^2(T,Y) \int_{B_k} \Biggl(\prod_{i = 1}^{k}\,dt_i\f{1}{\sqrt t_i}\Biggr) \f{1}{\sqrt{T -\sum_{i = 1}^k t_i}} < \infty,
     \end{aligned}
     \ee
     where 
$B_k = \{(t_1,\ldots,t_k) \in \R^k: t_i > 0, 1 \le i \le k,\; \sum_{i = 1}^{k} t_i < t-s\}$.

To bound the integral of $g_N$, we make the change of variables $x_i = t_iN^2 + y_i N$ so that 
\begin{align}
\int g_N &= C^{k+1} N^{k+2}  \f{(2N^2)^{2(\lfloor T_N N^2 + Y_N N \rfloor )}}{\pi^{(k+1)/2}e^{2T_N N^2}} \nonumber  \\
&\, \times \int_{\R^k_{\lfloor T_N N^2 + Y_N N \rfloor}} \int_{\Delta^k(T_N)} \prod_{i = 0}^k  \f{(t_{i+1} - t_i)^{2(\lfloor x_{i+1} \rfloor - \lfloor x_i \rfloor)} \ind\bigl\{\lfloor x_{i+1} \rfloor \ge \lfloor x_i \rfloor\bigr\}}{\big(2(\lfloor x_{i+1} \rfloor  - \lfloor x_i \rfloor) \big)!\sqrt{\max\big((\lfloor x_{i+1} \rfloor - \lfloor x_i \rfloor), 1\big)}} \prod_{i = 1}^k dt_i\,dx_i \nonumber\\
&= C^{k+1} N^{k+2}  \f{(2N^2T_N)^{2\lfloor T_N N^2 + Y_N N \rfloor } T_N^k }{\pi^{(k+1)/2}e^{2T_N N^2}\big(2\lfloor T_N N^2 + Y_N N \rfloor  + k\big)!} \nonumber\\
&\qquad\qquad  \times \int_{\R^k_{ \lfloor T_N N^2 + Y_N N \rfloor}} \prod_{i = 0}^k  \f{ \ind\bigl\{\lfloor x_{i+1} \rfloor \ge \lfloor x_i \rfloor\bigr\}}{\sqrt{\max\big((\lfloor x_{i+1} \rfloor - \lfloor x_i \rfloor),1\big)}} \prod_{i = 1}^k\,dx_i \nonumber\\
&\le C^{k+1} N^{k} \f{T_N^k}{\pi^{(k+1)/2}}  \f{\big(\lfloor T_N N^2 + Y_N N \rfloor !\big)^2 2^{2\lfloor T_N N^2 + Y_N N \rfloor}}{\big(2\lfloor T_N N^2 + Y_N N \rfloor + k\big)!} \nonumber\\
&\qquad \times p_N^2(T_N, Y_N) \int_{\R^k_{ \lfloor T_N N^2 + Y_N N \rfloor}} \prod_{i = 0}^k  \f{ \ind\bigl\{\lfloor x_{i+1} \rfloor \ge \lfloor x_i \rfloor\bigr\}}{\sqrt{\max\big((\lfloor x_{i+1} \rfloor - \lfloor x_i \rfloor), 1\big)}} \prod_{i = 1}^k\,dx_i \nonumber\\
&\sim C^{k+1} N^{1-k} \f{\sqrt{T}}{2^k \pi^{k/2}} \rho^2 (T, Y )  \int_{\R^k_{\lfloor T_N N^2 + Y_N N \rfloor}} \prod_{i = 0}^k  \f{ \ind\bigl\{\lfloor x_{i+1} \rfloor \ge \lfloor x_i \rfloor\bigr\}}{\sqrt{\max\big((\lfloor x_{i+1} \rfloor - \lfloor x_i \rfloor), 1\big)}} \prod_{i = 1}^k\,dx_i. \label{eq:int_left}
\end{align}
where the first equality is from the change of variables, the second is from taking the Dirichlet integral in the $t_i$ variables, the third is from Stirling's approximation, and the final asymptotic bound is from \eqref{eq:Stir3} with $n = \lfloor T_N N^2 + Y_N N \rfloor$.
The integral from \eqref{eq:int_left} equals
\begin{align*}
&\quad \sum_{0 = n_0 \le n_1 \le \cdots \le n_{k} \le n_{k+1} = \lfloor T_N N^2 + Y_N N \rfloor}  \prod_{i = 1}^{k+1}  \f{1}{\sqrt{\max\big((n_{i+1} - n_i), 1\big)}} = \sum_{\substack{m_i \ge 0 \\ \sum_{i = 1}^{k+1} m_i =\lfloor T_N N^2 + Y_N N \rfloor }} \prod_{i = 1}^{k+1} \f{1}{\sqrt{\max(m_i,1)}} \\
&= \int_{\R^k} \ind\bigl\{s_i > 0, 1 \le i \le k,\; s_{k+1} = \lfloor T_N N^2 + Y_N N \rfloor - \sum_{i = 1}^k \lfloor s_i  \rfloor \ge 0  \bigr\} \prod_{i = 1}^{k+1} \f{1}{\sqrt{\max(\lfloor s_i \rfloor,1)}}\prod_{i = 1}^k \,ds_i \\
&= N^{k - 1}\int_{B_k(N)} \Biggl(\prod_{i = 1}^k\,dt_i  \sqrt{\f{N^2}{\sqrt{\max(\lfloor t_i N^2 \rfloor, 1) }}} \Biggr)\sqrt{\f{N^2}{\max\Big(\big(\lfloor T_NN^2 + Y_N N \rfloor - \sum_{i = 1}^k \lfloor t_iN^2 \rfloor\big) , 1\Big)  }},
\end{align*}
where the last integral is taken over the set 
\[
     B_k(N) = \Bigl\{t_i > 0,\; 1 \le i \le k,\; \sum_{i = 1}^k \lfloor t_i N^2 \rfloor  \le \lfloor T_NN^2 + y_N N \rfloor \}.
     \]
Comparing to \eqref{rhointcomp}, to prove \eqref{eq:pN_full_conv}, it now suffices to show that 
\begin{align}
&\quad \, \lim_{N \to \infty} \int_{B_k(N)} \Biggl(\prod_{i = 1}^k\,dt_i  \sqrt{\f{N^2}{\sqrt{\max(\lfloor t_i N^2 \rfloor, 1) }}} \Biggr)\sqrt{\f{N^2}{\max\Big(\big(\lfloor T_NN^2 + Y_N N \rfloor  - \sum_{i = 1}^k \lfloor t_iN^2 \rfloor\big),1\Big)  }} \label{INkBig} \\
&= \int_{B_k}  \Biggl(\prod_{i = 1}^{k}\,dt_i\f{1}{\sqrt t_i}\Biggr) \f{1}{\sqrt{T - \sum_{i = 1}^k t_i}}, \nonumber 
\end{align}
and this is shown in Lemma \ref{lem:Dir_int_conv}

To prove \eqref{eq:full_int_bd}, we use the above computation of the integral, and instead of getting an asymptotic equivalence in \eqref{eq:int_left}, we use \eqref{eq:Stir4} to get that the integral in \eqref{eq:full_int_bd} is bounded above by
\begin{align*}
C^k p_N^2(T_N,Y_N) \sqrt{\f{\lfloor T_N N^2 + Y_N N \rfloor + 1}{N^2}  } \Biggl(\f{N^2}{2\lfloor T_N N^2 + Y_N N \rfloor + k}\Biggr)^k I_{N,k}(T_N,Y_N),
\end{align*}
where the constant $C$ depends on $T$ and the sequence $T_N$, but not $Y$ and $Y_N$,  and $I_{N,k}(T_N,Y_N)$ is the integral in \eqref{INkBig}.
The bound \eqref{eq:full_int_bd} follows by equation \eqref{IMkbd} of Lemma \ref{lem:Dir_int_conv}.
\end{proof}

\section{KMT coupling and Brownian bridge results}\label{sec:appBB}
\begin{lemma}\cite[Chapter 4, (3.40)]{Karatzas_Shreve} \label{lem:BB_max}
Let $\B \deq \B_{1,0}$ (a Brownian bridge with $\B(0) = \B(1) = 0$). For $u \ge 0$,
\[
\Pp\Big(\sup_{0 \le x \le 1} |\B(x)| > u\big) \le 2e^{-2u^2}.
\]
\end{lemma}
We will use a special log-inverse-gamma random walk case of the KMT coupling result of \cite{Dmitrov-Wu-2021} for comparing random walk bridges and their Brownian bridge limits. To match notation, let $-Y_1(\gamma),\ldots,-Y_{N}(\gamma)$ be i.i.d. log-inverse-gamma random variables with shape $\gamma > 0$ and scale $1$ (recall Definition \ref{def:p_meas}). For $\theta \in \R$, let $u \mapsto S_{N,\theta}(u;\gamma)$ be the continuously-interpolated random walk bridge on $[0,N]$ conditioned to have $S_{N,\theta}(0;\gamma) = 0$, $S_{N,\theta}(N;\gamma) = \theta$, and increments $S_{N,\theta}(i;\gamma)-S_{N,\theta}(i-1;\gamma)   = \f{1}{\sqrt{\psi_1(\gamma)}}\bigl(Y_i(\gamma)-\psi(\gamma)\bigr)$ for $i\in \{1,\ldots, N\}$. The digamma $\psi$ and trigamma $\psi_1$ functions are the mean and variance of $Y_i(\gamma)$.

\begin{theorem}\cite[Theorem 8.1]{Dmitrov-Wu-2021}\label{thm:KMTBB}
    For any $b > 0$ and $\gamma_0 > 0$, there exists constants $0 < C,a,\alpha' <\infty$ such that the following holds. For every positive integer $N$ and $\gamma \ge \gamma_0$, there is a probability space $(\Omega_N,\mathcal F_N,\Pp_N)$ on which are defined a variance $1$, slope $\f{\theta}{\sqrt N}$ Brownian bridge $\B_{1,\f{\theta}{\sqrt N}}(x), x \in [0,1]$ and the family of processes $\bigl(S_{N,\theta}(x;\gamma): x \in [0,N], \theta \in \R\bigr)$ such that, for any $u \ge 0$,
    \be \label{eq:KMT}
    \Pp_N\Biggl(\sup_{0 \le x \le N} \Bigl|\sqrt N \B_{1,\f{\theta}{\sqrt N}}(N^{-1} x)- S_{N,\theta}(x;\gamma)\Bigr| \ge u \log N\Biggr) \le C N^{\alpha' - a u} e^{b\theta^2/N}.
    \ee 
\end{theorem}

For our application, we need to allow for general shape and scale parameters that vary  with $N$.
\begin{lemma} \label{lem:BB_coup}
    Let $\theta \in \R$ and  $\beta > 0$. For $N \in \N$, let $\beta_N = \beta N^{-1/2}$,  $\gamma_N$ be such that $\psi(\gamma_N) = \log(\beta_N^{-2})$, and $\theta_N = \f{\theta}{\beta} \sqrt{N \psi_1(\gamma_N)}$. Let $(X_i)_{i \in \Z_N} \sim \nu_{\beta_N}^{N,(\theta_N)}$ (recall from Definition \ref{def:p_meas}). Let $\wt S_{N,\theta_N}:[0,N]\to \R$ be the continuously-interpolated random walk bridge define by setting $\wt S_{N,\theta_N}(0) = 0$ and $\wt S_{N,\theta_N}(i) = \sum_{j = 1}^i X_j$. Then, there exist $C,a,\alpha' > 0$ and, for each $N\in \N$, probability spaces  $(\Omega_N,\Ff_N,\Pp_N)$ on which $\wt S_{N,\theta_N}$ and a variance $\beta^2$, slope $\theta$ Brownian bridge $\wt \B_{\beta,\theta}$ are defined such that for all $N \in \N$ and $u \ge 0$,
\be \label{eq:KMTfinal}
\Pp_N\Biggl(\sup_{0 \le x \le 1} \Bigl|\wt \B_{\beta,\theta}(x) - \wt S_{N,\theta_N}(\lfloor Nx\rfloor ) \Bigr| \ge \f{u \log N}{\sqrt N}   \Biggr) \le CN^{\alpha' - au}.
\ee
\end{lemma}
\begin{proof}
We use Theorem \ref{thm:KMTBB}. Replace $\theta$ with $\theta'\sqrt N$, divide by $\sqrt N$, and use scaling invariance of Brownian bridge to rewrite \eqref{eq:KMT} as
\be \label{eq:KMT2}
\Pp_N\Biggl(\sup_{0 \le x \le 1} \Bigl|\B_{1,\theta'}(x) - \f{1}{\sqrt N} S_{N,\theta'\sqrt N}(Nx;\gamma_N)\Bigr| \ge u \f{\log N}{\sqrt N}\Biggr) \le C N^{\alpha' - am} e^{bz'}.
\ee
Here $\B_{1,\theta'}$ is a variance $1$, slope $\theta'$  Brownian bridge, and $S_{N,\theta' \sqrt N}$ is the continuously-interpolated random walk bridge on $[0,N]$ conditioned to have $S_{N,\theta' \sqrt N}(0) =0$, $S_{N,\theta' \sqrt N}(N) = \theta' \sqrt N$, and i.i.d. increments $S_{N,\theta' \sqrt N}(i)-S_{N,\theta' \sqrt N}(i-1)= Y_i(\gamma_N)$ for $i\in \{1,\ldots,N\}$ where $-Y_i(\gamma_N)$ has log-inverse-gamma distribution with shape $\gamma_N$ and scale $1$.

We will apply \eqref{eq:KMT2} with the following parameters:  $\theta' = -\f{\theta}{\beta}$, and $\gamma = \gamma_N$ is chosen so that $\psi(\gamma_N) = \log \beta_N^{-2}$, where $\beta_N = \beta N^{-1/2}$. Observe that, if $(X_i)_{i \in \Z_N}$ are i.i.d 
 log-inverse-gamma random variables with shape $\gamma$ and scale $\beta$, and if we set $Y_i = \log \beta^{-2} - X_i$, then $(-Y_i)_{i \in \Z_N}$ are i.i.d. log-inverse-gamma random variables with shape $\gamma$ and scale $1$. Thus, for i.i.d. log-inverse random variables $-Y_1(\gamma_N),\ldots,-Y_N(\gamma_N)$ with shape $\gamma_N$ and rate $1$, we set $X_i = X_i(\gamma_N,\beta_N) = \log \beta_N^{-2} - Y_i(\gamma_N)$. Then, $u\mapsto S_{N,\theta' \sqrt N}(u;\gamma)$ is the continuous-interpolated random walk bridge on $[0,N]$ conditioned to have $S_{N,\theta' \sqrt N}(0;\gamma) = 0$ and with increments  
$S_{N,\theta' \sqrt N}(i;\gamma)-S_{N,\theta' \sqrt N}(i-1;\gamma)=-\f{1}{\sqrt{\psi_1(\gamma_N)}} X_i(\gamma_N,\beta_N),
$
conditioned on 
\be \label{eq:Xi_cond}
\sum_{i = 1}^N  X_i(\gamma_N,\beta_N) = \f{\theta}{\beta}\sqrt{ N \psi_1(\gamma_N)} = \theta_N.
\ee
Observe now that if $x \mapsto \wt S_{N,\theta_N}(x)$ be the continuously interpolated random walk $[0,N] \to \R$ with increments
$X_i(\gamma_N,\beta_N)$ (without any rescaling), conditioned on \eqref{eq:Xi_cond}, then, under this conditioning, $\bigl(X_i(\gamma_N,\beta_N)\bigr)_{i \in \Z_N} \sim \nu_{\beta_N}^{N,(\theta_N)}$. This matches the assumption of the lemma. 

The asymptotic expansion for $\psi$ (e.g. \cite[Equation 5.11.2]{NIST-2003}) is $\psi(x) = \log x - \f{1}{2x} + O(x^{-2})$ as $x \to +\infty$. Hence, since we chose $\gamma_N$ so that $\log(\beta_N^{-2})  = \psi(\gamma_N)$, we see that $\gamma_N\to \infty$ as $N \to \infty$ and that  
\[
\log\f{\beta_N^{-2}}{\gamma_N} = \psi(\gamma_N) - \log \gamma_N = O\Bigl(\f{1}{\gamma_N}\Bigr).
\]
Then,
\[
\beta_N^{-2} - \gamma_N = \gamma_N (e^{\log(\beta_N^{-2})-\log(\gamma_N)} - 1) = \gamma_N \Bigl(e^{O(\gamma_N^{-1})} - 1)\Bigr) = O(1),\quad\text{so}\quad \f{N}{\gamma_N} = \beta^{-2} + O\Bigl(\f{1}{N}\Bigr), 
\]
where in the last step, we recall that
$\beta_N^{-2} = \beta^{-2} N$. By the asymptotic expansion for the trigamma function (e.g. \cite[Equation 5.15.8]{NIST-2003}), $\psi_1(x) = \f{1}{x} + \f{1}{2x^2} + O(x^{-3})$,
\be \label{sq_rt_asympt}
\sqrt{N \psi_1(\gamma_N)} = \sqrt{\f{N}{\gamma_N} + O\Bigl(\f{1}{N}\Bigr)} = \beta + O(N^{-1/2}).
\ee
 Then, \eqref{eq:KMT2} gives us 
\be \label{eq:KMT3}
\Pp_N\Biggl(\sup_{0 \le x \le 1} \Biggl|\beta \B_{1,-\f{\theta}{\beta}} + \f{\beta}{\sqrt{N\psi_1(\gamma_N)}}\wt S_{N,\theta}(Nx) \Biggr| \ge \f{u\beta  \log N}{\sqrt N}   \Biggr) \le CN^{\alpha' - au} e^{b\theta \beta^{-1}}.
\ee
But \eqref{sq_rt_asympt} implies $\f{\beta}{\sqrt{N \psi_1(\gamma_N)}} = 1 + O(N^{-1/2})$, so  we have
\begin{align*}
    &\quad \, \sup_{0 \le x \le 1} \Bigl|\beta \B_{1,-\f{\theta}{\beta}}(x)  + \wt S_{N,\theta}(Nx) \Bigr| \\
    &\le  \sup_{0 \le x \le 1} \Biggl|\beta\B_{1,-\f{\theta}{\beta}}(x) + \f{\beta}{\sqrt{N\psi_1(\gamma_N)}}\wt S_{N,\theta}(Nx) \Biggr| +\biggl|\f{\beta}{\sqrt{N\psi_1(\gamma_N)}} - 1\biggr| \sup_{0 \le x \le 1} \Bigl|\wt S_{N,\theta}(Nx)  \Bigr| \\
    &\le \sup_{0 \le x \le 1} \Biggl|\beta \B_{1,-\f{\theta}{\beta}}(x)+ \f{\beta}{\sqrt{N\psi_1(\gamma_N)}}\wt S_{N,\theta}(Nx) \Biggr| +\wt C N^{-1/2}  \sup_{0 \le x \le 1} \Bigl|\wt S_{N,\theta}(Nx)  \Bigr|,
\end{align*}
for a constant $\wt C > 0$. Hence, we have 
\begin{align*}
    &\quad \, \Pp_N\Biggl(\sup_{0 \le x \le 1} \Bigl|\beta \B_{1,-\f{\theta}{\beta}}(x) + \wt S_{N,\theta}(Nx) \Bigr| \ge \f{u \log N}{\sqrt N}   \Biggr)  \\
    &\le \Pp_N\Biggl(\sup_{0 \le x \le 1} \Biggl|\beta \B_{1,-\f{\theta}{\beta}}(x)+ \f{\beta}{\sqrt{N\psi_1(\gamma_N)}}\wt S_{N,\theta}(Nx) \Biggr| \ge \f{u\beta  \log N}{2\sqrt N}   \Biggr) + \Pp_N\Biggl(\sup_{0 \le x \le 1} \Bigl|\wt S_{N,\theta}(N,x) \Bigr| > \f{u\log N}{2 \wt C}\Biggr).
\end{align*}
The first term is handled by \eqref{eq:KMT3}. For the second term, we note
\[
\sup_{0 \le x \le 1} \Bigl|\wt S_{N,\theta}(N,x) \Bigr| \le \f{\sqrt{N\psi_1(\gamma_N)} }{\beta}\Biggl( \Biggl|\beta \B_{1,-\f{\theta}{\beta}}(x)+ \f{\beta}{\sqrt{N\psi_1(\gamma_N)}}\wt S_{N,\theta}(Nx) \Biggr|  +  \sup_{0 \le x \le 1} |\beta\B_{1,-\f{\theta}{\beta}}(x)| \Biggr).
\]
Applying \eqref{eq:KMT3} along with the tail bound $\Pp_N(\sup_{0 \le x \le 1} |\B_{1,0}(x)| > m) \le 2e^{-2m^2}$ from Lemma \ref{lem:BB_max} and the scaling invariance $\beta \B_{1,-\f{\theta}{\beta}}(x) \deq \beta \B_{1,0}(x) - \theta x$ , we obtain
\be \label{eq:KMT4}
\Pp_N\Biggl(\sup_{0 \le x \le 1} \Bigl|\wt \B_{\beta,\theta}(x) - \wt S_{N,\theta}(Nx) \Bigr| \ge \f{u \log N}{\sqrt N}   \Biggr) \le CN^{\alpha' - au},
\ee
for constants $C,\alpha',a > 0$, and we have defined $\wt \B_{\beta,\theta}(x) = -\beta \B_{1,-\f{\theta}{\beta}}(x)$. By scale invariance of Brownian bridge,  $\wt \B_{\beta,\theta}$ is a  variance $\beta^2$, slope $\theta$ Brownian bridge, and so the proof of \eqref{eq:KMTfinal} is nearly complete; in \eqref{eq:KMT4}, we only need to replace $Nx$ in the argument of $\wt S_{N,\theta}$ with $\lfloor Nx \rfloor$. 
We now write  
\begin{align}
&\quad \, \Pp\Biggl(\sup_{0 \le x \le 1} \Bigl|\wt \B_{\beta,\theta}(x) - \wt S_{N,\theta}(\lfloor Nx\rfloor ) \Bigr| \ge \f{u \log N}{\sqrt N}   \Biggr) \nonumber \\
&\le  \Pp\Biggl(\sup_{0 \le x \le 1} \Biggl|\wt \B(\f{\lfloor Nx \rfloor}{N})  - \wt S_{N,\theta}(\lfloor Nx\rfloor ) \Biggr| \ge \f{u \log N}{2\sqrt N}   \Biggr) \label{892}\\
&\qquad\qquad + \Pp\Biggl(\sup_{0 \le x \le 1} \Biggl|\wt \B_{\beta,\theta}(\f{\lfloor Nx \rfloor}{N}\Bigr)  - \wt \B_{\beta,\theta}(x) \Biggr| \ge \f{u \log N}{2\sqrt N}   \Biggr). \label{893}
\end{align}
We bound the term \eqref{892} using  \eqref{eq:KMT4}. For the term \eqref{893}, we use the distributional equality $\B_{\beta,\theta}(x) \deq \beta B(x) - \beta x B(1) + \theta x$, where $B$ is a standard Brownian motion. This term is bounded using the modulus of continuity of Brownian motion (see, e.g., \cite[Theorem 9.25]{Karatzas_Shreve}) and simple moment bounds on $B(1)$. 
\end{proof}

\begin{lemma} \label{lem:BB_to_BM}
Let $f_L \deq \B_{\beta \sqrt L,\theta L}$, and extend $f_L$  to a function $\R \to \R$ so that $f_L(x+1) - f_L(x) = f_L(1) - f_L(0)$ for all $x \in \R$. Define the rescaled function $f_L:\R \to \R$ by $\wt f_L(x) = f_L(\f{x}{L})$, 
Then, as $L \to \infty$, the rescaled $f_L$ converges in distribution to $B_{\beta,\theta}$ (two-sided Brownian motion with variance $\beta^2$ and drift $\theta$), with respect to the topology of uniform convergence on compact sets.  
\end{lemma}

\begin{proof}
We prove this by constructing a two-sided Brownian motion $B$ and a sequence $f_L' \deq f_L$ so that $f_L'$ converges uniformly on compact sets to $B$.

By Brownian scaling,  $\wt f_L:[0,L]\to \R$ has the law of $x\mapsto \beta\bigl(\wt B(x) - \f{x}{L}\wt B(L)\bigr) + \theta x$, where $B$ is a standard Brownian motion on $[0,\infty)$. It therefore suffices to consider the case of $\beta = 1$ and $\theta = 0$. Let $B$ be a standard, two-sided Brownian motion, and define a new Brownian motion $\wt B_L$ on $[0,L]$ as follows:
\[
\wt B_L(x) := \begin{cases}
B(x) &0 \le x \le \f{L}{2} \\
B\bigl(\f{L}{2}\bigr) + B(x-L) - B\bigl(-\f{L}{2}\bigr) &\f{L}{2} \le x \le L.
\end{cases}
\]
From this, define
\[
f_L'(x)  := \begin{cases}
\wt B_L(x) - \f{x}{L} \wt B_L(L), &0 \le x \le L \\
\wt B_L(L+x) - \f{L+x}{L} \wt B_L(L) &-L \le x \le 0 \\
\theta x & x \notin [-L,L],
\end{cases}
\] 
and note that, when restricted to processes on the interval $[-L,L]$, we have  $f_L' \deq \wt f_L$.
Under this coupling, 
\be \label{eq:fLB}
f_L'(x) = B(x) - \f{x}{L}\Biggl(B\Bigl(\f{L}{2}\Bigr) - B\Bigl(-\f{L}{2}\Bigr)\Biggr) \qquad \textrm{for } -\f{L}{2} \le x \le \f{L}{2},
\ee
and since $\f{1}{L}\Bigl(B\bigl(\f{L}{2}\bigr) - B\bigl(-\f{L}{2}\bigr)\Bigr) \to 0$ a.s., \eqref{eq:fLB} converges uniformly on compact sets to $B(x)$. 
\end{proof}

\begin{lemma} \label{lem:BB_tail}
Fix any $\theta_1 \neq \theta_2 \in \R$, and $\beta > 0$. Let $(L_n)_{n \ge 0}$ be a sequence converging to $+\infty$, and let  $(\Omega,\Ff,\Pp)$ be a probability space on which the processes $f_{1}^{(n)},f_{2}^{(n)}\in \Cpin[0,1]$ for $n\in \N$, and $f_1,f_2\in C(\R)$ are defined. Extend $f_1^{(n)}$ and $f_2^{(n)}$ to functions $\R \to \R$ by the condition $f_m^{(n)}(x+1) - f_m^{(n)}(x) = f_m^{(n)}(1) - f_m^{(n)}(0)$ for all $x \in \R$. Assume that the following properties of these processes hold:
\begin{enumerate} [label=\textup{(\roman*)}]
    \item For each $n\in \N$, $f_{1}^{(n)}$ and $f_{2}^{(n)}$ are independent, with $f_{m}^{(n)} \deq \B_{\beta \sqrt L_n,\theta_m}$ for $m = 1,2$.
    \item For each $\ve > 0$ and each compact set $K \subseteq \R$, 
    \be \label{close_assumptionApp}
    \limsup_{n \to \infty} \Pp\Biggl(\sup_{\substack{x \in K \\ m \in \{1,2\}}} \Bigl|f_{m}^{(n)}\Bigl(\f{x}{L_n}\Bigr) - f_m(x)   \Bigr| > \ve \Biggr) = 0.
    \ee
    \end{enumerate}
    Define $\wt f_m^{(n)}(x) = f_m^{(n)}\Bigl(\f{x}{L_n}\Bigr)$ for $m \in \{1,2\}$. Then if $\theta_1 > \theta_2$ or $\theta_2 > \theta_1$ we have convergence in probability of
    \be \label{802}
    \int_{0}^{L_n} e^{\wt f_2^{(n)}(y) - \wt f_1^{(n)}(y)}\,dy \overset{p}{\longrightarrow} \int_{0}^\infty e^{f_2(y) - f_1(y)}\,dy,\quad \textrm{or}\quad 
    \int_{-L_n}^0 e^{\wt f_2^{(n)}(y) - \wt f_1^{(n)}(y)}\,dy \overset{p}{\longrightarrow} \int_{-\infty}^0 e^{f_2(y) - f_1(y)}\,dy. 
    \ee  
    \end{lemma}
\begin{proof} 
We prove the $\theta_1 > \theta_2$  case of \eqref{802} as the $\theta_1 < \theta_2$ case follows by a symmetry argument. By independence of $f_1^{(n)}$ and $f_2^{(n)}$ and \eqref{close_assumptionApp}, $f_1$ and $f_2$ must be independent as well. By Lemma \ref{lem:BB_to_BM}, $f_m \deq B_{\beta,\theta_m}$ (two-sided Brownian motion with variance $\beta^2$ and drift $\theta_m$) for $m \in \{1,2\}$.  

Observe by scaling that the process $x\mapsto (f_1^{(n)}(x),f_2^{(n)}(x)$ for  $x \in [0,L_n]$ has the same law as 
\[
x\mapsto \bigl(\beta \B_1^{(n)}(x) + \theta_1 x,\beta \B_2^{(n)}(x) + \theta_2 x\bigr),
\]
where $\B_1^{(n)},\B_2^{(n)}$ are standard Brownian bridges on $[0,L_n]$ with $\B_1^{(n)}(L_n) = \B_2^{(n)}(L_n) = 0$. Furthermore,
\[
x\mapsto \bigl(\B_1^{(n)}(x),\B_2^{(n)}(x)\bigr) \deq x\mapsto \Bigl(B_1(x) - \f{x}{L_n}B_1(L_n),B_2(x) - \f{x}{L_n}B_2(L_n)\Bigr),
\]
where $B_1,B_2$ are independent standard Brownian motions. Hence, for any $T \le L_n$,
\begin{align*}
&\quad \, \Pp\Biggl(\int_{T}^{L_n} e^{\wt f_2^{(n)}(y) - \wt f_1^{(n)}(y)}\,dy > \ve    \Biggr) \\
&= \Pp\Biggl(\int_T^{L_n} e^{\beta \bigl(B_2(y) -  B_1(y) - \f{y}{L_n}(B_2(L_n) - B_1(L_n))\bigr) + \theta_2 y - \theta_1 y} > \ve    \Biggr) \\
&\le \Pp\Bigl(\f{B_2(L_n) - B_1(L_n)}{L_n} < \f{\theta_2 - \theta_1}{2\beta } \Bigr) + \Pp\Biggl(\int_T^{L_n} e^{\beta \bigl(B_2(y) -  B_1(y)\bigr) + \f{\theta_2 - \theta_1}{2} y} > \ve    \Biggr),
\end{align*}
and from standard facts about Brownian motion, since $\theta_1 > \theta_2$, this implies that for any $\ve > 0$
\be \label{806}
\limsup_{T \to \infty}\limsup_{n \to \infty} \Pp\Biggl(\int_{T}^{L_n} e^{\wt f_2^{(n)}(y) - \wt f_1^{(n)}(y)}\,dy > \ve    \Biggr) = 0.
\ee
Next, for any $\ve,T,D > 0$ with $T \le L_n$,
\begin{align}
&\quad \, \Pp\Biggl(\Bigl|\int_{0}^{L_n} e^{\wt f_2^{(n)}(y) - \wt f_1^{(n)}(y)}\,dy -\int_{0}^\infty e^{f_2(y) - f_1(y)}\,dy\Bigr| > \ve\Biggr) \nonumber  \\
&\le \Pp\Biggl(\sup_{0 \le y < \infty} e^{f_2(y) - f_1(y)} > D\Biggr) \label{803}\\
&+ \Pp\Biggl(\int_{0}^{T}\Bigl| e^{\wt f_2^{(n)}(y) - f_n(y) + f_1(y) - \wt f_1^{(n)}(y)} - 1 \Bigr|\,dy  > \f{\ve}{3D}\Biggr) \label{804} \\
&\,+ \Pp\Biggl(\int_{T}^{L_n} e^{\wt f_2^{(n)}(y) - \wt f_1^{(n)}(y)}\,dy > \f{\ve}{3}    \Biggr) + \Pp\Biggl(\int_{T}^\infty e^{f_2(y) - f_1(y)}\,dy > \f{\ve}{3}\Biggr). \label{805}
\end{align}
Since $f_1$ and $f_2$ are Brownian motions with the same variance and drifts $\theta_1 > \theta_2$, for any $\ve ' > 0$, the term in \eqref{803} is bounded by $\ve'$ for all sufficiently large $D$. Similarly, we may choose $T > 0$ sufficiently large so that the second term in \eqref{805} is smaller than $\ve'$. By the assumption \eqref{close_assumptionApp}, the term in \eqref{804} goes to $0$ as $n \to \infty$ for any choice of $D,T,\ve$. Hence, taking a $\limsup$ as $n \to \infty$, then limits as $D,T \to \infty$ and using \eqref{806}, along with the fact that $f_1$ and $f_2$ are Brownian motions with $\theta_1 > \theta_2$, we obtain
\[
\lim_{n \to \infty} \Pp\Biggl(\Bigl|\int_{0}^{L_n} e^{\wt f_2^{(n)}(y) - \wt f_1^{(n)}(y)}\,dy -\int_{0}^\infty e^{f_2(y) - f_1(y)}\,dy\Bigr| > \ve\Biggr) = 0. \qedhere
\]
\end{proof}

\section{Comparison with the structure of jointly invariant measures in other works} \label{sec:comparison} 
In this section, we compare $D^{N,2}$ defined in \eqref{eq:D_intro} to similar maps appearing in two related works. Section \ref{sec:inv_gamma} deals with the inverse-gamma polymer on the full line, while Section \ref{sec:TASEP} discusses multi-type TASEP.

\subsection{The inverse-gamma polymer on a line} \label{sec:inv_gamma}
The work \cite{Bates-Fan-Seppalainen} studies the jointly invariant measures for the inverse-gamma polymer on the full line. In this section, we summarize that result and compare the the periodic setting. In particular, we describe our initial proposal for the structure of the  jointly invariant measures of the periodic system of SDEs \eqref{eq:joint_OCY}. This proposal, in hindsight, was correct, but came in a more complicated form that was more difficult to prove than the route we ultimately took. See  Proposition \ref{prop:inv_gam_D_map} and Remark \ref{rmk:ig_comp} below. The periodic Pitman transform in Section \ref{sec:alg} is necessary for both routes.

To define the model, for fixed $\gamma > 0$, let $(W_{x})_{x \in \Z^2}$ be i.i.d. log-inverse-gamma random variables with shape $\gamma$ and scale $1$ (Definition \ref{def:p_meas}).  For $i,j,m,n \in \Z^2$ with $i \le m$ and $j \le n$, define the partition function
\[
Z(m,n \viiva i,j) := \sum_\pi \prod_{x \in \pi \setminus (i,j)} e^{W_{x}},
\]
where the sum is over all up and right oriented lattice paths 
\[
\pi = \bigl((x_i)_{i = r}^\ell: x_r =  (i,j), x_\ell = (m,n),\;\; x_i -x_{i-1} \in \{(0,1),(1,0)\}, r+1 \le i \le \ell \bigr)
\]
connecting $(i,j)$ and $(m,n)$. This the point-to-point partition function for the \textit{inverse-gamma polymer} \cite{Seppalainen-2012}. See the right side of Figure \ref{fig:disc_per_environment} for an illustration of such up and right lattice paths.

We define a discrete analogue of the full-line KPZ equation for initial condition  $f:\Z \to \R$ by setting
\[
h(k,\ell \viiva f) := \log\bigg(\sum_{j = -\infty}^\ell e^{f(j)} Z(k,\ell \viiva 1,j)\bigg).
\]
We say that a coupled (random) initial condition $(f_1,\ldots,f_k)$ satisfying $f_m(0) = 0$ for $1 \le m \le k$ is jointly invariant for the model if, for all $m > 1$, 
\be \label{eq:ig_invar}
\Bigl(\big(h(m,n \viiva f_1) - h(m,0 \viiva f_1)\big)_{n \in \Z},\ldots,\big(h(m,n \viiva f_k) - h(m,0 \viiva f_k)\big)_{n \in \Z }\Bigr) \deq \bigl(f_1(n),\ldots,f_k(n)\bigr). 
\ee

To describe the jointly invariant conditions, we first introduce some notation. 
Similarly as in \cite[Section 6.1]{Bates-Fan-Seppalainen}, consider the space $\mathcal I$ of sequences $\mbf X = (\mbf X_i)_{i \in \Z} \in \R^{\Z}$ such that the following limit exists:
\[
\mathfrak c(\mbf X) := \lim_{j \to \infty} \f{1}{j} \sum_{i = -j + 1}^0 X_{i}
\]
For $k \ge 1$, define the set
\[
\mathcal I_k^{\uparrow} = \{(\mbf X_1,\mbf X_2) \in \mathcal I \times \mathcal  I: \mathfrak c(\mbf X_1) < \mathfrak c(\mbf X_2) < \cdots < \mathfrak c(\mbf X_k)\}.
\]
Define a map $\wt D^2:\mathcal I_2^{\uparrow} \to \mathcal I$ as follows: for $(\mbf X_1,\mbf X_2) \in \mathcal I_2^{\uparrow}$, let $\wt D^2(\mbf X_1,\mbf X_2) = \bigl(\wt D_i^2(\mbf X_1,\mbf X_2)\bigr)_{i \in \Z}$, where
\be \label{eq:BFSDJ}
\wt D_i^2(\mbf X_1,\mbf X_2) = X_{2,i} + \log \f{J_i(\mbf X_1,\mbf X_2)}{J_{i-1}(\mbf X_1,\mbf X_2)},\qquad \text{where}\qquad J_i(\mbf X_1,\mbf X_2) = \sum_{j = 0}^\infty e^{X_{1,i - j}} \prod_{\ell = 0}^{j- 1}e^{X_{1,i-\ell} -X_{2,i-\ell}}.
\ee
Here, we interpret the empty product as $1$. We caution, that, while we have borrowed notation from \cite{Bates-Fan-Seppalainen}, we have defined things slightly differently. In that setting, the they take $W = e^{\mbf X_1}$ and $I = e^{\mbf X_2}$ as their inputs, and define the output of the map $D$ as the exponential of what we have defined. We have chosen to write it this way to match the conventions in the present paper. 

Analogously as in \eqref{eq:D_iter_intro} and \eqref{DNk_intro}, we iterate this map to obtain maps $\wt D^{k}:\mathcal I_k^{\uparrow} \to \mathcal I$ by 
\[
\wt D^1(\mbf X_1) = \mbf X_1,\qquad\text{and for }k > 1,\qquad \wt D^k(\mbf X_1,\ldots,\mbf X_k) = \wt D^2\bigl(\mbf X_1,\wt D^{k-1}(\mbf X_2,\ldots,\mbf X_k)\bigr),
\]
and a map $\wt \D^k:\mathcal I_k^{\uparrow} \to (\mathcal I)^k$ by 
\[
\wt \D^{k}(\mbf X_1,\ldots,\mbf X_k) = \bigl(\mbf X_1,\wt D^{2}(\mbf X_1,\mbf X_2),\ldots, \wt D^{k}(\mbf X_1,\ldots,\mbf X_k)\bigr).
\]
In this language, we cite the following result.
\begin{theorem} \cite[Theorem 4.1]{Bates-Fan-Seppalainen}.
Let $\alpha_k < \cdots < \alpha_1 < \alpha$, and let $(\mbf X_1,\ldots,\mbf X_k) \in (\mathcal I)^k$ be independent and such that, for $1 \le r \le k$, $\mbf X_r = (X_{r,i})_{i \in \Z}$ is an i.i.d. sequence of log-inverse gamma random variables with rate $\alpha_r$ and scale $1$. Let $(\mbf U_1,\ldots,\mbf U_k) = \wt \D^{N,k}(\mbf X_1,\ldots,\mbf X_k)$, and for $1 \le r \le k$, let $f_r:\Z \to \R$ satisfy $f_r(0)$ and $f_r(i) - f_{r}(i-1) = U_{r,i}$ for $i \in \Z$. Then, $(f_1,\ldots,f_r)$ is jointly invariant in the sense of \eqref{eq:ig_invar}. 
\end{theorem}
In Proposition \ref{prop:inv_gam_D_map}, we relate the map $\wt D^2$ to the map $ D^{N,2}:\R^{\Z_N} \times \R^{\Z_N} \to \R^{\Z_N}$ that we recall from \eqref{eq:D_intro}:
\[
D^{N,2}_i(\mbf X_1,\mbf X_2) = X_{2,i} + \log\Biggl(\f{\sum_{j \in \Z_N} e^{Y_{(i,j]}}     }{\sum_{j \in \Z_N} e^{Y_{(i-1,j]}  }}\Biggr),\qquad\text{where}\qquad \mbf Y = \mbf X_2 - \mbf X_1.
\]
To make this comparison, suppose we have two sequences $\mbf X_1,\mbf X_2 \in \R^{\Z_N}$ satisfying $\vecsum(\mbf X_1) < \vecsum(\mbf X_2)$. Extend these periodically to sequences $\mbf X_1,\mbf X_2$ by the convention that, for $i \in \Z$ and $m \in \{1,2\}$,  $X_{m,i} = X_{m,j}$ whenever $i \equiv j \mod N$. Then, for $m \in \{1,2\}$, 
\[
\mathfrak c(\mbf X_m) = \f{\vecsum(\mbf X_m)}{N},
\]
where, on the left, $\mbf X_m$ is viewed as an element of $\R^\Z$, and viewed as an element of $\R^{\Z_N}$ on the right. In particular, the extended tuple $(\mbf X_1,\mbf X_2)$ lies in $\mathcal I_2^{\uparrow}$. When this tuple has this periodic structure, we can write the output of $\wt{D}^2$ as a function of $(\mbf X_{1,i},\mbf X_{2,i})_{i \in \Z_N}$. We now demonstrate that the expression in \eqref{eq:BFSDJ} becomes much simpler in this setting, and somewhat resembles the map $D^{N,2}$ defined in \eqref{eq:D_intro}, although they are not the same, as we see from the following lemma. 
\begin{proposition} \label{prop:inv_gam_D_map}
Let $\mbf X_1,\mbf X_2 \in \R^{\Z_N}$ satisfy $\vecsum(\mbf X_1) < \vecsum(\mbf X_2)$. Extend these to a pair $(\mbf X_1,\mbf X_2) \in \mathcal I_2^{\uparrow}$ as described above. Then, for $i \in \Z_N$,
\be \label{Di_per}
\wt D_i^2(\mbf X_1,\mbf X_2) = D_{i+1}^{N,2}(\mbf X_1',\mbf X_2),
\ee
where $\mbf X_1' = (X_{1,\ell}')_{\ell \in \Z_N}$ is defined by $X_{1,\ell}' = X_{1,\ell - 1}$.
\end{proposition}
\begin{remark} \label{rmk:ig_comp}
If $\mbf X_1 \sim \nu_\beta^{N,(\theta)}$, then $\mbf X_1' \deq \mbf X_1$, so combined with the shift invariance in Corollary \ref{eq:U_shift}, $(\mbf X_1,\wt D^2(\mbf X_1,\mbf X_2)) \deq \bigl(\mbf X_1,D^{N,2}(\mbf X_1,\mbf X_2)\bigr) \sim \mu_\beta^{N,(\theta_1,\theta_2)}$ when $(\mbf X_1,\mbf X_2) \sim \nu_\beta^{N,(\theta_1,\theta_2)}$. In other words, the map $\wt D^{2}$, restricted to $N$-periodic inputs can be used in place of $D^{N,2}$ in describing the distribution $\mu_\beta^{N,(\theta_1,\ldots,\theta_k)}$. This was our preliminary guess as to the structure of the jointly invariant measures for the system of SDEs \eqref{eq:joint_OCY}. However, our ultimately use of the map $D^{N,2}$ is much more convenient when working with the system of SDEs in Theorem \ref{thm:OCY_joint}). In Section \ref{sec.simplecase}, we outlined a principle-based approach for finding an appropriate mapping in the $N = 2$ case, which led us to our form of $D^{N,2}$ and the jointly invariant measures. We note that the expression in \eqref{Di_per} is not $X_{2,i}$ plus a function of the differences $(X_{2,\ell} - X_{1,\ell})_{\ell \in \Z}$. So while a fortiori we see that $\wt D^2$ could have been used in place of $D^{N,2}$, the key simplifications described in \eqref{sec.simplecase} that gave the same diffusion terms between the original system of SDEs and the modified system \eqref{eq:X_sys}, as well as the vanishing of the Hessian term in It\^o's formula would not occur. The connection made in \eqref{Di_per} came as an observation after constructing the map $D^{N,2}$ from scratch. It should be noted that even with the relation \ref{Di_per}, the proof of joint invariance in the periodic inverse-gamma setting requires developing almost all of the periodic Pitman transform machinery from Section \ref{sec:alg}.
\end{remark}
\begin{proof}[Proof of Proposition \ref{prop:inv_gam_D_map}]
Changing the order of summation in \eqref{eq:BFSDJ} and then rearranging terms, we obtain
\begin{align*}
J_i(\mbf X_1,\mbf X_2) &= \sum_{j = -\infty}^i e^{X_{1,j}}\prod_{\ell = j +1 }^i e^{X_{1,\ell} - X_{2,\ell}}  =\sum_{j = -\infty}^i e^{X_{2,i+1 }} \prod_{\ell = j+1}^{i+1} e^{X_{1,\ell}' - X_{2,\ell}},
\end{align*}
where $X_{1,\ell}' = X_{1,\ell - 1}$.
Hence, from \eqref{eq:BFSDJ}, 
\[
\wt D_i^2(\mbf X_1,\mbf X_2) = X_{2,i} + \log \f{J_i(\mbf X_1,\mbf X_2)}{J_{i-1}(\mbf X_1,\mbf X_2)} = X_{2,i+1} + \log \Biggl(\f{\sum_{j = -\infty}^i \prod_{\ell = j+1}^{i+1} e^{X_{1,\ell}' - X_{2,\ell}}}{\sum_{j = -\infty}^{i-1} \prod_{\ell = j+1}^{i} e^{X_{1,\ell}' - X_{2,\ell}}}\Biggr).
\]
The proof now follows from the computation in Lemma \ref{lem:inv_gamma}. 
\end{proof}

\subsection{Multi-species TASEP on a ring} \label{sec:TASEP}
We also compare to the work in the setting of the totally asymmetric simple exclusion process (TASEP) on a ring from \cite{Ferrari-Martin-2007}. There, we also see a similar map that describes the jointly invariant initial condition. We first define the model as follows. Consider a ring with $N$ sites, indexed by $\Z_N$, and $k+1$ types of particles, denoted $\{1,\ldots,k,\infty\}$. The convention is that particles with smaller types have priority, and a particle of class $\infty$ denotes a hole. Multi-type TASEP is a continuous-time Markov process with state space $\{1,\ldots,k,\infty\}^{\Z_N}$, whose dynamics are described as follows. To each site of $\Z_N$ is associated an independent rate-one Poisson clocks. If the clock at site $i$ rings, then the particle at site $i$ attempts to switch positions with the particle at site $i -1$. An exchange takes place if and only if the class of the particle at site $i$ is smaller than the class of the particle at site $i-1$. 

It is clear that the number of each class of particle is preserved by the dynamics. This is the analogue of the preservation of slopes in the periodic KPZ equation. For a fixed number of particles of each type, there is a unique invariant measure. The invariant measures for multi-type TASEP in the case $k = 2$ were first discovered in  \cite{Derrida-Janowsky-Lebowitz-1993} using the matrix product ansatz. An alternate description using combinatorial methods was later given in \cite{ANGEL-2006} and shortly thereafter after, \cite{Ferrari-Martin-2007} gave yet another description of these invariant measures and extended to general $k$.  We discuss here the description from \cite{Ferrari-Martin-2007}.

As in \cite[Page 813]{Ferrari-Martin-2007}, for inputs $\mbf X_1,\mbf X_2 \in \{0,1\}^{\Z_N}$ that encode a process of services and arrivals, respectively,  to a queue, the process of departures $\wt D^{N,2}(\mbf X_1,\mbf X_2) = \bigl(\wt D_i^{N,2}(\mbf X_1,\mbf X_2)\bigr)_{i \in \Z_N} \!\!\!\in \{0,1\}^{\Z_N}$ is 
\be \label{DmapTASEP}
\begin{aligned}
\wt D_i^{N,2}(\mbf X_1,\mbf X_2) &:= X_{2,i} + \sup_{j \in \Z_N} \bigl[Y_{[i,j]}\bigr]_+ - \sup_{j \in \Z_N} \bigl[Y_{[i-1,j]}\bigr]_+ , \qquad\text{where}\qquad \mbf Y = \mbf X_2 - \mbf X_1,
\end{aligned}
\ee
and $[x]_+=\max(x,0)$
Just as in \eqref{eq:D_iter_intro}, extend $\wt D^{N,2}$ to maps $\wt D^{N,k}:(\{0,1\}^{\Z_N})^k \to \{0,1\}^{\Z_N}$ inductively by 
\[
\wt D^{N,1}(\mbf X_1) = \mbf X_1, \textrm{ and for }m > 1,\quad 
\wt D^{N,m}(\mbf X_1,\ldots,\mbf X_k) = \wt D^{N,2}\bigl(\mbf X_1, \wt D^{N,m-1}(\mbf X_2,\ldots,\mbf X_k)\bigr),
\]
and define $\wt \D^{N,k}:(\{0,1\}^{\Z_N})^k \to (\{0,1\}^{\Z_N})^k$ by 
\[
\wt \D^{N,k}(\mbf X_1,\ldots,\mbf X_k) = \bigl(\wt D^{N,1}(\mbf X_1),\wt D^{N,2}(\mbf X_1,\mbf X_2),\ldots,\wt D^{N,k}(\mbf X_1,\ldots,\mbf X_k)\bigr)
\]
The output of $\wt \D^{N,k}$ can be understood via departures of a multi-line queue. We point the reader to \cite[Section 2.2]{Ferrari-Martin-2007} for the queuing interpretation. We shall encode the configuration of $k$-type TASEP with a vector $(\mbf U_1,\ldots,\mbf U_k) \in (\{0,1\})^k$ that satisfies $\mbf U_m < \mbf U_{m+1}$ for $1 \le m \le k-1$. Here, $\mbf U_1$ denotes the set of first class particles (a first class particle at site $i$ if $ U_{1,i} = 1$ and no first class particle at site $i$ otherwise). Then, for $m \ge 1$, $\mbf U_m$ denotes the configuration of all particles of classes $1,\ldots,m$. That is, $U_{m,i} = 1$ if there is a particle of class $\le m$ at site $i$, and $U_{m,i} = 0$ otherwise. We now cite the following. 
\begin{theorem} \cite[Theorem 2.2]{Ferrari-Martin-2007}
Let $0 \le n_k \le \cdots \le n_1 \le N$, and let $(\mbf X_1,\ldots,\mbf X_k) \in (\{0,1\}^{\Z_N})^k$ be independent and so that each $\mbf X_m$ is uniformly distributed on all $\binom{N}{n_m}$ configurations of  $\{0,1\}^{\Z_N}$ with $n_m$ $1$s and $N - n_m$ $0$s. Set $(\mbf U_1,\ldots,\mbf U_k) = \wt \D^{N,k}(\mbf X_1,\ldots,\mbf X_k)$. Then, the distribution of $(\mbf U_1,\ldots,\mbf U_k)$ is the unique stationary measure for multi-type TASEP with $n_1$ first class particles, $n_2 - n_1$ second class particles, $n_3 - n_2$ third class particles, and so on.  
\end{theorem}
\begin{remark}
The statement in \cite[Theorem 2.2]{Ferrari-Martin-2007} is of a slightly different form then presented here. There, maps are constructed that give output in $\{1,\ldots,k,\infty\}$ instead of $\{0,1\}$. The equivalence of the formulation we present may be seen, for example, from \cite[Section 4.2]{Busa-Sepp-Sore-22b} (which deals with TASEP on the full line, but has a similar structure).
\end{remark}

 The definition of $\wt D^{N,2}$ in \eqref{DmapTASEP} is nearly a zero-temperature analogue of the map $D^{N,2}$ in the present paper. Note the slight difference of $Y_{[i,j]}$ in \eqref{DmapTASEP} compared to $Y_{(i,j]}$ in \eqref{eq:D_intro}. Note also that the inductive structure to go beyond $k = 2$ is the same in both settings. By convergence of conditioned Bernoulli random walks to Brownian bridge, one could prove convergence to the periodic stationary horizon defined in Proposition \ref{prop:betalim}\ref{itm:betainflim}. The analogous convergence to the full-line stationary horizon for TASEP on the full line was shown in \cite{Busa-Sepp-Sore-22b}. The argument in this periodic setting would be much simpler for two reasons. First, the supremum in \eqref{eq:grtilde} is over a compact set, while the analogous object on the full line involves a supremum over an infinite interval, so \cite{Busa-Sepp-Sore-22b} had to control the tails of the supremum uniformly in the scaling parameter $N$. Second, the bulk of the technicality in \cite{Busa-Sepp-Sore-22b} lies in showing tightness of the process as a function of the slope parameter $\theta$. Tightness in this setting follows readily from compactness, as in the proof of Proposition \ref{prop:betalim}.   

 For the  general asymmetric simple exclusion process (ASEP) on a ring particles attempt jumps in the counterclockwise direction with rate one and in the clockwise direction with rate $q$ for some parameter $q \in [0,1)$. The multi-type invariant measures for ASEP on a ring were first discovered in \cite{Prolhac-Evans-Mallick-2009} using the method known as the matrix product ansatz and \cite{Martin-2020} later gave this stationary measure a queuing interpretation. Similarly, as for TASEP, this can be defined in terms of maps of independent $\{0,1\}^{\Z_N}$-valued inputs. However, unlike that case, the map is no longer deterministic; it is a random map whose law is determined by the parameter $q$. This makes taking asymptotics for the ASEP multi-type stationary measures in this setting much more difficult. 
\bibliographystyle{alpha}
\bibliography{references_file}

\begin{thebibliography}{GRASS25}

\bibitem[ACH24]{aggarwal2024scalinglimitcoloredasep}
Amol Aggarwal, Ivan Corwin, and Milind Hegde.
\newblock {Scaling limit of the colored ASEP and stochastic six-vertex models}.
\newblock {\em Preprint:arXiv:2403.01341}, 2024.

\bibitem[AD95]{Aldous-Diaconis-1995}
D.~Aldous and P.~Diaconis.
\newblock Hammersley's interacting particle process and longest increasing subsequences.
\newblock {\em Probab. Theory Related Fields}, 103(2):199--213, 1995.

\bibitem[Agg18]{10.1215/00127094-2017-0029}
Amol Aggarwal.
\newblock Current fluctuations of the stationary {ASEP} and six-vertex model.
\newblock {\em Duke Math. J.}, 167(2):269--384, 2018.

\bibitem[AKQ14a]{Alberts-Khanin-Quastel-2014a}
Tom Alberts, Konstantin Khanin, and Jeremy Quastel.
\newblock The continuum directed random polymer.
\newblock {\em J. Stat. Phys.}, 154(1-2):305--326, 2014.

\bibitem[AKQ14b]{Alberts-Khanin-Quastel-2014b}
Tom Alberts, Konstantin Khanin, and Jeremy Quastel.
\newblock The intermediate disorder regime for directed polymers in dimension {$1+1$}.
\newblock {\em Ann. Probab.}, 42(3):1212--1256, 2014.

\bibitem[AMM23a]{Ayyer-Mandelshtam-Martin-2023b}
Arvind Ayyer, Olya Mandelshtam, and James Martin.
\newblock The multispecies zero range process and modified {M}acdonald polynomials.
\newblock {\em S\'em. Lothar. Combin.}, 89B:Art. 70, 12, 2023.

\bibitem[AMM23b]{Ayyer-Mandelshtam-Martin-2023a}
Arvind Ayyer, Olya Mandelshtam, and James~B. Martin.
\newblock Modified {M}acdonald polynomials and the multispecies zero-range process: {I}.
\newblock {\em Algebr. Comb.}, 6(1):243--284, 2023.

\bibitem[Ang06]{ANGEL-2006}
Omer Angel.
\newblock The stationary measure of a 2-type totally asymmetric exclusion process.
\newblock {\em Journal of Combinatorial Theory, Series A}, 113(4):625--635, 2006.

\bibitem[ANP23]{Aggarwal-Nicoletti-Petrov-2023}
Amol {Aggarwal}, Matthew {Nicoletti}, and Leonid {Petrov}.
\newblock {Colored Interacting Particle Systems on the Ring: Stationary Measures from Yang-Baxter Equation}.
\newblock {\em Preprint:arXiv:2309.11865}, 2023.
\newblock To appear in Compositio Math.

\bibitem[BBO05]{Biane-Bougerol-OConnell-2005}
Philippe Biane, Philippe Bougerol, and Neil O'Connell.
\newblock Littelmann paths and {B}rownian paths.
\newblock {\em Duke Math. J.}, 130(1):127--167, 2005.

\bibitem[BC95]{Bertini1995}
Lorenzo Bertini and Nicoletta Cancrini.
\newblock The stochastic heat equation: {F}eynman-{K}ac formula and intermittence.
\newblock {\em J. Statist. Phys.}, 78(5-6):1377--1401, 1995.

\bibitem[BC14]{Borodin-Corwin-2014}
Alexei Borodin and Ivan Corwin.
\newblock Macdonald processes.
\newblock {\em Probab. Theory Related Fields}, 158(1-2):225--400, 2014.

\bibitem[BC17]{Barraquand2017}
Guillaume Barraquand and Ivan Corwin.
\newblock Random-walk in beta-distributed random environment.
\newblock {\em Probab. Theory Related Fields}, 167(3-4):1057--1116, 2017.

\bibitem[BC22]{Bukh-Cox-2022}
Boris Bukh and Christopher Cox.
\newblock Periodic words, common subsequences and frogs.
\newblock {\em Ann. Appl. Probab.}, 32(2):1295--1332, 2022.

\bibitem[BC23]{10.1214/23-AOP1634}
Guillaume Barraquand and Ivan Corwin.
\newblock {Stationary measures for the log-gamma polymer and KPZ equation in half-space}.
\newblock {\em The Annals of Probability}, 51(5):1830 -- 1869, 2023.

\bibitem[BCK14]{Bakhtin-Cator-Konstantin-2014}
Yuri Bakhtin, Eric Cator, and Konstantin Khanin.
\newblock Space-time stationary solutions for the {B}urgers equation.
\newblock {\em J. Amer. Math. Soc.}, 27(1):193--238, 2014.

\bibitem[BD22]{Bougerol22}
Philippe Bougerol and Manon Defosseux.
\newblock Pitman transformations and brownian motion in the interval seen as an affine alcove.
\newblock {\em Scientific Annals of the ENS}, 55:429--472, 2022.

\bibitem[BFS25]{Bates-Fan-Seppalainen}
Erik Bates, Wai-Tong Fan, and Timo Sepp\"{a}l\"{a}inen.
\newblock Intertwining the {B}usemann process of the directed polymer model.
\newblock {\em Electron. J. Probab.}, 30:Paper No. 50, 80, 2025.

\bibitem[BG97]{BertiniGiacomin}
Lorenzo Bertini and Giambattista Giacomin.
\newblock Stochastic {B}urgers and {KPZ} equations from particle systems.
\newblock {\em Comm. Math. Phys.}, 183(3):571--607, 1997.

\bibitem[Bia09]{Biane09}
Philippe Biane.
\newblock From pitman's theorem to crystals.
\newblock {\em Adv. Std. Pure Math.}, 55:1--13, 2009.

\bibitem[Bil99]{billing}
Patrick Billingsley.
\newblock {\em Convergence of probability measures}.
\newblock Wiley Series in Probability and Statistics: Probability and Statistics. John Wiley \& Sons, Inc., New York, second edition, 1999.
\newblock A Wiley-Interscience Publication.

\bibitem[BL19a]{BaikLiu}
Jinho Baik and Zhipeng Liu.
\newblock Multipoint distribution of periodic {TASEP}.
\newblock {\em J. Amer. Math. Soc.}, 32(3):609--674, 2019.

\bibitem[BL19b]{bakhtin2019thermodynamic}
Yuri Bakhtin and Liying Li.
\newblock Thermodynamic limit for directed polymers and stationary solutions of the burgers equation.
\newblock {\em Communications on Pure and Applied Mathematics}, 72(3):536--619, 2019.

\bibitem[{Bru}03]{Brunet03}
\'Eric {Brunet}.
\newblock {Fluctuations of the winding number of a directed polymer in a random medium}.
\newblock {\em Physical Review E}, 68(4):041101, 2003.

\bibitem[BS22]{Busani-Seppalainen-2020}
Ofer Busani and Timo Sepp\"al\"ainen.
\newblock Non-existence of bi-infinite polymers.
\newblock {\em Electron. J. Probab.}, 27:Paper No. 14, 40, 2022.

\bibitem[BSS22]{Busa-Sepp-Sore-22b}
Ofer {Busani}, Timo {Sepp{\"a}l{\"a}inen}, and Evan {Sorensen}.
\newblock {Scaling limit of the TASEP speed process}.
\newblock {\em Preprint:arXiv:2211.04651}, 2022.
\newblock To appear in Ann. Inst. Henri Poincar\'{e} Probab. Stat.

\bibitem[BSS24a]{Busa-Sepp-Sore-23}
Ofer {Busani}, Timo {Sepp{\"a}l{\"a}inen}, and Evan {Sorensen}.
\newblock {Scaling limit of multi-type invariant measures via the directed landscape}.
\newblock {\em Int. Math. Res. Not. IMRN}, 2024.

\bibitem[BSS24b]{Busa-Sepp-Sore-22a}
Ofer Busani, Timo Sepp\"{a}l\"{a}inen, and Evan Sorensen.
\newblock The stationary horizon and semi-infinite geodesics in the directed landscape.
\newblock {\em Ann. Probab.}, 52(1):1--66, 2024.

\bibitem[Bur56]{Burke1956}
Paul~J. Burke.
\newblock The output of a queuing system.
\newblock {\em Operations Res.}, 4:699--704 (1957), 1956.

\bibitem[Bus24]{Busani-2021}
Ofer Busani.
\newblock Diffusive scaling limit of the {B}usemann process in last passage percolation.
\newblock {\em Ann. Probab.}, 52(5):1650--1712, 2024.

\bibitem[BW22]{BorodinWheeler}
Alexei Borodin and Michael Wheeler.
\newblock Colored stochastic vertex models and their spectral theory.
\newblock {\em Ast\'erisque}, 437:ix+225, 2022.

\bibitem[CG17]{Corwin2017}
Ivan Corwin and Yu~Gu.
\newblock Kardar-{P}arisi-{Z}hang equation and large deviations for random walks in weak random environments.
\newblock {\em J. Stat. Phys.}, 166(1):150--168, 2017.

\bibitem[CGST20]{Corwin2020}
Ivan Corwin, Promit Ghosal, Hao Shen, and Li-Cheng Tsai.
\newblock Stochastic {PDE} limit of the six vertex model.
\newblock {\em Comm. Math. Phys.}, 375(3):1945--2038, 2020.

\bibitem[CKNP21]{CKNP21}
Le~Chen, Davar Khoshnevisan, David Nualart, and Fei Pu.
\newblock Spatial ergodicity for {SPDE}s via {P}oincar\'e-type inequalities.
\newblock {\em Electron. J. Probab.}, 26:Paper No. 140, 37, 2021.

\bibitem[CN17]{10.1214/17-EJP32}
Ivan Corwin and Mihai Nica.
\newblock Intermediate disorder directed polymers and the multi-layer extension of the stochastic heat equation.
\newblock {\em Electron. J. Probab.}, 22:Paper No. 13, 49, 2017.

\bibitem[Cor21]{CorwinInvariance}
Ivan Corwin.
\newblock {Invariance of polymer partition functions under the geometric RSK correspondence}.
\newblock {\em Advanced Studies in Pure Mathematics}, 2021:89--136, 2021.

\bibitem[COSZ14]{10.1215/00127094-2410289}
Ivan Corwin, Neil O’Connell, Timo Sepp{\"a}l{\"a}inen, and Nikolaos Zygouras.
\newblock {Tropical combinatorics and Whittaker functions}.
\newblock {\em Duke Mathematical Journal}, 163(3):513 -- 563, 2014.

\bibitem[CP16]{Corwin2016}
Ivan Corwin and Leonid Petrov.
\newblock Stochastic higher spin vertex models on the line.
\newblock {\em Comm. Math. Phys.}, 343(2):651--700, 2016.

\bibitem[CS21]{CroydonSasada21}
David Croydon and Makiko Sasada.
\newblock Discrete integrable systems and pitman's transformation.
\newblock {\em Adv. Std. Pure Math.}, 87:381--402, 2021.

\bibitem[CSZ17]{CSZ}
Francesco Caravenna, Rongfeng Sun, and Nikos Zygouras.
\newblock Polynomial chaos and scaling limits of disordered systems.
\newblock {\em J. Eur. Math. Soc.}, 19(1):1--65, 2017.

\bibitem[CT17]{10.1214/16-AOP1101}
Ivan Corwin and Li-Cheng Tsai.
\newblock K{PZ} equation limit of higher-spin exclusion processes.
\newblock {\em Ann. Probab.}, 45(3):1771--1798, 2017.

\bibitem[DG24]{Dunlap-Gu-23}
Alexander Dunlap and Yu~Gu.
\newblock Jointly stationary solutions of periodic {B}urgers flow.
\newblock {\em J. Funct. Anal.}, 287(12):Paper No. 110656, 43, 2024.

\bibitem[DGK23]{ADYGTK22}
Alexander Dunlap, Yu~Gu, and Tomasz Komorowski.
\newblock Fluctuation exponents of the {KPZ} equation on a large torus.
\newblock {\em Comm. Pure Appl. Math.}, 76(11):3104--3149, 2023.

\bibitem[DGR21]{dunlap2021stationary}
Alexander Dunlap, Cole Graham, and Lenya Ryzhik.
\newblock Stationary solutions to the stochastic burgers equation on the line.
\newblock {\em Communications in Mathematical Physics}, 382(2):875--949, 2021.

\bibitem[DJLS93]{Derrida-Janowsky-Lebowitz-1993}
B.~Derrida, S.~A. Janowsky, J.~L. Lebowitz, and E.~R. Speer.
\newblock Exact solution of the totally asymmetric simple exclusion process: shock profiles.
\newblock {\em J. Statist. Phys.}, 73(5-6):813--842, 1993.

\bibitem[DNV22]{Dauvegne-Nica-Virag-2021}
Duncan Dauvergne, Mihai Nica, and B\'alint Vir\'ag.
\newblock R{SK} in last passage percolation: a unified approach.
\newblock {\em Probab. Surv.}, 19:65--112, 2022.

\bibitem[DS24]{Dunlap-Sorensen-2024}
Alexander {Dunlap} and Evan {Sorensen}.
\newblock {Viscous shock fluctuations in KPZ}.
\newblock {\em Preprint:arXiv:2406.06502}, 2024.

\bibitem[Dud89]{dudl}
Richard~M. Dudley.
\newblock {\em Real analysis and probability}.
\newblock The Wadsworth \& Brooks/Cole Mathematics Series. Wadsworth \& Brooks/Cole Advanced Books \& Software, Pacific Grove, CA, 1989.

\bibitem[{Dun}24]{Dunlap-2024}
Alexander {Dunlap}.
\newblock {Simultaneous global inviscid Burgers flows with periodic Poisson forcing}.
\newblock {\em Preprint:arXiv:2406.06896}, 2024.
\newblock To appear in Annales Henri Lebesgue.

\bibitem[DV24]{Dauvergne-Virag-2024}
Duncan {Dauvergne} and B{\'a}lint {Vir{\'a}g}.
\newblock {The directed landscape from Brownian motion}.
\newblock {\em Preprint:arXiv:2405.00194}, 2024.

\bibitem[DW21]{Dmitrov-Wu-2021}
Evgeni Dimitrov and Xuan Wu.
\newblock K{MT} coupling for random walk bridges.
\newblock {\em Probab. Theory Related Fields}, 179(3-4):649--732, 2021.

\bibitem[Ech82]{Echeverria-1982}
Pedro Echeverr\'ia.
\newblock A criterion for invariant measures of {M}arkov processes.
\newblock {\em Z. Wahrsch. Verw. Gebiete}, 61(1):1--16, 1982.

\bibitem[EK86]{ethi-kurt}
Stewart~N. Ethier and Thomas~G. Kurtz.
\newblock {\em Markov processes: Characterization and convergence}.
\newblock Wiley Series in Probability and Mathematical Statistics. John Wiley \& Sons Inc., New York, 1986.

\bibitem[EKMS97]{EKMS-1997}
Weinan E, Konstantin Khanin, Alexandre Mazel, and Yakov Sinai.
\newblock Probability distribution functions for the random forced burgers equation.
\newblock {\em Phys. Rev. Lett.}, 78:1904--1907, 1997.

\bibitem[EKMS00]{EKMS-2000}
Weinan E, K.~Khanin, A.~Mazel, and Ya. Sinai.
\newblock Invariant measures for {B}urgers equation with stochastic forcing.
\newblock {\em Ann. of Math. (2)}, 151(3):877--960, 2000.

\bibitem[FM06]{Ferrari-Martin-2005}
Pablo~A. Ferrari and J.~B. Martin.
\newblock Multi-class processes, dual points and {$M/M/1$} queues.
\newblock {\em Markov Process. Related Fields}, 12(2):175--201, 2006.

\bibitem[FM07]{Ferrari-Martin-2007}
Pablo~A. Ferrari and James~B. Martin.
\newblock Stationary distributions of multi-type totally asymmetric exclusion processes.
\newblock {\em Ann. Probab.}, 35(3):807--832, 2007.

\bibitem[FM09]{Ferrari-Martin-2009}
Pablo~A. Ferrari and James~B. Martin.
\newblock Multiclass {H}ammersley-{A}ldous-{D}iaconis process and multiclass-customer queues.
\newblock {\em Ann. Inst. Henri Poincar\'{e} Probab. Stat.}, 45(1):250--265, 2009.

\bibitem[FNS77]{FNS77}
Dieter Forster, David~R. Nelson, and Michael~J. Stephen.
\newblock Large-distance and long-time properties of a randomly stirred fluid.
\newblock {\em Phys. Rev. A}, 16:732--749, 1977.

\bibitem[FQ15]{Funaki-Quastel-2015}
Tadahisa Funaki and Jeremy Quastel.
\newblock K{PZ} equation, its renormalization and invariant measures.
\newblock {\em Stoch. Partial Differ. Equ. Anal. Comput.}, 3(2):159--220, 2015.

\bibitem[Fri64]{friedman2008partial}
Avner Friedman.
\newblock {\em Partial differential equations of parabolic type}.
\newblock Prentice-Hall, Inc., Englewood Cliffs, NJ, 1964.

\bibitem[FS20]{Fan-Seppalainen-20}
Wai-Tong~(Louis) Fan and Timo Sepp{\"a}l{\"a}inen.
\newblock {Joint distribution of Busemann functions for the exactly solvable corner growth model}.
\newblock {\em Probability and Mathematical Physics}, 1(1):55--100, 2020.

\bibitem[GIKP05]{GIKP-2005}
Diogo Gomes, Renato Iturriaga, Konstantin Khanin, and Pablo Padilla.
\newblock Viscosity limit of stationary distributions for the random forced {B}urgers equation.
\newblock {\em Mosc. Math. J.}, 5(3):613--631, 743, 2005.

\bibitem[GJ14]{Goncalves2014}
Patr\'icia Gon{\c c}alves and Milton Jara.
\newblock Nonlinear fluctuations of weakly asymmetric interacting particle systems.
\newblock {\em Arch. Ration. Mech. Anal.}, 212(2):597--644, 2014.

\bibitem[GJRA25]{Groathouse-Janjigian-Rassoul-21}
Sean Groathouse, Christopher Janjigian, and Firas Rassoul-Agha.
\newblock Non-existence of non-trivial bi-infinite geodesics in geometric last passage percolation.
\newblock {\em J. Stat. Phys.}, 192(6):Paper No. 81, 47, 2025.

\bibitem[GK22]{GK211}
Yu~Gu and Tomasz Komorowski.
\newblock High temperature behaviors of the directed polymer on a cylinder.
\newblock {\em J. Stat. Phys.}, 186(3):Paper No. 48, 15, 2022.

\bibitem[GK23]{YGTK22}
Yu~Gu and Tomasz Komorowski.
\newblock Fluctuations of the winding number of a directed polymer on a cylinder.
\newblock {\em SIAM J. Math. Anal.}, 55(4):3262--3286, 2023.

\bibitem[GK24a]{GK23}
Yu~{Gu} and Tomasz {Komorowski}.
\newblock {Effective diffusivities in periodic KPZ}.
\newblock {\em Probab. Theory Related Fields}, 2024.

\bibitem[GK24b]{GK21}
Yu~Gu and Tomasz Komorowski.
\newblock {KPZ on torus: Gaussian fluctuations}.
\newblock {\em Annales de l'Institut Henri Poincaré, Probabilités et Statistiques}, 60(3):1570 -- 1618, 2024.

\bibitem[GP17]{Gubinelli2017}
Massimiliano Gubinelli and Nicolas Perkowski.
\newblock K{PZ} reloaded.
\newblock {\em Comm. Math. Phys.}, 349(1):165--269, 2017.

\bibitem[GP20]{GubPerk20}
Massimiliano Gubinelli and Nicolas Perkowski.
\newblock The infinitesimal generator of the stochastic {B}urgers equation.
\newblock {\em Probab. Theory Related Fields}, 178(3-4):1067--1124, 2020.

\bibitem[GRASS25]{GRASS-23}
Sean Groathouse, Firas Rassoul-Agha, Timo Sepp\"{a}l\"{a}inen, and Evan Sorensen.
\newblock Jointly invariant measures for the {K}ardar--{P}arisi--{Z}hang equation.
\newblock {\em Probab. Theory Related Fields}, 192(1-2):303--372, 2025.

\bibitem[Hai13]{HairerKPZ}
Martin Hairer.
\newblock Solving the {KPZ} equation.
\newblock {\em Ann. of Math. (2)}, 178(2):559--664, 2013.

\bibitem[HB76]{Hsu-Burke-1976}
J.~Hsu and P.~Burke.
\newblock Behavior of tandem buffers with geometric input and markovian output.
\newblock {\em IEEE Transactions on Communications}, 24(3):358--361, 1976.

\bibitem[HL22]{HL22}
Yaozhong Hu and Khoa L\^e.
\newblock Asymptotics of the density of parabolic {A}nderson random fields.
\newblock {\em Ann. Inst. Henri Poincar\'e{} Probab. Stat.}, 58(1):105--133, 2022.

\bibitem[HM18]{Hairer-Mattingly-2018}
M.~Hairer and J.~Mattingly.
\newblock The strong {F}eller property for singular stochastic {PDE}s.
\newblock {\em Ann. Inst. Henri Poincar\'e{} Probab. Stat.}, 54(3):1314--1340, 2018.

\bibitem[HW90]{harrison1990}
J.~M. Harrison and R.~J. Williams.
\newblock On the quasireversibility of a multiclass {B}rownian service station.
\newblock {\em Ann. Probab.}, 18(3):1249--1268, 1990.

\bibitem[IK03]{Iturriaga-Khanin-2003}
R.~Iturriaga and K.~Khanin.
\newblock Burgers turbulence and random {L}agrangian systems.
\newblock {\em Comm. Math. Phys.}, 232(3):377--428, 2003.

\bibitem[JRA20]{Janjigian-Rassoul-2020b}
Christopher Janjigian and Firas Rassoul-Agha.
\newblock Busemann functions and {G}ibbs measures in directed polymer models on {$\mathbb Z^2$}.
\newblock {\em Ann. Probab.}, 48(2):778--816, 2020.

\bibitem[JRAS22]{janjigian2022ergodicitysynchronizationkardarparisizhangequation}
Christopher Janjigian, Firas Rassoul-Agha, and Timo Seppäläinen.
\newblock {Ergodicity and synchronization of the Kardar-Parisi-Zhang equation}.
\newblock {\em Preprint:arXiv:2211.06779}, 2022.

\bibitem[Kel11]{Kelly-2011}
F.~P. Kelly.
\newblock {\em Reversibility and stochastic networks}.
\newblock Cambridge Mathematical Library. Cambridge University Press, Cambridge, revised edition, 2011.

\bibitem[KLO12]{TKCLSO12}
Tomasz Komorowski, Claudio Landim, and Stefano Olla.
\newblock {\em Fluctuations in {M}arkov processes}, volume 345 of {\em Grundlehren der mathematischen Wissenschaften [Fundamental Principles of Mathematical Sciences]}.
\newblock Springer, Heidelberg, 2012.
\newblock Time symmetry and martingale approximation.

\bibitem[KOS24]{Kuniba-Okado-Scrimshaw-24}
Atsuo {Kuniba}, Masato {Okado}, and Travis {Scrimshaw}.
\newblock {A strange five vertex model and multispecies ASEP on a ring}.
\newblock {\em Preprint:arXiv:2408.12092}, 2024.

\bibitem[KPZ86]{KPZ}
Mehran Kardar, Giorgio Parisi, and Yi-Cheng Zhang.
\newblock Dynamic scaling of growing interfaces.
\newblock {\em Phys. Rev. Lett.}, 56:889--892, 1986.

\bibitem[KS88]{Karatzas_Shreve}
Ioannis Karatzas and Steven~E. Shreve.
\newblock {\em Brownian motion and stochastic calculus}, volume 113 of {\em Graduate Texts in Mathematics}.
\newblock Springer-Verlag, New York, 1988.

\bibitem[LG16]{LeGall-book}
Jean-Fran\c~cois Le~Gall.
\newblock {\em Brownian motion, martingales, and stochastic calculus}, volume 274 of {\em Graduate Texts in Mathematics}.
\newblock Springer, [Cham], french edition, 2016.

\bibitem[Lin20]{Lin2019}
Yier Lin.
\newblock K{PZ} equation limit of stochastic higher spin six vertex model.
\newblock {\em Math. Phys. Anal. Geom.}, 23(1):Paper No. 1, 118, 2020.

\bibitem[Lin23]{10.1214/23-EJP1022}
Yier Lin.
\newblock Classification of stationary distributions for the stochastic vertex models.
\newblock {\em Electron. J. Probab.}, 28:Paper No. 122, 40, 2023.

\bibitem[Loz03]{NIST-2003}
Daniel~W. Lozier.
\newblock N{IST} {D}igital {L}ibrary of {M}athematical {F}unctions.
\newblock {\em Ann. Math. Artif. Intell.}, 38(1-3):105--119, 2003.
\newblock Mathematical knowledge management.

\bibitem[Mar20]{Martin-2020}
James~B. Martin.
\newblock Stationary distributions of the multi-type {ASEP}.
\newblock {\em Electron. J. Probab.}, 25:Paper No. 43, 41, 2020.

\bibitem[MFQR]{MFQR-note}
Gregorio Moreno~Flores, Jeremy Quastel, and Daniel Remenik.
\newblock Unpublished preprint.

\bibitem[Mue91]{Mueller91}
Carl Mueller.
\newblock On the support of solutions to the heat equation with noise.
\newblock {\em Stochastics Stochastics Rep.}, 37(4):225--245, 1991.

\bibitem[Nic21]{Nica-2021}
Mihai Nica.
\newblock Intermediate disorder limits for multi-layer semi-discrete directed polymers.
\newblock {\em Electron. J. Probab.}, 26:Paper No. 62, 50, 2021.

\bibitem[Nua06]{Nua06}
David Nualart.
\newblock {\em The {M}alliavin calculus and related topics}.
\newblock Probability and its Applications (New York). Springer-Verlag, Berlin, second edition, 2006.

\bibitem[NY04]{Noumi-Yamada-2004}
Masatoshi Noumi and Yasuhiko Yamada.
\newblock Tropical robinson-schensted-knuth correspondence and birational weyl group actions.
\newblock {\em Adv. Stud. Pure Math.}, pages 371--442, 2004.

\bibitem[O'C03]{O'Connell-2003}
Neil O'Connell.
\newblock A path-transformation for random walks and the {R}obinson-{S}chensted correspondence.
\newblock {\em Trans. Amer. Math. Soc.}, 355(9):3669--3697, 2003.

\bibitem[O'C12]{O'Connell-2012}
Neil O'Connell.
\newblock Directed polymers and the quantum {T}oda lattice.
\newblock {\em Ann. Probab.}, 40(2):437--458, 2012.

\bibitem[OSZ14]{O’Connell2014}
Neil O'Connell, Timo Sepp\"al\"ainen, and Nikos Zygouras.
\newblock Geometric {RSK} correspondence, {W}hittaker functions and symmetrized random polymers.
\newblock {\em Invent. Math.}, 197(2):361--416, 2014.

\bibitem[OW16]{O’Connell2016}
Neil O'Connell and Jon Warren.
\newblock A multi-layer extension of the stochastic heat equation.
\newblock {\em Comm. Math. Phys.}, 341(1):1--33, 2016.

\bibitem[OY01]{brownian_queues}
Neil O'Connell and Marc Yor.
\newblock Brownian analogues of {B}urke's theorem.
\newblock {\em Stochastic Process. Appl.}, 96(2):285--304, 2001.

\bibitem[OY02]{rep_non_colliding}
Neil O'Connell and Marc Yor.
\newblock A representation for non-colliding random walks.
\newblock {\em Electron. Comm. Probab.}, 7:1--12, 2002.

\bibitem[Par22]{10.1214/22-EJP775}
Shalin Parekh.
\newblock Positive random walks and an identity for half-space {SPDE}s.
\newblock {\em Electron. J. Probab.}, 27:Paper No. 45, 47, 2022.

\bibitem[Par23]{PAREKH2023351}
Shalin Parekh.
\newblock Convergence of {ASEP} to {KPZ} with basic coupling of the dynamics.
\newblock {\em Stochastic Process. Appl.}, 160:351--370, 2023.

\bibitem[Par25]{parekh2023ergodicityresultsopenkpz}
Shalin Parekh.
\newblock Ergodicity results for the open {KPZ} equation.
\newblock {\em Ann. Inst. Henri Poincar\'{e} Probab. Stat.}, 61(3):--, 2025.

\bibitem[PEM09]{Prolhac-Evans-Mallick-2009}
S.~Prolhac, M.~R. Evans, and K.~Mallick.
\newblock The matrix product solution of the multispecies partially asymmetric exclusion process.
\newblock {\em J. Phys. A}, 42(16):165004, 25, 2009.

\bibitem[Pit75]{Pitman1975}
J.~W. Pitman.
\newblock One-dimensional {B}rownian motion and the three-dimensional {B}essel process.
\newblock {\em Advances in Appl. Probability}, 7(3):511--526, 1975.

\bibitem[{R C}22]{R}
{R Core Team}.
\newblock {\em R: A Language and Environment for Statistical Computing}.
\newblock R Foundation for Statistical Computing, Vienna, Austria, 2022.

\bibitem[Ros22]{TR22}
Tommaso~C. Rosati.
\newblock Synchronization for {KPZ}.
\newblock {\em Stoch. Dyn.}, 22(4):Paper No. 2250010, 46, 2022.

\bibitem[Sep12]{Seppalainen-2012}
Timo Sepp\"{a}l\"{a}inen.
\newblock Scaling for a one-dimensional directed polymer with boundary conditions.
\newblock {\em Ann. Probab.}, 40(1):19--73, 2012.

\bibitem[Sin91]{Sinai1991}
Ya.~G. Sina\u{\i}.
\newblock Two results concerning asymptotic behavior of solutions of the {B}urgers equation with force.
\newblock {\em J. Statist. Phys.}, 64(1-2):1--12, 1991.

\bibitem[Sin96]{Sinai1996}
Ya.~G. Sinai.
\newblock {\em Burgers system driven by a periodic stochastic flow}, pages 347--353.
\newblock Springer, Tokyo, 1996.

\bibitem[Sor23]{Sorensen-thesis}
Evan Sorensen.
\newblock The stationary horizon as the central multi-type invariant measure in the {KPZ} universality class.
\newblock {\em PhD thesis, University of Wisconisn--Madison}, 2023.
\newblock Available at \url{https://arxiv.org/abs/2306.09584}.

\bibitem[Sri98]{Srivastava-1998}
S.~M. Srivastava.
\newblock {\em A course on {B}orel sets}, volume 180 of {\em Graduate Texts in Mathematics}.
\newblock Springer-Verlag, New York, 1998.

\bibitem[SS23]{Seppalainen-Sorensen-21b}
Timo Sepp\"{a}l\"{a}inen and Evan Sorensen.
\newblock Global structure of semi-infinite geodesics and competition interfaces in {B}rownian last-passage percolation.
\newblock {\em Probab. Math. Phys.}, 4(3):667--760, 2023.

\bibitem[Wal86]{walsh1986introduction}
John~B. Walsh.
\newblock An introduction to stochastic partial differential equations.
\newblock {\em \'Ecole d'\'et\'e{} de probabilit\'es de {S}aint-{F}lour, {XIV}---1984}, 1180:265--439, 1986.

\bibitem[Wu20]{Wu2020}
Xuan Wu.
\newblock Intermediate disorder regime for half-space directed polymers.
\newblock {\em J. Stat. Phys.}, 181(6):2372--2403, 2020.

\end{thebibliography}
\end{document}